\documentclass[11pt,a4paper,titlepage,twoside
]{book}

\usepackage[utf8x]{inputenc}
\usepackage[T1]{fontenc}
\usepackage{kpfonts}

\usepackage[a4paper,includeheadfoot,pdftex,textwidth=16cm,textheight=24cm,
bottom=3.6cm]{geometry}
\usepackage[svgnames]{xcolor}%https://www.latextemplates.com/svgnames-colors
\usepackage{graphicx}

\usepackage[bookmarks=true,
pdfborder={0 0 1},%citebordercolor=violet, 
colorlinks=true,urlcolor=blue,citecolor=Purple, 
linkcolor=NavyBlue,hypertexnames=false]{hyperref}

\usepackage{enumitem}
\setlist{parsep=0pt} % decrease enumerate/itemize separation
\setlist[itemize,enumerate]{nolistsep,itemsep=1pt,topsep=1pt} 
\setlist{leftmargin=5mm}

\usepackage{fancybox}
\usepackage[Lenny]{fncychap}
\usepackage{fancyhdr}
\usepackage{amsmath}
\usepackage{amsfonts}
\usepackage{amssymb}
\usepackage{amsthm}
\usepackage{ upgreek }

\usepackage{bbm}

\usepackage{mathtools}%,xparse}
\usepackage{thmtools}

\usepackage{tikz}
\usetikzlibrary{matrix,arrows,calc}
\usetikzlibrary{decorations.pathmorphing}
\usetikzlibrary{decorations.markings}
\usetikzlibrary{cd}
\usepgflibrary{fpu}

\usepackage[margin=10pt,font=small,labelfont=bf,
labelsep=endash]{caption}

\theoremstyle{plain}

\newcommand{\myrulewidth}{0pt}
\definecolor{ThmColor}{rgb}{0.85,0.85,0.99}
\definecolor{DefColor}{rgb}{0.81,0.92,0.97}
\definecolor{RemColor}{rgb}{0.92,0.85,0.92}
\definecolor{ExoColor}{rgb}{0.81,0.99,0.81}
\declaretheorem[name=Theorem, numberwithin=section,
shaded={rulecolor=Black,rulewidth=\myrulewidth,bgcolor=ThmColor}]{theorem}
\declaretheorem[name=Proposition, sibling=theorem,
shaded={rulecolor=Black,rulewidth=\myrulewidth,bgcolor=ThmColor}]{proposition}
\declaretheorem[name=Corollary, sibling=theorem,
shaded={rulecolor=Black,rulewidth=\myrulewidth,bgcolor=ThmColor}]{corollary}
\declaretheorem[name=Lemma, sibling=theorem,
shaded={rulecolor=Black,rulewidth=\myrulewidth,bgcolor=ThmColor}]{lemma}

\declaretheorem[name=Definition, sibling=theorem,
shaded={rulecolor=Black,rulewidth=\myrulewidth,bgcolor=DefColor}]{definition}

\declaretheoremstyle[
    headfont=\bfseries,
    notefont=\bfseries,
    headpunct={.},
    numbered=yes
]{customrem}

\theoremstyle{customrem}
\declaretheorem[name=Example, sibling=theorem, qed={$\diamond$},
shaded={rulecolor=Black,rulewidth=\myrulewidth,bgcolor=RemColor}]{example}

\declaretheorem[name=Remark, sibling=theorem, qed={$\diamond$}, 
shaded={rulecolor=Black,rulewidth=\myrulewidth,bgcolor=RemColor}]{remark}
\declaretheorem[name=Exercise, sibling=theorem, qed={$\clubsuit$},
shaded={rulecolor=Black,rulewidth=\myrulewidth,bgcolor=ExoColor}]{exercise}

\usepackage{cleveref}
\crefname{exercise}{exercise}{exercises}
\usepackage{autonum} %numerotation intelligente

\input{sarajevo.sty}
\input{diagrams_tikz.sty}
\input{diagrams_phi4.sty}
\input{regstruct_new.sty}

\begin{document}

\pagestyle{empty}
\newgeometry{margin=1in}

\hypersetup{pageanchor=false}

\thispagestyle{empty}

\vspace*{1cm}
\begin{center}

{\Huge\bfseries\scshape
An introduction to \\[1mm]
singular stochastic PDEs: \\[1mm]  
Allen--Cahn equations, metastability \\[1mm]
and regularity structures \\[1mm]
}%\\[8mm]

\vspace*{12mm}
{\large Nils Berglund}\\[2mm]
{\large Institut Denis Poisson -- UMR 7013}\\[2mm]
{\large Universit\'e d'Orl\'eans, Universit\'e de Tours, CNRS}

\vspace*{12mm}
{\Large Lecture notes\\[4mm]
Sarajevo Stochastic Analysis Winter School 2019}\\
\vspace*{12mm}

\begin{figure}[h]
\begin{center}
\includegraphics[width=9cm]{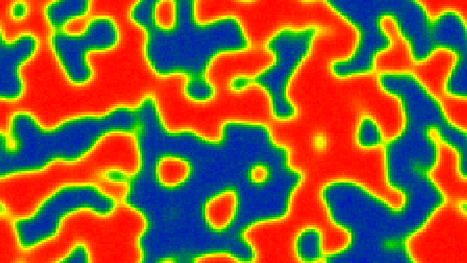}
\end{center}
\end{figure}

\vspace*{27mm}
--- Version of September 12, 2019 ---\\[2mm]
%--- Version of \today\ ---\\[2mm]

\end{center}
\hypersetup{pageanchor=true}

\cleardoublepage

\pagestyle{fancy}
\fancyhead[RO,LE]{\thepage}
\fancyhead[LO]{\nouppercase{\rightmark}}
\fancyhead[RE]{\nouppercase{\leftmark}}
\cfoot{}
\setcounter{page}{1}
\pagenumbering{roman}
\restoregeometry

\tableofcontents

\cleardoublepage

%%%%%%%%%%%%%%%%%%%%%%%%%%%%%%%%%%%%%%%%%%%%%%%%%%%%%%%%%%%%%%%%%%%%%%%%%%%%%%%%

\chapter*{Preface}

These notes have been prepared for a series of lectures given at the 
Sarajevo Stochastic Analysis Winter School, from January 28 to February 1, 2019. 
There already exist several excellent lecture notes and reviews on the subject, 
such as~\cite{Hairer_LN_2009} on (non-singular) stochastic PDEs, 
and~\cite{Hairer_LN_2015,Chandra_Weber_LN17} on singular stochastic PDEs and 
regularity structures. The present notes have two main specificities. The first 
one is that they focus on a particular example, the Allen--Cahn equation, which 
allows to introduce several of the difficulties of the theory in a gradual way, 
by increasing the space dimension step by step. The hope is that while this 
limits the generality of the theory presented, this limitation is more that made 
up by a gain in clarity. The second specific aspect of these notes is that they 
go beyond existence and uniqueness of solutions, by covering a few recent 
results on convergence to equilibrium and metastability in these system. 

The author wishes to thank the organisers of the Winter School, Frank Proske, 
Abdol-Reza Mansouri and Oussama Amine, for the invitation that provided the 
motivation to compile these notes; his co-workers on SPDEs, Ajay Chandra, 
Giacomo Di Ges\`u, Barbara Gentz, Christian Kuehn and Hendrik Weber, who 
greatly contributed to some results presented here; Yvain Bruned for remarks on 
an early version of the notes; and Martin Hairer and Lorenzo Zambotti for 
patiently explaining some subtleties of the theory of regularity structures. 

Thanks are also due to the Isaac Newton Institute in Cambridge, as well as  
David Brydges, Alessandro Giuliani, Massimiliano Gubinelli, Antti Kupiainen, 
Hendrik Weber and Lorenzo Zambotti for organising the semester \lq\lq Scaling 
limits, rough paths, quantum field theory\rq\rq. In particular the tutorial 
lectures given during the opening workshop of the semester proved of great help 
in preparing these notes. 

Finally, thanks are due to Adrian Martini, for pointing out some typos in the 
original version of these notes. 

\vfill

\cleardoublepage
%\newpage
\setcounter{page}{1}
\pagenumbering{arabic}

%%%%%%%%%%%%%%%%%%%%%%%%%%%%%%%%%%%%%%%%%%%%%%%%%%%%%%%%%%%%%%%%%%%%%%%%%%%%%%%%

\chapter{A system of interacting diffusions}
\label{ch:diff} 

In this chapter, we will analyse the dynamics of a system of $N\geqs2$ coupled 
stochastic differential equations (SDEs), which will converge, in a certain 
sense, to a stochastic Allen--Cahn PDE. This will serve two purposes: firstly, 
it will introduce a natural physical model that motivates the use of stochastic 
PDEs, and secondly, it will allow us to recapitulate some notions from the 
theory of SDEs that will be important in the infinite-dimensional case. 

The system of interacting diffusions is given by 
\begin{equation}
\label{eq:SDE} 
 \6y^i_t = \bigbrak{y^i_t - (y^i_t)^3} \6t 
 + \frac\gamma2 \bigbrak{y^{i+1}_t - 2y^i_t + y^{i-1}_t} \6t 
 + \sqrt{2\eps} \6W^i_t\;, 
 \qquad
 i = 1,\dots, N\;, 
\end{equation} 
in the following setting.
\begin{itemize}
\item 	The real variables $y^1,\dots,y^N\in\R$ can be thought of as the 
positions of $N$ particles. We use periodic boundary conditions, that is, 
$y^0=y^N$ and $y^{N+1}=y^1$. This can also be indicated by considering $i$ as an 
element of the cyclic group $\Lambda=\Z/N\Z$.

\item 	The term $\brak{y^i_t - (y^i_t)^3} \6t$ tends to push each particle 
towards $+1$ if $y^i_t>0$ and towards $-1$ if $y^i_t<0$, that is, we have a 
bistable local dynamics. 

\item 	The term $\smash{\frac\gamma2} \brak{y^{i+1}_t - 2y^i_t + y^{i-1}_t} 
\6t$ is a discretised Laplacian interaction, which describes a ferromagnetic 
nearest-neighbour coupling pushing the $y^i_t$ towards each other (if 
$\gamma>0$). 

\item 	The $W^i_t$ are independent standard Wiener processes on a probability 
space $(\Omega,\cF,\fP)$, modelling thermal noise. 
\end{itemize}

The system~\eqref{eq:SDE} can be thought of as a version of the Ising model 
with continuous spins. We are mainly interested in its long-time dynamics for 
large values of $N$ and $\gamma$. The parameter $\eps\geqs0$ will be either 
small or of order $1$. 

%%%%%%%%%%%%%%%%%%%%%%%%%%%%%%%%%%%%%%%%%%%%%%%%%%%%%%%%%%%%%%%%%%%%%%%%%%%%%%%%

\section{Deterministic dynamics}
\label{sec:diffdet} 

In the deterministic case $\eps=0$, the system~\eqref{eq:SDE} reduces to the 
system of ordinary differential equations (ODEs) 
\begin{equation}
\label{eq:ODE} 
 \dot y^i = y^i - (y^i)^3 
 + \frac\gamma2 \bigbrak{y^{i+1} - 2y^i + y^{i-1}}\;.
\end{equation} 
A first observation is that~\eqref{eq:ODE} has gradient form 
\begin{equation}
 \dot y = -\nabla V(y)\;,
\end{equation} 
where $V$ is the quartic potential given by 
\begin{equation}
\label{eq:V_finite} 
 V(y) := \sum_{i\in\Lambda} U(y^i) + \frac{\gamma}{4} \sum_{i\in\Lambda} 
(y^{i+1}-y^i)^2\;, 
 \qquad
 U(\xi) = \frac14 \bigpar{\xi^2 - 1}^2\;.
\end{equation} 
The stationary points of~\eqref{eq:ODE} are thus the critical points of $V$. 
Note that $V$ is invariant under the group $G$ generated by cyclic permutations 
of the $y^i$, the reflection $y^i\mapsto y^{N+1-i}$, and the point symmetry 
$y\mapsto -y$. 

The simplest situation arises for $\gamma=0$: then the set of all critical  
points of $V$ is $\set{-1,0,1}^\Lambda$ and has cardinality $3^N$. The  critical 
points in $\set{-1,1}^\Lambda$ are all local minima of $V$, and thus describe 
$2^N$ stable stationary points of~\eqref{eq:ODE}. In fact, one can show 
\cite[Proposition~2.1]{BFG06a} that this situation perturbs to small positive 
$\gamma$, in the sense that there are still $3^N$ stationary points, $2^N$ of 
which are stable, for all $\gamma$ up to a critical value $\gamma_0(N)$ which is 
larger than $\frac14$ for all $N$. 

We are more interested, however, in the case of large $\gamma$ of the order 
$N^2$. Indeed, this is the natural scaling in which the discrete Laplacian 
in~\eqref{eq:ODE} approaches the continuous Laplacian. 

\begin{proposition}%[{\cite[Proposition~2.2]{BFG06a}}]
\label{prop:det_potential} 
Let 
\begin{equation}
\label{eq:gamma1} 
 \gamma_1(N) := \frac{1}{2\sin^2(\pi/N)}
 = \frac{N^2}{2\pi^2} \Bigbrak{1+\BigOrder{\frac{1}{N^2}}}\;.
\end{equation} 
Then the potential $V$ admits exactly $3$ stationary points $0$ and 
$\pm(1,\dots,1)$ if and only if $\gamma\geqs\gamma_1(N)$. 
The points $\pm(1,\dots,1)$ are always local minima of $V$, while $0$ is a 
saddle with exactly one linearly unstable direction if and only if 
$\gamma>\gamma_1(N)$.
\end{proposition}

\begin{proof}
First note that the discrete Laplacian is represented by a Toeplitz matrix (its 
entries are constant on diagonals), whose eigenvalues can be computed 
explicitly by discrete Fourier transform. They are given by $-\lambda^N_k$ 
where 
\begin{equation}
 \lambda^N_k = \lambda^N_{N-k} := 2\sin^2\biggpar{\frac{k\pi}{N}}\;, 
 \qquad 
 k = 0,\dots, N-1\;.
\end{equation} 
In particular, $\lambda^N_0=0$ and 
$\gamma_1(N) = 1/(\lambda^N_1-\lambda^N_0)$ is the inverse of the spectral 
gap. Consider now the function 
\begin{equation}
 W(y) := \frac12 \sum_{i\in\Lambda} \bigpar{y^{i+1} - y^i}^2\;.
\end{equation} 
Note that $W$ vanishes on the diagonal $\set{y^1=\dots=y^N}$ and is otherwise 
positive. An easy computation (cf.~\cite[Proposition~3.1]{BFG06a}) shows that 
\begin{equation}
 \dtot{}{t} W(y_t) \leqs 2\Bigpar{1-\frac{\gamma}{\gamma_1(N)}} W(y_t)
 - \frac1N W(y_t)^2\;.
\end{equation} 
Thus when $\gamma\geqs\gamma_1(N)$, $t\mapsto W(y_t)$ decreases for $y_t$ 
outside the diagonal, i.e., $W$ is a Lyapunov function. Therefore, when  
$\gamma\geqs\gamma_1(N)$ all orbits of~\eqref{eq:ODE} will converge to the 
diagonal. The only stationary points on the diagonal are $0$ and 
$\pm(1,\dots,1)$. The eigenvalues of the linearisation of~\eqref{eq:ODE} at $0$ 
are given by $-\mu^N_k$ where 
\begin{equation}
\label{eq:muNk} 
 \mu^N_k := -1+\gamma\lambda^N_k\;, 
 \qquad 
 k = 0, \dots, N-1\;.
\end{equation} 
If $\gamma>\gamma_1(N)$, then only $-\mu^N_0$ is positive, while for 
$\gamma<\gamma_1(N)$, there are at least three positive eigenvalues (two if 
$N=2$). Since $V(y)$ is bounded below and goes to infinity as 
$\norm{y}\to\infty$, there must be critical points outside the diagonal in 
the latter case. Finally, the eigenvalues of the linearisation at 
$\pm(1,\dots,1)$ are given by $-\nu^N_k$ where 
\begin{equation}
\label{eq:nuNk} 
 \nu^N_k := 2+\gamma\lambda^N_k\;, 
 \qquad 
 k = 0, \dots, N-1\;.
\end{equation} 
These are always negative, showing that these points are always 
stable (they are local minima of $V$). 
\end{proof}

\begin{exercise}
\label{exo:offdiagonal}
Show that the potential satisfies the lower bound
\begin{equation}
 \label{eq:offdiagonal}
 V(y_0 + y_\perp) \geqs V(y_0) + 
\frac12\Bigpar{\frac{\gamma}{\gamma_1(N)}-1}\,\norm{y_\perp}^2
\end{equation} 
whenever $y_0$ belongs to the diagonal $\set{y^1=\dots=y^N}$ and 
$y_\perp$ is orthogonal to the diagonal.
\end{exercise}

This result shows that if $\gamma > \gamma_1(N)$, then $V$ is a double-well 
potential, with a saddle having one unstable direction located at the origin, 
and the two minima located at $\pm(1,\dots,1)$. The unstable manifolds of the 
saddle are subsets of the diagonal, and connect $0$ to the two local minima. 
This is the case we will mainly be interested in in what follows. 

\begin{figure}[htb]
\begin{center}
\includegraphics[width=0.8\textwidth]{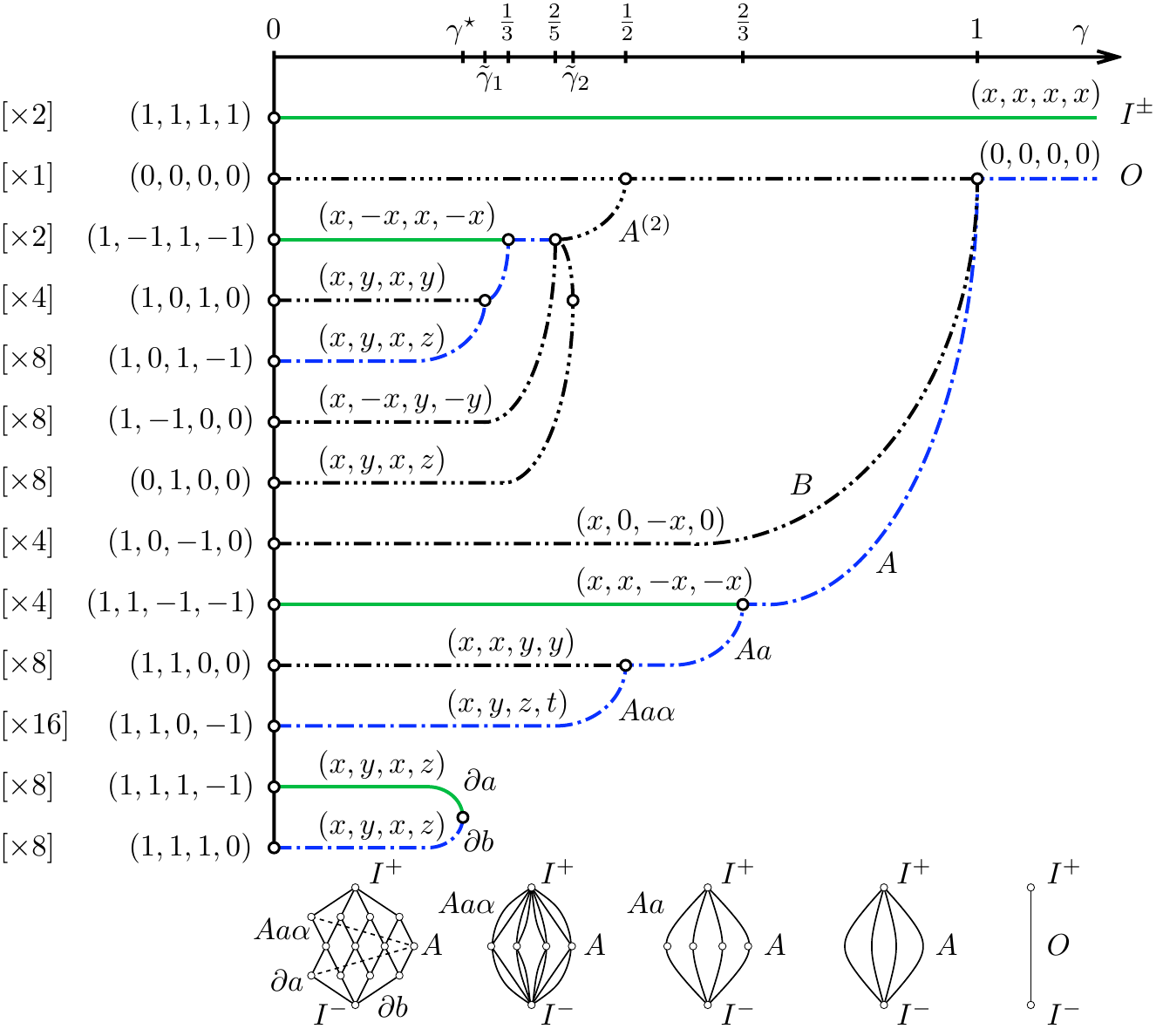}
\end{center}
\vspace{-1mm}
 \caption[]{Schematic bifurcation diagram for~\eqref{eq:ODE} in the case $N=4$. 
Each branch in this diagram represents a family of stationary points globally 
invariant under the symmetry group $G$, with only one representative of the 
family displayed on the left side of the diagram. Green full lines represent 
stable stationary points (local minima of the potential $V$), while broken lines 
with $k$ dots represent stationary points with $k$ unstable directions (saddles 
of Morse index $k$ of $V$).}
 \label{fig_bifurcation}
\end{figure}

As $\gamma$ decreases below $\gamma_1(N)$, the situation becomes much more 
involved, and is not understood in full detail. What is known is that at 
$\gamma_1(N)$, the system undergoes a $G$-symmetric analogue of a pitchfork 
bifurcation, at which the origin becomes unstable in more than one direction, 
while expelling a multiple of $N$ new stationary points with lower symmetry. As 
$\gamma$ decreases further, these points undergo further symmetry-breaking 
bifurcations, as illustrated by \Cref{fig_bifurcation} in the case $N=4$. 
However, all these further bifurcation occur for $\gamma=\order{N^2}$. In other 
words, for any $\tilde\gamma>0$, the number of stationary points is at most of 
order $N$ when $\gamma > \tilde\gamma N^2$ and $N$ is sufficiently large. 

\begin{exercise}
In the case $N=2$, find all stationary points of the potential $V$ and 
determine their stability as a function of $\gamma\geqs0$. 
\Hint Express the potential $V$ in variables $z_\pm = y_1\pm y_2$.
\end{exercise}

%%%%%%%%%%%%%%%%%%%%%%%%%%%%%%%%%%%%%%%%%%%%%%%%%%%%%%%%%%%%%%%%%%%%%%%%%%%%%%%%

\section{Existence and uniqueness of solutions}
\label{sec:diffex} 

We now return to the SDE~\eqref{eq:SDE} for $\eps>0$. We can write it in 
compact form as 
\begin{equation}
\label{eq:SDE_V} 
 \6y_t = -\nabla V(y_t) \6t + \sqrt{2\eps} \6W_t\;.
\end{equation} 
We first note that $y\mapsto\nabla V(y)$ is continuous (and therefore 
measurable) and locally (but not globally) Lipschitz. We can thus apply a 
standard result, such as~\cite[Theorem~5.2.1]{Oeksendal} to the process 
$y_{t\wedge\tau}$ where $\tau$ is the first-exit time from a large ball to 
obtain local existence and uniqueness of a strong solution to~\eqref{eq:SDE_V}.
In particular, local existence follows from the fact that the map 
$y^{(k)}\mapsto y^{(k+1)}$ defined by 
\begin{equation}
 y_t^{(k+1)} = y_0 - \int_0^{t\wedge\tau} \nabla V(y^{(k)}_s) \6t + 
\sqrt{2\eps} W_t
\end{equation} 
is a contraction, via Banach's fixed point theorem. 

To extend this result to global existence of a strong solution, we have to rule 
out finite-time explosion of solutions. This is, however, not hard to achieve, 
since the quartic growth of the potential prevents sample paths from straying 
very far from the origin (coercivity). Here are two possible ways of obtaining 
global 
existence:

\begin{enumerate}
\item 	One checks that the drift and diffusion coefficients 
in~\eqref{eq:SDE_V} satisfy the one-sided growth condition 
\begin{equation}
\label{eq:growth} 
 y \cdot \bigpar{-\nabla V(y)} + \frac12 \bigpar{\sqrt{2\eps}}^2 \leqs 
K(1+\norm{y}^2)
 \qquad \forall y\in\R^\Lambda
\end{equation} 
for a constant $K>0$. See for instance~\cite[Theorem~3.5]{Mao_book}. 

\item 	Let 
\begin{equation}
\label{eq:generator} 
 \cL := \eps\Delta - \nabla V \cdot \nabla
\end{equation} 
denote the infinitesimal generator of the diffusion~\eqref{eq:SDE_V}. Then 
we can take advantage of the fact that $V$ is itself a Lyapunov function 
for the system, in the sense that there exists a constant $c\geqs0$ such that  
\begin{equation}
\label{eq:Lyap-CD0} 
 (\cL V)(y) \leqs c V(y) 
 \qquad \forall y\in\R^\Lambda\;.
\end{equation} 
This implies that the process is non-explosive according to 
\cite[Theorem~2.1]{Meyn_Tweedie_1993b}. 
\end{enumerate}

\begin{exercise}
\label{ex:Lyap} 
Check that the SDE~\eqref{eq:SDE} satisfies 
the one-sided growth condition~\eqref{eq:growth} as well as the Lyapunov  
condition~\eqref{eq:Lyap-CD0}. 
Show that, in addition, there exists $R<\infty$ such that $\cL V(y)$ is 
negative outside the ball $\set{\norm{y}\leqs R}$. Can we do better than that?
\end{exercise}

The main result of this section is thus as follows.

\begin{proposition}
For any initial condition $y_0\in\R^\Lambda$, the SDE~\eqref{eq:SDE} admits an 
almost surely continuous, pathwise unique strong solution $(y_t)_{t\geqs0}$ 
which is global in time.
\end{proposition}

%%%%%%%%%%%%%%%%%%%%%%%%%%%%%%%%%%%%%%%%%%%%%%%%%%%%%%%%%%%%%%%%%%%%%%%%%%%%%%%%

\section{Invariant measure and reversibility}
\label{sec:diffinv} 

In order to determine an invariant probability measure for the 
diffusion~\eqref{eq:SDE_V}, it is useful to rewrite the infinitesimal 
generator~\eqref{eq:generator} in the form
\begin{equation}
 \label{eq:generator_rev}
 \cL = \eps \e^{V/\eps} \nabla \cdot \e^{-V/\eps} \nabla\;,
\end{equation} 
with the usual convention that any differential operator acts on everything on 
its right. This is indeed equivalent to~\eqref{eq:generator} by Leibniz' rule. 

\begin{proposition}
Assume $\eps>0$ and the potential $V$ is such that 
\begin{equation}
\label{eq:defZ} 
 \cZ := \int_{\R^\Lambda} \e^{-V(y)/\eps} \6y < \infty\;.
\end{equation}
Then the diffusion~\eqref{eq:SDE_V} admits a unique invariant probability 
measure, with density 
\begin{equation}
\label{eq:pi_finite_dim} 
 \pi(y) := \frac1\cZ \e^{-V(y)/\eps}
\end{equation} 
with respect to Lebesgue measure $\6y$, known as a \emph{Gibbs measure}. 
\end{proposition}
\begin{proof}
Let $f$ be a bounded measurable test function in the domain of $\cL$. Writing 
$\cL^\dagger$ for the adjoint of $\cL$, we have  
\begin{align}
\pscal{\cL^\dagger\pi}{f}_{L^2} 
&= \pscal{\pi}{\cL f}_{L^2} \\
&= \int_{\R^\Lambda} \pi(y)\eps \e^{V(y)/\eps} \nabla \cdot 
\bigpar{\e^{-V(y)/\eps}\nabla f(y)}\6y \\
&= \frac{\eps}{\cZ} \int_{\R^\Lambda} \nabla \cdot 
\bigpar{\e^{-V(y)/\eps}\nabla f(y)}\6y \\
&= 0 
\end{align}
by the divergence theorem. 
Thus $\cL^\dagger\pi = 0$, and Kolmogorov's backward equation (or 
Fokker--Planck equation) shows that $\pi$ is invariant.  

Uniqueness of $\pi$ follows from the fact that the diffusion is irreducible 
(with respect to Lebesgue measure) by uniform ellipticity of the noise. One way 
of seeing this is to apply Theorem~3.3 in~\cite{Meyn_Tweedie_1993b} to show 
that the process is Harris recurrent (it almost surely hits any set of positive 
measure), owing to the improved estimate on $\cL V$ obtained in \Cref{ex:Lyap}.
\end{proof}

\begin{exercise}
Use integration by parts to compute an explicit expression for $\cL^\dagger$, 
and use it to verify that $\cL^\dagger\pi=0$. 
\end{exercise}

Another consequence of the specific form~\eqref{eq:generator_rev} of the 
generator is that the diffusion is \emph{reversible}. 

\begin{proposition}
The infinitesimal generator $\cL$ is self-adjoint with respect to the weighted 
inner product 
\begin{equation}
\label{eq:pscal_pi} 
 \pscal{f}{g}_\pi 
 := \int_{\R^\Lambda} \e^{-V(y)/\eps} \overline{f(y)} g(y) \6y\;.
\end{equation} 
As a consequence, the transition probability density $p_t$ of the diffusion 
satisfies the detailed balance condition 
\begin{equation}
\label{eq:det_balance} 
 \pi(y) p_t(y,z) = \pi(z) p_t(z,y) 
\end{equation} 
for all $y,z\in\R^\Lambda$ and all $t>0$. 
\end{proposition}
\begin{proof}
The divergence theorem yields 
\begin{equation}
 \pscal{f}{\cL g}_\pi 
= \eps \int_{\R^\Lambda} \overline{f(y)} \nabla \cdot \bigbrak{\e^{-V(y)/\eps} 
\nabla g(y)} \6y 
= - \eps \int \overline{\nabla f(y)} \cdot \nabla g(y) \e^{-V(y)/\eps} \6y\;.
\end{equation} 
Since $\overline f$ and $g$ play a symmetric role in this expression, it is 
equal to $\pscal{\cL f}{g}_\pi$. It follows that the Markov semigroup $P_t = 
\e^{t\cL}$ is also self-adjoint, that is, 
\begin{equation}
 \pscal{f}{P_t g}_\pi 
 = \pscal{P_t f}{g}_\pi\;.
\end{equation} 
Taking for $f$ and $g$ functions converging to Dirac distributions 
$\delta_y$ and $\delta_z$ yields~\eqref{eq:det_balance}. 
\end{proof}

\begin{remark}
The reason we have omitted the normalisation $\cZ^{-1}$ in the weighted inner 
product~\eqref{eq:pscal_pi} is purely for later notational convenience. Of 
course, it does not change anything in the presented results. 
\end{remark}

%%%%%%%%%%%%%%%%%%%%%%%%%%%%%%%%%%%%%%%%%%%%%%%%%%%%%%%%%%%%%%%%%%%%%%%%%%%%%%%%

\section{Large deviations}
\label{sec:diffldp} 

We now focus on the weak-noise regime $0<\eps\ll 1$. The theory of large 
deviations provides a useful tool to estimate the probability of rare events as 
$\eps$ decreases to $0$.

\begin{definition}[Large-deviation principle]
\label{def:ldp} 
Fix a time horizon $T \in (0,\infty)$. We say that a function $\cI=\cI_{[0,T]}$ 
from 
the space $\cC_0$ of continuous paths $\gamma:[0,T]\to\R^\Lambda$ to $\R_+$ is 
a 
\emph{good rate function} if it is lower semi-continuous (its sublevel sets are 
closed) and has compact level sets. The process $(y_t)_{t\in[0,T]}$ satisfies a 
sample-paths \emph{large-deviation principle} (LDP) with good rate function 
$\cI$ 
if 
\begin{align}
\liminf_{\eps\to0} 2\eps \log \bigprob{(y_t)_{t\in[0,T]} \in O}
&\geqs - \inf_{\gamma\in O} \cI_{[0,T]}(\gamma) \\
\limsup_{\eps\to0} 2\eps \log \bigprob{(y_t)_{t\in[0,T]} \in C}
&\leqs - \inf_{\gamma\in C} \cI_{[0,T]}(\gamma) 
\end{align}
holds for all open $O\subset\cC_0$ and all closed $C\subset\cC_0$.
\end{definition}

Roughly speaking, the large-deviation principle says that for a 
sufficiently nice set $\Gamma$ of paths $\gamma:[0,T]\to\R^\Lambda$, 
we have 
\begin{equation}
 \bigprob{(y_t)_{t\in[0,T]} \in \Gamma}
 \simeq \e^{-\inf_\Gamma \cI_{[0,T]}/(2\eps)}
\end{equation} 
in the sense of logarithmic equivalence. Here are two illustrative examples:

\begin{enumerate}
\item 	Let $\gamma_0:[0,T]\to\R^\Lambda$ be a given continuous path, and let 
$\Gamma$ be the set of paths such that $\norm{\gamma(t)-\gamma_0(t)}<\delta$ 
for 
all $t\in[0,T]$. As $\delta$ decreases to $0$, the infimum of $\cI_{[0,T]}$ 
over 
$\Gamma$ will converge to $\cI(\gamma_0)$. Therefore, the probability of sample 
paths tracking $\gamma_0$ up to a small error $\delta$ will be close to 
$\e^{-\cI_{[0,T]}(\gamma_0)/(2\eps)}$ for small $\delta$. 

\item 	Let $D\subset\R^\Lambda$ be a bounded, connected subset of 
$\R^\Lambda$, fix a point $y_0\in D$, and let $\Gamma$ be the set of continuous 
paths starting in $y_0$ and leaving $D$ at least once during the time interval 
$(0,T)$. Then we have 
\begin{equation}
 \bigprobin{y_0}{\tau_D < T} \simeq \e^{-\inf_\Gamma \cI_{[0,T]}/(2\eps)}\;,
\end{equation} 
where $\tau_D = \inf\setsuch{t>0}{y_t\notin D}$ is the first-exit time from $D$.
\end{enumerate}

In the case of scaled Brownian motion $\sqrt{2\eps} W_t$, 
Schilder~\cite{Schilder66} obtained a large-deviation principle with rate 
function 
\begin{equation}
 \cI_{[0,T]}(\gamma) := 
 \begin{cases}
 \displaystyle
 \frac12 \int_0^T \norm{\dot\gamma(t)}^2 \6t 
 & \text{if $\gamma\in H^1$\;, } \\
 + \infty 
 & \text{otherwise\;.}
 \end{cases}
\end{equation} 
Freidlin and Wenzell~\cite{FW} extended this result to general 
diffusions. Their proof uses the Cameron--Martin--Girsanov formula to reduce 
the general problem to the special case of scaled Brownian motion. In the case 
of the SDE~\eqref{eq:SDE_V}, the LDP takes the following form.

\begin{theorem}[LDP for gradient diffusions]
The diffusion~\eqref{eq:SDE_V} satisfies a large-deviation principle with good 
rate function
\begin{equation}
 \cI_{[0,T]}(\gamma) := 
 \begin{cases}
 \displaystyle
 \frac12 \int_0^T \norm{\dot\gamma(t) + \nabla V(\gamma(t))}^2 \6t 
 & \text{if $\gamma\in H^1$\;, } \\
 + \infty 
 & \text{otherwise\;.}
 \end{cases}
\end{equation}
\end{theorem}

For SDEs with more general drift coefficients $f(y)$, the term 
$\norm{\dot\gamma(t) + \nabla V(\gamma(t))}^2$ in the rate function has to be 
replaced by $\norm{\dot\gamma(t) - f(\gamma(t))}^2$, and there also exists a 
version for SDEs with general diffusion coefficients. However, the gradient form 
of the system~\eqref{eq:SDE_V} entails a substantial simplification, since we 
have the identity 
\begin{align}
\cI_{[0,T]}(\gamma) 
&= \frac12 \int_0^T \norm{\dot\gamma(t) - \nabla V(\gamma(t))}^2 \6t
+ 2 \int_0^T \dot\gamma(t) \cdot \nabla V(\gamma(t)) \6t \\
&= \frac12 \int_0^T \norm{\dot\gamma(t) - \nabla V(\gamma(t))}^2 \6t
+ 2 \bigbrak{V(\gamma(T)) - V(\gamma(0))}\;.
\label{eq:ratefct_gradient} 
\end{align}
If for example we want to analyse properties of the first-exit time from a 
potential well $D$, starting from its bottom $y^*$, the integral 
in~\eqref{eq:ratefct_gradient} can be made arbitrarily small by taking $T$ 
large and letting $\gamma(t)$ be the solution of the equation with reversed 
drift $\dot\gamma = + \nabla V(\gamma)$. The rate function will thus be 
dominated by the minimum of the potential difference $2 \brak{V(y) - V(y^*)}$ 
over all $y$ on the boundary of $D$. 

\begin{exercise}
Compute the rate function $\cI_{[0,T]}$ in the case of the Ornstein--Uhlenbeck 
process 
\begin{equation}
 \6y_t = - y_t \6t + \sqrt{2\eps} \6W_t\;,
\end{equation} 
and use it to analyse the cumulative distribution function of 
$\tau=\inf\setsuch{t>0}{\abs{y_t}>L}$ as $\eps\to 0$. 
\end{exercise}

%%%%%%%%%%%%%%%%%%%%%%%%%%%%%%%%%%%%%%%%%%%%%%%%%%%%%%%%%%%%%%%%%%%%%%%%%%%%%%%%

\section{Metastability}
\label{sec:diffmeta} 

The existence of a unique invariant probability measure having been settled, 
the next natural question to ask is whether the system will converge to this 
measure. In fact, as hinted at in \Cref{ex:Lyap}, in the case of the 
diffusion~\eqref{eq:SDE} we have the even stronger Lyapunov property 
\begin{equation}
 (\cL V)(y) \leqs -cV(y) + d
\end{equation} 
for two constants $c>0$ and $d\geqs0$. According 
to~\cite[Theorem~6.1]{Meyn_Tweedie_1993b}, we have an exponential ergodicity 
result of the form 
\begin{equation}
 \sup_{f\colon\abs{f}\leqs V+1} \bigabs{\expecin{y}{f(y_t)} - \pscal{\pi}{f}} 
 \leqs C \brak{V(y)+1} \e^{-\beta t}
 \qquad \forall y\in\R^\Lambda
\end{equation} 
for some constants $\beta, C>0$. However, this result does not yield a good 
control on the constants $\beta$ and $C$, and in particular $\beta$ can behave 
quite badly in terms of $\eps$. This is a manifestation of the phenomenon of 
\emph{metastability}, which is related to the fact that for small $\eps$, the 
system may take extremely long to move between potential wells. A classical 
approach to quantifying this phenomenon is thus to investigate the law of this 
interwell transition time.  

%%%%%%%%%%%%%%%%%%%%%%%%%%%%%%%%%%%%%%%%%%%%%%%%%%%%%%%%%%%%%%%%%%%%%%%%%%%%%%%%

\subsection{Arrhenius law}
\label{ssec:arrhenius} 

For simplicity, we will only consider cases where $V$ is a double-well 
potential, as is the case for the system~\eqref{eq:SDE} for $\gamma > 
\gamma_1(N)$ according to \Cref{prop:det_potential}. We can always assume that 
the saddle is located at the origin, and denote the two local minima by 
$y^*_\pm$. Assume that the diffusion starts in $y^*_-$, and given a small 
constant $\delta>0$, denote by 
\begin{equation}
 \tau_+ := \inf\bigsetsuch{t>0}{\norm{y_t-y^*_+}<\delta}
\end{equation} 
the first-hitting time of a small neighbourhood of $y^*_+$. 

A first type of results on the law of $\tau_+$ can be obtained by applying the 
theory of large deviations outlined in \Cref{sec:diffldp}. 

\begin{theorem}[Large-deviation results for interwell transitions]
\label{thm:arrhenius} 
Let $H = V(0) - V(y^*_-)$ be the potential difference between the starting 
minimum and the saddle. Then 
\begin{equation}
\label{eq:Arrhenius} 
 \lim_{\eps\to 0} \eps \log\expecin{y^*_-}{\tau_+} = H\;.
\end{equation} 
Furthermore, for any $\eta>0$, 
\begin{equation}
\label{eq:Arrhenius_concentration} 
 \lim_{\eps\to 0}
 \Bigprobin{y^*_-}{\e^{(H-\eta)/\eps} \leqs \tau_+ \leqs \e^{(H+\eta)/\eps}} = 
1\;.
\end{equation} 
Finally, let $D_-$ denote the basin of attraction of $y^*_-$ under the 
deterministic dynamics. Then the location $y_{\tau_{D_-}}$ of first exit from 
$D_-$ satisfies 
\begin{equation}
\label{eq:exit_concentration} 
 \lim_{\eps\to0} \Bigprobin{y^*_-}{y_{\tau_{D_-}}\in \cN} = 0
\end{equation} 
for any closed $\cN\subset\partial D_-$ that does not contain the saddle at $0$.
\end{theorem}

Relation~\eqref{eq:Arrhenius} is called \emph{Arrhenius' law}. It states that 
the expected transition time behaves like $\e^{H/\eps}$ in the sense of 
logarithmic equivalence. The exponential dependence in $H/\eps$ goes back to 
works by van t'Hoff and Arrhenius in the late 19th century~\cite{Arrhenius}. 
Relation~\eqref{eq:Arrhenius_concentration} shows that the law of $\tau_+$ 
concentrates around its expectation, albeit in a rather weak sense. The last 
result~\eqref{eq:exit_concentration} states that the exit location from the 
starting potential well concentrates near the saddle in the vanishing noise 
limit. 

\Cref{thm:arrhenius} can be obtained in two main steps. Firstly, analogous 
results for the first exit from a bounded subset of $D_-$ follow from 
Theorems~2.1, 4.1 and 4.2 in~\cite[Chapter 4]{FW}. The main idea 
of the proof is that successive attempts to exit this subset are almost 
independent, and have an exponentially small probability of success as 
discussed at the end of \Cref{sec:diffldp}. Secondly, these estimates can be 
extended to the distribution of $\tau_+$ by showing that the process behaves 
like a Markov chain jumping between neighbourhoods of critical points of $V$, 
see Theorems~5.1 and 5.3 in~\cite[Chapter 6]{FW}. 

Particularising to the diffusion~\eqref{eq:SDE} for $\gamma > \gamma_1(N)$, we 
obtain that Arrhenius' law~\eqref{eq:Arrhenius} holds with 
\begin{equation}
\label{eq:H_synchro} 
 H = V(0) - V(-1,\dots,-1) = \frac{N}{4}\;.
\end{equation} 
We point out that at this stage, we do not claim any control on the speed of 
convergence in~\eqref{eq:Arrhenius}, \eqref{eq:Arrhenius_concentration} 
and~\eqref{eq:exit_concentration} as a function of $N$. This is a more subtle 
point that we will come back to later on.

\begin{remark}
The case where $\gamma$ lies below $\gamma_1(N)$ but is still of order $N^2$ 
can be analysed in a similar way. The main difference is that instead of a 
single saddle located at the origin, the number of relevant saddles is 
proportional to $N$. Thus $H$ is no longer given by~\eqref{eq:H_synchro}, and 
the exit location from the starting well is concentrated in the union of 
the relevant saddles. See~\cite[Theorem~2.10]{BFG06a} 
and~\cite[Theorem~2.4]{BFG06b}. 
\end{remark}

%%%%%%%%%%%%%%%%%%%%%%%%%%%%%%%%%%%%%%%%%%%%%%%%%%%%%%%%%%%%%%%%%%%%%%%%%%%%%%%%

\subsection{Potential theory}
\label{ssec:potential} 

One of the approaches allowing to obtain sharper asymptotics on metastable 
transition times is based on potential theory. The starting point is the 
observation that owing to Dynkin's formula, several probabilistic quantities of 
interest solve boundary value problems involving the infinitesimal generator 
$\cL$ of the process. 

Let $A \subset \R^\Lambda$ be a closed set with smooth boundary, and let 
$\tau_A=\inf\setsuch{t>0}{y_t\in A}$ be the first-hitting time of $A$. Then the 
map $y\mapsto w_A(y) = \expecin{y}{\tau_A}$ satisfies the Poisson problem 
\begin{equation}
\label{eq:Poisson} 
 \begin{cases}
  (\cL w_A)(y) = -1 &\qquad y\in A^c\;, \\
  w_A(y) = 0        &\qquad y\in A\;.
 \end{cases}
\end{equation} 
This is a particular case of~\cite[Corollary 9.1.2]{Oeksendal}. 

\begin{figure}[htb]
\begin{center}
%\vspace{-5mm}
\begin{tikzpicture}[>=stealth',main node/.style={draw,circle,fill=white,minimum
size=3pt,inner sep=0pt},scale=0.9%,x=1.5cm,y=0.3cm
]

% grid to help positioning
%\draw[help lines] (-5,-4) grid (5,1);

% axes

\draw[->,thick] (-5.5,0) -> (5.5,0);
\draw[->,thick] (0,-4.5) -> (0,2.0);

% potential

\draw[blue,thick] plot[smooth,tension=.6]
  coordinates{(-4.6,1.5) (-2.6,-4) (-0.05,0) (2.4,-3) (4.4,1.5)};
 
% vertical lines

\draw[blue!50,semithick,dashed] (-2.55,0) -- (-2.55,-4);
\draw[blue!50,semithick,dashed] (2.3,0) -- (2.3,-3);

% local minima and maxima

\node[main node,blue,fill=white,semithick] at (0,0) {}; 
\node[main node,semithick] at (2.3,0) {}; 
\node[main node,blue,fill=white,semithick] at (2.3,-3) {}; 
\node[main node,semithick] at (-2.55,0) {}; 
\node[main node,blue,fill=white,semithick] at (-2.55,-4) {}; 
\node[main node,semithick] at (-2.0,0) {}; 

\node[] at (-2.55,0.3) {$y^*_-$};
\node[] at (0.3,0.3) {$z^*$};
\node[] at (2.3,0.3) {$y^*_+$};
\node[] at (-2,0.3) {$a$};

\node[] at (5.0,-0.27) {$y$};
\node[] at (0.5,1.5) {$V(y)$};

\end{tikzpicture}
\vspace{-3mm}
\end{center}
\caption[]{A one-dimensional double-well potential.  
}
\label{fig:double_well}
\end{figure}
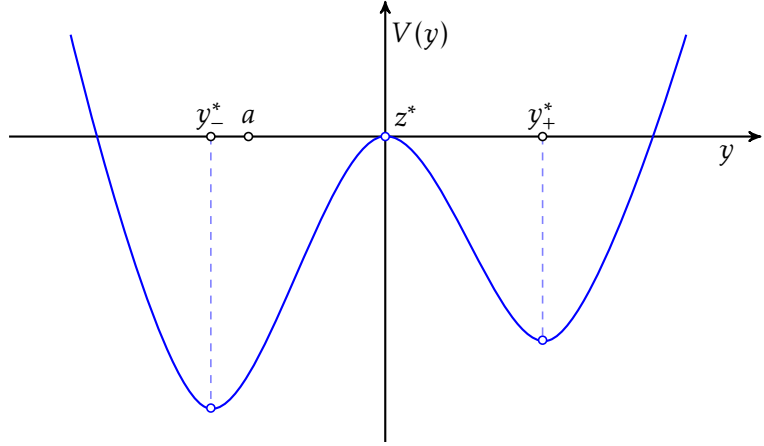

\begin{exercise}
\label{exo:1dKramers} 
Consider the SDE~\eqref{eq:SDE_V} in the one-dimensional case $y\in\R$, for a 
potential $V$ growing sufficiently fast at infinity. Show that for 
$A=(-\infty,a]$, the solution of~\eqref{eq:Poisson} is given by 
\begin{equation}
 w_A(y) = \frac{1}{\eps} \int_a^y \int_{y_2}^\infty \e^{[V(y_2)-V(y_1)]/\eps} 
\6y_1 \6y_2 
\qquad \forall y > a\;.
\end{equation} 
Assume now that $V$ is a double-well potential, with two local minima located 
at $y^*_- < a$ and $y^*_+ > a$, and a saddle (local maximum) at $z^* \in 
(a,y^*_+)$ (\Cref{fig:double_well}). Use Laplace asymptotics to show that
\begin{equation}
\label{eq:Kramers} 
 w_A(y) = \frac{2\pi}{\sqrt{\abs{V''(z^*)}V''(y^*_+)}}
 \e^{[V(z^*) - V(y^*_+)]/\eps}
 \bigbrak{1+\Order{\sqrt{\eps}}}\;.
\end{equation} 
Relation~\eqref{eq:Kramers} is known as \emph{Kramers' law}. 
\end{exercise}

In general, there is no explicit solution to the Poisson 
equation~\eqref{eq:Poisson}. However, the solution can be represented as 
\begin{equation}
\label{eq:wA_integral} 
 w_A(y) = -\int_{A^c} G_{A^c}(y,z)\6z\;,
\end{equation} 
where $G_{A^c}$ is the \emph{Green function} associated with $A^c$, which 
solves 
\begin{equation}
  \begin{cases}
  (\cL G_{A^c})(y,z) = \delta(y-z) &\qquad y\in A^c\;, \\
  G_{A^c}(y,z) = 0        &\qquad y\in A\;.
 \end{cases}
\end{equation} 
Reversibility of the SDE~\eqref{eq:SDE_V} implies that $G_{A^c}$ satisfies the 
detailed-balance relation
\begin{equation}
\label{eq:G_det_balance} 
 \e^{-V(y)/\eps} G_{A^c}(y,z)
 = \e^{-V(z)/\eps} G_{A^c}(z,y)\;.
\end{equation}
Let now $A$ and $B$ be two disjoint closed sets with smooth boundary. 
A second quantity of interest is the \emph{committor function} 
$h_{AB}(y) = \probin{y}{\tau_A < \tau_B}$, also called \emph{equilibrium 
potential}. It satisfies the Dirichlet problem 
\begin{equation}
\label{eq:Dirichlet} 
 \begin{cases}
  (\cL h_{AB})(y) = 0  &\qquad y\in (A\cup B)^c\;, \\
  h_{AB}(y) = 1        &\qquad y\in A\;, \\
  h_{AB}(y) = 0        &\qquad y\in B\;.
 \end{cases}
\end{equation} 

\begin{exercise}
\label{exo:1d-committor} 
Consider again the SDE~\eqref{eq:SDE_V} in the one-dimensional case $y\in\R$. 
Given $a<b$, let $A=(-\infty,a]$ and $B=[b,\infty)$. Show that 
\begin{equation}
\label{eq:hAB_1d} 
 h_{AB}(y) = 
 \frac{\displaystyle\int_y^b \e^{V(z)/\eps}\6z}
 {\displaystyle\int_a^b\e^{V(z)/\eps}\6z}
 \qquad \forall y\in [a,b]\;.
\end{equation} 
Use Laplace asymptotics to determine the behaviour of $h_{AB}(y)$ as 
$\eps\to0$.  
\end{exercise}

There exists again an integral representation of the solution 
to~\eqref{eq:Dirichlet} in terms of the Green function, namely 
\begin{equation}
\label{eq:hAB_integral} 
 h_{AB}(y) = -\int_{\partial A} G_{B^c}(y,z) e_{AB}(\6z)\;.
\end{equation} 
Here $e_{AB}$ is a measure concentrated on $\partial A$ called the 
\defwd{equilibrium measure}. It is defined by 
\begin{equation}
 e_{AB}(\6y) := (-\cL h_{AB})(dy)
\end{equation} 
interpreted in the weak sense (i.e., both sides have to be integrated against 
test functions). The \defwd{capacity} is then defined by 
\begin{equation}
 \capacity(A,B) := \int_{\partial A} \e^{-V(y)/\eps} e_{AB}(\6y)\;.
\end{equation} 
Note that the capacity is the normalisation required to make 
\begin{equation}
 \nu_{AB}(\6y) := \frac{1}{\capacity(A,B)} \e^{-V(y)/\eps} e_{AB}(\6y)
\end{equation} 
a probability measure on $\partial A$. 

\begin{remark}
The name \emph{capacity} is due to an analogy with electrostatics. Indeed, in 
the case without potential $V=0$, \eqref{eq:Dirichlet} is the equation for the 
electric potential of a capacitor, consisting of a conductor $A$ at potential 
$1$ and a grounded conductor $B$. The integral relation~\eqref{eq:hAB_integral} 
expresses the fact that $h_{AB}$ is the electric potential created by a charge 
density $e_{AB}$ on $\partial A$ with zero boundary conditions on $B$. Since 
the 
potential difference between the conductors is equal to $1$, the capacity is 
equal to the total charge on the conductor $A$. 
\end{remark}

The main result that makes the potential-theoretic approach so successful is 
the following relation between expected first-hitting time and capacity.

\begin{theorem}
\label{thm:magic_formula} 
For any disjoint sets $A$ and $B$ with smooth boundary, 
\begin{equation}
\label{eq:magic_formula} 
 \expecin{\nu_{AB}}{\tau_B} 
 := \int_{\partial A} w_B(y)\nu_{AB}(\6y) 
 = \frac{1}{\capacity(A,B)} \int_{B^c} \e^{-V(y)/\eps} h_{AB}(y) \6y\;.
\end{equation} 
\end{theorem}
\begin{proof}
It follows from the integral representation~\eqref{eq:wA_integral} of $w_B$, 
the detailed balance~\eqref{eq:G_det_balance} of the Green function, and the 
integral representation~\eqref{eq:hAB_integral} of $h_{AB}$, that 
\begin{align}
\capacity(A,B) \int_{\partial A} w_B(y) \nu_{AB}(\6y) 
&= - \int_{\partial A} \int_{B^c} G_{B^c}(y,z)\6z\, \e^{-V(y)/\eps} e_{AB}(\6y) 
\\
&= -\int_{B^c} \int_{\partial A} G_{B^c}(z,y) e_{AB}(\6y) \e^{-V(z)/\eps} \6z 
\\
&= \int_{B^c} h_{AB}(z) \e^{-V(z)/\eps} \6z\;.
\end{align}
Dividing by the capacity yields~\eqref{eq:magic_formula}. 
\end{proof}

The exact relation~\eqref{eq:magic_formula} is useful because there exist 
variational principles that often allow to obtain good upper and lower bounds 
on the capacity. The first of these principles states that the capacity is a 
minimiser of the \emph{Dirichlet form}.

\begin{definition}[Dirichlet form]
The \emph{Dirichlet form} associated with the diffusion~\eqref{eq:SDE_V} is the 
quadratic form acting on test functions $f:\R^\Lambda\to\R$ in the domain of 
$\cL$, defined by 
\begin{equation}
\label{eq:Dirichlet_form} 
 \cE(f) := \pscal{f}{-\cL f}_\pi = \eps \int_{\R^\Lambda} \e^{-V(y)/\eps} 
\norm{\nabla f(y)}^2 \6y\;.
\end{equation} 
\end{definition}

The integral expression for the Dirichlet form in~\eqref{eq:Dirichlet_form} is a 
consequence of the first Green identity, which is essentially the divergence 
theorem (using the fact that $\e^{-V/\eps}[\eps\nabla f\cdot\nabla g + g\cL f] = 
\eps\nabla\cdot(g\e^{-V/\eps}\nabla f)$) and says that for a set $D$ with smooth 
boundary, 
\begin{equation}
\label{eq:Green1} 
 \int_D \e^{-V(y)/\eps} \bigbrak{\eps \nabla f(y)\cdot\nabla g(y)
 + g(y)(\cL f)(y)} \6y 
 = \eps\int_{\partial D} \e^{-V(y)/\eps} g(y) \partial_{n(y)}f(y) \sigma(\6y)\;,
\end{equation} 
where $\smash{\partial_{n(y)}}f(y) = \nabla f(y)\cdot n(y)$ denotes the  
derivative in the direction of the unit outer normal $n(y)$ at a point 
$y\in\partial D$, and $\sigma(\6y)$ is the Lebesgue measure on $\partial D$. In 
the particular case $D=\R^\Lambda$, the right-hand side of~\eqref{eq:Green1} 
vanishes and one obtains the integral expression in~\eqref{eq:Dirichlet_form}. 

The quadratic form $\cE$ can be extended by polarisation to a bilinear form 
given by 
\begin{equation}
 \cE(f,g) = \pscal{f}{-\cL g}_\pi = \eps \int_{\R^\Lambda} \e^{-V(y)/\eps} 
\overline{\nabla f(y)}\cdot \nabla g(y) \6y\;, 
\end{equation} 
which satisfies the Cauchy--Schwarz inequality 
$\cE(f,g)^2 \leqs \cE(f)\cE(g)$.

\begin{theorem}[Dirichlet principle]
Let $\cH_{AB}$ be the set of functions $h:\R^\Lambda\to[0,1]$ that are in 
the domain of $\cE$ and such that $h\vert_A=1$ and $h\vert_B=0$. Then 
\begin{equation}
 \capacity(A,B) = \cE(h_{AB}) = \inf_{h\in \cH_{AB}} \cE(h)\;.
\end{equation}  
\end{theorem}
\begin{proof}
Pick any $h\in\cH_{AB}$. Since $h_{AB}$ satisfies the Dirichlet 
problem~\eqref{eq:Dirichlet}, we have 
\begin{align}
\cE(h,h_{AB}) 
&= \eps \int_{\R^\Lambda} \e^{-V(y)/\eps} h(y) (-\cL h_{AB})(y) \6y \\
&= \eps \int_{\partial A} \e^{-V(y)/\eps} e_{AB}(\6y) \\
&= \capacity(A,B)\;.
\end{align}
Since $h_{AB}\in\cH_{AB}$, we obtain $\cE(h_{AB})=\capacity(A,B)$. 
The Cauchy--Schwarz inequality then yields 
$\cE(h_{AB})^2 = \cE(h,h_{AB})^2 \leqs \cE(h)\cE(h_{AB})$, i.e.\ 
$\cE(h)\geqs\cE(h_{AB})$ . A sketch is given
in~\Cref{fig:Dirichlet_form}. 
\end{proof}

\begin{figure}[htb]
\begin{center}
%\vspace{-5mm}
\begin{tikzpicture}[main node/.style={circle,minimum
size=0.1cm,inner sep=0pt,fill=white,draw},scale=0.9]

% grid to help positioning
%\draw[help lines] (-5,-1) grid (5,5);

\draw[semithick] (0,0) -- (3,2);
\draw[semithick] (0,0) -- (4,0.5);
\draw[] (2.7,1.8) -- (2.9,1.5) -- (3.2,1.7);
\draw[semithick,teal] (5,-1) -- (2,3.5);

\node[main node] at (0,0) {};
\node[main node] at (3,2) {};
\node[main node] at (4,0.5) {};

\node[] at (-0.3,-0.2) {$0$};
\node[] at (3.5, 2.2) {$h_{AB}$};
\node[] at (4.4, 0.6) {$h$};
\node[teal] at (5.4, -0.8) {$\cH_{AB}$};

\end{tikzpicture}
\vspace{-3mm}
\end{center}
\caption[]{Dirichlet principle. The capacity minimises the distance to the 
origin given by the Dirichlet form over all functions $h$ satisfying the 
boundary conditions.}
\label{fig:Dirichlet_form}
\end{figure}
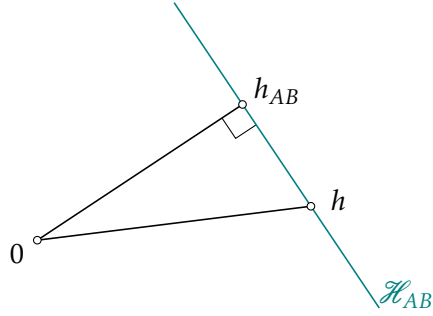

There also exists a variational principle allowing to obtain lower bounds on 
the capacity, the \emph{Thomson principle}, which involves divergence-free 
flows. 

\begin{definition}[Divergence-free unit flow]
A \defwd{divergence-free unit $AB$-flow} is a vector field $\ph$ of class 
$\cC^1$ on $(A\cup B)^c$ such that $\nabla\cdot \ph=0$ in $(A\cup B)^c$ and 
\begin{equation}
\label{eq:def_unit_flow} 
 \int_{\partial A} \ph(y) \cdot n_A(y) \sigma(\6y) = 1 
 = -\int_{\partial B} \ph(y) \cdot n_B(y) \sigma(\6y)\;,
\end{equation} 
where $n_A(y)$, $n_B(y)$ denote the unit outward normal vectors at 
$y\in\partial A$ and $y\in\partial B$ respectively. We denote by 
$\mathfrak{U}^1_{AB}$ the set of divergence-free unit 
$AB$-flows.
The \defwd{harmonic unit flow} is defined as 
\begin{equation}
 \ph_{AB}(y) := -\frac{\eps}{\capacity(A,B)} \e^{-V(y)/\eps} \nabla h_{AB}(y)
\end{equation}
and it satisfies $\ph_{AB} \in \mathfrak{U}^1_{AB}$.
\end{definition}

The fact that $\ph_{AB}$ is divergence-free follows from the 
form~\eqref{eq:generator_rev} of the generator and the fact that $h_{AB}$ is 
harmonic. Its satisfies~\eqref{eq:def_unit_flow} as a consequence of the 
divergence theorem. 

We define a bilinear form on $\mathfrak{U}^1_{AB}$ by 
\begin{equation}
\label{eq:def_D} 
 \cD(\ph,\psi) := \frac{1}{\eps} \int_{(A\cup B)^c} \e^{V(y)/\eps} \ph(y)\cdot 
\psi(y)\6y\;,
\end{equation} 
and set $\cD(\ph,\ph) := \cD(\ph)$. 

\begin{proposition}[Thomson principle]
\label{prop:thomson}
The capacity satisfies 
\begin{equation}
\label{eq:Thomson} 
 \capacity(A,B) 
 = \frac{1}{\cD(\ph_{AB})}
 = \sup_{\ph\in\mathfrak{U}^1_{AB}} \frac{1}{\cD(\ph)}\;.
\end{equation} 
\end{proposition}
\begin{proof}
Set $\Psi_{h_{AB}} = \capacity(A,B)\ph_{AB} = -\eps 
\e^{-V/\eps} \nabla h_{AB}$. Then 
\begin{equation}
 \cD(\Psi_{h_{AB}}) 
 = \frac{\eps^2}{\eps} \int_{(A\cup B)^c} \e^{V(y)/\eps} 
\abs{\e^{-V(y)/\eps}\nabla h_{AB}(y)}^2 \6y 
 = \capacity(A,B)\;.
\end{equation} 
This implies $\cD(\ph_{AB}) = (\capacity(A,B))^{-1}$ by bilinearity. 
For any unit flow $\ph\in\mathfrak{U}^1_{AB}$, we have 
\begin{align}
\cD(\Psi_{h_{AB}},\ph)  
&= \frac{1}{\eps} \int_{(A\cup B)^c} \e^{V(y)/\eps}  
\Psi_{h_{AB}}(y) \cdot \ph(y) \6y \\
&= -\int_{(A\cup B)^c} \nabla h_{AB}(y) \cdot \ph(y) \6y \\ 
&= -\int_{(A\cup B)^c} \nabla\cdot \bigpar{h_{AB}(y) \ph(y)} \6y
+ \int_{(A\cup B)^c} h_{AB}(y) \nabla\cdot \ph(y) \6y\;.
\end{align}
The second integral vanishes because $\ph$ is divergence-free. By the 
divergence 
theorem, the first term is equal to 
\begin{equation}
 -\int_{\partial A\cup\partial B} h_{AB}(y) \ph(y) \cdot n(y)\sigma(\6y)\;,
\end{equation} 
where $n(y)$ is the unit normal vector pointing \emph{inside} $A\cup B$. The 
integral on $\partial B$ vanishes, while the integral on $\partial A$ is equal 
to $1$ because $n(y)=-n_A(y)$ and $\ph$ 
satisfies~\eqref{eq:def_unit_flow}. 
This implies  
%\begin{equation}
$\cD(\Psi_{h_{AB}},\ph) = 1$, 
%\end{equation} 
and thus, by the Cauchy--Schwarz inequality, 
\begin{equation}
 1 = \cD(\Psi_{h_{AB}},\ph)^2 \leqs \cD(\Psi_{h_{AB}})\cD(\ph) 
 = \capacity(A,B) \cD(\ph)\;,
\end{equation} 
showing that indeed $\capacity(A,B) \geqs 1/\cD(\ph)$. 
\end{proof}

\begin{remark}
The potential-theoretic approach described in this section admits a simple 
analogue in the setting of discrete Markov chains. This will not play any role 
in what follows, but for completeness and since it may help intuition, we 
summarise this theory in \Cref{ch:app_markov}.
\end{remark}

%%%%%%%%%%%%%%%%%%%%%%%%%%%%%%%%%%%%%%%%%%%%%%%%%%%%%%%%%%%%%%%%%%%%%%%%%%%%%%%%

\subsection{Eyring--Kramers law}
\label{ssec:kramers} 

We now apply the potential-theoretic approach to the system of interacting 
diffusions, in order to obtain sharper asymptotics on the transition time 
between the two states $y^*_\pm = \pm(1,\dots,1)$. For simplicity, we only 
consider the case $\gamma > \gamma_1(N)$, when the potential is a double-well 
potential with a saddle located at the origin. The theory works, however, in 
much greater generality. 

One simplification due to the condition $\gamma > \gamma_1(N)$ is that we will 
be able to use the following symmetry argument.

\begin{lemma}
\label{lem:symmetry_hAB} 
Let $A$ and $B$ satisfy $B=-A$. Then 
\begin{equation}
\int_{(A\cup B)^c} \e^{-V(y)/\eps} h_{AB}(y) \6y
= \frac12 \cZ
\end{equation} 
where $\cZ$ is the \emph{partition function} of the system defined 
in~\eqref{eq:defZ}. 
\end{lemma}

\begin{exercise}
Prove \Cref{lem:symmetry_hAB}, using the fact that $V(-y)=V(y)$ and the 
relations $h_{AB}(y) = h_{BA}(-y)$ and $h_{AB}(y) = 1 - h_{BA}(y)$. 
\end{exercise}

To apply \Cref{thm:magic_formula}, it is thus sufficient to estimate the 
partition function $\cZ$ and the capacity $\capacity(A,B)$. In order to do so, 
it 
turns out to be useful to make a change of variables. Let $(e_0,\dots,e_{N-1})$ 
be an orthonormal basis of $\R^\Lambda$, where 
\begin{equation}
 e_0 := \frac{1}{\sqrt{N}} \transpose{(1,\dots,1)}\;.
\end{equation} 
The precise form of the other basis vectors will not matter -- one possibility 
is to use those appearing in the discrete Fourier transform. We denote by 
$\R^\Lambda_\perp$ the span of $e_1,\dots,e_{N-1}$. The change of variables is 
given by 
\begin{equation}
 y = y_0 e_0 + \sqrt{\eps} y_\perp
\end{equation} 
where 
\begin{equation}
 y_0 = y \cdot e_0 = \frac{1}{\sqrt{N}} \sum_{i=1}^N y^i\;, 
 \qquad 
 y_\perp = \frac{1}{\sqrt{\eps}} \bigpar{y - y_0 e_0} \in \R^\Lambda_\perp\;.
\end{equation} 
The role of the factor $\sqrt{\eps}$ is to highlight the scaling properties of 
some quantities with $\eps$. Due to this factor, the Jacobian of the 
transformation $y\mapsto(y_0,y_\perp)$ is equal to $\eps^{(N-1)/2}$. Performing 
the change of variables in the potential $V$, and using the fact that the sum 
of the coordinates of $y_\perp$ vanishes, we obtain 
\begin{equation}
\label{eq:potential_fourier} 
 \frac{1}{\eps} V(y) = \frac{1}{\eps} V_0(y_0)
 + \frac12 \pscal{y_\perp}{Q_\perp(y_0)y_\perp} 
 + R_\eps(y_0,y_\perp)\;.
\end{equation} 
Here
\begin{equation}
 V_0(y_0) := N U \Bigpar{\frac{y_0}{\sqrt{N}}}
 = \frac{1}{4N}y_0^4 - \frac12 y_0^2 + \frac{N}{4}\;,
\end{equation} 
while $\pscal{y_\perp}{Q_\perp(y_0)y_\perp}$ is the quadratic form defined by 
\begin{equation}
 Q_\perp(y_0) := \Bigpar{\frac{3}{N}y_0^2-1} \one - \gamma \Delta_{\perp,N}
\end{equation} 
where $\Delta_{\perp,N}$ is the discrete Laplacian acting on 
$\R^\Lambda_\perp$, and $R_\eps$ is a remainder given by 
\begin{equation}
R_\eps(y_0,y_\perp) = 
 \sqrt{\frac{\eps}{N}}y_0 \sum_{i\in\Lambda} \bigpar{y^i_\perp}^3 
 + \frac{\eps}{4} \sum_{i\in\Lambda} \bigpar{y^i_\perp}^4\;.
\end{equation} 

\begin{proposition}
The partition function has the asymptotic form 
\begin{equation}
\label{eq:bound_Z} 
 \cZ = 2\sqrt{\frac{(2\pi\eps)^N}{2\det Q_\perp(-\sqrt{N})}} 
\bigbrak{1+\Order{\eps}}
 = 2\sqrt{\frac{(2\pi\eps)^N}{\prod_{k=0}^{N-1} \nu^N_k}} 
\bigbrak{1+\Order{\eps}}\;,
\end{equation} 
where the $\nu^N_k$ are the eigenvalues~\eqref{eq:nuNk} of the Hessian of 
$V(y^*_-)$. 
\end{proposition}
\begin{proof}
This result follows rather directly from standard Laplace asymptotics, but we 
will give some details of the proof as they will be useful later on. 
Using~\eqref{eq:potential_fourier}, we get 
\begin{equation}
\label{eq:proof_Z1} 
 \cZ = \eps^{(N-1)/2} \int_{-\infty}^\infty \e^{-V_0(y_0)/\eps}
 \int_{\R^\Lambda_\perp} \e^{-\pscal{y_\perp}{Q_\perp(y_0)y_\perp}/2}
 \e^{-R_\eps(y_0,y_\perp)}
 \6y_\perp \6y_0\;.
\end{equation} 
The idea is to view the integral over $\R^\Lambda_\perp$ as an expectation 
under the Gaussian measure $g_N(y_0)$ with covariance $Q_\perp(y_0)^{-1}$. 
Taking the normalisation into account, we obtain   
\begin{equation}
 \cZ = \eps^{(N-1)/2} \int_{-\infty}^\infty \e^{-V_0(y_0)/\eps}
 \sqrt{\frac{(2\pi)^{N-1}}{\det Q_\perp(y_0)}} 
\bigexpecin{g_N(y_0)}{\e^{-R_\eps(y_0,y_\perp)}} 
\6y_0\;.
\end{equation} 
The expectation is bounded uniformly in $\eps$ and $N$, because the potential 
satisfies the quadratic lower bound~\eqref{eq:offdiagonal} derived in 
\Cref{exo:offdiagonal}. Therefore it converges to $1$ as $\eps\to0$ by the 
dominated convergence theorem. The result then follows by one-dimensional 
Laplace asymptotics for the integral over $y_0$, since $V_0$ has quadratic 
minima in $\pm\sqrt{N}$. The fact that the error has order $\eps$ instead of 
$\sqrt{\eps}$ is due to the fact that the potential is even in $y_0$.
\end{proof}

\begin{remark}
We have not claimed that the error term $\Order{\eps}$ is~\eqref{eq:bound_Z} is 
uniform in $N$. In fact, this is indeed the case, but proving it needs a little 
bit more work. We will come back to this point in \Cref{sec:1dmeta}. 
\end{remark}

To simplify the computation of the capacity, we will assume that the sets $A$ 
and $B$ are of the form 
\begin{equation}
\label{eq:condAB} 
 A = -B = \bigsetsuch{y}{\abs{y_0+\sqrt{N}}\leqs\delta, y_\perp\in D_\perp}, 
\end{equation} 
where $\delta\ll\sqrt{N}$ and $D_\perp$ is a ball sufficiently large for $A\cup 
B$ to contain most of the mass of the invariant measure $\pi$, in the sense 
that $\pi((A\cup B)^c) = \Order{\eps}$. This holds for $D_\perp$ of radius of 
order $\sqrt{\log(\eps^{-1})}$, see for 
instance~\cite[Lemma~5.9]{Berglund_DiGesu_Weber_16}. 

\begin{proposition}
For $A$ and $B$ satisfying~\eqref{eq:condAB}, one has 
\begin{equation}
\label{eq:bound_capAB} 
 \capacity(A,B) 
 = \frac{1}{2\pi} \sqrt{\frac{(2\pi\eps)^N\abs{\mu^N_0}}{2\det Q_\perp(0)}} 
 \e^{-N/(4\eps)} \bigbrak{1+\Order{\eps}}
 = \frac{\abs{\mu^N_0}}{2\pi} 
 \sqrt{\frac{(2\pi\eps)^N}{\prod_{k=0}^{N-1}\mu^N_k}} 
 \e^{-N/(4\eps)}
\bigbrak{1+\Order{\eps}}\;,
\end{equation} 
where the $\mu^N_k$ are the eigenvalues~\eqref{eq:muNk} of the Hessian of 
$V(0)$. 
\end{proposition}
\begin{proof}
We will apply the Dirichlet and Thomson principles with appropriate choices of 
$h$ and $\ph$. For the upper bound, we use a function depending only on the 
coordinate $y_0$, and which is simply given by the one-dimensional 
committor~\eqref{eq:hAB_1d} derived in \Cref{exo:1d-committor}:
\begin{equation}
 h(y) = h_0(y_0) = \frac{1}{c_0} \int_y^a \e^{V_0(\xi)/\eps}\6\xi\;, 
 \qquad 
 c_0 := \int_{-a}^a \e^{V_0(\xi)/\eps}\6\xi\;, 
\end{equation} 
where $a=\sqrt{N}-\delta$ is the left boundary of $B$. For $\abs{y_0}>a$, 
$h_0(y_0)$ is continuously extended by constant values $1$ or $0$. Inserting 
this in the Dirichlet form, we obtain 
\begin{equation}
 \cE(h) = \frac{\eps}{c_0^2} \eps^{(N-1)/2} 
 \int_{-a}^a \e^{V_0(y_0)/\eps}
 \int_{\R^\Lambda_\perp} \e^{-\pscal{y_\perp}{Q_\perp(y_0)y_\perp}/2}
 \e^{-R_\eps(y_0,y_\perp)}
 \6y_\perp \6y_0\;.
\end{equation} 
Observe that the sign of $V_0$ has changed in the exponent with respect 
to~\eqref{eq:proof_Z1}, which means that $y_0$ close to $0$ will now dominate 
the integral. Writing again the integral over $\R^\Lambda_\perp$ as an 
expectation under the Gaussian measure $g_N$ yields an upper bound of the 
desired form, noting that $c_0 = \e^{V_0(0)/\eps} 
[2\pi\eps/\abs{\mu^N_0}]^{1/2}[1+\Order{\eps}]$ and $V_0(0)=N/4$. 

For the lower bound, we apply the Thomson principle with the unit 
flow\footnote{We could have included the quartic part of the potential in the 
exponent as well, yielding a better control on the decay for large $y_\perp$. 
This is, however, not needed if we do not care about uniformity in $N$.}
\begin{equation}
 \ph(y) = \frac{1}{K} \indicator{D_\perp}(y_\perp) 
\e^{-\pscal{y_\perp}{Q_\perp(0)y_\perp}/2} e_0\;, 
\qquad 
K := \eps^{(N-1)/2} \int_{D_\perp} 
\e^{-\pscal{y_\perp}{Q_\perp(0)y_\perp}/2}\6y_\perp\;.
\end{equation} 
Since $\ph$ depends only on $y_\perp$ and is directed along $e_0$, it is indeed 
divergence-free. In addition, it has intensity $1$ by definition of $K$. We 
obtain 
\begin{equation}
 \cD(\ph) = \frac{1}{\eps K} \int_{-a}^a \e^{V_0(y_0)/\eps}
 \bigexpecin{g_N(0)}{\indicator{D_\perp} 
 \e^{\pscal{\cdot}{[Q_N(\ph_0)-Q_\perp(0)]\cdot}/2}
 \e^{R_\eps(y_0,y_\perp)}} \6y_0\;.
\end{equation} 
The expectation is bounded because $D_\perp$ is bounded, and equal to 
$1+\Order{\eps}$ thanks to the assumption on $D_\perp$. The result then follows 
by computing the Gaussian integral $K$ and performing Laplace asymptotics on 
the integral over $y_0$. 
\end{proof}

Combining the above estimates, we obtain the following sharp asymptotics on 
the transition time, which is the main result of this section.

\begin{theorem}[Eyring--Kramers law for the double-well situation]
Assume $\gamma > \gamma_1(N)$ and $B$ is as in~\eqref{eq:condAB}. Then 
\begin{equation}
\label{eq:Eyring_Kramers_finite-dim} 
 \bigexpecin{y^*_-}{\tau_B} 
 = \frac{2\pi}{\abs{\mu^N_0}} \sqrt{\frac{\abs{\det\Hess V(0)}}{\det\Hess 
V(y^*_-)}} \e^{[V(0)-V(y^*_-)]/\eps} \bigbrak{1+\Order{\eps}}\;,
\end{equation} 
where $\Hess V(y)$ denotes the Hessian matrix of $V$ at $y$. 
\end{theorem}
\begin{proof}
The result when starting with the distribution $\nu_{AB}$ follows directly 
by inserting~\eqref{eq:bound_Z} and~\eqref{eq:bound_capAB} in the exact 
relation obtained in \Cref{thm:magic_formula}. To extend this to solutions 
starting in $y^*_-$, there are two possibilities. One of them is to use Harnack 
inequalities, which bound the oscillation of harmonic functions, to show that 
$w_B(y)$ does not depend too badly on $y$, cf.~\cite[Lemma~4.6]{BEGK}. An 
alternative is to use a coupling argument as 
in~\cite{Martinelli_Olivieri_Scoppola_89}. 
\end{proof}

Note that~\eqref{eq:Eyring_Kramers_finite-dim} is indeed a generalisation to 
higher dimension of the one-dimensional expression~\eqref{eq:Kramers} obtained 
in \Cref{exo:1dKramers}. 

\begin{remark}
Similar results as~\eqref{eq:Eyring_Kramers_finite-dim} hold in much more 
general finite-dimensional situations, with a less sharp control of the error 
term, including situations with more than $2$ wells~\cite{BEGK}. Furthermore, 
the spectral gap of the generator $\cL$ can be shown to be exponentially close 
to the inverse of the expected transition 
time~\eqref{eq:Eyring_Kramers_finite-dim}~\cite{BGK}. 
\end{remark}

%%%%%%%%%%%%%%%%%%%%%%%%%%%%%%%%%%%%%%%%%%%%%%%%%%%%%%%%%%%%%%%%%%%%%%%%%%%%%%%%

\section{Bibliographical notes}
\label{sec:diffbib} 

The system~\eqref{eq:SDE} of coupled diffusions was introduced 
in~\cite{BFG06a,BFG06b} to understand the general theory of metastability in a 
specific example. These works provide a number of results on the potential 
landscape, both for small and large coupling $\gamma$, and asymptotic results 
for large $N$. In particular, \Cref{prop:det_potential} 
is~{\cite[Proposition~2.2]{BFG06a}}.

General results on solutions of SDEs can be found in the 
monographs~\cite{McKean69,Oeksendal,KaratzasShreve,Mao_book}. 
The use of Lyapunov functions to prove non-explosion, Harris recurrence, and 
various ergodicity results has been developed by Meyn and Tweedie in a series 
of works~\cite{Meyn_Tweedie_92,Meyn_Tweedie_1993a,Meyn_Tweedie_1993b}, as well 
as the monograph~\cite{MeynTweedie_book}.

The theory of large deviations for SDEs is developed by Freidlin and Wenzell 
in the monograph~\cite{FW}. Other general monographs on large deviations 
include~\cite{DZ,DS}. 

The potential-theoretic approach to metastability was mainly developed 
in~\cite{BEGK_MC} for Markov chains, and in~\cite{BEGK,BGK,Eckhoff05} for 
reversible diffusions. A comprehensive account of the potential-theoretic 
approach can be found in the monograph~\cite{Bovier_denHollander_book}. Short 
overviews are also found in~\cite{Slowik_12,Berglund_irs_MPRF}. The Thomson 
principle is proved (in a more general, non-reversible setting) 
in~\cite{Landim_Mariani_Seo_17}.  

Another successful approach to sharp asymptotics for metastable transition 
times is based on semiclassical analysis of the Witten Laplacian, and was 
initiated in~\cite{HelfferKleinNier04,HelfferNier05}. Extensions can be found, 
e.g., in~\cite{LePeutrec_2010,LePeutrec_2011,LPNV_2013}. 

There exist several extensions of the results on the Eyring--Kramers 
formula presented here. The case of saddles with vanishing Hessian determinant 
has been considered in~\cite{Berglund_Gentz_MPRF}. Situations with many 
degeneracies due to symmetries have been considered in~\cite{BD15} for 
markovian jump processes, and~\cite{Dutercq_thesis,BD16} for diffusions.  
Uniformity in $N$ of the error terms in the Eyring--Kramers formula for the 
system~\eqref{eq:SDE} was obtained in~\cite{BarretBovierMeleard}. Results on 
the spectral gap that are uniform in $N$ have been obtained 
in~\cite{DiGesu_LePeutrec17}.

%%%%%%%%%%%%%%%%%%%%%%%%%%%%%%%%%%%%%%%%%%%%%%%%%%%%%%%%%%%%%%%%%%%%%%%%%%%%%%%%

\chapter{Allen--Cahn SPDE in one space dimension}
\label{ch:dim1} 

Consider the formal limit of the system~\eqref{eq:SDE} of coupled diffusions as 
$N\to\infty$ with $\gamma\sim N^2$. Given a parameter $L>0$, that we will 
choose below as a function of the limit of $\gamma/N^2$, we define a function 
$\phi(t,x)$ by setting 
\begin{equation}
 y^i_t = \phi\biggpar{t, \frac{i}{N}L}\;, 
 \qquad 
 i\in\Lambda
\end{equation}
and interpolating linearly (or with some higher-order polynomials) between 
lattice points. The discrete Laplacian then formally satisfies, for $x = 
(i/N)L \in (0,L]$, 
\begin{align}
 y^{i+1}_t - 2y^i_t + y^{i-1}_t
 &= \phi\biggpar{t, x + \frac{L}{N}} - 2 \phi(t,x) + 
\phi\biggpar{t, x-\frac{L}{N}} \\
 &= \frac{L^2}{N^2} \partial_{xx}\phi(t,x) 
 + \BigOrder{\frac{L^4}{N^4}}\;.
\end{align} 
Taking $L$ such that 
\begin{equation}
 L^2 = \lim_{N\to\infty} \frac{2N^2}{\gamma}\;, 
\end{equation} 
we see that the SDE~\eqref{eq:SDE} converges formally, as $N\to\infty$, to the 
equation 
\begin{equation}
\label{eq:AC-1d} 
 \partial_t \phi(t,x) = \partial_{xx}\phi(t,x) + \phi(t,x) - \phi(t,x)^3 
 + \sqrt{2\eps} \xi(t,x)\;,
\end{equation} 
where $\xi(t,x)$ is a stochastic process called \emph{space-time white noise}, 
which we will have to properly define. In what follows, we will write $\Delta 
\phi$ instead of $\partial_{xx}\phi$, even though $x$ is one-dimensional, 
because the same notation will apply for higher-dimensional $x$. Note in 
particular that for the critical value $\gamma_1(N)$ of $\gamma$ obtained in 
\Cref{prop:det_potential}, we have 
\begin{equation}
 \lim_{N\to\infty} \frac{2N^2}{\gamma_1(N)} = (2\pi)^2\;,
\end{equation} 
indicating that the value $L=2\pi$ will play a special role. 

\begin{remark}
\label{rem:Phi4} 
The equation with a negative coefficient in front of $\phi$ 
\begin{equation}
\label{eq:phi4-1d} 
 \partial_t \phi(t,x) = \Delta\phi(t,x) - m^2\phi(t,x) - \phi(t,x)^3 
 + \sqrt{2\eps} \xi(t,x)
\end{equation} 
(where $m^2\geqs 0$) is called the \emph{$\Phi^4$ 
model}~\cite{Glimm_Jaffe_68,Glimm1975,Glimm_Jaffe_81} (massive $\Phi^4$ model if 
$m>0$) or \emph{stochastic quantisation equation}~\cite{Parisi_Wu}, and plays an 
important role in Quantum Field Theory. The solution theory is the same 
for~\eqref{eq:AC-1d} and~\eqref{eq:phi4-1d}, but their long-time behaviour is 
very different, since~\eqref{eq:AC-1d} displays metastability 
while~\eqref{eq:phi4-1d} does not.
\end{remark}

%%%%%%%%%%%%%%%%%%%%%%%%%%%%%%%%%%%%%%%%%%%%%%%%%%%%%%%%%%%%%%%%%%%%%%%%%%%%%%%%

\section{Deterministic dynamics}
\label{sec:1ddet} 

We start by briefly analysing~\eqref{eq:AC-1d} in the deterministic case 
$\eps=0$, where it takes the form of the PDE 
\begin{equation}
\label{eq:AC-1d-det} 
 \partial_t \phi(t,x) = \Delta\phi(t,x) + \phi(t,x) - \phi(t,x)^3\;. 
\end{equation} 
This equation is commonly known as \emph{Allen--Cahn equation}, though one 
finds other names in the literature, including \emph{Chaffee--Infante 
equation}, or \emph{real Ginzburg--Landau equation}.

In the present setting, $x$ belongs to the scaled circle $\Lambda=\R/(L\Z)$, 
which implicitly implies that we consider~\eqref{eq:AC-1d-det} with periodic 
boundary conditions. 

A first useful observation is that the right-hand side of~\eqref{eq:AC-1d-det} 
derives again from a potential, obtained as the continuum limit of the 
potential~\eqref{eq:V_finite}, which is given by 
\begin{equation}
\label{eq:potential_infdim} 
 V(\phi) := \int_0^L \biggbrak{\frac12 \norm{\nabla\phi(x)}^2 + 
 \frac14 \Bigpar{\phi(x)^2 -1}^2}\;.
\end{equation} 
Here we have written $\nabla\phi(x)$ instead of $\partial_x\phi(x)$, to have a 
notation also valid in higher dimensions. Indeed, the G\^ateaux derivative of 
$V$ at $\phi$ in the direction $\psi$ is given by 
\begin{align}
\nabla_\psi V(\phi) 
&:= \dpar{}{\lambda} V(\phi+\lambda\psi) \Bigr\vert_{\lambda=0} \\
&= \int_0^L \bigbrak{\nabla\phi(x) \cdot \nabla\psi(x) - \phi(x)\psi(x) + 
\phi(x)^3\psi(x)} \6x \\
&= - \pscal{\Delta\phi+\phi-\phi^3}{\psi}_{L^2}\;,
\end{align}
where we have used integration by parts and the periodic boundary conditions in 
the last step. In particular, this shows that stationary solutions 
of~\eqref{eq:AC-1d-det} are critical points of $V$. 

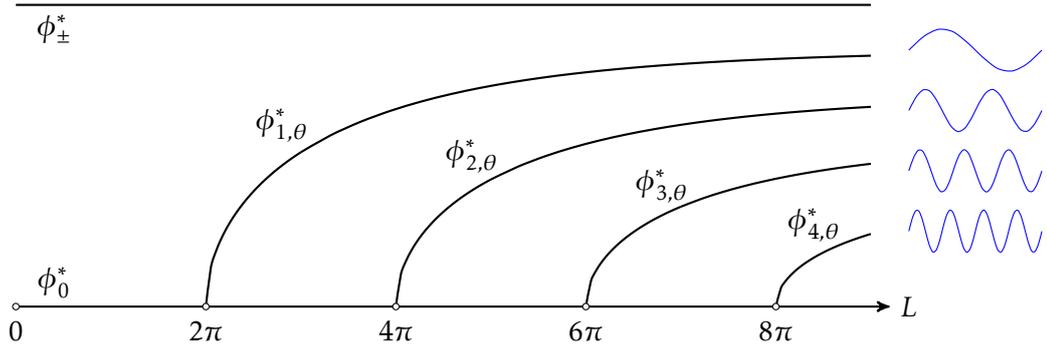
\begin{figure}[tb]
\begin{center}
\begin{tikzpicture}[>=stealth',main
node/.style={circle,inner sep=0.03cm,fill=white,draw},x=2.5cm,y=4cm,
declare function={
branch(\x,\a) = (1 - 0.15*\a )*sqrt(1 - exp(0.9*(-\x + \a )));
kink(\x,\w) = 0.07*(sin(\w * 180*\x + sin(3* \w * 180*\x)/3
+ sin(5* \w * 180*\x)/5);
}]

\draw[->,thick] (0,0) -- (4.6,0);
\draw[-,thick] (0,1) -- (4.5,1);

% curve starting in \pi

\draw[black,thick,-,smooth,domain=1:1.5,samples=20,/pgf/fpu,
/pgf/fpu/output format=fixed] plot (\x,{branch(\x,1)});
\draw[black,thick,-,smooth,domain=1.5:4.5,samples=20,/pgf/fpu,
/pgf/fpu/output format=fixed] plot (\x,{branch(\x,1)});

% curve starting in 2\pi

\draw[black,thick,-,smooth,domain=2:2.5,samples=20,/pgf/fpu,
/pgf/fpu/output format=fixed] plot (\x,{branch(\x,2)});
\draw[black,thick,-,smooth,domain=2.5:4.5,samples=15,/pgf/fpu,
/pgf/fpu/output format=fixed] plot (\x,{branch(\x,2)});

% curve starting in 3\pi

\draw[black,thick,-,smooth,domain=3:3.5,samples=20,/pgf/fpu,
/pgf/fpu/output format=fixed] plot (\x,{branch(\x,3)});
\draw[black,thick,-,smooth,domain=3.5:4.5,samples=10,/pgf/fpu,
/pgf/fpu/output format=fixed] plot (\x,{branch(\x,3)});

% curve starting in 4\pi

\draw[black,thick,-,smooth,domain=4:4.5,samples=20,/pgf/fpu,
/pgf/fpu/output format=fixed] plot (\x,{branch(\x,4)});

% kinks with w=1

\draw[blue,-,smooth,domain=0:1,samples=10,/pgf/fpu,
/pgf/fpu/output format=fixed] plot ({4.7 + 0.7*\x}, {0.85 + kink(\x,2)});

% kinks with w=2

\draw[blue,-,smooth,domain=0:1,samples=20,/pgf/fpu,
/pgf/fpu/output format=fixed] plot ({4.7 + 0.7*\x}, {0.65 + kink(\x,4)});

% kinks with w=3

\draw[blue,-,smooth,domain=0:1,samples=30,/pgf/fpu,
/pgf/fpu/output format=fixed] plot ({4.7 + 0.7*\x}, {0.45 + kink(\x,6)});

% kinks with w=3

\draw[blue,-,smooth,domain=0:1,samples=40,/pgf/fpu,
/pgf/fpu/output format=fixed] plot ({4.7 + 0.7*\x}, {0.25 + kink(\x,8)});

\node[main node] at (0,0) {};
\node[main node] at (1,0) {};
\node[main node] at (2,0) {};
\node[main node] at (3,0) {};
\node[main node] at (4,0) {};

\node[] at (0,-0.08) {$0$};
\node[] at (1,-0.08) {$2\pi$};
\node[] at (2,-0.08) {$4\pi$};
\node[] at (3,-0.08) {$6\pi$};
\node[] at (4,-0.08) {$8\pi$};
\node[] at (4.7,0) {$L$};

\node[] at (0.2,0.08) {$\phi^*_0$};
\node[] at (1.4,0.6) {$\phi^*_{1,\theta}$};
\node[] at (2.4,0.5) {$\phi^*_{2,\theta}$};
\node[] at (3.4,0.4) {$\phi^*_{3,\theta}$};
\node[] at (4.2,0.27) {$\phi^*_{4,\theta}$};
\node[] at (0.2,0.92) {$\phi^*_{\pm}$};

\end{tikzpicture}
\vspace{-3mm}
\end{center}
\caption[]{Schematic bifurcation diagram of the deterministic Allen--Cahn 
equation~\eqref{eq:AC-1d-det}. The vertical coordinate represents roughly the 
amplitude of the stationary solution.}
\label{fig:stationary_1d}
\end{figure}

The role of the Hessian of $V$ at $\phi$ is played by the 
bilinear form mapping two periodic functions $\psi_1$, $\psi_2$ to 
\begin{align}
 \nabla^2_{\psi_1,\psi_2} V(\phi) 
&:= \dpar{^2}{\lambda_1\partial\lambda_2} 
 V(\phi+\lambda_1\psi_1+\lambda_2\psi_2)
 \Bigr\vert_{\lambda_1=\lambda_2=0}\\
&= \int_0^L \bigbrak{\nabla\psi_1(x) \cdot \nabla\psi_2(x) - 
\psi_1(x)\psi_2(x) + 3\phi(x)^2\psi_1(x)\psi_2(x)} \6x \\
&= \pscal{\psi_1}{[-\Delta-1+3\phi(\cdot)^2]\psi_2}_{L^2}\;. 
\label{eq:Hessian_AC-1d} 
\end{align} 
This reflects the fact that the linearisation of~\eqref{eq:AC-1d-det} around a 
stationary solution $\phi^*$ is given by the variational equation 
\begin{equation}
\label{eq:Sturm-Liouville} 
 \partial_t \psi(t,x) = \Delta \psi(t,x) + \bigbrak{1-3\phi^*(x)^2} \psi(t,x)\;.
\end{equation} 
Determining the stability of $\phi^*$ is equivalent to solving the 
Sturm--Liouville problem consisting in computing the spectrum 
of~\eqref{eq:Sturm-Liouville} with periodic boundary conditions. 

The PDE~\eqref{eq:AC-1d-det} has three obvious stationary solutions, which are 
constant in space, and natural analogues of the constant stationary solutions 
of the discrete system. We will denote by $\phi^*_\pm$ the constant solutions 
equal to $\pm1$, and by $\phi^*_0$ the constant solution equal to $0$. The 
following result describes the set of all stationary solutions and their 
stability. The bifurcation diagram is sketched in \Cref{fig:stationary_1d}.

\begin{proposition}
\label{prop:AC-1d-det} 
The stationary solutions $\phi^*_\pm$ are always stable. The solution
$\phi^*_0$ has one unstable direction if $0<L<2\pi$, and more generally $2k+1$ 
unstable directions if $2k\pi < L < 2(k+1)\pi$ for any $k\in\N_0$. 
In addition, at every multiple of $2\pi$, a one-parameter family of 
non-constant solutions bifurcates from the origin. These solutions are 
unstable, and of the form $\phi^*_{k,\theta}(x) = \phi^*_{k,0}(x+\theta)$.   
\end{proposition}
\begin{proof}
Setting $\phi=\pm1$ in~\eqref{eq:Hessian_AC-1d}, we find that the Hessian of 
$V$ at $\phi^*_\pm$ is given by $-\Delta+2$. This has eigenvalues 
\begin{equation}
 \nu_k = \biggpar{\frac{2k\pi}{L}}^2 + 2\;, 
 \qquad 
 k\in\Z\;,
\end{equation}
which are all positive. Therefore, the Sturm--Liouville 
equation~\eqref{eq:Sturm-Liouville} has negative spectrum, and $\phi^*_\pm$ are 
stable. In the case of $\phi^*_0$, we obtain a Hessian equal to $-\Delta-1$, 
which has eigenvalues
\begin{equation}
  \mu_k = \biggpar{\frac{2k\pi}{L}}^2 - 1\;, 
 \qquad 
 k\in\Z\;.
\end{equation}
Therefore, $\mu_0$ is always negative, while two additional $\mu_k$ become 
negative whenever $L$ exceeds $2k\pi$. 
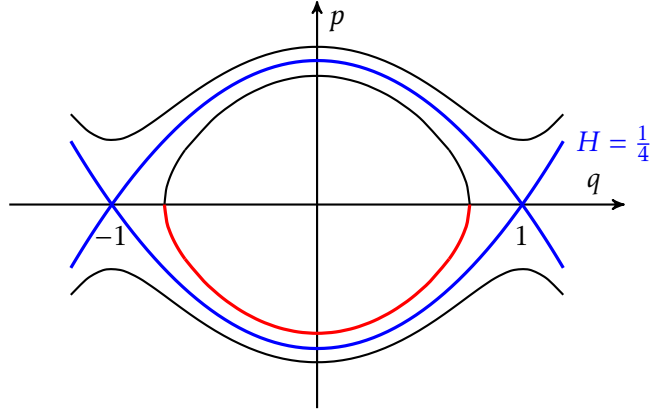
\begin{figure}[tb]
\begin{center}
  \begin{tikzpicture}[>=stealth',main node/.style={circle,minimum
size=0.25cm,fill=white,draw},x=3cm,y=3cm,scale=0.9,
declare function={
pplus(\q,\H) = sqrt(2* \H - \q^2 + 0.5 * \q^4);
qmax(\H) = sqrt(1-sqrt(1-4* \H));
}]

\draw[->,thick] (-1.5,0) -- (1.5,0);
\draw[->,thick] (0,-1) -- (0,1);

% level curve H=1/4 (drawn in several parts for better resolution)

\draw[blue,very thick,-,smooth,domain=-1:1,samples=30,/pgf/fpu,
/pgf/fpu/output format=fixed] plot (\x,{pplus(\x,1/4)});
\draw[blue,very thick,-,smooth,domain=-1:1,samples=30,/pgf/fpu,
/pgf/fpu/output format=fixed] plot (\x,{-pplus(\x,1/4)});

\draw[blue,very thick,-,smooth,domain=-1.2:-1,samples=10,/pgf/fpu,
/pgf/fpu/output format=fixed] plot (\x,{pplus(\x,1/4)});
\draw[blue,very thick,-,smooth,domain=-1.2:-1,samples=10,/pgf/fpu,
/pgf/fpu/output format=fixed] plot (\x,{-pplus(\x,1/4)});

\draw[blue,very thick,-,smooth,domain=1:1.2,samples=10,/pgf/fpu,
/pgf/fpu/output format=fixed] plot (\x,{pplus(\x,1/4)});
\draw[blue,very thick,-,smooth,domain=1:1.2,samples=10,/pgf/fpu,
/pgf/fpu/output format=fixed] plot (\x,{-pplus(\x,1/4)});

\node[blue] at (1.45,0.3) {$H=\frac14$};

%\node[blue] at (-0.8,0.9) {$H=\frac12(u')^2 + \frac12u^2 - \frac14u^4$};

% level curve H=0.3

\draw[black,thick,-,smooth,domain=-1.2:1.2,samples=30,/pgf/fpu,
/pgf/fpu/output format=fixed] plot (\x,{pplus(\x,0.3)});
\draw[black,thick,-,smooth,domain=-1.2:1.2,samples=30,/pgf/fpu,
/pgf/fpu/output format=fixed] plot (\x,{-pplus(\x,0.3)});

% level curve H=0.2 (drawn in several parts for better resolution)

\pgfmathsetmacro{\qmaxx}{qmax(0.2)}
\pgfmathsetmacro{\qmed}{qmax(0.2)-0.1}

\draw[black,thick,-,smooth,domain={-\qmaxx}:{-\qmed},samples=15,/pgf/fpu,
/pgf/fpu/output format=fixed] plot (\x,{pplus(\x,0.2)});
\draw[black,thick,-,smooth,domain={-\qmed}:\qmed,samples=10,/pgf/fpu,
/pgf/fpu/output format=fixed] plot (\x,{pplus(\x,0.2)});
\draw[black,thick,-,smooth,domain=\qmed:\qmaxx,samples=10,/pgf/fpu,
/pgf/fpu/output format=fixed] plot (\x,{pplus(\x,0.2)}) -- (\qmaxx,0);

% red part of level curve H=0.2 (drawn in several parts for better resolution)

\draw[red,very thick,-,smooth,domain={-\qmed}:\qmed,samples=15,/pgf/fpu,
/pgf/fpu/output format=fixed] plot (\x,{-pplus(\x,0.2)});

\draw[red,very thick,-,smooth,domain={-\qmed}:{-\qmaxx},samples=10,/pgf/fpu,
/pgf/fpu/output format=fixed] plot (\x,{-pplus(\x,0.2)}) -- ({-\qmaxx},0);
\draw[red,very thick,-,smooth,domain=\qmed:\qmaxx,samples=10,/pgf/fpu,
/pgf/fpu/output format=fixed] plot (\x,{-pplus(\x,0.2)}) -- (\qmaxx,0);

% axis labels etc

\node[] at (1.35,0.1) {$q$};
\node[] at (0.1,0.9) {$p$};
\node[] at (-1,-0.15) {$-1$};
\node[] at (1,-0.15) {$1$};

\end{tikzpicture}
\vspace{-3mm}
\end{center}
\caption[]{Solutions of the Hamilton equations~\eqref{eq:Hamilton}. Closed 
orbits correspond to periodic boundary conditions, while the orbit shown in red 
corresponds to Neumann boundary conditions. The curves $\set{H=\frac14}$ are 
separatrices reached in the limit of the period going to infinity.}
\label{fig:Hamiltonian}
\end{figure}
The other stationary solutions are periodic solutions of 
\begin{equation}
 \phi''(x) = \phi(x)^3 - \phi(x)
\end{equation} 
of period $L$. 
This second-order ODE is equivalent to the Hamiltonian system
\begin{align}
q'(x) &= p(x) \\
p'(x) &= q(x)^3 - q(x)\;,
\label{eq:Hamilton} 
\end{align}
which derives from the Hamiltonian
\begin{equation}
 H(p,q) = \frac12 p^2 + \frac12 q^2 - \frac14 q^4\;.
\end{equation} 
The energy $H$ is constant along solutions of~\eqref{eq:Hamilton}, and 
therefore orbits are contained in level curves of $H$. Some level curves are 
shown in \Cref{fig:Hamiltonian}. Non-constant stationary solutions 
of~\eqref{eq:AC-1d-det} correspond to closed level curves whose period is of 
the form $L/k$ for some $k\in\N$ (as they can be tracked multiple times). 
Expressing $p=q'$ in terms of $H$ and $q$ and integrating, one obtains that the 
period of the closed level curve $\set{H(q,p)=E}$ is given by 
\begin{equation}
 T(E) = 2\int_{q_-(E)}^{q_+(E)} \frac{\6q}{\sqrt{2E-q^2+\frac12q^4}}\;,
 \qquad 
 q_\pm(E) = \pm \sqrt{1-\sqrt{1-4E}}\;,
\end{equation} 
which is a Jacobi elliptic integral that is known to be an increasing function 
of $E$, starting at $2\pi$ for $E=0$. 

It follows that whenever $L$ is a multiple of $2\pi$, a new family of 
non-constant stationary solutions bifurcates from $\phi^*_0$. This is a family 
of solutions because one can choose the origin $x=0$ anywhere on the closed 
level curve. The fact that these solutions are unstable has topological reasons 
(conservation of the degree of stationary points of functions depending on a 
parameter), and can be checked locally by a bifurcation analysis. 
\end{proof}

In what follows, we will again mostly concentrate for simplicity on the case 
$0<L<2\pi$, when $V$ is a double-well potential with two local minima located 
at $\phi^*_\pm$ and a saddle at the origin. However, the results hold in more 
generality. 

\begin{exercise}
Discuss the case of zero-flux Neumann boundary conditions 
\begin{equation}
 \dpar{\phi}{x}(0) = \dpar{\phi}{x}(L) = 0\;.
\end{equation} 
This setting corresponds to the orbit shown in red in \Cref{fig:Hamiltonian}. 
Derive the bifurcation diagram representing the stationary solutions and their 
stability as a function of $L$. 
\end{exercise}

%%%%%%%%%%%%%%%%%%%%%%%%%%%%%%%%%%%%%%%%%%%%%%%%%%%%%%%%%%%%%%%%%%%%%%%%%%%%%%%%

\section{Space-time white noise}
\label{sec:1dnoise}

We now turn to the precise definition of the space-time white noise process 
$\xi$ formally introduced in~\eqref{eq:AC-1d}. It should have the following 
properties:
\begin{itemize}
\item 	each $\xi(t,x)$ should be a Gaussian random variable;
\item 	each $\xi(t,x)$ should be centred;
\item 	the values of $\xi$ at different space-time points should be 
independent, and thus uncorrelated, which is sometimes written informally as 
\begin{equation}
\label{eq:xi_cov_formal} 
 \bigexpec{\xi(t,x)\xi(s,y)} = \delta(t-s)\delta(x-y)\;.
\end{equation}
\end{itemize}

It turns out that there is no random function having the above properties, but 
that there exists a random Schwartz distribution doing the job, i.e., one can 
consider $\xi$ as a random linear functional acting on test functions. 
We will denote by $\cH$ be the Hilbert space $L^2(\R\times\Lambda)$, where we 
recall that $\Lambda=\R/(L\Z)$. Let $\cS'(\cH)$ be the space of Schwartz 
distributions, and denote by $\pscal{\zeta}{\ph}$ the duality pairing between a 
distribution $\zeta\in\cS'(\cH)$ and a test function $\ph\in\cH$. 

\begin{definition}[Space-time white noise]
\label{def:space-time-white-noise} 
Space-time white noise on $\R\times\Lambda$ is a random distribution $\xi$ on a 
probability space $(\Omega,\cF,\fP)$ such that for any smooth test function 
$\ph\in\cH$, $\pscal{\xi}{\ph}$ is a centred Gaussian random variable of 
variance $\norm{\ph}_{\cH}^2$, and the covariances are given by 
\begin{equation}
\label{eq:xi_cov} 
 \bigexpec{\pscal{\xi}{\ph_1}\pscal{\xi}{\ph_2}}
 = \pscal{\ph_1}{\ph_2}_{\cH}
\end{equation} 
for any two smooth test functions $\ph_1, \ph_2\in\cH$. 
\end{definition}

The fact that such a process $\xi$ exists is a consequence of general results 
on Gaussian measure theory. Details can be found, for instance, in Chapter~3 of 
Martin Hairer's lecture notes~\cite{Hairer_LN_2009}. 

Note that we recover the expression~\eqref{eq:xi_cov_formal} by formally 
replacing the test functions $\ph_1$ and $\ph_2$ by Dirac distributions. Let us 
consider a couple more examples.

\begin{example}
%\item 	
Let $A_1, A_2\subset\Lambda$ be open sets, fix $T_1,T_2>0$ and 
take 
\begin{equation}
 \ph_i(t,x) = \indexfct{0\leqs t\leqs T_i} \indexfct{x\in A_i}\;, 
 \qquad
 i= 1,2\;.
\end{equation}
These test functions are not smooth, but can be obtained as limits of smooth 
functions. Define 
\begin{equation}
 W^{A_i}_{T_i} = \pscal{\xi}{\ph_i}
 = \int_0^{T_i}\int_{A_i}  \xi(t,x)\6x\6t
 \qquad
 i= 1,2
\end{equation}
(where the last integral is of course formal). 
Then~\eqref{eq:xi_cov} takes the form 
\begin{equation}
 \bigexpec{W^{A_1}_{T_1}W^{A_2}_{T_2}} = (T_1\wedge T_2) \abs{A_1\cap A_2}\;,
\end{equation} 
where $\abs{\cdot}$ denotes the Lebesgue measure of a set. Thus the 
$W^{A_i}_{T_i}$ behave like Brownian motions, which are independent if the sets 
$A_1$ and $A_2$ are disjoint. 
\end{example}

%\item 	
\begin{example}
\label{ex:noise_fourier} 
Let $\set{e_k}_{k\in\Z}$ be an orthonormal basis of $L^2(\Lambda)$, for 
instance a Fourier basis, fix $T_1,T_2>0$ and 
take 
\begin{equation}
 \ph_i(t,x) = \indexfct{0\leqs t\leqs T_i} e_{k_i}(x)\;, 
 \qquad
 i= 1,2\;.
\end{equation}
Let 
\begin{equation}
 W^{(k_i)}_{T_i} = \pscal{\xi}{\ph_i}
 = \int_0^{T_i} \pscal{\xi(t,\cdot)}{e_{k_i}} \6t
 \qquad
 i= 1,2\;.
\end{equation}
Then~\eqref{eq:xi_cov} yields 
\begin{equation}
 \bigexpec{W^{(k_1)}_{T_1}W^{(k_2)}_{T_2}} = (T_1\wedge T_2) \delta_{k_1k_2}\;,
\end{equation} 
showing that the $W^{(k_i)}_T$ are independent Brownian motions. 
\end{example}

An important property of space-time white noise is related to scaling, 
in a way that should be reminiscent of Brownian motion. Given $\tau,\lambda>0$, 
define a scaling operator acting on test functions by 
\begin{equation}
\label{eq:S-tau-lambda} 
 (\cS^{\tau,\lambda}\ph)(t,x)
 := \frac{1}{\tau\lambda} \ph \biggpar{\frac{t}{\tau},\frac{x}{\lambda}}\;.
\end{equation} 
Then we define a scaled version $\xi_{\tau,\lambda}$ of space-time white noise 
by 
\begin{equation}
\label{eq:xi_tau_lambda} 
 \pscal{\xi_{\tau,\lambda}}{\ph} = \pscal{\xi}{\cS^{\tau,\lambda}\ph}
\end{equation} 
for any test function $\ph$.

\begin{exercise}
Show that if $\xi$ is replaced by a smooth function $f$ 
in~\eqref{eq:xi_tau_lambda}, then one simply has $f_{\tau,\lambda}(t,x) = 
f(\tau t,\lambda x)$.
\end{exercise}

\begin{proposition}[Scaling property of space-time white noise]
\label{prop:xi_scaling}
We have 
\begin{equation}
 \xi_{\tau,\lambda} \eqinlaw \frac{1}{\sqrt{\tau\lambda}}\xi\;,
\end{equation} 
where $\eqinlaw$ denotes equality in distribution. 
\end{proposition}
\begin{proof}
Both processes are Gaussian and centred. Therefore, it suffices to show that 
they have the same covariance. Given two compactly supported test functions 
$\ph_1, \ph_2$, we have 
by~\eqref{eq:xi_cov} 
\begin{align}
\bigexpec{\pscal{\xi_{\tau,\lambda}}{\ph_1} \pscal{\xi_{\tau,\lambda}}{\ph_2}}
&= \bigexpec{\pscal{\xi}{\cS^{\tau,\lambda}\ph_1} 
{\pscal{\xi}{\cS^{\tau,\lambda}\ph_2}}} \\
&= \int_{-\infty}^\infty \int_\Lambda (\cS^{\tau,\lambda}\ph_1)(t,x) 
(\cS^{\tau,\lambda}\ph_2)(t,x) \6x\6t \\
&= \frac{1}{\tau\lambda} \pscal{\ph_1}{\ph_2}_\cH\;.
\label{eq:proof_xi_scaling} 
\end{align}
This is indeed the covariance of $\xi/\sqrt{\lambda\tau}$. 
\end{proof}

In what follows, it will be important to find some functions spaces to which 
the relevant objects belong. A first important family of such spaces 
are scaled versions of the classical H\"older spaces. The scaling comes from 
the fact that it turns out to be useful to work with the \emph{parabolic 
distance} 
\begin{equation}
 \label{eq:parabolic_distance-1d}
 \norm{(t,x) - (s,y)}_\fraks := \abs{t-s}^{1/2} + \abs{x-y}
\end{equation} 
between space-time points, owing to the parabolic nature of the 
SPDE~\eqref{eq:AC-1d}, which is first-order in time and second-order in space. 
This is not a genuine distance because it fails to satisfy the triangle 
inequality, but this will not be important. 

\begin{definition}[Parabolic H\"older space]
\label{def:Holder_0-1} 
For $0<\alpha<1$, the parabolic H\"older space 
$\cC^\alpha_\fraks(\R\times\Lambda)$ consists of all continuous functions 
$f:\R\times\Lambda\to\R$ such that for every compact $\fK\subset\R\times\Lambda$
\begin{equation}
 \norm{f}_{\cC^\alpha_\fraks(\fK)}
 := \sup_{z\in\fK} \bigabs{f(z)} + 
 \sup_{\substack{z_1, z_2\in\fK \\ z_1\neq z_2}}
 \frac{\abs{f(z_1)-f(z_2)}}{\norm{z_1-z_2}_\fraks^\alpha} < \infty\;.
\end{equation} 
\end{definition}

There exists an extension of these spaces to negative index $\alpha$, which are 
a scaled version of a particular class of Besov spaces, commonly denoted 
$\cB^\alpha_{\infty,\infty}$. Since these are closely related to the above 
H\"older spaces, as we will see in the next section, we keep the notation 
$\cC^\alpha_\fraks$ as in~\cite{Hairer2014}. To define these spaces, we 
slightly adapt the notation~\eqref{eq:S-tau-lambda} to allow for zooming in 
at any space-time point, setting 
\begin{equation}
\label{eq:S-lambda-z} 
 (\cS^{\lambda}_{t,x}\ph)(s,y)
 := \frac{1}{\lambda^3} \ph 
\biggpar{\frac{s-t}{\lambda^2},\frac{y-x}{\lambda}}\;.
\end{equation} 
Given a positive integer $r$, we further let $B_r$ be the set of smooth test 
functions $\ph:\R\times\Lambda\to\R$ supported in the unit 
$\norm{\cdot}_\fraks$-ball, such that  
\begin{equation}
\label{eq:def_Br} 
 \norm{\ph}_{\cC^r} := \sup_{k \colon \abs{k}_\fraks \leqs r} 
\sup_{z\in\R\times\Lambda} \bigabs{\DD^k \ph(z)} \leqs 1\;.
\end{equation} 
Here $k=(k_0,k_1)\in\N_0^2$ is a multiindex, $\abs{k}_\fraks := 
2k_0+k_1$ is its scaled norm, and $\DD^k := 
\partial_t^{k_0}\partial_x^{k_1}$. 

\begin{definition}[Parabolic negative index Besov--H\"older space]
\label{def:Holder-negative} 
For $\alpha < 0$, the space $\cC^\alpha_\fraks(\R\times\Lambda)$ consists of 
all 
Schwartz distributions $\zeta\in\cS'(\R\times\Lambda)$ such that for every 
compact $\fK\subset\R\times\Lambda$
\begin{equation}
 \norm{\zeta}_{\cC^\alpha_\fraks(\fK)}
 := \sup_{z\in\fK} \sup_{\ph\in B_r} \sup_{\lambda\in(0,1]}
 \biggabs{\frac{\pscal{\zeta}{\cS^\lambda_z\ph}}{\lambda^\alpha}} < \infty\;,
\end{equation} 
where $r=\intpartplus{-\alpha}$. 
\end{definition}

In the case where $\zeta=\xi$ is space-time white noise, it follows 
from~\eqref{eq:proof_xi_scaling} that 
\begin{equation}
\label{eq:xi_moments} 
 \bigexpec{\pscal{\xi}{\cS^\lambda_z\ph}^2} 
 = \frac{1}{\lambda^3} \norm{\ph}_\cH^2\;.
\end{equation} 
Since for Gaussian random variables, second moments determine all other 
moments, this suggests that $\xi$ might belong to $\cC^{-3/2}_\fraks$. The 
following important result shows that this is almost the case. 

\begin{theorem}[Regularity of space-time white noise]
\label{thm:space_time_white_noise_reg} 
We have $\xi \in \cC^{-3/2-\kappa}_\fraks$ for any $\kappa>0$. 
\end{theorem}
\begin{proof}[\Sketch]
We follow the argument outlined in~\cite[Theorem~2.7]{Chandra_Weber_LN17}.
First recall the fact that if $X$ is a Gaussian random variable, then there 
exists, for every integer $p>0$, a finite constant $C_p$ such that 
\begin{equation}
\label{eq:moments_Gaussian} 
 \bigexpec{\abs{X}^p} \leqs C_p \Bigpar{\bigexpec{X^2}}^{p/2}
\end{equation} 
(see also \Cref{exo:Isserlis}). 
It thus follows from~\eqref{eq:xi_moments} that for every integer $p>0$,  
\begin{equation}
\label{eq:proof_xi_moments} 
 \Bigexpec{\bigabs{\pscal{\xi}{\cS^\lambda_z\ph}}^p}
 \leqs C_p \lambda^{-3p/2} \norm{\ph}_\cH^p\;.
\end{equation} 
Fix a compact $\fK \subset \R\times\Lambda$, and let $\bar\fK$ be the set of 
points in $\R\times\Lambda$ at parabolic distance at most $1$ from $\fK$ 
(called the \emph{$1$-fattening} of $\fK$). Given $k\in\N_0$, we introduce a 
dyadic lattice discretisation of $\bar\fK$ on scale $2^{-k}$, given by 
\begin{equation}
 \fK_k = \bigpar{2^{-2k}\Z \times 2^{-k}\Z} \cap \bar\fK\;.
\end{equation} 
The crucial observation is that one can show that for any $\alpha\in\R$, there 
exists a test function $\ph$ and a constant $C$ such that   
\begin{equation}
 \norm{\xi}_{\cC^\alpha_\fraks(\fK)} 
 \leqs C \sup_{k\geqs0} \sup_{z\in\fK_k} 2^{k\alpha} 
 \bigabs{\pscal{\xi}{\cS^{2^{-k}}_z\ph}}\;.
\end{equation} 
The fact that $\ph$ can be chosen independent of $k$ is not trivial, but there 
are various ways of constructing such a $\ph$ (using for instance wavelets, or 
Paley--Littlewood blocks). Bounding the suprema by sums and taking the $p$th 
power, we arrive at 
\begin{equation}
 \Bigpar{\norm{\xi}_{\cC^\alpha_\fraks(\fK)}}^p 
 \leqs C^p \sum_{k\geqs0} \sum_{z\in\fK_k} 2^{k\alpha p} 
 \bigabs{\pscal{\xi}{\cS^{2^{-k}}_z\ph}}^p\;.
\end{equation} 
Note that so far, the argument is purely deterministic. At this point, we take 
expectations and use~\eqref{eq:proof_xi_moments} to arrive at 
\begin{equation}
 \Bigexpec{\Bigpar{\norm{\xi}_{\cC^\alpha_\fraks(\fK)}}^p} 
 \leqs %C^p C_p \norm{\ph}_\cH^p 
 \overline C_p \sum_{k\geqs0} 2^{3k} 2^{k\alpha p} 2^{3kp/2}
\end{equation} 
for some constant $\overline C_p$, where we have used the fact that $\fK_k$ 
contains of the order of $2^{3k}$ points. If $\alpha < \smash{- \frac32 - 
\frac3p}$, one can sum the geometric series, yielding the bound 
\begin{equation}
 \Bigexpec{\Bigpar{\norm{\xi}_{\cC^\alpha_\fraks(\fK)}}^p} < \infty 
 \qquad 
 \forall \alpha < -\frac32 - \frac3p\;.
\end{equation}
It then follows from a version of Kolmogorov's continuity theorem that there 
exists a modification of $\xi$ with bounded $\cC^\alpha_\fraks(\fK)$-norm. 
Taking $p$ large enough yields the result. 
\end{proof}

%%%%%%%%%%%%%%%%%%%%%%%%%%%%%%%%%%%%%%%%%%%%%%%%%%%%%%%%%%%%%%%%%%%%%%%%%%%%%%%%

\section{The stochastic heat equation}
\label{sec:1dheat} 

Before turning to the Allen--Cahn SPDE~\eqref{eq:AC-1d}, we consider the 
linear \emph{stochastic heat equation} on $\R_+\times\Lambda$, given by 
\begin{equation}
\label{eq:stoch_heat} 
\partial_t \phi(t,x) = \Delta\phi(t,x) + \xi(t,x)\;.
\end{equation} 
Here we have taken $2\eps=1$, since the case of general $\eps$ can easily be 
recovered by scaling. 

First recall that the \emph{heat equation} 
\begin{equation}
 \partial_t \phi(t,x) = \Delta\phi(t,x)\;, 
 \qquad \phi(0,x) = \phi_0(x) 
\end{equation} 
admits the solution $\phi(t,x) = (\e^{t\Delta}\phi_0)(x)$, where $\e^{t\Delta}$ 
is the \emph{heat semigroup} defined by 
\begin{equation}
 (\e^{t\Delta}\phi_0)(x) := \int_\Lambda P(t,x-y)\phi_0(y)\6y\;.
\end{equation} 
Here $P(t,x)$ is the \emph{heat kernel}. If $\Lambda$ were equal to $\R$, it 
would be given by  
\begin{equation}
 P_{\,\R}(t,x) := \frac{1}{\sqrt{4\pi t}} \e^{-x^2/(4t)} \indexfct{t>0}\;.
\end{equation} 
Since we are working in $\Lambda=\R/(L\Z)$, however, we have to use a 
periodicised version of the heat kernel, defined by 
\begin{equation}
\label{eq:heat_kernel_1d} 
 P(t,x) = P_\Lambda(t,x) := \sum_{k\in\Z} P_{\,\R}(t,x-kL)\;.
\end{equation} 
Next recall that the \emph{forced heat equation}
\begin{equation}
  \partial_t \phi(t,x) = \Delta\phi(t,x) + f(t,x)\;, 
 \qquad \phi(0,x) = \phi_0(x)\;,
\end{equation} 
where $f$ is some smooth forcing, can be solved by the method of variation of 
constant, also known as \emph{Duhamel principle} in the theory of PDEs. The 
result is 
\begin{align}
 \phi(t,x) 
 &= (\e^{t\Delta}\phi_0)(x) + \int_0^t \bigpar{\e^{(t-s)\Delta}f}(s,x)\6s \\
 &= (\e^{t\Delta}\phi_0)(x) + \int_0^t \int_\Lambda P(t-s,x-y)f(s,y)\6y\6s\;.
\label{eq:Duhamel} 
\end{align} 
Note that the second term on the right-hand side is the space-time convolution
$(P*f)(t,x)$. Therefore, we will sometimes write~\eqref{eq:Duhamel} in the 
short form 
\begin{equation}
 \phi = P\phi_0 + P*f\;,
\end{equation} 
keeping in mind that the star denotes space-time convolution, while no star 
stands for convolution in space only. 

It is thus natural to define the solution of the stochastic heat 
equation~\eqref{eq:stoch_heat} by 
\begin{equation}
\label{eq:SHE-fixed-point} 
 \phi = P\phi_0 + P*\xi\;,
\end{equation} 
where $(P*\xi)(t,x)$ is called the \emph{stochastic convolution}. It seems 
reasonable to expect that this definition makes sense, since the stochastic 
convolution looks similar to pairing $\xi$ with a test function. One should be 
careful, however, with the fact that $P$ has a singularity at the origin, and 
is thus not strictly speaking a test function.   

We will discuss two equivalent ways of analysing the stochastic convolution. 
The first one is based on Fourier series and fractional Sobolev spaces, while 
the second one uses Besov-- H\"older spaces and the so-called Schauder estimate.

%%%%%%%%%%%%%%%%%%%%%%%%%%%%%%%%%%%%%%%%%%%%%%%%%%%%%%%%%%%%%%%%%%%%%%%%%%%%%%%%

\subsection{Fourier series and fractional Sobolev spaces}
\label{ssec:1d-Sobolev} 

Define an orthonormal basis $(e_k)_{k\in\Z}$ of $L^2(\Lambda)$ by
\begin{equation}
 e_k(x) := \sqrt{\dfrac2L} \cos\biggpar{\dfrac{2k\pi x}{L}} 
 \quad\text{if $k>0$\;,}
 \qquad 
 e_0(x) := \dfrac1{\sqrt{L}}\;, 
 \qquad
 e_k(x) := \sqrt{\dfrac2L} \sin\biggpar{\dfrac{2k\pi x}{L}} 
 \quad\text{if $k<0$\;.} 
 \label{eq:Fourier_basis} 
\end{equation} 
These basis functions satisfy the eigenvalue problem 
\begin{equation}
\label{eq:Laplace_eigenvalues} 
 \Delta e_k = -\lambda_k e_k\;, \qquad 
 \lambda_k := \biggpar{\frac{2k\pi}{L}}^2\;.
\end{equation} 
We will write the expansion of $\phi\in L^2(\Lambda)$ in the Fourier 
basis~\eqref{eq:Fourier_basis} 
\begin{equation}
\label{eq:Fourier_series} 
 \phi(x) = \sum_{k\in\Z} \hat \phi_k e_k(x)\;.
\end{equation}

\begin{remark}
It might seem more convenient to use a complex Fourier basis of the form 
$e_k(x)=\e^{2\icx k\pi x/L}/\sqrt{L}$. While this simplifies certain 
computations involving nonlinear terms, it has the drawback that the modes $k$ 
and $-k$ become correlated. This is not really a problem, but makes it simpler 
to work with real Fourier series, at least in the case of linear equations. 
\end{remark}

Fourier series are intimately related to the scale of \emph{fractional Sobolev 
spaces} (also called \emph{Bessel potential spaces}).

\begin{definition}[Fractional Sobolev spaces]
For $s\geqs0$, the fractional Sobolev space $H^s(\Lambda)$ is given by the 
subspace of functions $\phi\in L^2(\Lambda)$ such that 
\begin{equation}
\label{eq:Sobolev_norm} 
 \norm{\phi}_{H^s}^2 := \sum_{k\in\Z} (1+k^2)^s \hat\phi_k^2 < \infty\;.
\end{equation} 
In particular, $H^0(\Lambda) = L^2(\Lambda)$. For $s<0$, $H^s(\Lambda)$ is 
the closure of $L^2(\Lambda)$ under the norm~\eqref{eq:Sobolev_norm}.
\end{definition}

A first use of fractional Sobolev spaces is that they allow to quantify the 
regularising effect of the heat semigroup.

\begin{lemma}
\label{lem:heat_Sobolev} 
For any $s\geqs0$, there exists a constant $C(s)<\infty$ such that 
\begin{equation}
 \norm{\e^{t\Delta}\phi_0}_{H^s} \leqs \bigpar{1+C(s)t^{-s/2}} 
\norm{\phi_0}_{L^2} 
\end{equation} 
holds for all $\phi_0\in L^2(\Lambda)$. 
\end{lemma}
%
% \begin{proof}
% In Fourier representation, we have 
% \begin{equation}
%  (\e^{t\Delta}\phi_0)(x) = \sum_{k\in\Z} \e^{-\lambda_k t} \hat\phi_{0,k} 
% e_k(x)\;.
% \end{equation}
% Computing the $H^s$-norm and using the fact that $x\mapsto x^s\e^{-x}$ is 
% bounded uniformly in $x$ by a constant depending on $s$. 
% \end{proof}

\begin{exercise}
Prove \Cref{lem:heat_Sobolev}, using the Fourier representation of 
$\e^{t\Delta}\phi_0$ and the fact that the map $x\mapsto x^s\e^{-x}$ is 
bounded uniformly in $x>0$ by a constant depending on $s$.
\end{exercise}

Projecting the stochastic heat equation~\eqref{eq:stoch_heat} on the basis 
function $e_k$, and using the observation on $\pscal{\xi}{\indexfct{0\leqs 
t\leqs T}e_k}$ made in \Cref{ex:noise_fourier}, we obtain that each Fourier 
mode 
evolves according to the SDE 
\begin{equation}
 \6\hat\phi_k = -\lambda_k \hat\phi_k \6t + \6W^{(k)}_t\;,
\end{equation} 
where the $(W^{(k)}_t)_{t\geqs0}$ are independent Wiener processes. The 
solution is given by
\begin{equation}
 \hat\phi_k(t) = \e^{-\lambda_k t}\hat\phi_k(0)
 + \int_0^t \e^{-\lambda_k(t-s)} \6W^{(k)}_s\;.
\end{equation} 
Hence $\hat\phi_0(t)$ is a Brownian motion, while all other $\hat\phi_k(t)$ are 
Orstein--Uhlenbeck processes. The stochastic convolution can thus be defined as 
\begin{equation}
\label{eq:stoch_conv_Fourier} 
 (P*\xi)(t,x) = \sum_{k\in\Z} \int_0^t \e^{-\lambda_k(t-s)} \6W^{(k)}_s 
e_k(x)\;,
\end{equation} 
which together with the initial-condition term $\e^{t\Delta}\phi_0$ defines the 
so-called \emph{mild solution} of the stochastic heat equation. Using It\^o's 
isometry, we obtain 
\begin{equation}
 \Bigexpec{\norm{(P*\xi)(t,\cdot)}_{H^s}^2}
 = t + \sum_{k\in\Z^*} (1+k^2)^s \frac{1-\e^{-2\lambda_k t}}{2\lambda_k}\;,
\end{equation} 
which is finite for finite $t$ for all $s<1/2$. In fact, we have the following 
result, which is a special case of~\cite[Theorem~5.13]{Hairer_LN_2009}. 

\begin{theorem}[Sobolev regularity of the stochastic convolution]
The stochastic convolution $(P*\xi)(t,\cdot)$ belongs to $H^s(\Lambda)$ for 
every $s < \frac12$ and $t\in\R_+$. 
\end{theorem}

\begin{exercise}
Compute the covariance $\bigexpec{(P*\xi)(t,x)(P*\xi)(s,y)}$ of the stochastic 
convolution~\eqref{eq:stoch_conv_Fourier}.  
\Hint It may help to distinguish the cases $t>s$, $t=s$ and $t<s$. 
\end{exercise}

%%%%%%%%%%%%%%%%%%%%%%%%%%%%%%%%%%%%%%%%%%%%%%%%%%%%%%%%%%%%%%%%%%%%%%%%%%%%%%%%

\subsection{Besov--H\"older spaces and Schauder estimate}
\label{ssec:1d-Besov} 

It is also possible to quantify the H\"older regularity of the stochastic 
convolution. One way of doing this is to use 
\Cref{thm:space_time_white_noise_reg} on the regularity of space-time white 
noise in conjunction with a regularising estimate for the heat semigroup, 
called \emph{Schauder estimate}. To state this result in full generality, we 
extend the definition of parabolic H\"older spaces given in 
\Cref{def:Holder_0-1} to exponents $\alpha>1$ in the following way.

\begin{definition}[Parabolic H\"older spaces of positive index]
\label{def:Holder-positive}
Let $\alpha\geqs0$. The space $\cC^\alpha_\fraks(\R\times\Lambda)$ consists of 
all $f:\R\times\Lambda\to\R$ such that there exist polynomials 
$\set{P_z}_{z\in\R\times\Lambda}$ of parabolic degree less than $\alpha$, such 
that for every compact $\fK\subset\R\times\Lambda$
\begin{equation}
\label{eq:norm_C-alpha} 
 \norm{f}_{\cC^\alpha_\fraks(\fK)}
 := \sup_{z\in\fK} \sup_{\ph\in B_0} \sup_{\lambda\in(0,1]}
 \biggabs{\frac{\pscal{f-P_z}{\cS^\lambda_z\ph}}{\lambda^\alpha}} < \infty\;,
\end{equation} 
where $B_0$ is defined right before~\eqref{eq:def_Br}. 
\end{definition}

At first sight, this definition may seem redundant with \Cref{def:Holder_0-1} 
when $0<\alpha<1$, but it is in fact equivalent. To see this, consider the case 
$z=(0,0)$, and take $P_0(t,x) = f(0,0)$. Then 
\begin{align}
 \pscal{f-P_0}{\cS^\lambda_z\ph}
 &= \int_\R \int_\Lambda \bigbrak{f(t,x)-f(0,0)} \frac{1}{\lambda^3}
 \ph\biggpar{\frac{t}{\lambda^2},\frac{x}{\lambda}} \6x\6t \\
 &= \int_\R \int_\Lambda \bigbrak{f(\lambda^2t,\lambda x)-f(0,0)}
 \ph(t,x) \6x\6t\;,
\end{align} 
which has indeed order $\lambda^\alpha$ if $f$ is $\alpha$-H\"older in the 
parabolic norm. For $\alpha>1$, one can easily check that $P_z$ must be  
the Taylor expansion of $f$ at $z$ to parabolic order $\intpart{\alpha}$. 

\begin{exercise}
\label{exo:Holder_derivative} 
Prove that if $f\in\cC^\alpha_\fraks$ and $k\in\N_0^2$ is a multiindex, then 
$\DD^k f\in\cC^{\alpha-\abs{k}_\fraks}_\fraks$, where derivatives are 
understood in the sense of distributions if $\alpha < \abs{k}_\fraks$. 
\Hint Show that 
\begin{equation}
 \pscal{\DD^k(f-P_z)}{\cS^\lambda_z\ph} = -\frac{1}{\lambda^{\abs{k}_\fraks}} 
\pscal{f-P_z}{\cS^\lambda_z\DD^k\ph}
\end{equation} 
using integration by parts. 
\end{exercise}

The main result quantifying the regularising properties of the heat semigroup 
is the following Schauder estimate. 

\begin{theorem}[Schauder estimate]
\label{thm:Schauder}
For every $\alpha\in\R\setminus\Z$ and $t>0$, there exists a constant $C$ such 
that for every $\zeta\in\cC^\alpha_\fraks(\R_+\times\Lambda)$, one has
\begin{equation}
 \norm{P*\zeta}_{\cC^{\alpha+2}_\fraks([0,t]\times\Lambda)} 
 \leqs C \norm{\zeta}_{\cC^\alpha_\fraks([0,t]\times\Lambda)}\;.
\end{equation} 
Therefore $P*\zeta \in \cC^{\alpha+2}_\fraks(\R_+\times\Lambda)$.  
The heat semigroup is said to be \emph{regularising of 
parabolic order $2$.}
\end{theorem}

\begin{remark}
This result is in general \emph{not} true if $\alpha$ is an integer. In fact, 
our definition of the spaces $\cC^\alpha_\fraks$ implies a kind of 
discontinuity at integer values of $\alpha$. For instance, $\alpha$-H\"older 
functions approach continuously differentiable functions as $\alpha$ decreases 
to $1$ from above, and Lipschitz functions as $\alpha$ increase to $1$ from 
below. This will, however, not cause any problems, since we will always be able 
to slightly decrease $\alpha$. 
\end{remark}

\begin{proof}[\Sketch]
The proof we sketch here is not the only possible one, and perhaps not the 
shortest, but it will prepare us for the theory of regularity structures. 
The main idea is to decompose $P$ as a sum
\begin{equation}
\label{eq:P_sum} 
 P(z) = \sum_{n\geqs 0} P_n(z) + R(z)\;,
\end{equation} 
where $R$ is smooth and bounded, and each $P_n$ is supported in a ball 
$\setsuch{z}{\norm{z}_\fraks \leqs 2^{-n}}$, and satisfies the bound 
\begin{equation}
\label{eq:Pn_norm} 
 \bigabs{\DD^k P_n(z)} \leqs C 2^{(1+\abs{k}_\fraks)n} 
\end{equation} 
for every multiindex $k$. Such a decomposition is possible as shown 
in~\cite[Lemma~5.5]{Hairer2014}. We then have 
\begin{equation}
 (P*\zeta)(z) = \sum_{n\geqs0} \pscal{\zeta}{P_n(z-\cdot)} + 
\pscal{\zeta}{R(z-\cdot)}\;.
\end{equation} 
The term involving $R$ is smooth, so we only need to concern ourselves with 
the sum over $n$. 

Consider first the case $\alpha < -2$. The aim is to show that for a test 
function $\ph$, 
\begin{equation}
\label{eq:proof_Schauder_bound} 
 \biggabs{\sum_{n\geqs0} \int_{[0,t]\times\Lambda} 
 \pscal{\zeta}{P_n(\bar z-\cdot)}
 (\cS^\lambda_z \ph)(\bar z)\6\bar z} \lesssim \lambda^{\alpha+2}\;.
\end{equation} 
The proof proceeds differently depending on the relative value of $2^{-n}$ and 
$\lambda$. If $2^{-n} > \lambda$, it follows from the fact that 
$\zeta\in\cC^\alpha_\fraks$ and the properties of $P_n$ that 
\begin{equation}
\label{eq:proof_Schauder_nlarge} 
 \bigabs{\pscal{\zeta}{P_n(\bar z-\cdot)}} \lesssim 2^{-(\alpha+2)n}\;.
\end{equation} 
Indeed, $P_n$ can be bounded above by a scaled test function multiplied by  
$2^{-2n}$, cf.~\cite[Remark~2.21]{Hairer2014}. Since $\alpha+2<0$, the sum 
over all $n$ such that $2^{-n} > \lambda$ is dominated by the largest $n$, and 
satisfies~\eqref{eq:proof_Schauder_bound} as required.  

If $2^{-n}\leqs\lambda$, we write the quantity to be bounded as 
$\pscal{\zeta}{Y^\lambda_n}$, where, for fixed $z$, 
\begin{equation}
\label{eq:proof_Ylambda} 
 Y^\lambda_n(z') := \int_{[0,t]\times\Lambda} P_n(\bar z-z') 
 (\cS^\lambda_z \ph)(\bar z)\6\bar z\;.
\end{equation} 
Here we note that $Y^\lambda_n$ is supported in a ball of parabolic radius 
$2\lambda$ around $z$ (see \Cref{fig:Schauder}). Bounding $\cS^\lambda_z \ph$ 
by $\lambda^{-3}$ and using the properties of $P_n$, we obtain 
\begin{equation}
 \bigabs{Y^\lambda_n(z')} \lesssim 2^{-2n}\lambda^{-3}\;.
\end{equation} 
A similar bound can be obtained for derivatives of $Y^\lambda_n$ 
(cf.~\cite[Lemma~5.19]{Hairer2014}), and using once again the fact that 
$\zeta\in\cC^\alpha_\fraks$, we arrive at 
\begin{equation}
 \bigabs{\pscal{\zeta}{Y^\lambda_n}} \lesssim \lambda^\alpha 2^{-2n}\;.
\end{equation} 
Summing over all $n$ such that $2^{-n}\leqs\lambda$ yields again a bound of the 
form~\eqref{eq:proof_Schauder_bound}. 

\begin{figure}[tb]
\begin{center}
%\vspace{-5mm}
\begin{tikzpicture}[>=stealth',main node/.style={circle,minimum
size=0.1cm,inner sep=0pt,fill=white,draw},scale=1]

% grid to help positioning
%\draw[help lines] (-5,-1) grid (5,5);

\draw[blue,fill=blue!50, fill opacity=0.5] (0,0) circle (2.5); 
\draw[teal,fill=teal!50, fill opacity=0.5] (3,0) circle (1.3); 

\draw[semithick,->] (0,0) -- +(120:2.5);
\draw[semithick,->] (3,0) -- +(60:1.3);

\node[main node] at (0,0) {};
\node[main node] at (3,0) {};

\node[] at (0,-0.3) {$z$};
\node[] at (3,-0.3) {$z'$};
\node[] at (-0.3,1.2) {$\lambda$};
\node[] at (3.7,0.3) {$2^{-n}$};

\node[blue] at (-1,-1) {$\cS^\lambda_z\ph$};
\node[teal] at (3.5,-0.8) {$P_n$};

\end{tikzpicture}
\vspace{-3mm}
\end{center}
\caption[]{Domain of the integral~\eqref{eq:proof_Ylambda} defining 
$Y^\lambda_n(z')$. The functions $\cS^\lambda_z$ and $P_n(\bar z-\cdot)$ are 
supported, respectively, in a ball of (parabolic) radius $\lambda$ centred in 
$z$, and in a ball of radius $2^{-n}$ centred in $z'$. Therefore, 
$Y^\lambda_n(z')$ vanishes if $\norm{z'-z}_\fraks > \lambda+2^{-n}$. The 
support of the integrand lies in the intersection of the balls, and is thus 
contained in a ball of radius $2^{-n}$.}
\label{fig:Schauder}
\end{figure}
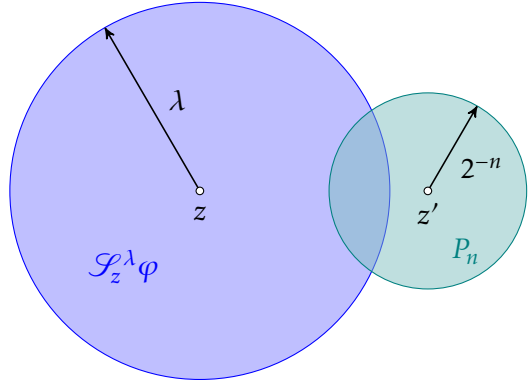

Consider next the case $-2<\alpha<-1$. Then, instead of the 
bound~\eqref{eq:proof_Schauder_bound}, we have to show that 
\begin{equation}
\label{eq:proof_Schauder_bound2} 
 \biggabs{\sum_{n\geqs0} \int_{[0,t]\times\Lambda} 
 \pscal{\zeta}{P_n(\bar z-\cdot) - P_n(z-\cdot)}
 (\cS^\lambda_z \ph)(\bar z)\6\bar z} \lesssim \lambda^{\alpha+2}\;.
\end{equation} 
The difference is that we now subtract the term $P_n(z-\cdot)$, which accounts 
for the polynomial $P_z$ in~\eqref{eq:norm_C-alpha}. If $2^{-n}>\lambda$, using 
a Taylor expansion, one obtains instead 
of~\eqref{eq:proof_Schauder_nlarge}
\begin{equation}
 \bigabs{\pscal{\zeta}{P_n(\bar z-\cdot) - P_n(z-\cdot)}} \lesssim 
 \sum_{k\in\set{(1,0),(0,1)}} \norm{z-\bar z}^{\abs{k}_\fraks}_\fraks
 2^{(\abs{k}_\fraks-\alpha-2)n}\;.
\end{equation} 
Combining this with the bound 
\begin{equation}
 \int_{[0,t]\times\Lambda} \norm{z-\bar z}^{\abs{k}_\fraks}_\fraks 
 (\cS^\lambda_z \ph)(\bar z)\6\bar z \lesssim \lambda^{\abs{k}_\fraks}
\end{equation} 
and summing over $n$ yields the result. 
For $2^{-n}\leqs\lambda$, the term $Y^\lambda_n(z')$ has to be replaced by 
\begin{equation}
 Y^\lambda_n(z') - P_n(z-z')
\end{equation} 
since the scaled test function integrates to $1$. When tested against $\zeta$, 
the new term $P_n(z-\cdot)$ produces an extra term $2^{-(\alpha+2)n}$, whose 
sum over $n$ again satisfies the required bound.

For large values of $\alpha$, the proof is similar, except that one has to 
extract more terms in the Taylor expansion of the $P_n$.
\end{proof}

We now directly obtain the nontrivial fact that the stochastic convolution is 
H\"older continuous of exponent almost $1/2$. This will become very useful when 
solving nonlinear equations with a fixed-point argument. 

\begin{corollary}[H\"older regularity of the stochastic convolution]
\label{cor:stoch_conv_Holder} 
The stochastic convolution $P*\xi$ belongs to 
$\cC^\alpha_\fraks(\R_+\times\Lambda)$ for all $\alpha < 1/2$.
\end{corollary}
\begin{proof}
This follows immediately by combining \Cref{thm:space_time_white_noise_reg} and 
the Schauder estimate of \Cref{thm:Schauder}. 
\end{proof}

%%%%%%%%%%%%%%%%%%%%%%%%%%%%%%%%%%%%%%%%%%%%%%%%%%%%%%%%%%%%%%%%%%%%%%%%%%%%%%%%

\section{Existence and uniqueness of solutions}
\label{sec:1dex} 

We turn now to semilinear SPDEs of the form 
\begin{equation}
\label{eq:AC-1d-rep} 
 \partial_t \phi(t,x) = \Delta\phi(t,x) + F(\phi(t,x)) 
 + \sqrt{2\eps} \xi(t,x)
\end{equation} 
with initial condition $\phi_0(x)$ and locally Lipschitz nonlinearity $F$. For 
instance, in the case of the Allen--Cahn SPDE~\eqref{eq:AC-1d}, 
$F(\phi)=\phi-\phi^3$.  By analogy with the Duhamel principle 
yielding~\eqref{eq:SHE-fixed-point} in the case of the stochastic heat 
equation, 
one way of defining a notion of solution to~\eqref{eq:AC-1d-rep} is as a fixed 
point of 
\begin{equation}
\label{eq:AC-1d-FP} 
 \phi = P\phi_0 + P*F(\phi) + \sqrt{2\eps} P*\xi\;.
\end{equation} 
A more explicit notation of this relation, analogous to~\eqref{eq:Duhamel}, is
\begin{equation}
 \label{eq:AC-1d-FP-semigroup}
 \phi(t) = \e^{t\Delta}\phi_0 
 + \int_0^t \e^{(t-s)\Delta} F(\phi(s)) \6s 
 + \sqrt{2\eps} \int_0^t \e^{(t-s)\Delta} \6\xi(s)\;.
\end{equation} 

\begin{definition}[Mild solution] \hfill
\begin{itemize}
\item 	A stochastic process $t\mapsto \phi(t)$ with values in a Banach space 
$\cB$ satisfying~\eqref{eq:AC-1d-FP-semigroup} almost surely for every $t>0$ is 
called a \emph{mild solution} of the SPDE~\eqref{eq:AC-1d-rep}. 

\item 	A \emph{local mild solution} of~\eqref{eq:AC-1d-rep} is a pair 
$(\phi,\tau)$, where $t\mapsto \phi(t)$ is a stochastic process and $\tau$ is a 
stopping time such that~\eqref{eq:AC-1d-FP-semigroup} is satisfied almost 
surely 
for all $t\leqs\tau$. 

\item 	A local mild solution $(\phi,\tau)$ is \emph{maximal} if for any local 
mild solution $(\tilde\phi,\tilde\tau)$, one has $\tilde\tau \leqs \tau$ almost 
surely.
\end{itemize}
\end{definition}

Existence and uniqueness of a maximal local mild solution can be proved in a 
way which is analogous to what is done for ODEs and SDEs. 

\begin{theorem}[Existence of a unique maximal local mild solution]
\label{thm:semilinear_local_existence} 
Let $\cC^0$ be the Banach space of continuous functions $f:\Lambda\to\R$, 
equipped with the supremum norm. For every $\phi_0\in\cC^0$, the semilinear  
equation~\eqref{eq:AC-1d-rep} admits a unique maximal local mild 
solution $\phi$ taking values in $\cC^0$. Moreover, this solution has 
continuous sample paths, and 
satisfies $\lim_{t\nearrow\tau}\norm{\phi(t)}_{\cC^0} = \infty$ almost surely 
on the set $\set{\tau < \infty}$. 
\end{theorem}
\begin{proof}
Given $T>0$ and a continuous function $g:\R_+\to\cC^0$, we define a map $\cT$ 
from the space of continuous functions from $[0,T]$ to $\cC^0$ to itself by 
\begin{equation}
 (\cT\phi)(t) := \int_0^t \e^{(t-s)\Delta} F(\phi(s))\6s + g(t)\;, 
 \qquad t\in[0,T]\;.
\end{equation} 
H\"older's inequality shows that $\norm{\e^{t\Delta}f}_{\cC^0} \leqs 
\norm{f}_{\cC^0}$ for any $t\geqs0$. Therefore, we get  
\begin{align}
 \sup_{t\in[0,T]} \norm{(\cT\phi)(t) - (\cT\psi)(t)}_{\cC^0}
 &\leqs T \sup_{t\in[0,T]} \norm{F(\phi(t)) - F(\psi(t))}_{\cC^0} \\
 \sup_{t\in[0,T]} \norm{(\cT\phi)(t) - g(t)}_{\cC^0}
 &\leqs T \sup_{t\in[0,T]} \norm{F(\phi(t))}_{\cC^0}\;.  
\end{align} 
Now fix a constant $R>0$. Since $F$ is locally Lipschitz, we obtain the 
existence of a constant $M=M(R)>0$ such that for any $\phi,\psi$ in a ball of 
radius $R$ around $g$, one has 
\begin{align}
 \sup_{t\in[0,T]} \norm{(\cT\phi)(t) - (\cT\psi)(t)}_{\cC^0}
 &\leqs TM \sup_{t\in[0,T]} \norm{\phi(t)-\psi(t)}_{\cC^0} \\
 \sup_{t\in[0,T]} \norm{(\cT\phi)(t) - g(t)}_{\cC^0}
 &\leqs TM \;.  
\end{align}
Choosing $T$ small enough, one can guarantee that $\cT$ maps the ball of radius 
$R$ into itself, and is a contraction of constant $\lambda<1/2$ there. 
Therefore, Banach's fixed point theorem shows that $\cT$ admits a unique fixed 
point in that ball. Choosing 
\begin{equation}
 g(t) = \e^{t\Delta}\phi_0 
 + \sqrt{2\eps} \int_0^t \e^{(t-s)\Delta} \6\xi_s\;,
\end{equation} 
which is in $\cC^0$ for all $t$ according to \Cref{cor:stoch_conv_Holder},
we obtain the existence of a unique local mild solution on $[0,T]$. Note that 
this is a pathwise solution, that is, for each realisation of $\xi$, we find 
such a solution, where $T$ may depend on the particular realisation of the 
noise. 

To show maximality, we argue exactly as in the deterministic case. Since $R$ is 
finite, we may restart the process to extend it to a slightly longer time 
interval. This procedure can be iterated as long as the solution remains 
bounded, proving that either $\tau=\infty$, or $\phi(t)$ diverges as $t$ 
approaches $\tau$. Finally, uniqueness and continuity can be shown in the same 
way as for finite-dimensional SDEs, see for 
instance~\cite[Theorem~5.2.1]{Oeksendal}. 
\end{proof}

To obtain global existence of solutions, we have again to exploit the sign of 
the nonlinearity. There exist various one-sided conditions on the nonlinearity 
allowing to do that. One of them is that $F$ satisfy the following 
property: for any $R>0$, there exists a $C>0$ such that 
\begin{equation}
\label{eq:F_dissipative} 
 u F(u+v) \leqs Cu^2 
 \qquad \text{when $\abs{v}\leqs R$\;.}
\end{equation} 
One easily checks that this holds in the case $F(u)=u-u^3$ of the Allen--Cahn 
equation. 

\begin{theorem}[Global existence of the solution]
\label{thm:semilinear_global_existence} 
If $F$ is locally Lipschitz and satisfies~\eqref{eq:F_dissipative}, 
then for any $\phi_0\in\cC^0$, the solution of the semilinear stochastic 
PDE~\eqref{eq:AC-1d-rep} exists globally in time.
\end{theorem}
\begin{proof}
We follow the proof given in~\cite[Proposition~6.23]{Hairer_LN_2009}, using a 
Lyapunov function to show that $\phi(t)$ cannot blow up. Let $W_\Delta(t)$ 
denote the stochastic convolution in~\eqref{eq:AC-1d-FP-semigroup}. Then 
$\psi(t) = \phi(t)-\sqrt{2\eps}W_\Delta(t)$ satisfies 
\begin{equation}
 \psi(t) = \e^{t\Delta}\phi_0 
 + \int_0^t \e^{(t-s)\Delta} G(s) \6s\;, 
 \qquad 
 G(s) = F(\psi(s)+\sqrt{2\eps}W_\Delta(s))\;.
\end{equation} 
Continuity of sample paths and the Markov property imply  
\begin{equation}
 \psi(t+h) = \e^{h\Delta} \bigbrak{\psi(t)+hG(t)} + \order{h}\;.
\end{equation} 
Consider now the Lyapunov function 
\begin{equation}
 \cV(\psi) := \sup_{x\in\Lambda} \psi(x)^2\;.
\end{equation} 
Note that since the heat semigroup is contracting,\footnote{This holds actually 
for any convex function $\cV$, cf.~\cite[Lemma~6.22]{Hairer_LN_2009}.} 
$\cV(\e^{h\Delta}\psi) \leqs \cV(\psi)$. Thus we have 
\begin{equation}
 \limsup_{h\to0} \frac{1}{h} \bigbrak{\cV(\psi(t+h)) - \cV(\psi(t))}
 \leqs \limsup_{h\to0} \frac{1}{h} \bigbrak{\cV(\psi(t)+hG(t)) - 
\cV(\psi(t))}\;.
\end{equation} 
Now it follows from the one-sided condition~\eqref{eq:F_dissipative} that as 
long as $\sqrt{2\eps}\abs{W_\Delta(t)}\leqs R$, one has   
\begin{align}
\cV(\psi(t)+hG(t)) 
&= \sup_{x\in\Lambda} \Bigpar{\psi(t,x)^2 + 2h\psi(t,x)G(t,x) + h^2G(t,x)^2} \\
&\leqs \cV(\psi(t)) + 2hC \cV(\psi(t)) + h^2 \cV(G(t))\;. 
\end{align}
It follows that 
\begin{equation}
 \limsup_{h\to0} \frac{1}{h} \bigbrak{\cV(\psi(t+h)) - \cV(\psi(t))}
 \leqs 2C\cV(\psi(t))\;.
\end{equation} 
By a Gronwall argument, $\psi(t)$ cannot blow up unless $W_\Delta(t)$ does, 
which is not the case. 
\end{proof}

We note that for SPDEs with confining nonlinearities, one can obtain much 
better bounds on solutions. For instance, \cite[Proposition~3.4]{Cerrai_1999} 
applied to $F(u)=u-u^3$ provides the bound
\begin{equation}
 \sup_{\phi_0\in\cC^0} \norm{\phi(t)}_{\cC^0}
 \leqs \frac{C}{\sqrt{t}} \Bigpar{1+\sqrt{2\eps} \sup_{0\leqs s\leqs t} 
\norm{W_\Delta(s)}_{\cC^0}} + \sqrt{2\eps} \norm{W_\Delta(t)}_{\cC^0}\;.
\end{equation} 
The remarkable point is that $\phi(t)$ decreases like $1/\sqrt{t}$ for small 
times, whatever the size of the initial condition. This property is known as 
\emph{coming down from infinity}, and is very useful to obtain continuation 
results for solutions with very little prior knowledge. This phenomenon 
actually already exists for very simple ODEs, as shows the following exercise. 

\begin{exercise}
Show that the solution of the ODE 
\begin{equation}
 \dot x = -x^3\;, 
 \qquad 
 x(0) = x_0
\end{equation}
can be bounded by a function independent of the value of $x_0$. 
\end{exercise}

%%%%%%%%%%%%%%%%%%%%%%%%%%%%%%%%%%%%%%%%%%%%%%%%%%%%%%%%%%%%%%%%%%%%%%%%%%%%%%%%

\section{Invariant measure}
\label{sec:1dinv} 

%%%%%%%%%%%%%%%%%%%%%%%%%%%%%%%%%%%%%%%%%%%%%%%%%%%%%%%%%%%%%%%%%%%%%%%%%%%%%%%%

\subsection{The Gaussian free field}
\label{ssec:1dGFF} 

Consider first the case of linear SPDEs. The stochastic heat 
equation~\eqref{eq:stoch_heat} does not admit an invariant measure, since we 
have seen that its zeroth Fourier mode performs a Brownian motion. This 
problem, however, is easily cured by considering the equation
\begin{equation}
\label{eq:stoch_heat_dissipative} 
 \partial_t \phi(t,x) = \Delta\phi(t,x) - \phi(t,x) + \sqrt{2\eps}\xi(t,x)\;,
\end{equation} 
where we have reintroduced the parameter $\eps$ to keep track of its effect. 
The modes of the Fourier series~\eqref{eq:Fourier_series} then obey the SDEs 
\begin{equation}
 \6\hat\phi_k(t) = -(\lambda_k+1)\hat\phi_k(t)\6t + \sqrt{2\eps}\6W^{(k)}_t\;,
\end{equation} 
where $-\lambda_k$ are the eigenvalues of the Laplacian, 
cf.~\eqref{eq:Laplace_eigenvalues}. Hence each $\hat\phi_k$ is an 
Ornstein--Uhlenbeck process given by 
\begin{equation}
 \hat\phi_k(t) = \e^{-(\lambda_k+1)t} \hat\phi_k(0) 
 + \sqrt{2\eps}\int_0^t \e^{-(\lambda_k+1)(t-s)} \6W^{(k)}_s\;.
\end{equation} 
Asymptotically as $t\to\infty$, each $\hat\phi_k$ follows a centred normal law 
of variance $\eps/(\lambda_k+1)$. This motivates the following definition.

\begin{definition}[Gaussian free field]
The \defwd{massive Gaussian free field (GFF)} with covariance 
$(-\Delta+1)^{-1}$ on $\Lambda$ is the random 
function $\GFF\in L^2(\Lambda)$ with Fourier series 
\begin{equation}
\label{eq:def_GFF} 
 \GFF(x) := \sum_{k\in\Z} \frac{Z_k}{\sqrt{\lambda_k+1}} 
e_k(x)\;,
\end{equation} 
where the $Z_k\sim\cN(0,1)$ are i.i.d.\ standard normal random variables. 
\end{definition}

This definition is justified by the fact that $\GFF$ is centred, and thus has 
covariance
\begin{equation}
\label{eq:covariance_GFF} 
\bigexpec{\GFF(x)\GFF(y)} 
= \sum_{k\in\Z} \frac{1}{\lambda_k+1} e_k(x)e_k(y) 
= \sum_{k\in\Z} e_k(x) (-\Delta+1)^{-1} e_k(y)\;.
\end{equation}
Note in particular that the trace of the covariance is given by 
\begin{equation}
 \int_\Lambda \bigexpec{\GFF(x)^2} \6x 
 = \Tr\Bigbrak{(-\Delta+1)^{-1}}
 = \sum_{k\in\Z}\frac{1}{\lambda_k+1}
 < \infty\;.
\end{equation} 
The covariance operator is said to be \emph{trace class}. This is indeed a 
necessary and sufficient condition for a Gaussian invariant measure to exist, 
cf.~\cite[Theorem~6.2.1]{DaPrato_Zabczyk_Ergodicity}; see 
also~\cite[Proposition~3.15]{Hairer_LN_2009}.  

\begin{exercise}
Show that for any test functions $\ph_1, \ph_2:L^2(\Lambda)\to\R$, one has 
\begin{equation}
 \bigexpec{\pscal{\GFF}{\ph_1}\pscal{\GFF}{\ph_2}}
 = \pscal{\ph_1}{(-\Delta+1)^{-1}\ph_2}_{L^2}\;,
% \tag*{\qedhere} 
\end{equation} 
where we have identified elements of $L^2(\Lambda)$ and its dual in the 
canonical way. 
\end{exercise}

More generally, $\sqrt{\eps}\GFF$ is a Gaussian free field with covariance 
$\eps(-\Delta+1)^{-1}$, which defines a Gaussian probability measure 
$\smash{\muGFF^{(\eps)}}$ on $L^2(\Lambda)$. This measure is invariant under 
the dynamics of~\eqref{eq:stoch_heat_dissipative}, meaning that 
\begin{equation}
\label{eq:def_muGFFeps} 
 \Bigexpecin{\muGFF^{(\eps)}}{f(\phi_t)} = \Bigexpec{f(\sqrt{\eps}\GFF)}
 \qquad \forall t\geqs 0
\end{equation} 
for any integrable functional $f:L^2(\Lambda)\to\R$. Note that in particular, 
any linear functional $\ell:L^2(\Lambda)\to\R$ follows a centred normal 
distribution, which is in fact a general way to characterise Gaussian measures 
on infinite-dimensional spaces. See for 
instance~\cite[Definition~3.2]{Hairer_LN_2009}. 

\begin{exercise}
Show that for a linear functional $\ell:L^2(\Omega)\to\R$, one has 
\begin{equation}
 \bigexpec{\e^{\icx\pscal{\ell}{\GFF}}} 
 = \exp\biggset{-\frac12\pscal{\ell}{(-\Delta+1)^{-1}\ell}_{L^2}}\;. 
\end{equation} 
What is the variance of $\pscal{\ell}{\GFF}$? 
\end{exercise}

%%%%%%%%%%%%%%%%%%%%%%%%%%%%%%%%%%%%%%%%%%%%%%%%%%%%%%%%%%%%%%%%%%%%%%%%%%%%%%%%

\subsection{Invariant Gibbs measures}
\label{ssec:1dGibbs} 

We return now to semilinear SPDEs of the form 
\begin{equation}
\label{eq:AC-1d-rep2} 
 \partial_t \phi(t,x) = \Delta\phi(t,x) + F(\phi(t,x)) 
 + \sqrt{2\eps} \xi(t,x)\;,
\end{equation} 
where the nonlinearity $F$ is such that \Cref{thm:semilinear_global_existence} 
applies. We denote its Markov semigroup 
\begin{equation}
 P_t(\phi_0,A) := \bigprobin{\phi_0}{\phi_t\in A}
\end{equation} 
(where $P_t$ should not be confused with the heat kernel). More generally, for 
a test function $f:\cH=L^2(\Lambda)\to\R$, we write 
\begin{equation}
 (P_t f)(\phi_0) := \bigexpecin{\phi_0}{f(\phi_t)}\;.
\end{equation} 
One can show that $P_t$ has the \emph{Feller property}, meaning that it maps 
continuous functions into continuous functions. A probability measure $\pi$ on 
$\cH$ is \emph{invariant} if 
\begin{equation}
 (\pi P_t)(A) := \int_{\cH} \pi(\6\phi) P_t(\phi,A) = \pi(A)
\end{equation} 
for every $t\geqs0$ and every measurable $A\subset \cH$. 

A classical way of establishing the existence of an invariant probability 
measure, due to Krylov and Bogoliubov, relies on a tightness argument.

\begin{definition}[Tightness]
A family $\set{\mu_t}$ of probability measures on $\cH$ is~\emph{tight} if for 
any $\delta>0$, there exists a compact set $K\subset\cH$ such that 
$\mu_t(K)\geqs1-\delta$ for all $t$. 
\end{definition}

Indeed one can show (cf.~\cite[Corollary~3.1.2]{DaPrato_Zabczyk_Ergodicity}) 
that if the family of ergodic averages  
\begin{equation}
 \biggsetsuch{\frac{1}{T} \int_0^T P_t(\phi_0,\cdot)\6t}{T\geqs 1}
\end{equation} 
is tight, then there exists an invariant probability measure. In the present 
situation, tightness follows from stochastic stability of $\phi_t$, meaning 
that for any $\phi_0\in\cH$ and any $\delta>0$, there exists $R>0$ such that 
for all $T\geqs1$, one has 
\begin{equation}
 \frac{1}{T} \int_0^T \bigprobin{\phi_0}{\norm{\phi_t}>R}\6t \leqs \delta\;.
\end{equation} 
See~\cite[Theorem~6.1.2]{DaPrato_Zabczyk_Ergodicity} for details. 

In the case of the semilinear SPDE~\eqref{eq:AC-1d-rep2}, we can 
apply~\cite[Theorem~6.3.3]{DaPrato_Zabczyk_Ergodicity} to obtain the following 
exponential ergodicity result. The main ingredient of the proof is a suitable 
growth bound on $\expecin{\phi_0}{\phi_t}$. 

\begin{theorem}[Existence of a unique invariant measure and exponential 
ergodicity]
When $\eps>0$, there exists a unique invariant probability measure $\pi$ 
for~\eqref{eq:AC-1d-rep2}, which is strongly mixing\footnote{Strong mixing 
means that $\pscal{P_tf}{g}$ converges, as $t\to\infty$, to 
$\pscal{f}{1}\pscal{1}{g}$, where the inner products are in 
$L^2(\cH,\pi)$.} and such that $\nu P_t$ converges weakly to $\pi$ as 
$t\to\infty$ for any probability measure $\nu$ on $\cH$. Furthermore, there 
exist constants $C,\beta>0$ such that 
\begin{equation}
\label{eq:AC-1d-ergodicity} 
 \bigabs{(P_tf)(\phi_0) - \pscal{\pi}{f}} 
 \leqs C \bigpar{1+\norm{\phi_0}_{\cH}} \e^{-\beta t}
 \norm{f}_{\math{Lip}}
\end{equation} 
for all $t\geqs0$, all $\phi_0\in\cH$, and all bounded Lipschitz continuous 
test functions $f:\cH\to\R$ with Lipschitz constant $\norm{f}_{\math{Lip}}$. 
\end{theorem}

In the particular case $F(\phi)=-\phi$ in~\eqref{eq:AC-1d-rep2}, the 
invariant measure is given by $\muGFF^{(\eps)}(\6\phi)$, the distribution of 
the scaled GFF defined in~\eqref{eq:def_muGFFeps}. 

In the general case, the semilinear equation~\eqref{eq:AC-1d-rep2} can 
be written in gradient form as 
\begin{equation}
  \partial_t \phi(t,x) = \bigbrak{\Delta-\one}\phi(t,x) - \nabla 
\widetilde V(\phi(t,x)) + \sqrt{2\eps} \xi(t,x)\;,
\end{equation} 
where 
\begin{equation}
 \widetilde V(\phi) = \int_\Lambda \widetilde U(\phi(x))\6x\;,
 \qquad 
 \widetilde U'(\phi) = \phi-F(\phi)\;.
\end{equation} 
Then it turns out that one has the following explicit expression for the 
invariant measure, which is in fact a natural extension of the Gibbs 
measure~\eqref{eq:pi_finite_dim} we encountered in the finite-dimensional case. 
Indeed, the following result is a particularisation 
of~\cite[Theorem~8.6.3]{DaPrato_Zabczyk_Ergodicity}.

\begin{theorem}[Invariance and reversibility of the Gibbs measure]
For $\eps>0$, the unique invariant measure of~\eqref{eq:AC-1d-rep2} is 
absolutely continuous with respect to the scaled Gaussian free field, with 
density $\e^{\widetilde V/\eps}$:
\begin{equation}
\label{eq:Gibbs_1dSPDE} 
 \pi(\6\phi) = \frac{1}{\cZ_0} \e^{-\widetilde V(\phi)/\eps} 
 \muGFF^{(\eps)}(\6\phi)\;,
%  \qquad
%  \cZ_0 = \int_{\cH} \e^{-\widetilde V(\phi)/\eps} 
%  \muGFF^{(2\eps)}(\6\phi)\;.
\end{equation} 
where $\cZ_0$ is the normalisation. 
In addition, $\pi$ satisfies the detailed-balance property 
\begin{equation}
\label{eq:Gibbs_det_balance} 
 \int_{\cH} f(\phi) (P_t g)(\phi)\pi(\6\phi)
 = \int_{\cH} g(\phi) (P_t f)(\phi)\pi(\6\phi)
\end{equation} 
for all bounded test functions $f, g:\cH\to\R$. 
\end{theorem}

\begin{remark}
As in the case of Markov chains, reversibility of a measure implies its 
invariance, as can be seen by choosing $f(\phi)=1$ and 
$g(\phi)=\indicator{A}(\phi)$ in~\eqref{eq:Gibbs_det_balance}. Thus it is in 
fact sufficient to prove that the measure defined in~\eqref{eq:Gibbs_1dSPDE} 
satisfies detailed balance to prove the theorem. The proof proceeds by using a 
finite-dimensional spectral Galerkin approximation of~\eqref{eq:AC-1d-rep2} and 
passing to the limit. 
\end{remark}

To see why (in the case $F(\phi)=\phi-\phi^3$) this is indeed the Gibbs measure 
associated with the potential $V$ defined in~\eqref{eq:potential_infdim}, it 
suffices to notice that 
\begin{equation}
 V(\phi) - \widetilde V(\phi) = \frac12 \int_{\Lambda} 
\bigbrak{\norm{\nabla\phi(x)}^2+\phi(x)^2}\6x\;,
\end{equation} 
which is exactly the quadratic form associated with the linear 
case $F(\phi)=-\phi$. This suggests using the sloppy notation 
\begin{equation}
 \pi(\6\phi) = \frac{1}{\cZ} \e^{-V(\phi)/\eps} 
 \6\phi\;,
 \qquad 
 \cZ = \int_{\cH} \e^{-V(\phi)/\eps} \6\phi\;,
\end{equation} 
but one should not forget that there is no Lebesgue measure on 
$\cH=L^2(\Lambda)$!

%%%%%%%%%%%%%%%%%%%%%%%%%%%%%%%%%%%%%%%%%%%%%%%%%%%%%%%%%%%%%%%%%%%%%%%%%%%%%%%%

\section{Large deviations}
\label{sec:1dldp} 

Large deviation estimates for the SPDE
\begin{equation}
\label{eq:AC-1d-rep3} 
 \partial_t \phi(t,x) = \Delta\phi(t,x) + F(\phi(t,x)) 
 + \sqrt{2\eps} \xi(t,x)
\end{equation}
as $0<\eps\ll1$ have been obtained in a way very similar to the 
finite-dimensional situation discussed in \Cref{sec:diffldp} 
\cite{Faris_JonaLasinio82,Freidlin88,ChenalMillet1997}. One starts by showing 
that scaled space-time white noise $\sqrt{2\eps}\xi(t,x)$ satisfies an LDP on 
$[0,T]$ with good rate function 
\begin{equation}
 \cI_{[0,T]}(\gamma) = \frac12 \int_0^T \int_\Lambda \biggbrak{\dpar{}{t} 
\gamma(t,x)}^2 \6x\6t\;,
\end{equation} 
with the understanding that $\cI_{[0,T]}(\gamma)=+\infty$ if the integral on 
the 
right-hand side does not converge. Then one uses a contraction principle to 
extend the LDP to the general case. 

\begin{theorem}[Large-deviation principle for SPDEs on the one-dimensional 
torus]
The semilinear SPDE~\eqref{eq:AC-1d-rep3} satisfies an LDP on $[0,T]$ with good 
rate function 
\begin{equation}
\cI_{[0,T]}(\gamma) := 
 \begin{cases}
 \displaystyle
 \frac12 \int_0^T \int_\Lambda \biggbrak{\dpar{}{t} \gamma(t,x) - 
\Delta\gamma(t,x) - F(\gamma(t,x))}^2 \6x\6t
 & \text{if the integral is finite\;, } \\
 + \infty 
 & \text{otherwise\;.}
 \end{cases}
\end{equation}
\end{theorem}

From this result, one can for instance deduce information on the exit time and 
location from a domain in $\cH$ in exactly the same way as in the 
finite-dimensional case.

\begin{exercise}
Use a similar argument as in~\eqref{eq:ratefct_gradient} to show that the lower 
bound
\begin{equation}
 \cI_{[0,T]}(\phi) \geqs 2 \bigbrak{V(\phi_T)-V(\phi_0)}
\end{equation} 
holds for any given $T$, where $V$ is the potential 
given by~\eqref{eq:potential_infdim} in the case $F(\phi)=\phi-\phi^3$.
\end{exercise}

%%%%%%%%%%%%%%%%%%%%%%%%%%%%%%%%%%%%%%%%%%%%%%%%%%%%%%%%%%%%%%%%%%%%%%%%%%%%%%%%

\section{Metastability}
\label{sec:1dmeta} 

We examine now the dynamics of the stochastic Allen--Cahn equation
\begin{equation}
\label{eq:AC-1d-rep4} 
 \partial_t \phi(t,x) = \Delta\phi(t,x) + \phi(t,x) - \phi(t,x)^3 + 
\sqrt{2\eps}\xi(t,x) 
\end{equation} 
for small positive $\eps$. For simplicity, we concentrate on the case 
$L<2\pi$, when the associated potential $V$ is a double-well potential, 
cf.~\Cref{prop:AC-1d-det}. As in the finite-dimensional case, the 
estimate~\eqref{eq:AC-1d-ergodicity} shows that any initial measure will 
converge exponentially fast to the invariant measure $\pi$, but it provides no 
control on the dependence of the constants $C$ and $\beta$ on $\eps$. 

The large-deviation approach provides similar results as those in 
\Cref{thm:arrhenius}, including an Arrhenius law for the first-hitting time 
$\tau_+$ of a neighbourhood of $\phi^*_+$ of the form 
\begin{equation}
 \lim_{\eps\to0} \eps\log \bigexpecin{\phi^*_-}{\tau_+} = H\;,
\end{equation} 
where for $L<2\pi$, one has $H=V(\phi^*_0)-V(\phi^*_-) = L/4$.

\begin{remark}
For $L>2\pi$, $\phi^*_0$ is no longer a saddle with one unstable direction, and 
the role of the transition state is played by the family 
$\set{\phi^*_{1,\theta}}$ of non-constant stationary solutions that bifurcate 
from $\phi^*_0$ at $L=2\pi$. The value of $H$ can be expressed in terms of 
elliptic integrals, cf.~\cite{Maier_Stein_PRL_01,Maier_Stein_SPIE_2003}. 
\end{remark}

The next question is thus whether an analogue of the Eyring--Kramers 
formula~\eqref{eq:Eyring_Kramers_finite-dim} holds. We first have to check  
whether the ratio of determinants of Hessians in that formula makes any  
sense in the infinite-dimensional case. Recall that we have shown in 
\Cref{prop:AC-1d-det} that the Hessians are given by $-\Delta-1$ and 
$-\Delta+2$. Though each of them has an infinite determinant, it turns out that 
their ratio converges. To see this, let us write this ratio as 
\begin{equation}
\label{eq:Fredholm} 
 \det\Bigpar{(-\Delta-1)(-\Delta+2)^{-1}}
 = \det\Bigpar{\one - 3(-\Delta+2)^{-1}}\;.
\end{equation} 
This is a so-called \emph{Fredholm determinant}. To see that it indeed 
converges, we note that it is equal to $-\smash{\frac12}\det\bigpar{\one - 
3(-\Delta_\perp+2)^{-1}}$ where $\Delta_\perp$ acts on mean-zero functions, and
write 
\begin{align}
 \log\det\Bigpar{\one - 3(-\Delta_\perp+2)^{-1}}
 = \Tr\log\Bigpar{\one - 3(-\Delta_\perp+2)^{-1}} 
 = -\sum_{n\geqs1} \frac{3^n}{n} \Tr\Bigpar{(-\Delta_\perp+2)^{-n}}\;.
\end{align} 
The series converges because $(-\Delta_\perp+2)^{-1}$ is trace class (and 
therefore all its powers as well), and its eigenvalues decrease like 
$((2k\pi/L)^2+2)^{-1}$. 

\begin{remark}
The ratio~\eqref{eq:Fredholm} can also be computed explicitly, writing it as 
\begin{equation}
 \prod_{k\in\Z} \frac{\nu_k}{\mu_k}
 = \prod_{k\in\Z} 
\frac{\Bigpar{\frac{2k\pi}{L}}^2+2}{\Bigpar{\frac{2k\pi}{L}}^2-1}
 = -2 \left( \prod_{k\in\N} 
\frac{1+2\Bigpar{\frac{L}{2k\pi}}^2}{1-\Bigpar{\frac{L}{2k\pi}}^2}
\right)^2
 = - \frac{\sinh^2(L/\sqrt{2})}{\sin^2(L/2)}\;,
\end{equation} 
where we used two of Euler's infinite product identities in the last step. 
However, the interpretation in terms of Fredholm determinant will be more 
useful in higher dimensions.  
\end{remark}

\begin{exercise}
\label{exo:Fredholm} 
By working in Fourier space, show that 
\begin{equation}
 \det\Bigpar{\one + c(-\Delta+a)^{-1}}
 = \frac{1}{\det\Bigpar{\one -c(-\Delta+b)^{-1}}}
\end{equation} 
for every $a, b\in\R$ such that $a+c=b$ and the Fredholm determinants make 
sense.
\end{exercise}

In order to prove an Eyring--Kramers law, we will approximate the 
infinite-dimensional system by a finite-dimensional one and pass to the limit. 
One way of doing this is to take the limit $N\to\infty$ the lattice 
system~\eqref{eq:SDE}. While this is possible, we prefer the slightly easier 
approach of using a spectral Galerkin approximation, which has the advantage 
that the relevant eigenvalues are independent of $N$. For 
$\phi\in L^2(\Lambda)$, we define its Galerkin projection as 
\begin{equation}
\label{eq:Galerkin} 
 (\Pi_N \phi)(x) = \sum_{\substack{k\in\Z \\ \abs{k}\leqs N}} 
\pscal{\phi}{e_k}e_k(x)\;, 
\end{equation} 
where the $e_k$ denote the Fourier basis functions~\eqref{eq:Fourier_basis}. 
We then consider the projected equation 
\begin{equation}
 \partial_t \phi_N(t,x) = \Delta\phi_N(t,x) + \phi_N(t,x) - \Pi_N\phi_N(t,x)^3 
+ \sqrt{2\eps}\bigpar{\Pi_N\xi}(t,x)
\end{equation} 
on the finite-dimensional subspace $E_N$ of $L^2(\Lambda)$ spanned by the $e_k$ 
with $\abs{k}\leqs N$ (we have used the fact that $\Pi_N$ and $\Delta$ 
commute). This is equivalent to the finite-dimensional SDE for Fourier 
coefficients  
\begin{equation}
 \6\hat\phi_t = -\nabla \widehat V(\hat\phi_t)\6t + \sqrt{2\eps}\6W_t\;,  
\end{equation} 
where $\widehat V$ is simply the potential~\eqref{eq:potential_infdim} 
expressed in Fourier coordinates. This form is suitable for the 
potential-theoretic approach described in \Cref{ssec:potential}

One easily checks that for symmetric sets $B=-A\subset E_N$ 
\Cref{lem:symmetry_hAB} still applies here, so that it is essentially 
sufficient 
to estimate the partition function $\cZ$ and the capacity $\capacity(A,B)$ 
uniformly in the cut-off $N$. 

As in the finite-dimensional case, it turns out to be useful to perform the 
change of variables 
\begin{equation}
\label{eq:1d_transversal} 
 \phi_N(x) = \frac1{\sqrt L}\phi_0 + \sqrt{\eps}\phi_\perp(x)\;,
 \qquad
 \phi_0 = \pscal{e_0}{\phi_N} 
 = \frac1{\sqrt L} \int_\Lambda \phi_N(x)\6x\;, 
\end{equation} 
where $\phi_\perp$ has zero mean. Inserting this in the 
potential~\eqref{eq:potential_infdim} yields 
\begin{equation}
\label{eq:potential_1d_decomposition} 
 \frac{1}{\eps} V \biggpar{\frac1{\sqrt L}\phi_0 + \sqrt{\eps}\phi_\perp(x)} 
 = \frac{1}{\eps} V_0(\phi_0) 
 + \frac12 \pscal{\phi_\perp}{Q_\perp(\phi_0)\phi_\perp}
 + R_\eps(\phi_0,\phi_\perp)\;,
\end{equation} 
where
\begin{align}
V_0(\phi_0) &= \frac{L}{4} \biggpar{\frac{\phi_0^2}{L}-1}^2\;, \\
Q_\perp(\phi_0) &= 
-\Delta_{\perp}-1+\frac{3\phi_0^2}{L}\;, \\
R_\eps(\phi_0,\phi_\perp) 
&= \sqrt{\frac{\eps}{L}}\phi_0 \int_{\Lambda} \phi_\perp(x)^3 \6x
  + \frac{\eps}{4}\int_{\Lambda} \phi_\perp(x)^4 \6x\;.
\end{align}
Here the Laplacian $\Delta_\perp=\Delta_{\perp,N}$ acts on $E_{\perp,N}$, the 
subspace of zero-mean functions in $E_N$. Therefore, the eigenvalues of 
$Q_\perp(\phi_0)$ are given by $\mu_k + 3\phi_0^2/L \geqs 2\pi/L - 1 > 0$, where 
the $\mu_k=\lambda_k-1$ are the eigenvalues of $-\Delta_\perp-1$. We will denote 
by $g_N(\phi_0)$ the Gaussian measure with covariance $Q_\perp(\phi_0)^{-1}$ on 
$E_{\perp,N}$, called a (truncated) Gaussian free field with mass 
$m=(3\phi_0^2/L)^{1/2}$. A crucial estimate is the following one.

\begin{lemma}
\label{lem:moment_GFF} 
There exist constants $C_n$, uniform in $N$ and $\phi_0$, such that the bound   
\begin{equation}
 \biggexpecin{g_N(\phi_0)}{\int_\Lambda \phi_\perp(x)^{2n} \6x} 
 \leqs C_n 
\end{equation} 
holds for all even moments of $\phi_\perp$, while the odd moments vanish.
\end{lemma}
\begin{proof}
Let $\cK_N = \setsuch{k\in\Z}{k\neq0, \abs{k}\leqs N}$. Then we can represent 
$\phi_\perp$ as 
\begin{equation}
 \phi_\perp(x) = \sum_{k\in\cK_N} \frac{Z_k}{\sqrt{\mu_k+m^2}} e_k(x)\;,
\end{equation} 
where the $Z_k$ are i.i.d.\ standard normal random variables. Therefore 
\begin{equation}
 \biggexpecin{g_N(\phi_0)}{\int_\Lambda \phi_\perp(x)^{2n} \6x} 
 = \sum_{k_1,\dots,k_{2n}\in\cK_N} 
 \frac{\expec{Z_{k_1}\dots Z_{k_{2n}}}} 
{\sqrt{\mu_{k_1}+m^2}\dots\sqrt{\mu_{k_{2n}}+m^2}}
\int_\Lambda e_{k_1}(x)^2 \dots e_{k_{2n}}(x)^2\6x\;.
\end{equation} 
The expectation vanishes unless the $k_i$ are pairwise equal. Therefore, we 
obtain 
\begin{equation}
 \biggexpecin{g_N(\phi_0)}{\int_\Lambda \phi_\perp(x)^{2n} \6x} 
 \lesssim \sum_{k_1,\dots,k_n\in\cK_N}
 \frac{1}{(\mu_{k_1}+m^2)\dots(\mu_{k_n}+m^2)}
 = \Biggpar{\sum_{k\in\cK_N} \frac{1}{\mu_k+m^2}}^n\;,
\end{equation} 
which is bounded uniformly in $N$ and $m$. Odd moments vanish by symmetry. 
\end{proof}

We will see in the next chapter that this bound can be substantially improved 
using Wick calculus. For now, it will be sufficient to obtain the following key 
estimate.

\begin{proposition}
\label{prop:exp_moment} 
We have 
\begin{equation}
 \e^{-c\eps} \leqs 
 \bigexpecin{g_N(\phi_0)}{\e^{-R_\eps(\phi_0,\cdot)}}
 \leqs 1 + \e^{C(1+\phi_0^2)}\eps\;,
\end{equation} 
where $c$ and $C$ are uniform in $N$. 
\end{proposition}
\begin{proof}
By \Cref{lem:moment_GFF}, $\expecin{g_N(\phi_0)}{R_\eps} = c_N\eps$, where 
$c_N\leqs c$ is bounded uniformly in $N$ and does not depend on $\phi_0$. 
Therefore the lower bound follows directly from Jensen's inequality, which 
yields 
\begin{equation}
 \bigexpecin{g_N(\phi_0)}{\e^{-R_\eps(\phi_0,\cdot)}}
 \geqs \exp\Bigset{-\bigexpecin{g_N(\phi_0)}{R_\eps(\phi_0,\cdot)}}
 \geqs \exp\bigset{-c\eps}\;.
\end{equation}
For the other direction, we first derive a rough bound on a slightly higher 
exponential moment, using a change of mass in the Gaussian measure. For 
any $p\geqs1$, we have
\begin{equation}
 \bigexpecin{g_N(\phi_0)}{\e^{-pR_\eps(\phi_0,\cdot)}} 
 = \sqrt{\frac{\det Q_\perp(\phi_0)}{\det Q_\perp(0)}}
 \Biggexpecin{g_N(0)}{\exp\biggset{
 -\frac{3\phi_0^2}{2L}\int_{\Lambda} \phi_\perp(x)^2 \6x
 -pR_\eps(\phi_0,\cdot)}}\;.
\end{equation} 
For the ratio of determinants, we note that 
\begin{equation}
 \log \Biggpar{\frac{\det Q_\perp(\phi_0)}{\det Q_\perp(0)}} 
 = \log\prod_{k\in\cK_N} \Biggpar{1+\frac{3\phi_0^2}{L\mu_k}}
 = \sum_{k\in\cK_N} \log\Biggpar{1+\frac{3\phi_0^2}{L\mu_k}}
 \leqs \frac{3\phi_0^2}{L} \sum_{k\in\cK_N} \frac{1}{\mu_k}
 =: 2C_N\phi_0^2\;,
\end{equation} 
where $C_N$ is bounded uniformly in $N$.
Furthermore, using completion of squares, we get 
\begin{align}
\frac{3\phi_0^2}{2L}\int_{\Lambda} \phi_\perp(x)^2 \6x
 +pR_\eps(\phi_0,\phi_\perp)
 &= \int_\Lambda \phi_\perp(x)^2 \Biggbrak{\frac{3\phi_0^2}{2L}
 + p\sqrt{\frac{\eps}{L}}\phi_0 \phi_\perp(x)
  + p\frac{\eps}{4}\phi_\perp(x)^2 }\6x \\
 &= \frac14 \int_\Lambda \phi_\perp(x)^2 \Biggbrak{ 
\Biggpar{\frac{2\sqrt{p}\phi_0}{\sqrt{L}} + \sqrt{\eps p}\phi_\perp(x)}^2 
+ \frac{2(3-2p)\phi_0^2}{L}} \6x\;,
\label{eq:proof_exp_moment} 
\end{align}
which is positive for $p\leqs3/2$. Therefore 
$\bigexpecin{g_N(\phi_0)}{\e^{-(3/2)R_\eps(\phi_0,\cdot)}} \leqs 
\e^{C_N\phi_0^2}$. We now decompose
\begin{align}
\bigexpecin{g_N(\phi_0)}{\e^{-R_\eps}}
&= \bigexpecin{g_N(\phi_0)}{\e^{-R_\eps}\indexfct{R_\eps\geqs0}}
+  \bigexpecin{g_N(\phi_0)}{\e^{-R_\eps}\indexfct{R_\eps<0}} \\
&\leqs \bigprobin{g_N(\phi_0)}{R_\eps\geqs0} 
+ \bigexpecin{g_N(\phi_0)}{(1-R_\eps+\e^{-R_\eps}-1+R_\eps) 
\indexfct{R_\eps<0}} \\
&\leqs \bigprobin{g_N(\phi_0)}{R_\eps\geqs0} 
+ \bigprobin{g_N(\phi_0)}{R_\eps<0} 
- \bigexpecin{g_N(\phi_0)}{R_\eps}
+ \frac12\bigexpecin{g_N(\phi_0)}{R_\eps^2\e^{-R_\eps}} \\
&\leqs 1 - \bigexpecin{g_N(\phi_0)}{R_\eps}
+ \frac12\bigexpecin{g_N(\phi_0)}{R_\eps^6}^{1/3} 
\bigexpecin{g_N(\phi_0)}{\e^{-(3/2)R_\eps}}^{2/3}\;,
\label{eq:proof_exp_moment2} 
\end{align}
where we have used the fact that $\e^{-x}-1+x \leqs \frac12x^2\e^{-x}$ for 
$x<0$ and H\"older's inequality. Since by \Cref{lem:moment_GFF}, 
%$\expecin{g_N(\phi_0)}{R_\eps} = \Order{\eps}$ and 
$\expecin{g_N(\phi_0)}{\smash{R_\eps^6}} = \Order{\eps^3\phi_0^6}$, the 
result follows from the bound on 
$\expecin{g_N(\phi_0)}{\smash{\e^{-(3/2)R_\eps(\phi_0,\cdot)}}}$ for an 
appropriate choice of $C$. 
\end{proof}

We can now estimate the partition function and capacity, in a similar way as we 
did in the finite-dimensional case. We define the sets $A$ and $B$ as 
\begin{equation}
\label{eq:condAB_1d} 
 A = -B = \bigsetsuch{y}{\abs{\phi_0+\sqrt{L}}\leqs\delta, \phi_\perp\in 
D_\perp}, 
\end{equation} 
with $D_\perp$ containing a ball of radius $\sqrt{\log(\eps^{-1})}$ in 
$H^s(\Lambda)$ for some $s<\frac12$. Then we have the following result. 

\begin{proposition}
\label{prop:Z-cap_AC-1d} 
The partition function and the capacity satisfy 
\begin{align}
\cZ_N := \int_{E_N} \e^{-V(\phi)/\eps} \6\phi 
&= 2 \frac{(2\pi\eps)^{N+1/2}} {\sqrt{\nu_0\det(-\Delta_\perp+2)}} 
\bigbrak{1+\Order{\eps}}\;, \\
\capacity(A,B) &= \frac{(2\pi\eps)^{N+1/2}\sqrt{\abs{\mu_0}}\e^{-V_0(0)/\eps}}
{2\pi\sqrt{\det(-\Delta_\perp-1)}}
\bigbrak{1+\Order{\eps}}\;, 
\end{align}
where $\Delta_\perp=\Delta_{\perp,N}$, $\nu_0=2$, $\mu_0=-1$, and the error 
terms are uniform in the cut-off $N$. 
\end{proposition}
\begin{proof}
The partition function can be written as 
\begin{equation}
 \cZ_N = \eps^N \int_{-\infty}^\infty \e^{-V_0(\phi_0)/\eps}
 \frac{(2\pi)^N}{\sqrt{\det Q_\perp(\phi_0)}}
 \bigexpecin{g_N(\phi_0)}{\e^{-R_\eps(\phi_0,\cdot)}} \6\phi_0\;.
\end{equation} 
The result then follows from \Cref{prop:exp_moment}, using Laplace asymptotics 
for the integral over $\phi_0$, which is dominated by values of $\phi_0$ near 
$\pm\sqrt{L}$. The error term $\eps\e^{C(1+\phi_0^2)}$ only yields an error of 
order $\eps$, because it is dominated by the factor $\e^{-V_0(\phi_0)/\eps}$. 

For the upper bound on the capacity, we use as before the Dirichlet principle, 
with a test function 
\begin{equation}
\label{eq:cap_h0} 
 h(\phi) = h_0(\phi_0) = \frac{1}{c_0} \int_{\phi_0}^a 
\e^{V_0(\xi)/\eps}\6\xi\;, 
\qquad 
c_0 =  \int_{-a}^a \e^{V_0(\xi)/\eps}\6\xi\;. 
\end{equation} 
Writing $\cE(h)$ in terms of 
$\bigexpecin{g_N(\phi_0)}{\e^{-R_\eps(\phi_0,\cdot)}}$ yields the upper bound 
in a similar way as for the partition function. For the lower bound on the 
capacity, we use the Thomson principle. Here it is useful to decompose the 
potential as 
\begin{equation}
 \frac{1}{\eps} V \biggpar{\frac1{\sqrt L}\phi_0 + \sqrt{\eps}\phi_\perp(x)} 
 = \frac{1}{\eps} V_0(\phi_0) 
 + V_\perp(\phi_\perp) + V_+(\phi_0,\phi_\perp)\;,
\end{equation} 
where
\begin{align}
V_\perp(\phi_\perp) &= \frac12 \pscal{\phi_\perp}{Q_\perp(0)\phi_\perp}
+ \frac{\eps}{4} \int_\Lambda \phi_\perp(x)^4 \6x\;, \\
V_+(\phi_0,\phi_\perp) &= \frac{3\phi_0^2}{2L} \int_\Lambda \phi_\perp(x)^2\6x 
+ \sqrt{\eps}\phi_0 \int_\Lambda \phi_\perp(x)^3\6x\;.
\end{align}
We define the divergence-free unit flow 
\begin{equation}
 \ph(\phi) = \frac{1}{K} \indicator{D_\perp}(\phi_\perp) 
\e^{-V_\perp(\phi_\perp)} e_0\;, 
\qquad 
K = \eps^N \int_{D_\perp}\e^{-V_\perp(\phi_\perp)} \6\phi_\perp\;.
\end{equation} 
Substituting in the quadratic form~\eqref{eq:def_D} yields 
\begin{equation}
 \cD(\ph) = \frac{1}{\eps K^2} \int_{-\infty}^\infty \e^{V_0(\phi_0)/\eps}
 \frac{(2\pi\eps)^N}{\sqrt{\det Q_\perp(0)}}
 \Biggexpecin{\mu_N(0)}{\indicator{D_\perp} \exp\biggset{-\frac{\eps}{4} 
\int_\Lambda \phi_\perp(x)^4 \6x + V_+(\phi_0,\cdot)}} \6\phi_0\;.
\end{equation} 
Proceeding as in~\Cref{prop:exp_moment}, one checks that the gaussian 
expectation is again equal to $1+\Order{\eps}$, despite the different sign of 
the cubic term and the extra quadratic term. A similar analysis can be 
performed for $K$. The lower bound then follows from Thomson's principle.
\end{proof}

\Cref{prop:Z-cap_AC-1d} is one of the essential building blocks of the 
following theorem, which is the main result of this section. 

\begin{theorem}[Eyring-Kramers law for the one-dimensional Allen--Cahn SPDE]
\label{thm:EK-1d} 
Assume that $L<2\pi$ and $A$ and $B$ are as in~\eqref{eq:condAB_1d}. Then 
\begin{equation}
 \bigexpecin{\phi^*_-}{\tau_B}
 = \frac{2\pi}{\abs{\mu_0}}
 \frac{\e^{[V(\phi^*_0)-V(\phi^*_-)]/\eps}}{\sqrt{\bigabs{\det\bigpar{\one + 
3(-\Delta-1)^{-1}}}}} \bigbrak{1+\Order{\eps}}\;,
\end{equation} 
where $V(\phi^*_0)-V(\phi^*_-) = L/4$ and 
the determinant is to be interpreted as the inverse of the Fredholm 
determinant~\eqref{eq:Fredholm} (cf.~\Cref{exo:Fredholm}). 
\end{theorem}
\begin{proof}[\Sketch]
For finite cut-off $N$, and when starting in the equilibrium distribution on 
$\partial A$, the result follows from \Cref{prop:Z-cap_AC-1d} and 
\Cref{thm:magic_formula}. To extend this to the infinite-dimensional equation, 
we have to check two more things. The first one is that first-hitting times of 
$B$ in the finite-dimensional system converge to those of the 
infinite-dimensional one. This can be done by combining an a priori estimate on 
$\expec{\tau_B^2}$ with a sample-path approximation argument given 
in~\cite[Theorem~3.1]{Blomker_Jentzen_13}, cf.~\cite[Theorem~3.1 
and~Proposition~3.4]{BG12a}. The second thing to be done is to prove that one 
can replace the initial distribution on $\partial A$ by a start in $\phi^*_-$. 
This can be achieved by using the coupling argument given 
in~\cite[Corollary~3.1]{Martinelli_Olivieri_Scoppola_89}, cf.~\cite[Theorem~3.5 
and~Proposition~3.6]{BG12a}. 
\end{proof}

%%%%%%%%%%%%%%%%%%%%%%%%%%%%%%%%%%%%%%%%%%%%%%%%%%%%%%%%%%%%%%%%%%%%%%%%%%%%%%%%

\section{Bibliographical notes}
\label{sec:1dbib} 

The deterministic PDE commonly called Allen--Cahn equation was introduced 
in~\cite{ChafeeInfante_74} and \cite{AllenCahn}. Detailed descriptions of its 
dynamics can be found for instance in~\cite{Carr_Pego89,Chen_2004}. 

The definitions of parabolic H\"older spaces and related Schauder estimates 
have been taken from~\cite{Hairer2014} and~\cite{Chandra_Weber_LN17}, while 
properties of Gaussian fields, in particular in relation with fractional 
Sobolev spaces, can be found in~\cite{Hairer_LN_2009}. 

Existence and uniqueness of solutions of the Allen--Cahn SPDE have been first 
established in~\cite{Faris_JonaLasinio82}. Regarding the general existence and 
uniqueness theory, as well as properties of invariant measures, we have mainly 
followed~\cite{DaPrato_Zabczyk_Ergodicity} with some input 
from~\cite{Hairer_LN_2009}. The sharper bounds on the decay of supremum norms 
are from~\cite{Cerrai_1999}. 

The large-deviation principle for the stochastic Allen--Cahn equation has been 
first obtained in~\cite{Faris_JonaLasinio82}, and later extended 
in~\cite{Freidlin88,Sowers92,ChenalMillet1997}. 

Eyring--Kramers laws for the one-dimensional Allen--Cahn equation have been 
obtained in~\cite{BG12a} (using spectral Galerkin approximations) and 
in~\cite{Barret15} (using lattice discretisations). Results in these papers are 
much more general than those presented here, as they apply to equations with a 
general confining nonlinearity, with some smoothness assumptions. Results 
in~\cite{BG12a} include the bifurcation regime and the dynamics for $L>2\pi$ 
for periodic boundary conditions, in which the transition states are 
degenerate. 
The proof presented here is somewhat different from the proofs given in those 
works, which did not use Gaussian measure theory and the Thomson principle, but 
were instead based on Hausdorff--Young inequalities and direct approximations 
of Dirichlet forms, in the spirit of~\cite{BEGK}. 

%%%%%%%%%%%%%%%%%%%%%%%%%%%%%%%%%%%%%%%%%%%%%%%%%%%%%%%%%%%%%%%%%%%%%%%%%%%%%%%%

\chapter{Allen--Cahn SPDE in two space dimensions}
\label{ch:dim2} 

The particle system considered in \Cref{ch:diff} makes also sense in two 
dimensions, that is, for a two-dimensional lattice of $N^2$ particles, each one 
interacting with its $4$ nearest neighbours. Proceeding as in \Cref{ch:dim1}, 
we obtain formally a continuum limit given by the SPDE 
\begin{equation}
\label{eq:AC-2d} 
 \partial_t \phi(t,x) = \Delta \phi(t,x) + \phi(t,x) - \phi(t,x)^3 
 + \sqrt{2\eps} \xi(t,x)\;,
\end{equation} 
where $x$ now belongs to the two-dimensional torus $\Lambda=(\R/(L\Z))^2$, and 
the Laplacian $\Delta=\partial_{x_1x_1}+\partial_{x_2x_2}$ acts on both 
components of $x$. The PDE obtained in the deterministic case $\eps=0$ is 
perfectly well-defined. 

\begin{exercise}
Compute the eigenvalues of the Laplacian $\Delta$ for periodic boundary 
conditions on $\Lambda$, and use it to determine the stability of the constant 
stationary solutions of the deterministic PDE. Argue that for $L<2\pi$, we are 
again in the double-well situation, as there can be no non-constant stationary 
solutions. Prove this fact, using 
\begin{equation}
 W(\phi) = \frac12 \int_\Lambda \norm{\nabla\phi(x)}^2\6x
\end{equation} 
as a Lyapunov function. 
\end{exercise}

It turns out, however, that the SPDE~\eqref{eq:AC-2d} is not well-posed. This 
is related to the fact that space-time white noise is more irregular in 
dimension $2$ than in dimension $1$. Before discussing a cure for this problem, 
we first have to understand in more detail where the problem comes from. 

%%%%%%%%%%%%%%%%%%%%%%%%%%%%%%%%%%%%%%%%%%%%%%%%%%%%%%%%%%%%%%%%%%%%%%%%%%%%%%%%

\section{The need for renormalisation}
\label{sec:dim2_renorm} 

Space-time white noise on $\R\times\Lambda$ can be defined exactly as in 
\Cref{def:space-time-white-noise}, except that $\Lambda$ now denotes the 
two-dimensional torus. The scaling operator in~\eqref{eq:S-tau-lambda} now 
takes the form 
\begin{equation}
 (\cS^{\tau,\lambda}\ph)(t,x)
 := \frac{1}{\tau\lambda^2} \ph \biggpar{\frac{t}{\tau},\frac{x}{\lambda}}\;,
\end{equation} 
and the scaling property of \Cref{prop:xi_scaling} becomes 
\begin{equation}
 \xi_{\tau,\lambda} \eqinlaw \frac{1}{\sqrt{\tau}\lambda}\xi\;.
\end{equation} 
The parabolic H\"older spaces are defined as before, with the parabolic norm 
defined by 
\begin{equation}
 \norm{(t,x)-(s,y)}_\fraks = \abs{t-s}^{1/2} + \sum_{i=1}^2 \abs{x_i-y_i}\;,
\end{equation} 
and the scaling operator in~\eqref{eq:S-lambda-z} replaced by 
\begin{equation}
 (\cS^{\lambda}_{t,x}\ph)(s,y)
 := \frac{1}{\lambda^4} \ph 
\biggpar{\frac{s-t}{\lambda^2},\frac{y-x}{\lambda}}\;.
\end{equation} 
Making the necessary changes in the proof of 
\Cref{thm:space_time_white_noise_reg}, one now finds that 
\begin{equation}
\label{eq:reg_noise_d2} 
 \xi\in\cC^{-2-\kappa}_\fraks
 \qquad \text{for any $\kappa>0$\;.}
\end{equation} 

\begin{exercise}
Prove~\eqref{eq:reg_noise_d2} by adapting the argument used in 
\Cref{thm:space_time_white_noise_reg}.
\end{exercise}

The Schauder estimate stated in \Cref{thm:Schauder} remains true in the present 
setting. The problem is, however, that the stochastic convolution $P*\xi$ 
belongs to $\cC^\alpha_\fraks$ for any $\alpha<0$ only, instead any 
$\alpha<\frac12$ in the one-dimensional case. Therefore, the stochastic 
convolution is only a distribution, albeit rather close to being a function. 
Since there is no canonical way of multiplying two distributions, the 
fixed-point equation~\eqref{eq:AC-1d-FP} becomes ill-defined, as it requires  
the nonlinear function $F$ to be applied to the stochastic convolution. 

As the stochastic convolution is a Gaussian process, its singularity can be 
traced to the behaviour of the covariance, which diverges near the diagonal 
$y=x$ (see also~\eqref{eq:covariance_GFF}). Therefore, a regularisation becomes 
possible by subtracting an \lq\lq infinite constant\rq\rq\ from the right-hand 
side of~\eqref{eq:AC-2d}, a process known as \emph{renormalisation}. An 
important tool to understand how it works is Wick calculus for Gaussian random 
variables.  

%%%%%%%%%%%%%%%%%%%%%%%%%%%%%%%%%%%%%%%%%%%%%%%%%%%%%%%%%%%%%%%%%%%%%%%%%%%%%%%%

\section{Wick calculus}
\label{sec:2d_wick} 

%%%%%%%%%%%%%%%%%%%%%%%%%%%%%%%%%%%%%%%%%%%%%%%%%%%%%%%%%%%%%%%%%%%%%%%%%%%%%%%%

\subsection{Isserlis' theorem}
\label{ssec:2d_isserlis} 

Consider a Gaussian random vector $X$ on $\R^N$ with covariance matrix $C$. 
Recall that this means that for any test function $F$, one has 
\begin{equation}
 \bigexpec{F(X)} = \frac{1}{\cZ} \int_{\R^N} F(x) \e^{-\pscal{x}{C^{-1}x}/2} 
\6x\;, 
\end{equation} 
where $C$ has entries $C_{ij}=\expec{X_iX_j}$, and $\cZ$ is such that 
$\expec{1}=1$. Moments of $X$ can easily be computed by the following 
integration-by-parts formula, which is a finite-dimensional version of the 
Schwinger--Dyson equations from Quantum Field Theory. 

\begin{lemma}
\label{lem:ibp} 
For any $i\in\set{1,\dots,N}$ we have 
\begin{equation}
\label{eq:ibp} 
 \bigexpec{X_i F(X)} = \sum_{j=1}^N C_{ij} \bigexpec{\partial_j F(X)}
\end{equation} 
for all real-valued differentiable $F$ such that both sides of the equality are 
well-defined. 
\end{lemma}
\begin{proof}
Leibniz' rule yields 
\begin{equation}
 \dpar{}{x_j} \Bigpar{F(x)\e^{-\pscal{x}{C^{-1}x}/2}}
 = \biggbrak{\dpar{F}{x_j}(x)
 -\frac12 F(x)\dpar{}{x_j} \pscal{x}{C^{-1}x} }
 \e^{-\pscal{x}{C^{-1}x}/2} \;.
\end{equation} 
Since we have 
\begin{equation}
 \frac12 \sum_{j=1}^N C_{ij} \dpar{}{x_j} \pscal{x}{C^{-1}x} 
 = \sum_{j,k=1}^N C_{ij}C^{-1}_{jk} x_k 
 = \sum_{k=1}^N \delta_{ik} x_k = x_i\;,
\end{equation} 
it follows that 
\begin{equation}
 x_i F(x)\e^{-\pscal{x}{C^{-1}x}/2} 
 = \sum_{j=1}^N C_{ij} 
 \biggbrak{-\dpar{}{x_j} \Bigpar{F(x)\e^{-\pscal{x}{C^{-1}x}/2}} + 
 \dpar{F}{x_j} \e^{-\pscal{x}{C^{-1}x}/2}}\;.  
\end{equation} 
Integrating over the whole space, we see that the boundary terms vanish, and
the result follows at once.
\end{proof}

An immediate consequence of this result is the following theorem, due to 
Leon Isserlis. 

\begin{theorem}[Isserlis]
\label{thm:Isserlis} 
 For any $1\leqs n\leqs N/2$, we have 
 \begin{align}
  \bigexpec{X_1\dots X_{2n}} &= 
  \sum_{\text{pairings $\cP$ of $\set{1,\dots, 2n}$}}
  \;\prod_{\set{i,j}\in\cP} \bigexpec{X_iX_j}\;, \\
  \bigexpec{X_1\dots X_{2n-1}} &= 0\;,
 \end{align} 
 where a \emph{pairing} of $\set{1,\dots, 2n}$ is a partition 
$\cP=\set{\set{i_1,j_1},\dots,\set{i_n,j_n}}$ of this set into disjoint subsets 
of two elements.
\end{theorem}
\begin{proof}
By induction on $n$, applying~\eqref{eq:ibp} with $i=1$ and $F(X)$ of the form 
$X_2\dots 
X_m$. 
\end{proof}

\begin{remark}
An alternative proof relies on the identity 
$\expec{\e^{\icx\pscal{\ell}{X}}} = \e^{-\pscal{\ell}{C\ell}/2}$, expanding 
both sides in powers of $\ell=\transpose{(\ell_1,\dots,\ell_n)}$ and 
identifying coefficients. 
\end{remark}

For instance, in the case $2n=4$, we obtain 
\begin{equation}
 \bigexpec{X_1X_2X_3X_4} 
 = \bigexpec{X_1X_2}\bigexpec{X_3X_4} + \bigexpec{X_1X_3}\bigexpec{X_2X_4} + 
\bigexpec{X_1X_4}\bigexpec{X_2X_3}\;.
\end{equation} 
A convenient graphical way of representing this relation is the following:
\begin{equation}
 \bigexpec{X_1X_2X_3X_4} 
 =
\raisebox{-25pt}{\tikz{
 \path[use as bounding box] (-0.5, -1.5) rectangle (1.5,0.5);
 \node[mynode,label=above:{\scriptsize 1}] (1) at (0,0) {};
 \node[mynode,label=below:{\scriptsize 2}] (2) at (0,-1) {};
 \node[mynode,label=above:{\scriptsize 3}] (3) at (1,0) {};
 \node[mynode,label=below:{\scriptsize 4}] (4) at (1,-1) {};
 \draw[thick,bare] (1) -- (2);  
 \draw[thick,bare] (3) -- (4);  
 \drawbox;
 }}
 +
\raisebox{-25pt}{\tikz{
 \path[use as bounding box] (-0.5, -1.5) rectangle (1.5,0.5);
 \node[mynode,label=left:{\scriptsize 1}] (1) at (0,0) {};
 \node[mynode,label=left:{\scriptsize 2}] (2) at (0,-1) {};
 \node[mynode,label=right:{\scriptsize 3}] (3) at (1,0) {};
 \node[mynode,label=right:{\scriptsize 4}] (4) at (1,-1) {};
 \draw[thick,bare] (1) -- (3);  
 \draw[thick,bare] (2) -- (4);  
 \drawbox;
 }}
+
\raisebox{-25pt}{\tikz{
 \path[use as bounding box] (-0.5, -1.5) rectangle (1.5,0.5);
 \node[mynode,label=above left:{\scriptsize 1}] (1) at (0,0) {};
 \node[mynode,label=below left:{\scriptsize 2}] (2) at (0,-1) {};
 \node[mynode,label=above right:{\scriptsize 3}] (3) at (1,0) {};
 \node[mynode,label=below right:{\scriptsize 4}] (4) at (1,-1) {};
 \draw[thick,bare] (1) -- (4);  
 \draw[thick,bare] (3) -- (2);  
 \drawbox;
 }}
\end{equation} 

\begin{exercise}
\label{exo:Isserlis} 
Determine the number of pairings of $\set{1,\dots,2n}$. Use this information to 
compute $\expec{X^{2n}}$ when $X$ follows a standard normal distribution, and 
prove~\eqref{eq:moments_Gaussian}. 
\end{exercise}

%%%%%%%%%%%%%%%%%%%%%%%%%%%%%%%%%%%%%%%%%%%%%%%%%%%%%%%%%%%%%%%%%%%%%%%%%%%%%%%%

\subsection{Hermite polynomials}
\label{ssec:2d_hermite} 

Consider the one-dimensional Ornstein--Uhlenbeck process, solving the SDE  
\begin{equation}
 \6X_t = -X_t \6t + \sqrt{2\eps}\6W_t\;.
\end{equation} 
We know from \Cref{sec:diffinv} that its invariant distribution $\pi$ has 
density $\e^{-x^2/(2\eps)}/\sqrt{2\pi\eps}$ with respect to Lebesgue measure, 
while its generator 
\begin{equation}
 \cL = \eps \e^{x^2/(2\eps)} \dtot{}{x} \e^{-x^2/(2\eps)} \dtot{}{x}
\end{equation} 
is self-adjoint in $L^2(\R,\pi(\6x))$. Furthermore, one easily checks that the 
operators 
\begin{align}
 a &= \e^{x^2/(2\eps)} \biggpar{x + \eps \dtot{}{x}} 
\e^{-x^2/(2\eps)}
 = \eps \dtot{}{x}\;, \\
 a^\dagger &= -\e^{x^2/(2\eps)} \eps \dtot{}{x} \e^{-x^2/(2\eps)}
 = x - \eps \dtot{}{x}
\end{align} 
are mutually adjoint in $L^2(\R,\pi(\6x))$, and satisfy the relations 
\begin{equation}
 a^\dagger a = -\eps\cL\;, 
 \qquad
 aa^\dagger - a^\dagger a = \eps\;.
\end{equation} 
These relations are useful to find eigenfunctions of $\cL$. Indeed, we 
obviously have $\cL H_0=0$, where $H_0(x)=1$ is constant. Now if $H$ is an 
eigenfunction of $\cL$ with eigenvalue $-\lambda$, we have 
\begin{equation}
 -\cL(a^\dagger H) = \frac{1}{\eps} a^\dagger a a^\dagger H 
 = \frac{1}{\eps} a^\dagger (a^\dagger a + \eps)H
 = (\lambda+1) a^\dagger H\;,
\end{equation} 
showing that $a^\dagger H$ is an eigenfunction of $\cL$ with eigenvalue 
$-(\lambda+1)$. We conclude that the functions in the family $(H_n)_{n\in\N_0}$ 
defined recursively by 
\begin{equation}
\label{eq:hermite_recursive} 
 H_{n+1}(x;\eps) = (a^\dagger H_n)(x;\eps)
 = x H_n(x;\eps) - \eps\dtot{}{x} H_n(x;\eps)
\end{equation} 
are orthogonal, and satisfy 
\begin{equation}
 \cL H_n(x;\eps) = - n H_n(x;\eps)\;, 
 \qquad n\in\N_0\;.
\end{equation} 
The corresponding eigenfunctions of $\cL^\dagger$ are simply given by
$\e^{-x^2/(2\eps)}H_n(x;\eps)$. 

\begin{remark}
The operators $a^\dagger$ and $a$ behave in the same way as creation and 
annihilation operators for the quantum harmonic oscillator. In fact, the 
generator $\cL$ is conjugated to the Hamiltonian of the harmonic oscillator via 
the transformation $\cL\mapsto\e^{-x^2/(4\eps)}\cL \e^{x^2/(4\eps)}$, and the 
operators $a^\dagger$ and $a$ are conjugated in the same way to those of the 
harmonic oscillator. 
\end{remark}

In fact, one can show that the family $(H_n)_{n\in\N_0}$ forms a complete 
orthonormal basis of $L^2(\R,\pi(\6x))$. See for 
instance~\cite[Lemma~1.1.2]{nualart2006malliavin}. 

\begin{definition}[Hermite polynomials]
\label{def:Hermite} 
The polynomial 
\begin{equation}
 H_n(x;\eps) = \bigpar{a^\dagger}^n H_0(x;\eps)
 = (-\eps)^n \e^{x^2/(2\eps)} \dtot{^n}{x^n} \e^{-x^2/(2\eps)}
\end{equation} 
is called the \emph{$n$th Hermite polynomial} with variance $\eps$. 
\end{definition}

Hermite polynomials can easily be computed by the recurrence 
relation~\eqref{eq:hermite_recursive}. The first few of them are given by 
\begin{align}
H_0(x;\eps) &= 1 \\
H_1(x;\eps) &= x \\
H_2(x;\eps) &= x^2-\eps \\
H_3(x;\eps) &= x^3-3\eps x \\
H_4(x;\eps) &= x^4-6\eps x^2+3\eps^2\;.
\end{align}
A useful alternative expression for Hermite polynomials relies on their 
generating function.

\begin{lemma}
\label{lem:hermite_generating} 
For every $t,x\in\R$ and $\eps\geqs 0$, one has 
\begin{equation}
\label{eq:hermite_generating} 
G(t,x;\eps) :=
\e^{tx - \eps t^2/2} = \sum_{n=0}^\infty \frac{t^n}{n!} H_n(x;\eps)\;.
\end{equation} 
\end{lemma}

\begin{exercise}
Prove \Cref{lem:hermite_generating}. \Hint Take the derivative 
with respect to $t$ on both sides, and use the recurrence  
relation~\eqref{eq:hermite_recursive}. 
Another proof (with a slightly different normalisation) can be found 
in~\cite[Lemma~1.1.1]{nualart2006malliavin}. 
\end{exercise}

A very important consequence of \Cref{lem:hermite_generating} is the following 
identity, which can be seen as an analogue of Isserlis' theorem for Hermite 
polynomials of Gaussian random variables. 

\begin{lemma}
\label{lem:hermite_moments} 
Let $X$ and $Y$ be jointly Gaussian centred random variables, of respective 
variance $\eps_1$ and $\eps_2$. Then for any $n,m\geqs0$, one has 
\begin{equation}
\label{eq:Wick_identity} 
 \bigexpec{H_n(X;\eps_1) H_m(Y;\eps_2)} = 
 \begin{cases}
  n! \bigexpec{XY}^n & \text{if $n=m$\;,} \\
  0 & \text{otherwise\;.}
 \end{cases}
\end{equation} 
\end{lemma}
\begin{proof}
Let $C$ denote the covariance matrix of $Z=\transpose{(X,Y)}$. Then the 
expression 
\begin{equation}
 \Bigexpec{\e^{\pscal{\ell}{Z}}} = \exp\biggset{\frac12\pscal{\ell}{C\ell}}
\end{equation} 
for the Laplace transform of $Z$, applied with $\ell=\transpose{(t,s)}$, 
implies that 
\begin{equation}
 \Bigexpec{\e^{tX-\eps_1t^2/2}\e^{sY-\eps_2s^2/2}}
 = \e^{ts\expec{XY}}\;.
\end{equation} 
Using~\eqref{eq:hermite_generating} and expanding both sides yields 
\begin{equation}
 \sum_{n,m=0}^\infty \frac{t^ns^m}{n!m!} \bigexpec{H_n(X;\eps_1) H_m(Y;\eps_2)} 
= \sum_{k=0}^\infty \frac{(ts)^k}{k!} \bigexpec{XY}^k\;.
\end{equation} 
Identifying coefficients of the power series gives the result. 
\end{proof}

A convenient graphical way of representing~\eqref{eq:Wick_identity}, for 
instance in the case $n=4$, is
\begin{equation}
\label{eq:Wick_graphical} 
\biggexpec{\FDVfour{vedge}{vedge}{vedge}{vedge}
\FDVfour{vedge}{vedge}{vedge}{vedge}} =  
4!\FDQzero{vedge}{vedge}{vedge}{vedge}\;.
\end{equation} 
Here each diagram with four legs represents a term $H_4$, and each edge on the 
right-hand side represents a covariance $\expec{XY}$. The combinatorial factor 
$4!$ counts the number of ways of pairing the four legs, with the rule that 
only legs from different diagrams can be paired. This is the main difference 
with Isserlis' theorem, which allowed for all possible pairings. Similar 
identities as~\eqref{eq:Wick_identity} can be obtained for products of more 
than two Hermite polynomials, by summing over all pairings of legs belonging to 
different terms. 

Note that if we put $m=0$ in~\eqref{eq:Wick_identity}, we obtain  
\begin{equation}
\label{eq:Hermite_centered} 
 \bigexpec{H_n(X;\eps)} = \delta_{n0}\;,
\end{equation} 
which shows that Hermite polynomials of Gaussian random variables are centred 
for $n\geqs1$. 
Another useful consequence of \Cref{lem:hermite_generating} is the following 
binomial formula for Hermite polynomials.

\begin{lemma}
\label{lem:hermite_binomial}
For any $x,y\in\R$, $\eps_1,\eps_2\geqs0$ and $n\in\N_0$, 
\begin{equation}
 H_n(x+y,\eps_1+\eps_2) = \sum_{m=0}^n \binom{n}{m} 
 H_m(x,\eps_1)H_{n-m}(y,\eps_2)\;.
\end{equation} 
\end{lemma}
\begin{proof}
Expanding the identity
\begin{equation}
 \e^{t(x+y)-(\eps_1+\eps_2)t^2/2}
  = \e^{tx-\eps_1t^2/2}\e^{ty-\eps_2t^2/2}
\end{equation} 
yields 
\begin{equation}
 \sum_{n=0}^\infty \frac{t^n}{n!} H_n(x+y;\eps_1+\eps_2)
= \sum_{m=0}^\infty \frac{t^m}{m!} H_m(x;\eps_1)
 \sum_{k=0}^\infty \frac{t^k}{k!} H_k(y;\eps_2)\;. 
\end{equation} 
Comparing coefficients of $t^n$ gives the result. 
\end{proof}

%%%%%%%%%%%%%%%%%%%%%%%%%%%%%%%%%%%%%%%%%%%%%%%%%%%%%%%%%%%%%%%%%%%%%%%%%%%%%%%%

\subsection{The two-dimensional Gaussian free field}
\label{ssec:2d_GFF} 

We define a Fourier basis of $\Lambda=(\R/(L\Z))^2$ by 
\begin{equation}
\label{eq:Fourier_2d} 
 e_k(x) = e_{k_1}(x_1) e_{k_2}(x_2)\;, 
 \qquad 
 k=(k_1,k_2)\in\Z^2\;,
\end{equation} 
where the $e_{k_i}$ are the one-dimensional basis functions defined 
in~\eqref{eq:Fourier_basis}. The slight abuse of notation made by using the 
same symbol for one- and two-dimensional basis functions will not matter in 
what follows. The $e_k$ are eigenfunctions of the Laplacian on $\Lambda$ with 
periodic boundary conditions, satisfying
\begin{equation}
 \Delta e_k = -\lambda_k e_k\;, 
 \qquad 
 \lambda_k = \biggpar{\frac{2\norm{k}\pi}{L}}^2\;.
\end{equation} 
For $N\in\N$, define sets of indices $\cK_N = \setsuch{k\in\Z^2}{\abs{k}\leqs 
N}$ and $\cK_N^+ = \setsuch{k\in\cK_N}{k_1,k_2>0}$, where 
$\abs{k}=\abs{k_1}+\abs{k_2}$. Let $E_N$ be the space spanned by 
$\setsuch{e_k}{k\in\cK_N}$. Note the identities 
\begin{align}
 e_{(k_1,k_2)}(x)^2 + e_{(k_1,-k_2)}(x)^2 
 +  e_{(-k_1,k_2)}(x)^2 + e_{(-k_1,-k_2)}(x)^2 &= \frac{4}{L^2} 
 && \forall (k_1,k_2)\in \cK_N^+\;, \\
 e_{(k_1,0)}(x)^2 + e_{(-k_1,0)}(x)^2 
 =  e_{(0,k_2)}(x)^2 + e_{(0,-k_2)}(x)^2 &= \frac{2}{L^2} 
 && \forall k_1, k_2 \neq 0\;. 
 \label{eq:fourier_2d_symmetry} 
\end{align} 

\begin{definition}[Two-dimensional Gaussian free field]
The truncated two-dimensional \emph{Gaussian free field (GFF)} with covariance 
$(-\Delta_N+1)^{-1}$ on $\Lambda$ is defined as 
\begin{equation}
 \GFFN(x) := \sum_{k\in\cK_N}
 \frac{Z_k}{\sqrt{\lambda_k+1}} e_k(x)\;,
\end{equation} 
where $\Delta_N$ is the restriction of $\Delta$ to $E_N$ and 
the $Z_k$ are i.i.d.\ standard normal random variables. 
\end{definition}

Thanks to the symmetry relations~\eqref{eq:fourier_2d_symmetry}, 
$\expec{\GFFN(x)^2}$ does not depend on $x$, and is thus equal to 
\begin{align}
 C_N &:= \frac{1}{L^2} \int_\Lambda \expec{\GFFN(x)^2} \6x 
 = \frac{1}{L^2} \int_\Lambda \sum_{k\in\cK_N} \frac{1}{\lambda_k+1} 
e_k(x)^2\6x \\
 &= \frac{1}{L^2} \sum_{k\in\cK_N} \frac{1}{\lambda_k+1}
 = \frac{1}{L^2} \Tr\bigbrak{(-\Delta_N+1)^{-1}}\;.
\label{eq:def_CN} 
\end{align} 
Since $k$ is two-dimensional and $\lambda_k$ grows like $\norm{k}^2$, this 
constant diverges like $\log N$. 

\begin{exercise}
Show that 
\begin{equation}
\label{eq:asymp_CN} 
C_N = \frac{\log N}{2\pi} + \Order{1} 
\end{equation} 
as $N\to\infty$. \Hint View~\eqref{eq:def_CN} as a Riemann sum and integrate 
using polar coordinates. 
\end{exercise}

Thus we see that the non-truncated two-dimensional GFF has infinite variance at 
every point! 
Fortunately, since this variance is constant, we can subtract it in 
order to get a meaningful object in the limit, a procedure which is akin to an 
It\^o correction. This leads to the following definition.

\begin{definition}[Wick powers of the GFF]
For any integer $n\in\N$, the $n$th Wick power of the truncated GFF is defined 
as 
\begin{equation}
 \Wick{\GFFN^n(x)} =
 \Wick{\GFFN^n(x)}_{C_N} := H_n(\GFFN(x);C_N)\;,
 \qquad 
 x\in\Lambda\;.
\end{equation} 
\end{definition}

Note in particular that owing to~\eqref{eq:Hermite_centered}, all Wick powers 
are centred random variables at any $x\in\Lambda$. Furthermore, 
\Cref{lem:hermite_moments} implies the following important result, which 
is a generalisation of \Cref{lem:moment_GFF} to the two-dimensional case.

\begin{proposition}
\label{prop:Wick_powers} 
For every $n\in\N$, we have
\begin{equation}
 \sup_{N\geqs1} \Biggexpec{\Biggpar{\frac{1}{L^2} \int_\Lambda 
\Wick{\GFFN^n(x)}\6x}^2}  < \infty\;.
\end{equation} 
\end{proposition}
\begin{proof}
By~\eqref{eq:Wick_identity}, we have 
\begin{align}
\Biggexpec{\Biggpar{\int_\Lambda \Wick{\GFFN^n(x)}\6x}^2}
&= \int_\Lambda\int_\Lambda \bigexpec{ \Wick{\GFFN^n(x)}  
\Wick{\GFFN^n(y)}} \6x\6y \\
&= n!\int_\Lambda\int_\Lambda \bigexpec{\GFFN(x)\GFFN(y)}^n \6x\6y \\
&= n!\int_\Lambda\int_\Lambda \Biggpar{\sum_{k,\ell\in\cK_N} 
\frac{\expec{Z_kZ_\ell}}{\sqrt{(\lambda_k+1)(\lambda_\ell+1)}} 
e_k(x)e_\ell(y)}^n \6x\6y \\
&= n!\int_\Lambda\int_\Lambda \Biggpar{\sum_{k\in\cK_N} 
\frac{1}{\lambda_k+1} e_k(x)e_k(y)}^n \6x\6y \\
% &= n! \sum_{k_1,\dots,k_n\in\cK_N} 
% \frac{1}{(\lambda_{k_1}+1)\dots(\lambda_{k_n}+1)} 
% \int_\Lambda\int_\Lambda e_{k_1}(x)\dots e_{k_n}(x) e_{k_1}(y)\dots 
% e_{k_n}(y) 
% \6x\6y \\
&= n! \sum_{k_1,\dots,k_n\in\cK_N} 
\frac{1}{(\lambda_{k_1}+1)\dots(\lambda_{k_n}+1)} 
\Biggpar{\int_\Lambda e_{k_1}(x)\dots e_{k_n}(x)\6x}^2\;.
\end{align}
By orthogonality of the eigenfunctions, the integral in the last expression 
vanishes unless some linear combination $k_1\pm k_2 \pm \dots \pm k_n$ is equal 
to zero. Therefore, we have the bound 
\begin{equation}
 \Biggexpec{\Biggpar{\int_\Lambda \Wick{\GFFN^n(x)}\6x}^2} 
 \lesssim 
 \sum_{\substack{k_1,\dots,k_n\in\cK_N \\ k_1+\dots+k_n=0}} 
\frac{1}{(\lambda_{k_1}+1)\dots(\lambda_{k_n}+1)}\;.
\end{equation} 
Thanks to the condition that the sum of the $k_i$ vanishes, this sum is bounded 
uniformly in $N$, by a Young-type inequality, see for 
instance~\cite[Lemma~3.10]{zhu2015three}.
\end{proof}

\begin{remark}
In the graphical notation introduced in~\eqref{eq:Wick_graphical}, we  
have $C_N = \FDtadpolenolabel{vedge}$, since $C_N$ involves the covariance of 
any $\GFFN(x)$ with itself. 
\end{remark}

%%%%%%%%%%%%%%%%%%%%%%%%%%%%%%%%%%%%%%%%%%%%%%%%%%%%%%%%%%%%%%%%%%%%%%%%%%%%%%%%

\subsection{Wick powers of the stochastic convolution}
\label{ssec:2d_heat} 

We consider now the stochastic heat equation 
\begin{equation}
\label{eq:heat_2d} 
 \partial_t \phi(t,x) = \Delta\phi(t,x) + %\sqrt{2\eps}
 \xi_N(t,x)\;,
\end{equation} 
where $\xi_N = \Pi_N\xi$ is a spectral Galerkin approximation of space-time 
white noise (cf.~\eqref{eq:Galerkin}). Similarly to~\eqref{eq:Duhamel}, the 
solution of~\eqref{eq:heat_2d} can be written as 
\begin{equation}
 \phi = P\phi_0 + %\sqrt{2\eps} 
 P*\xi_N\;,
\end{equation} 
where $P$ denotes the analogue of the heat kernel~\eqref{eq:heat_kernel_1d} in 
two dimensions, namely 
\begin{equation}
\label{eq:heat_kernel_2d} 
 P(t,x) := \sum_{k\in\Z^2} P_{\,\R^2}(t,x-kL)\;, 
 \qquad 
 P_{\,\R^2}(t,x) := \frac{1}{4\pi t} \e^{-\norm{x}^2/(4t)} 
\indexfct{t>0}\;. 
\end{equation} 
The stochastic convolution $P*\xi_N$ can be analysed as in \Cref{sec:1dheat}. In 
particular, subtracting the Brownian motion $\smash{W^{(0)}_t}$ of the zeroth 
Fourier mode, we find that $(P*\xi_N)(t,x) - \smash{W^{(0)}_t}$ converges as 
$t\to\infty$ to a Gaussian free field with covariance 
$(-\Delta_{\perp,N})^{-1}$, where $\Delta_{\perp,N}=\Pi_N\Delta_\perp$ acts on 
zero-mean functions. By \Cref{prop:Wick_powers}, defining $C_N$ as 
in~\eqref{eq:def_CN}, all Wick powers $\Wick{(P*\xi_N)^n}_{C_N}$ will have a 
variance uniformly bounded in $N$, and should have well-defined limits (one 
easily checks that adding a constant to the $\lambda_k$ in~\eqref{eq:def_CN} 
does not change its divergent part). 

There is a slightly different way to regularise the stochastic heat equation, 
which turns out to be equivalent to~\eqref{eq:heat_2d}. Let 
$\varrho:\R\times\Lambda\to\R$ be smooth, compactly supported and of 
integral~$1$. Then for any $\delta\in(0,1]$, define the \emph{mollifier}  
\begin{equation}
 \varrho_\delta(t,x) = \cS^{\delta}_0\varrho(t,x) 
 := \frac{1}{\delta^4} \varrho\biggpar{\frac{t}{\delta^2},\frac{x}{\delta}}\;.
\end{equation} 
We can then set $\xi^\delta(t,x) = (\varrho_\delta *\xi)(t,x)$ (interpreted as 
testing $\xi$ against $\varrho((t,x)-\cdot)$), and consider the regularised 
stochastic heat equation 
\begin{equation}
  \partial_t \phi(t,x) = \Delta\phi(t,x) + %\sqrt{2\eps}
  \xi^\delta(t,x)\;,
\end{equation} 
which admits solutions in the classical sense since $\xi^\delta$ is smooth, 
namely
\begin{equation}
 \phi = P\phi_0 + %\sqrt{2\eps} 
 P*\xi^\delta\;,
\end{equation} 
with a stochastic convolution $P*\xi^\delta$ given by 
\begin{equation}
\label{eq:stoch_conv_integral} 
 (P*\xi^\delta)(t,x) 
 = \int_{-\infty}^\infty \int_\Lambda P(t-s,x-y)\xi^\delta(s,y) \6y\6s 
 = \pscal{\xi}{P*\varrho_\delta}(t,x)\;.
\end{equation} 
Since this object and its powers will play a central role in the following, we 
introduce the graphical notation (borrowed 
from~\cite{Hairer2014,Chandra_Weber_LN17})
\begin{equation}
 (P*\xi^\delta)(t,x) = \RSI_\delta(t,x)\;, 
 \qquad 
 (P*\xi^\delta)(t,x)^2 = \RSV_\delta(t,x)\;, 
 \qquad 
 (P*\xi^\delta)(t,x)^3 = \RSW_\delta(t,x)\;. 
\end{equation} 
Writing $P_\delta = P*\varrho_\delta$ and using the defining 
property~\eqref{eq:xi_cov} of space-time white noise, we obtain  
\begin{equation}
\label{eq:def_Cdelta} 
 \bigexpec{\RSV_\delta(t,x)} 
 = \bigexpec{\pscal{\xi}{P_\delta}(t,x)^2}
 = \pscal{P_\delta}{P_\delta}_{L^2(\R\times\Lambda)} 
 = \int_{-\infty}^\infty \int_{\Lambda} P_\delta(t,x)^2\6x\6t
 =: C_\delta\;,
\end{equation} 
which is independent of $(t,x)$.

\begin{proposition}
\label{prop:asymp_Cdelta} 
Assume that $\varrho(t,x)=\varrho^{(0)}(t)\varrho^{(1)}(x)$ where 
$\varrho^{(1)}$ is even. Then the constant $C_\delta$ has the asymptotic 
behaviour
\begin{equation}
\label{eq:asymp_Cdelta} 
 C_\delta = \frac{\log(\delta^{-1})}{4\pi} + \Order{1}
 \qquad 
 \text{as $\delta\to0$\;.}
\end{equation} 
\end{proposition}
\begin{proof}
We present a proof which is perhaps not the shortest possible, but will 
highlight the link with other object such as the Green function. First note 
that the Markov property implies that for any $t,s>0$ and $x\in\Lambda$, 
\begin{equation}
 \int_\Lambda P(t,x-y) P(s,y) \6y = P(t+s,x)\;.
\end{equation} 
Furthermore, for any $x\in\Lambda$, 
\begin{equation}
 \int_{-\infty}^\infty P(t,x)\6t 
 = \int_0^\infty \e^{t\Delta}(x)\6t 
 = - G(x)\;,
\end{equation}
where $G=\Delta^{-1}$ is the Green function of the Laplacian acting on 
zero-mean functions, that is, the solution of $(\Delta G)(x)=\delta(x)$. Taking 
into account the properties of $\varrho$, one finds that for $\delta>0$, these 
relations become 
\begin{align}
\int_\Lambda P_\delta(t,x-y) P_\delta(s,y) \6y &= \tilde P_\delta(t+s,x)\;, \\
 \int_{-\infty}^\infty P_\delta(t,x)\6t 
 = - G_\delta(x)\;,
 \label{eq:Pdelta_convolution} 
\end{align}
where $G_\delta=G*\varrho_\delta^{(1)}$ and $\tilde 
P_\delta=P*(\varrho*\varrho)_\delta$ is a regularised version of $P$ with a 
different mollifier. It follows, using that $P$ is even in $x$, that 
\begin{equation}
 \int_{-\infty}^\infty \int_{\Lambda} P_\delta(t,x)^2 \6x\6t
 = \int_{-\infty}^\infty \tilde P_\delta(2t,0) \6t 
 = -\frac12 \tilde G_\delta(0)\;,
\end{equation} 
where again $\tilde G_\delta$ is a regularisation of $G$ with a different 
mollifier. Using the fact that the Green function of the two-dimensional 
Laplacian behaves like $\log{\norm{x}}/(2\pi)$ near the origin, one easily 
obtains $\tilde G_\delta(0) = \log(\delta)/(2\pi) + \Order{1}$, which 
proves~\eqref{eq:asymp_Cdelta}. 
\end{proof}

\begin{remark}
Note that the expressions~\eqref{eq:asymp_Cdelta} and~\eqref{eq:asymp_CN} have 
the same divergent behaviour. Furthermore, the divergent part of $C_\delta$ is 
independent of the mollifier $\varrho$. The term of order $1$, however, will 
depend on $\varrho$. In other words, there is no canonical choice for the 
bounded part of the renormalisation constants.  
\end{remark}

\Cref{prop:Wick_powers} shows that the Wick powers 
\begin{align}
\Wick{\RSI_\delta} &= H_1(\RSI_\delta;C_\delta) 
= \RSI_\delta\;, \\
\label{eq:Wick_powers_limit} 
\Wick{\RSV_\delta} &= H_2(\RSI_\delta;C_\delta) 
= \RSV_\delta - C_\delta\;, \\
\Wick{\RSW_\delta} &= H_3(\RSI_\delta;C_\delta) 
= \RSW_\delta - 3C_\delta\RSI_\delta
\end{align}
have a variance uniformly bounded in $\delta$. In fact, one can show that they 
admit well-defined limits $\RSI$, $\RSV$ and $\RSW$ as $\delta\searrow0$ in 
terms of iterated stochastic integrals, which belong to $\cC^{-\kappa}_\fraks$ 
for any $\kappa>0$. See~\cite[Appendix~A]{Chandra_Weber_LN17}, 
\cite[Section~10.1]{Hairer2014} and~\cite[Lemma~3.2]{daPratoDebussche} for 
details. The same works of course for higher powers of the stochastic 
convolution. 

\begin{remark}
A formal way of writing these limits is to introduce 
\begin{align}
\xi(\6z_1)\diamond\xi(\6z_2)
={}& \xi(\6z_1)\xi(\6z_2) - \delta(z_1-z_2)\;, \\
\xi(\6z_1)\diamond\xi(\6z_2)\diamond\xi(\6z_3)
={}& \xi(\6z_1)\xi(\6z_2)\xi(\6z_3) \\ 
{}&- \xi(\6z_1)\delta(z_2-z_3)
- \xi(\6z_2)\delta(z_3-z_1)
- \xi(\6z_3)\delta(z_1-z_2)\;. 
\end{align}
The quantities on the left-hand side, which define the limiting objects $\RSV$ 
and $\RSW$ when integrated against a product of $P_\delta$, can be given a 
rigorous meaning as elements of the Wiener chaos decomposition, 
see~\cite[Section~1.1]{nualart2006malliavin},  
and~\cite[Appendix~A]{Chandra_Weber_LN17}.  
\end{remark}

%%%%%%%%%%%%%%%%%%%%%%%%%%%%%%%%%%%%%%%%%%%%%%%%%%%%%%%%%%%%%%%%%%%%%%%%%%%%%%%%

\section{Existence and uniqueness of solutions}
\label{sec:2d_existence} 

The discussion so far suggests that the renormalised version of the 
two-dimensional stochastic Allen--Cahn equation we should consider is 
\begin{align}
 \partial_t \phi_\delta(t,x) 
 &= \Delta\phi_\delta(t,x) + \phi_\delta(t,x) - 
\Wick{\phi_\delta(t,x)^3}_{C_\delta} + \xi^\delta(t,x) \\
 &= \Delta\phi_\delta(t,x) + \phi_\delta(t,x) - \Bigpar{\phi_\delta(t,x)^3 - 
3C_\delta\phi_\delta(t,x)} + \xi^\delta(t,x)\;.
\label{eq:AC-2d-reg} 
\end{align}
For $\delta>0$, this is a smooth PDE (for any realisation of the noise), whose 
solution with initial condition $\phi_{\delta,0}$ satisfies the fixed-point 
equation 
\begin{equation}
\label{eq:FP_2d} 
 \phi_\delta = P\phi_{\delta,0} + \RSI_\delta + 
P*\bigbrak{\phi_\delta-\phi_\delta^3+3C_\delta\phi_\delta}\;.
\end{equation} 
The idea exploited by Da Prato and Debussche in \cite{daPratoDebussche} is 
that since the most irregular term on the right-hand side should be the 
stochastic convolution $\RSI_\delta$, one can write 
\begin{equation}
 \phi_\delta = \RSI_\delta + \psi_\delta\;,
\end{equation} 
where $\psi_\delta$ is expected to be more regular than $\phi_\delta$. 
Substituting in~\eqref{eq:FP_2d}, we obtain 
\begin{align}
\psi_\delta 
&= P\psi_{\delta,0} + P*\bigbrak{\RSI_\delta + \psi_\delta - (\RSI_\delta + 
\psi_\delta)^3 
+ 3C_\delta(\RSI_\delta + \psi_\delta)} \\
&= P\psi_{\delta,0} + P*\bigbrak{\RSI_\delta + \psi_\delta - \psi_\delta^3 - 
3\RSI_\delta\psi_\delta^2 - 3(\RSV_\delta-C_\delta)\psi_\delta - 
(\RSW_\delta-3C_\delta\RSI)}\;.
\end{align}
Taking formally the limit $\delta\searrow0$ and 
using~\eqref{eq:Wick_powers_limit} yields the fixed-point equation 
\begin{equation}
\label{eq:FP_2d_psi} 
 \psi = P\psi_0 + P*\bigbrak{\RSI + \psi - \psi^3 
- 3\RSI\psi^2 - 3\RSV\psi - \RSW}\;.
\end{equation} 
The key result allowing to analyse this fixed-point equation is the following, 
see~\cite[Proposition~4.14]{Hairer2014} as well 
as~\cite[Section~2.6]{Bahouri_Chemin_Danchin_book}.

\begin{theorem}[Product in Besov spaces]
\label{thm:product_Holder} 
Let $\alpha,\beta\in\R$ satisfy $\alpha+\beta>0$. Then there exists a bilinear 
map $B:\cC^\alpha_\fraks\times\cC^\beta_\fraks \to 
\cC^{\alpha\wedge\beta}_\fraks$ with the following properties:
\begin{enumerate}
\item 	If $f\in\cC^\alpha_\fraks$ and $g\in\cC^\beta_\fraks$ are continuous 
functions, then $B(f,g)(z) = f(z)g(z)$.
\item 	For arbitrary $f\in\cC^\alpha_\fraks$ and $g\in\cC^\beta_\fraks$, one 
has the bound 
\begin{equation}
 \norm{B(f,g)}_{\cC^{\alpha\wedge\beta}_\fraks} 
 \lesssim \norm{f}_{\cC^\alpha_\fraks} \norm{g}_{\cC^\beta_\fraks}\;.
\end{equation} 
\end{enumerate}
If $\alpha+\beta\leqs0$, then no bilinear map satisfying these two properties 
exists. 
\end{theorem}

We then have the following existence and uniqueness result. 

\begin{theorem}[Global existence and uniqueness for the two-dimensional 
Allen--Cahn SPDE]
\label{thm:AC-2d-global} 
Assume $\psi_0\in\cC^{-\kappa}_\fraks(\Lambda)$.  
For sufficiently small $\kappa>0$, equation \eqref{eq:FP_2d_psi} with initial 
condition $\psi_0$ admits an almost surely unique global solution in 
$\cC^{\alpha}_\fraks(\R_+\times\Lambda)$ for any $\alpha < 2-\kappa$. 
Therefore, the solution of \eqref{eq:AC-2d-reg} with initial condition 
$\varrho_\delta *\psi_0$ converges almost surely, as $\delta\searrow0$, to a 
stochastic process $(\psi_t)_{t\geqs0}$ such that $\psi-\RSI \in 
\cC^{\alpha}_\fraks(\R_+\times\Lambda)$. 
\end{theorem}
\begin{proof}[\Sketch]
We first prove a local existence result, in the spirit 
of~\Cref{thm:semilinear_local_existence}. We know that $\RSI, \RSV, \RSW$ 
belong 
to $\cC^{-\kappa}_\fraks$ for any $\kappa>0$, while an analogue of the Schauder 
estimate in \Cref{thm:Schauder} shows that 
$P\psi_0\in\cC^{2-\kappa}_\fraks([0,T]\times\Lambda)$.  Using 
\Cref{thm:product_Holder}, we obtain for any $\alpha,T>0$ the estimate 
\begin{equation}
 \norm{\RSI + \psi - \psi^3 
- 3\RSI\psi^2 - 3\RSV\psi - \RSW}_{\cC^{-\kappa}_\fraks([0,T]\times\Lambda)}
 \lesssim \bigpar{\norm{\psi}_{\cC^{\alpha}_\fraks([0,T]\times\Lambda)} + M}^3
\end{equation} 
for some constant $M$. The Schauder estimate shows that the convolution of this 
quantity with $P$ belongs to $\cC^{2-\kappa}_\fraks([0,T]\times\Lambda)$.
This is not sufficient to obtain a contraction, since the norm does not become 
small for small $T$. Such a $T$-dependence can, however, be obtained by 
sacrificing a bit of regularity. One can indeed show that, for any $\sigma>0$,
\begin{equation}
 \norm{P*[\RSI + \psi - \psi^3 
- 3\RSI\psi^2 - 3\RSV\psi - 
\RSW]}_{\cC^{2-\kappa-\sigma}_\fraks([0,T]\times\Lambda)}
 \lesssim 
 T^\sigma \bigpar{\norm{\psi}_{\cC^{\alpha}_\fraks([0,T]\times\Lambda)} + 
M}^3\;.
\end{equation} 
Choosing $\alpha=2-\kappa-\sigma$, we obtain that the right-hand side 
of~\eqref{eq:FP_2d_psi} maps a ball into itself for sufficiently small $T$. A 
similar argument shows that this map is a contraction in that ball for small 
$T$, so that the existence of a local solution follows from Banach's 
fixed-point theorem. Since $\sigma$ was arbitrary, the fixed point has the 
claimed regularity.  

Global existence then follows from showing non-explosion, by using an 
appropriate Lyapunov function and a Gronwall-type argument. 
See~\cite[p.~1914]{daPratoDebussche} 
and~\cite[Section~3.4]{Tsatsoulis_Weber_16}, as well 
as~\cite[Section~9]{Mourrat_Weber_17} for an argument working when $\Lambda$ 
is replaced by $\R^2$. 
\end{proof}

\begin{remark}
This result appears with different variants of function spaces in the 
literature. The proof given in~\cite{daPratoDebussche} uses spaces mixing 
Lebesgue spaces $L^p$ in the time variable and Besov spaces 
$\cB^{\alpha}_{q,r}$ 
in the space variable. In~\cite[Section~3]{Tsatsoulis_Weber_16}, one finds a 
formulation with H\"older regularity in space and continuity in time. The 
formulation here follows~\cite{Hairer2014}. The result remains true for any 
polynomial nonlinearity of odd degree and negative leading-order coefficient. 
\end{remark}

%%%%%%%%%%%%%%%%%%%%%%%%%%%%%%%%%%%%%%%%%%%%%%%%%%%%%%%%%%%%%%%%%%%%%%%%%%%%%%%%

\section{Invariant measure}
\label{sec:2d_invariant} 

We now consider the stochastic Allen--Cahn equation written in the form
\begin{equation}
 \partial_t \phi(t,x) = \Delta\phi(t,x) + \phi(t,x) - 
\Wick{\phi(t,x)^3} + \sqrt{2\eps}\xi(t,x)\;, 
\label{eq:AC-2d-eps} 
\end{equation}
where we have reintroduced the small parameter in front of the noise for later 
use. Here the Wick power is with respect to the constant $\eps 
\overline{C}_\delta = 2\eps C_\delta$ since the variance of the noise has been 
multiplied by $2\eps$. 

By analogy with~\eqref{eq:Gibbs_1dSPDE}, a natural candidate for the invariant 
measure of~\eqref{eq:AC-2d-eps} is 
\begin{equation}
\label{eq:potential-2d} 
 \pi(\6\phi) := \frac{1}{\cZ_0} \e^{-\widetilde V(\phi)/\eps}
\muGFF^{(\eps)}(\6\phi)\;, 
\qquad
\widetilde V(\phi) := \int_\Lambda 
\biggbrak{\frac14 \Wick{\phi^4(x)} - \Wick{\phi^2(x)} + \frac14}\6x\;,
\end{equation} 
where $\muGFF^{(\eps)}$ is the measure of the Gaussian free field with 
covariance $\eps(-\Delta+1)^{-1}$ (and the constant term $1/4$ plays no role 
but will be convenient in computations). Note that the integral over $\Lambda$ 
has finite expectation and variance by~\Cref{prop:Wick_powers}. 
However, this does not automatically imply that the partition function $\cZ_0$ 
is finite. 

In the case of the $\Phi^4$ model (without the linear term $\phi(t,x)$ 
in~\eqref{eq:AC-2d-eps}, cf. \Cref{rem:Phi4}), this problem is in fact an old 
problem in Quantum Field Theory, which has been solved by various methods. A 
particularly elegant approach has recently been developed by Barashkov and 
Gubinelli~\cite{Barashkov_Gubinelli_18}. To adapt it to the nonconvex 
potential of the Allen--Cahn equation, we start by performing, in a similar way 
as in~\Cref{sec:1dmeta}, cf.~\eqref{eq:1d_transversal}, the change of variables 
\begin{equation}
\label{eq:potential_decompose} 
 \phi(x) = \frac{1}{L}\phi_0 + \sqrt{\eps}\phi_\perp(x)\;,
\end{equation} 
where $\phi_\perp$ has zero mean. Using \Cref{lem:hermite_binomial} to 
transform Wick powers, we obtain 
\begin{equation}
 \frac1\eps \widetilde V(\phi)
= \frac1\eps \biggbrak{V_0(\phi_0) - \frac12\phi_0^2} + 
\frac12\cQ_\perp(\phi_0,\phi_\perp) + R_\eps(\phi_0,\phi_\perp)\;,
\end{equation} 
where
\begin{align}
V_0(\phi_0) &= \frac{L^2}{4} \biggpar{\frac{\phi_0^2}{L^2}-1}^2\;, \\
\cQ_\perp(\phi_0,\phi_\perp) &= \biggpar{\frac{3\phi_0^2}{L^2}-2} 
\int_\Lambda\Wick{\phi_\perp(x)^2}\6x\;, \\
R_\eps(\phi_0,\phi_\perp) 
&= \frac{\sqrt{\eps}}{L}\phi_0 \int_{\Lambda} \Wick{\phi_\perp(x)^3} \6x
  + \frac{\eps}{4}\int_{\Lambda} \Wick{\phi_\perp(x)^4} \6x\;.
\label{eq:V0Reps} 
\end{align}
This yields the expression 
\begin{equation}
\label{eq:Z0_perp} 
\cZ_0 = \bigexpecin{\muGFF^{(\eps)}}{\e^{-\widetilde V/\eps}} 
= \frac{1}{\sqrt{2\pi\eps}} \int_{-\infty}^\infty \e^{-V_0(\phi_0)/\eps}
\Bigexpecin{\muGFFperp}{\e^{-\cQ_\perp(\phi_0,\cdot)/2 - 
R_\eps(\phi_0,\cdot)}} \6\phi_0
\end{equation}
for the partition function, where $\muGFFperp$ has covariance 
$(-\Delta_\perp+1)^{-1}$. 

Our aim is now to bound the expectation in~\eqref{eq:Z0_perp}. One way of doing 
this, which was used in~\cite{Berglund_DiGesu_Weber_16}, is based on the 
so-called Nelson estimate, a refinement of \Cref{prop:Wick_powers}. We outline 
here instead the approach used in~\cite{Barashkov_Gubinelli_18}, which relies on 
the Bou\'e--Dupuis formula. The idea is that if $\mu_\perp(a)$ is a Gaussian 
measure with covariance $(-\Delta_\perp+a)^{-1}$, it can be regularised in the 
following way. Let $\rho:\R_+\to\R_+$ be a compactly supported smooth function 
such that $\rho(0)=0$, and set, for each $k\in\Z^2$, 
\begin{equation}
 v_k(t) := \frac{\sigma_k(t)}{\sqrt{\lambda_k^2+a}}\;, 
 \qquad
 \sigma_k(t)^2 := \dtot{}{t} \rho\biggpar{\frac{\norm{k}}{t}}\;.
\end{equation} 
Then introduce the stochastic process 
\begin{equation}
 Y_t(x) := \sum_{k\in\Z^2\setminus\set{0}} \int_0^t v_k(s)\6W^{(k)}_s e_k(x)\;,
\end{equation} 
where the $W^{(k)}_t$ are independent Wiener processes. Using It\^o's isometry, 
we obtain 
\begin{equation}
 \bigexpec{\pscal{Y_t}{\ph}\pscal{Y_s}{\psi}}
 = \sum_{k\in\Z^2\setminus\set{0}} \int_0^{t\wedge s} v_k(u)^2\6u 
\,\overline{\hat\ph_k}\hat\psi_k
= \sum_{k\in\Z^2\setminus\set{0}} \rho\biggpar{\frac{\norm{k}}{t\wedge s}}
\,\overline{\hat\ph_k}\hat\psi_k
\end{equation} 
for any test functions $\ph$ and $\psi$ with Fourier coefficients $\hat\ph_k$ 
and $\hat\psi_k$. In particular, 
\begin{equation}
 \lim_{t\to\infty} \bigexpec{\pscal{Y_t}{\ph}\pscal{Y_t}{\psi}}
 = \sum_{k\in\Z^2\setminus\set{0}} \frac{1}{\lambda_k^2+a} 
\,\overline{\hat\ph_k}\hat\psi_k
 = \pscal{\ph}{(-\Delta_\perp+a)^{-1}\psi}\;.
\end{equation} 
Then the Bou\'e--Dupuis variational formula, whose proof relies on the Girsanov 
transformation~\cite[Section~2]{Barashkov_Gubinelli_18}, reads as follows. 

\begin{theorem}[Bou\'e--Dupuis formula, \cite{Boue_Dupuis_98}]
For any $T>0$ and functional $\cV$, we have 
\begin{equation}
\label{eq:Boue-Dupuis} 
 - \log \bigexpec{\e^{-\cV(Y_T)}}
 = \inf_{u\in\mathbb{H}} \Biggexpec{\cV(Y_T+Z_T(u)) + 
\frac{\abs{\Lambda}}2 \int_0^T \norm{u_t}_{L^2}^2\6t}\;,
\end{equation} 
where 
\begin{equation}
 Z_T(u,x) := \sum_{k\in\Z^2\setminus\set{0}} \int_0^T u_k(t)\6t\, e_k(x)
\end{equation} 
and $\mathbb{H}$ is the space of progressively measurable processes which are 
almost surely in $L^2(\R_+\times\Lambda)$. 
\end{theorem}

Note that here the parameter $T$ does not represent time, but a continuous 
regularisation parameter such that in the limit $T\to\infty$, one recovers the 
expectation under the Gaussian measure $\muGFFperp$. This leads to the 
following estimate, which extends \Cref{prop:exp_moment} to the two-dimensional 
case (the bounds' $\eps$-dependence can be improved, cf.\ 
\Cref{prop:Z-cap_AC-2d}).

\begin{proposition}
\label{prop:exp_moment_2d} 
Assume that $L<2\pi$. Then there exists a constant $c$ such that  
\begin{equation}
 1 \leqs 
 \Bigexpecin{\muGFFperp}{\e^{-\cQ_\perp(\phi_0,\cdot)/2 - 
R_\eps(\phi_0,\cdot)}}
 \leqs \e^{c\phi_0^4}\;.
\end{equation} 
As a consequence, the partition function $\cZ_0$ is bounded. 
\end{proposition}
\begin{proof}[\Sketch]
The lower  bound follows from Jensen's inequality and the fact that Wick 
powers are centred. For the upper bound, we first perform a shift to the 
Gaussian measure $\mu_\perp$ of covariance $(-\Delta_\perp-1)^{-1}$ (which is 
allowed since the eigenvalues of $-\Delta_\perp$ are greater than $1$ for 
$L<2\pi$). One obtains 
\begin{equation}
 \Bigexpecin{\muGFFperp}{\e^{-\cQ_\perp(\phi_0,\cdot)/2 - R_\eps(\phi_0,\cdot)}}
 = K\Bigexpecin{\mu_\perp}{\e^{-\cV(\phi_0,\cdot)}}\;,
\end{equation} 
where 
\begin{equation}
 \cV(\phi_0,\phi_\perp) := R_\eps(\phi_0,\phi_\perp) + \frac{3\phi_0^2}{2L^2}
 \int_{\Lambda} \Wick{\phi_\perp(x)^2} \6x\;,
\end{equation} 
and $K$ is a constant. This is actually a nontrivial point, to which we will 
come back in \Cref{sec:2d_metastability}, see \Cref{lem:Gaussian_shift}. 
Expanding the Wick powers of $Y_T+Z_T(u)$ with the help of 
\Cref{prop:Wick_powers}, and inserting this in~\eqref{eq:Boue-Dupuis}, we get 
\begin{equation}
 - \log \bigexpec{\e^{-\cV(\phi_0,\cdot)}}
 = \inf_{u\in\mathbb{H}} \Bigexpec{\Psi_T(u)+\Phi_T(u)}\;,
\end{equation} 
where, since all Wick powers of $Y_T$ have zero expectation, 
\begin{align}
\Psi_T(u) ={}& \frac{\eps}{4} \int_\Lambda Z_T^4(u) \6x 
+ \frac{3\phi_0^2}{2L^2} \int_\Lambda Z_T^2(u) \6x
+ \frac{\abs{\Lambda}}{2} \int_0^T \norm{u_t}_{L^2}^2\6t\;,\\
\Phi_T(u) ={}& \eps \int_\Lambda \Wick{Y_T^3}Z_T(u) \6x
+ \frac32\eps \int_\Lambda \Wick{Y_T^2}Z_T^2(u) \6x
+ \eps \int_\Lambda Y_TZ_T^3(u) \6x \\
&{}+ \frac{\sqrt{\eps}}{L}\phi_0 \Biggbrak{ 
3\int_\Lambda \Wick{Y_T^2}Z_T(u) \6x
+ 3\int_\Lambda Y_TZ_T^2(u) \6x
+ \int_\Lambda Z_T^3(u) \6x}
+ \frac{3\phi_0^2}{2L^2} \int_\Lambda Y_T Z_T(u)\6x\;.
\end{align}
Here one has to show that for each choice of $u\in\mathbb{H}$, the positive 
term 
$\Psi_T(u)$ dominates $\Phi_T(u)$ uniformly in $T$ up to a remainder of order 
$\phi_0^4$. This is a purely deterministic argument, which is based on 
functional inequalities, cf.~\cite[Section~3]{Barashkov_Gubinelli_18}. The 
claim on $\cZ_0$ then follows from Laplace asymptotics on~\eqref{eq:Z0_perp}. 
\end{proof}

The existence of the Gibbs measure $\pi$ being established, it is natural to ask 
whether $\pi$ is indeed invariant under the dynamics of the Allen--Cahn 
equation~\eqref{eq:AC-2d-eps}, whether it is the unique invariant measure, and 
what are its ergodic properties. This task is not straightforward at all, since 
even the fact that solutions of~\eqref{eq:AC-2d-eps} enjoy the Markov property 
is not automatic, but requires a proof. All these properties have in fact 
been established recently. We summarise them in the following theorem, which 
incorporates results 
from~\cite{Rockner_Zhu_Zhu_15,Tsatsoulis_Weber_16,Rockner_Zhu_Zhu_17, 
Hairer_Mattingly_18}.

\begin{theorem}[Ergodic properties of the two-dimensional Allen--Cahn equation]
The solution of~\eqref{eq:AC-2d-eps} defines a Markov process 
$(\phi_t)_{t\geqs0}$ with transition semi-group $(P_t)_{t\geqs0}$ with respect 
to the filtration $(\cF_t)_{t\geqs0}$ generated by all random variables 
$\pscal{\xi}{\ph}$ obtained by testing space-time white noise against functions 
$\ph$ supported on $[0,t]\times\Lambda$. The semigroup $(P_t)_{t\geqs0}$ 
satisfies the \emph{strong Feller property}, that is, it maps bounded 
measurable functions to continuous functions for every $t>0$. 
The Gibbs measure $\pi(\6\phi)$ is the unique invariant measure of the 
dynamics, and is reversible. Finally, the dynamics is exponentially mixing, in 
the sense that there exist constants $C,\beta>0$ such that 
\begin{equation}
\label{eq:exp_mixing_2d} 
 \norm{P_t(\phi_0)-\pi}_{\math{TV}} 
 \leqs C\e^{-\beta t} \norm{\delta_{\phi_0}-\pi}_{\math{TV}}
\end{equation} 
holds for all $\phi_0\in\cC^{-\kappa}_\fraks(\Lambda)$ and all $t>0$, where 
$\norm{\cdot}_{\math{TV}}$ denotes the total variation distance. 
\end{theorem}

%%%%%%%%%%%%%%%%%%%%%%%%%%%%%%%%%%%%%%%%%%%%%%%%%%%%%%%%%%%%%%%%%%%%%%%%%%%%%%%%

\section{Large deviations}
\label{sec:2d_ldp} 

Consider the Allen--Cahn equation on the torus $\Lambda$ with weak mollified 
space-time white noise
\begin{equation}
 \partial_t \phi(t,x) 
 = \Delta\phi(t,x) + \phi(t,x) - \Bigpar{\phi(t,x)^3 - 
3\eps\overline{C}_\delta\phi(t,x)} + \sqrt{2\eps}\xi^\delta(t,x)\;,
\label{eq:AC-2d-regeps} 
\end{equation}
where $\overline{C}_\delta=2C_\delta$. Note that we have now two small parameter 
$\delta$ and $\eps$, and that the counterterm $3\eps\overline{C}_\delta\phi$ has 
been scaled accordingly (recall from~\eqref{eq:asymp_Cdelta} that 
$\overline{C}_\delta$ diverges logarithmically). For fixed $\eps>0$, we know 
that this equation admits a well-defined limit as $\delta\searrow0$. When 
$\delta,\eps>0$, \eqref{eq:AC-2d-regeps} is a well-posed PDE for any realisation 
of the noise.  
In~\cite{HairerWeber}, Hairer and Weber have established the 
following large-deviation principle.

\begin{theorem}[Large-deviation principle for the renormalised equation]
\label{thm:LDP_2d} 
Let $\phi_{\delta,\eps}$ denote the solution of~\eqref{eq:AC-2d-regeps}, and 
let $\eps\mapsto\delta(\eps)\geqs0$ be a function such that
\begin{equation}
 \label{eq:lim_delta}
 \lim_{\eps\to0} \delta(\eps) = 0\;.
\end{equation} 
Then the family $(\phi_{\delta(\eps),\eps})_{\eps>0}$ with fixed initial 
condition $\phi_0$ satisfies an LDP on $[0,T]$ with good rate function 
\begin{equation}
\label{eq:lpd_AC-2d} 
\cI_{[0,T]}(\gamma) := 
 \begin{cases}
 \displaystyle
 \frac12 \int_0^T \int_\Lambda \biggbrak{\dpar{}{t} \gamma(t,x) - 
\Delta\gamma(t,x) - \gamma(t,x) + \gamma(t,x)^3}^2 \6x\6t
 & \text{if the integral is finite\;, } \\
 + \infty 
 & \text{otherwise\;.}
 \end{cases}
\end{equation}
\end{theorem}

The remarkable thing about this result is that the counterterm 
$3\eps\overline{C}_\delta\phi$ does not appear in the rate function, which 
suggests that the Wick power consisting of cubic term and counterterm is the 
object that should really be considered physically relevant. 

On the other hand, for $\delta>0$, we may also consider the variant 
of~\eqref{eq:AC-2d-regeps} without 
counterterm 
\begin{equation}
 \partial_t \phi(t,x) 
 = \Delta\phi(t,x) + C\phi(t,x) - \phi(t,x)^3 
 + \sqrt{2\eps}\xi^\delta(t,x)\;,
\label{eq:AC-2d-noC} 
\end{equation}
which does not admit a limit as $\delta\searrow0$. 
Here, one can also obtain a large-deviation result, which reads as follows. 

\begin{theorem}[Large-deviation principle for equation without renormalisation]
Let $\tilde\phi_{\delta,\eps}$ denote the solution of~\eqref{eq:AC-2d-noC}, and 
let $\eps\mapsto\delta(\eps)\geqs0$ be a function 
satisfying~\eqref{eq:lim_delta} as well as 
\begin{equation}
 \lim_{\eps\to0} \eps\log\Bigpar{\delta(\eps)^{-1}} = \lambda\in[0,\infty)\;.
\end{equation} 
Then the family $(\tilde\phi_{\delta(\eps),\eps})_{\eps>0}$ with fixed initial 
condition $\phi_0$ satisfies an LDP on $[0,T]$ with good rate function 
\begin{equation}
\cI_{[0,T]}(\gamma) := 
 \begin{cases}
 \displaystyle
 \frac12 \int_0^T \int_\Lambda \biggbrak{\dpar{}{t} \gamma(t,x) - 
\Delta\gamma(t,x) + C^{(\lambda)} \gamma(t,x) + \gamma(t,x)^3}^2 \6x\6t
 & \text{if the integral is finite\;, } \\
 + \infty 
 & \text{otherwise\;,}
 \end{cases}
\end{equation}
where $C^{(\lambda)} = C - 3\lambda^2/(4\pi)$.
% \begin{equation}
%  C^{(\lambda)} = C - \frac{3}{4\pi} \lambda^2\;.
% \end{equation} 
\end{theorem}

Note that by the asymptotics~\eqref{eq:asymp_Cdelta} of $C_\delta$, the results 
of both theorems indeed overlap if $C$ is such that $C_\lambda=0$. 

%%%%%%%%%%%%%%%%%%%%%%%%%%%%%%%%%%%%%%%%%%%%%%%%%%%%%%%%%%%%%%%%%%%%%%%%%%%%%%%%

\section{Metastability}
\label{sec:2d_metastability} 

We consider now the long-time dynamics of the renormalised 
equation~\eqref{eq:AC-2d-regeps} for small positive $\eps$, again for simplicity 
in the case $L<2\pi$, when the only stationary solutions of the deterministic 
equation have constant value $0$ or $\pm1$. As before, we denote these 
stationary solutions $\phi^*_0$ and $\phi^*_\pm$. Note that at first glance, one 
might think that the stable stationary solutions of~\eqref{eq:AC-2d-regeps} are 
located in $\pm[1+3\eps \smash{\overline{C}_\delta}]^{1/2}$ instead of $\pm1$. 
However, the large-deviation estimate~\eqref{eq:lpd_AC-2d} suggests otherwise, 
and we shall see that indeed the system behaves in some sense as if the 
counterterm were absent. 

As in the one-dimensional case, we have no good control on the 
$\eps$-dependence of the constant $\beta$ controlling the speed of convergence 
in the exponential mixing result~\eqref{eq:exp_mixing_2d}. A natural approach 
to characterise this speed is thus to establish asymptotic results on the 
expected transition time $\tau_+$ from $\phi^*_-$ to $\phi^*_+$. The LDP given 
in \Cref{thm:LDP_2d} can be used to show that, just as in the one-dimensional 
case, this time satisfies the Arrhenius law 
\begin{equation}
 \lim_{\eps\to0} \eps\log \bigexpecin{\phi^*_-}{\tau_+} = H\;,
\end{equation} 
where for $L<2\pi$, one has $H=L^2/4$. Note that the fact that the rate 
function~\eqref{eq:lpd_AC-2d} does not depend on $C_\delta$ is crucial here. 

The next step is to derive Eyring--Kramers asymptotics for 
$\bigexpecin{\phi^*_-}{\tau_+}$. One may indeed wonder whether 
there exists an analogue to \Cref{thm:EK-1d} on the one-dimensional 
Allen--Cahn equation. Here, however, we note that the Fredholm determinant 
\begin{equation}
  \det\Bigpar{(-\Delta-1)(-\Delta+2)^{-1}}
 = \det\Bigpar{\one - 3(-\Delta+2)^{-1}}\;.
\end{equation} 
that would appear in the prefactor is \emph{not} well-defined, since 
$(-\Delta+2)^{-1}$ is not trace class in two dimensions. But one should not 
forget that the Eyring--Kramers formula also involves a potential difference. 
If the potential $V$ is defined  by 
\begin{equation}
\label{eq:def_pot_2d} 
 V(\phi) = \int_\Lambda \biggbrak{\frac12\norm{\nabla\phi(x)}^2 + 
\frac14\Wick{\phi(x)^4} - \frac12\Wick{\phi(x)^2} + \frac14} \6x\;,
\end{equation} 
one can check that 
\begin{equation}
 V(\phi^*_0) - V(\phi^*_-) = \frac{L^2}{4} + 
\frac{3}{2}L^2\eps \overline{C}_\delta\;.
\end{equation} 
It turns out that the extra term in $\overline{C}_\delta$ exactly compensates 
the divergence of the Fredholm determinant, replacing it by a so-called 
\emph{Carleman--Fredholm} (renormalised) determinant. 

To see this, we will work with the spectral Galerkin approximation of the 
SPDE given by 
\begin{equation}
\label{eq:AC-2d-Galerkin} 
 \partial_t \phi_N(t,x) = \Delta\phi_N(t,x) + \phi_N(t,x) - \Pi_N\phi_N(t,x)^3 
 + 3\eps C_N \phi_N(t,x) + \sqrt{2\eps}\Pi_N\xi(t,x)\;,
\end{equation} 
where $\Pi_N$ is the projection on Fourier modes with wave vector $k$ satisfying 
$\abs{k} = \abs{k_1} + \abs{k_2} \leqs N$. As before, we denote by $\cK_N$ the 
set of these $k$. In order to keep track of the choice of the order-one part of 
the renormalisation constant, which is in principle arbitrary, it will be 
convenient to choose $C_N$ such that 
\begin{equation}
\label{eq:CN_Galerkin} 
 L^2 C_N = \Tr\Bigpar{\abs{-\Delta_{N}-1}^{-1}} + \theta
 = \sum_{k\in\cK_N} \frac{1}{\abs{\lambda_k-1}} + \theta
 = L^2\frac{\log N}{2\pi} + \Order{1}\;,
\end{equation} 
where $\theta\in\R$ is a free parameter. The renormalised potential is then 
given by~\eqref{eq:def_pot_2d}  where Wick powers are with respect to $\eps 
C_N$. Using the decomposition~\eqref{eq:potential_decompose} of $\phi$ into its 
mean and oscillatory parts, we obtain
\begin{equation}
 \frac{1}{\eps} V(\phi) = \frac{1}{\eps} V_0(\phi_0) 
 + \frac12 \pscal{\phi_\perp}{Q_\perp(\phi_0)\phi_\perp} 
 + \frac12 \Bigpar{1-m^2(\phi_0)}L^2C_N 
 + R_\eps(\phi_0,\phi_\perp)\;,
\end{equation} 
where $V_0$ and $R_\eps$ are as in~\eqref{eq:V0Reps}, $Q_\perp(\phi_0) = 
-\Delta_\perp -1 + m^2(\phi_0)$ and $m^2(\phi_0)=3\phi_0^2/L^2$. Note that 
compared to the one-dimensional case, 
cf.~\eqref{eq:potential_1d_decomposition}, there is an additional term 
proportional to $C_N$, which is due to our using Wick powers in the potential.

The key observation here is the following explicit expression for a change of 
mass in a Gaussian measure. 

\begin{lemma}
\label{lem:Gaussian_shift} 
For $a>0$, let $\mu(a)$ be the Gaussian measure with covariance 
$(-\Delta_N+a)^{-1}$. Then for any $\mu(a)$-integrable random variable 
$F(\phi_N)$ and any $b>0$, one has 
\begin{equation}
 \bigexpecin{\mu(a)}{F(\phi_N)} 
 = \sqrt{\CF_N(a-b;b)} \, \biggexpecin{\mu(b)}{\exp\biggset{-\frac{a-b}2 
\int_{\Lambda}\Wick{\phi_N(x)^2}\6x }F(\phi_N)}\;,
\end{equation} 
where the Wick power is with respect to $\Tr(-\Delta+b)^{-1}/L^2$, and for 
any $b\notin\spec(\Delta_N)$, we define 
\begin{equation}
\label{eq:Carleman_Fredholm} 
 \CF_N(c;b) 
 = \det{}_2\Bigpar{\one + c(-\Delta_N+b)^{-1}}
:= \det\Bigpar{\one + c(-\Delta_N+b)^{-1}} 
\e^{-c\Tr(-\Delta_N+b)^{-1}}\;.
\end{equation} 
\end{lemma}

The modified determinant $\det{}_2(\one + \cL) = \det(\one+\cL)\e^{-\Tr\cL}$  
of a linear operator $\cL$, as appearing in~\eqref{eq:Carleman_Fredholm}, is 
called its  \emph{Carleman--Fredholm determinant}, and is convergent whenever  
$\cL$ is \emph{Hilbert--Schmidt}, meaning that only $\cL^\dagger\cL$ needs to 
be trace class, see~\cite[Chapter~5]{Simon_Trace_ideals}. 

\begin{exercise}
Use a Fourier basis $\set{e_k}_{k\in\cK_N}$ in order to 
\begin{itemize}
\item 	show that there is a constant $M$ uniform in $N$, $b$ and $c$ such 
that $\e^{-Mc^2/b^2} \leqs \CF_N(c;b) \leqs \e^{Mc^2/b^2}$; 
\item 	show that $\CF_N(a-b;b)^{-1} = \CF_N(b-a;a) 
\exp\Bigset{(a-b)^2\Tr[(-\Delta+a)^{-1}(-\Delta+b)^{-1}]}$;
\item 	prove \Cref{lem:Gaussian_shift}, using Parseval's identity.
\qedhere
\end{itemize}
\end{exercise}

As in the one-dimensional case, the Galerkin 
approximation~\eqref{eq:AC-2d-Galerkin} is equivalent to a finite-dimensional 
It\^o SDE 
\begin{equation}
 \6\hat\phi_t = -\nabla \widehat V_N(\hat\phi_t)\6t + \sqrt{2\eps}\6W_t\;,  
\end{equation}
with a potential $\widehat V_N$ depending now explicitly on $C_N$. We can thus 
apply as before the potential-theoretic approach. To be able to apply the 
symmetry argument of \Cref{lem:symmetry_hAB}, we choose again symmetric sets
\begin{equation}
\label{eq:condAB_2d} 
 A = -B = \bigsetsuch{y}{\abs{\phi_0+L}\leqs\delta, \phi_\perp\in D_\perp}, 
\end{equation} 
with $D_\perp$ containing this time a ball of radius $\sqrt{\log(\eps^{-1})}$ 
in $H^s(\Lambda)$ for some $s<0$. Then we have the following 
result~\cite{Berglund_DiGesu_Weber_16}. 

\begin{proposition}
\label{prop:Z-cap_AC-2d}
The partition function and capacity satisfy 
\begin{equation}
\frac{\cZ_N}{2\capacity(A,B)} 
= \frac{2\pi}{\abs{\mu_0}} 
\frac{\e^{-3\theta/2}}{\sqrt{\abs{\CF_N(3;-1))}}}
\e^{V_0(0)/\eps} \bigbrak{1+\Order{\eps}}\;,
\end{equation}
where $\nu_0=2$, $\mu_0=-1$, and the 
error term is uniform in the cut-off $N$. 
\end{proposition}
\begin{proof}[\Sketch]
We write $\cK_N^* = \cK_N\setminus\set{0}$, and define $\CF_N^\perp(c;b)$ as 
in~\eqref{eq:Carleman_Fredholm}, with $\Delta_{N}$ replaced by 
$\Delta_{\perp,N}$. Interpreting as before the integral over $\phi_\perp$ as an 
expectation under the Gaussian measure $\mu_\perp(-1+m^2(\phi_0))$, which is 
well-defined since the smallest eigenvalue of $\Delta_{\perp,N}$ is 
$(2\pi/L)^2$, we find for the partition function
\begin{align}
\cZ_N
&= \int_{E_N} \e^{-V(\phi)/\eps} \6\phi \\
&= \int_{-\infty}^\infty \e^{-V_0(\phi_0)/\eps} 
\prod_{k\in\cK_N^*} \sqrt{\frac{2\pi\eps}{\lambda_k-1+m^2(\phi_0)}}
\e^{(1-m^2(\phi_0))L^2C_N/2}
\bigexpecin{\mu_\perp(-1+m^2(\phi_0))}{\e^{-R_\eps}} \6\phi_0 \\
&= \prod_{k\in\cK_N^*} \sqrt{\frac{2\pi\eps}{\lambda_k-1}}
\int_{-\infty}^\infty \e^{-V_0(\phi_0)/\eps} 
\sqrt{\frac{\e^{L^2C_N} \e^{-m^2(\phi_0)(\abs{\mu_0}^{-1}+\theta)}} 
{\CF_N^\perp(m^2(\phi_0);-1)}}
\bigexpecin{\mu_\perp(-1+m^2(\phi_0))}{\e^{-R_\eps}} \6\phi_0\;.
\end{align}
Proceeding as in \Cref{prop:exp_moment}, one can show that 
$\bigexpecin{\mu_\perp(-1+m^2(\phi_0))}{\e^{-R_\eps}} = 1 + 
\Order{\eps\e^{\phi_0^4}}$. As in that proposition, the argument actually 
requires Gaussian changes of mass, using here \Cref{lem:Gaussian_shift}, to 
obtain the required estimate, 
cf.~\cite[Proposition~5.7]{Berglund_DiGesu_Weber_16}. When performing the 
integral over $\phi_0$, we note that $-V_0(\phi_0)$ is maximal at $\phi_0=\pm 
L^2$, where $m^2(\phi_0)=3$, so that by Laplace asymptotics we obtain 
\begin{equation}
\cZ_N 
= 2 \sqrt{\frac{2\pi\eps \e^{L^2C_N-3(\abs{\mu_0}^{-1}+\theta)}} 
{\nu_0 \CF_N^\perp(3;-1)}}
\prod_{k\in\cK_N^*} \sqrt{\frac{2\pi\eps}{\lambda_k-1}} 
\bigbrak{1+\Order{\eps}}\;.
\end{equation} 
For the capacity, using the Dirichlet and Thomson principles, we find in a 
similar way
\begin{equation}
 \capacity(A,B) 
 = \frac{\eps}{c_0^2} 
 \prod_{k\in\cK_N^*} \sqrt{\frac{2\pi\eps}{\lambda_k-1}}
\int_{-a}^a \e^{V_0(\phi_0)/\eps} 
\sqrt{\frac{\e^{L^2C_N} \e^{-m^2(\phi_0)(\abs{\mu_0}^{-1}+\theta)}} 
{\CF_N^\perp(m^2(\phi_0);-1)}}
\bigexpecin{\mu_\perp(-1+m^2(\phi_0))}{\e^{-R_\eps}} \6\phi_0\;,
\end{equation} 
where $c_0$ is as in~\eqref{eq:cap_h0}. This time, the integral over $\phi_0$ 
is dominated by $\phi_0$ near $0$, where $m^2=0$, and the result is 
\begin{equation}
 \capacity(A,B) 
= \frac{\eps}{c_0} 
\sqrt{\frac{\e^{L^2C_N}}{\CF_N^\perp(0;-1)}}
\prod_{k\in\cK_N^*} \sqrt{\frac{2\pi\eps}{\lambda_k-1}} 
\bigbrak{1+\Order{\eps}}\;.
\end{equation} 
The result follows, using the facts that $c_0 = 
\sqrt{2\pi\eps\abs{\mu_0}^{-1}}\e^{V_0(0)/\eps}[1+\Order{\eps}]$, that  
$\CF_N^\perp(0;-1)=1$, and that $\nu_0\abs{\mu_0}^{-1}\CF_N^\perp(3;-1) = 
\e^{-3\abs{\mu_0}^{-1}}\abs{\CF_N(3;-1)}$.
\end{proof}

To turn \Cref{prop:Z-cap_AC-2d} into a limiting result, one needs to perform 
some post-processing, as discussed in the proof of \Cref{thm:EK-1d}. Namely, one 
has to show that the expected transition times of the spectral Galerkin 
approximation converge to the one of the limiting system, and derive a coupling 
result to get rid of the equilibrium distribution on $\partial A$. This has been 
achieved in~\cite{Tsatsoulis_Weber_18}, with the following result. 

\begin{theorem}[Eyring-Kramers law for the two-dimensional Allen--Cahn SPDE]
\label{thm:EK-2d} 
Assume that $L<2\pi$ and $A$ and $B$ are as in~\eqref{eq:condAB_2d}. Then, for 
the choice~\eqref{eq:CN_Galerkin} of regularisation $C_N$, one has 
\begin{equation}
\lim_{N\to\infty} 
 \bigexpecin{\phi^*_-}{\tau_B}
 = \frac{2\pi}{\abs{\mu_0}}
 \frac{\e^{-3\theta/2}\e^{L^2/(4\eps)}}{\sqrt{\bigabs{\det_{2}\brak{\one + 
 3(-\Delta-1)^{-1}}}}}\bigbrak{1+\Order{\eps}}\;,
\end{equation} 
where the determinant is to be interpreted as the Carleman--Fredholm 
determinant~\eqref{eq:Carleman_Fredholm}. 
\end{theorem}

%%%%%%%%%%%%%%%%%%%%%%%%%%%%%%%%%%%%%%%%%%%%%%%%%%%%%%%%%%%%%%%%%%%%%%%%%%%%%%%%

\section{Bibliographical notes}
\label{sec:2dbib} 

Besides the system of coupled diffusions, the Allen--Cahn SPDE also describes 
particular scaling limits of Ising--Kac models, in which each spin interacts 
with a large number of other spins~\cite{Mourrat_Weber_17_Kac}. 

Background on Wick calculus can be found 
in~\cite[Chapter~1]{nualart2006malliavin}. The graphical notation for Wick 
powers has been introduced in~\cite{Hairer2014}, and we adopted the version used 
in~\cite{Chandra_Weber_LN17} for their regularisation. 

Construction of solutions to stochastic quantisation equations was investigated 
in~\cite{Albeverio_Rockner_91}, using the theory of Dirichlet forms. Existence 
and uniqueness of strong solutions to stochastic quantisation equations on the 
two-dimensional torus, as well as the absence of explosion, was first proved 
in~\cite{daPratoDebussche}. The result was extended to the whole plane 
in~\cite{Mourrat_Weber_17}. 

Connections between the Gibbs measure and the stochastic quantisation equation 
were investigated in~\cite{JonaLasinio_Mitter_85}. The idea to use the 
Bou\'e--Dupuis formula to bound the partition function was introduced 
in~\cite{Barashkov_Gubinelli_18}. The fact that solutions to stochastic 
quantisation equations satisfy the Markov property and are reversible with 
respect to the Gibbs measure was proved in~\cite{Rockner_Zhu_Zhu_15} using 
Dirichlet forms, while uniqueness of the Gibbs measure and convergence to it 
were obtained in~\cite{Rockner_Zhu_Zhu_17}. The strong Feller property, as well 
as the Markov property and exponential mixing were proved 
in~\cite{Tsatsoulis_Weber_16} using a dissipative bound, while the strong Feller 
property was also proved (for more general equations) 
in~\cite{Hairer_Mattingly_18}, using the theory of regularity structures.  

The large-deviation results are from~\cite{HairerWeber}. The 
uniform-in-$N$ Eyring--Kramers asymptotics on expected transition times were 
obtained in~\cite{Berglund_DiGesu_Weber_16}, and extended to a full proof of 
the Eyring--Kramers formula in~\cite{Tsatsoulis_Weber_18}. 

%%%%%%%%%%%%%%%%%%%%%%%%%%%%%%%%%%%%%%%%%%%%%%%%%%%%%%%%%%%%%%%%%%%%%%%%%%%%%%%%

\chapter{Allen--Cahn SPDE in three space dimensions}
\label{ch:dim3} 

We finally turn to the analysis of the three-dimensional Allen--Cahn equation, 
whose non-renormalised version is formally written as  
\begin{equation}
\label{eq:AC-3d} 
 \partial_t \phi(t,x) = \Delta \phi(t,x) + \phi(t,x) - \phi(t,x)^3 
 + \sqrt{2\eps} 
 \xi(t,x)\;,
\end{equation} 
where $x$ now belongs to the torus $\Lambda=(\R/(L\Z))^3$. The deterministic PDE 
obtained for $\eps=0$ is again well-defined. However, we will see that the 
stochastic PDE remains ill-posed even when the nonlinear term $\phi(t,x)^3$ is 
replaced by its Wick power $\Wick{\phi(t,x)^3}$. This is a consequence of 
space-time white noise being even more singular in three than in two space 
dimensions.

There exist by now several methods to make sense of the singular 
equation~\eqref{eq:AC-3d} after proper renormalisation, namely \emph{regularity 
structures}~\cite{Hairer2014}, \emph{paracontrolled 
calculus}~\cite{Gubinelli_Imkeller_Perkowski_15,Catellier_Chouk_18} and a 
Wilsonian renormalisation group approach~\cite{Kupiainen_16}. In this chapter, 
we will describe the approach based on regularity structures. First, however, 
we need to understand the reason for the ill-posedness in more detail. 

%%%%%%%%%%%%%%%%%%%%%%%%%%%%%%%%%%%%%%%%%%%%%%%%%%%%%%%%%%%%%%%%%%%%%%%%%%%%%%%%

\section{Failure of the previous approaches}
\label{sec:3d_failure} 

The definition of space-time white noise $\xi$ in one time and three space 
dimensions proceeds as before, with some straightforward adjustments. In 
particular, $\xi$ has now the scaling behaviour 
\begin{equation}
 \xi_{\tau,\lambda} \eqinlaw \frac{1}{\sqrt{\tau\lambda^3}}\xi\;.
\end{equation}
With the parabolic norm 
\begin{equation}
 \norm{(t,x)-(s,y)}_\fraks = \abs{t-s}^{1/2} + \sum_{i=1}^3 \abs{x_i-y_i}\;,
\end{equation} 
and the scaling operator 
\begin{equation}
 (\cS^{\lambda}_{t,x}\ph)(s,y)
 := \frac{1}{\lambda^5} \ph 
\biggpar{\frac{s-t}{\lambda^2},\frac{y-x}{\lambda}}\;,
\end{equation} 
\Cref{thm:space_time_white_noise_reg} becomes  
\begin{equation}
\label{eq:reg_noise_d3} 
 \xi\in\cC^{-5/2-\kappa}_\fraks
 \qquad \text{for any $\kappa>0$\;.}
\end{equation} 
The Schauder estimate stated in \Cref{thm:Schauder} remains true, showing that  
the stochastic convolution $P*\xi$ now belongs to $\cC^\alpha_\fraks$ for  
$\alpha<-\frac12$ only. It is thus much more singular than in dimension two, 
where 
it was very close to being a function. 

In analogy with what we did in dimension two, it seems natural 
to consider the regularised equation (for $2\eps=1$)
\begin{align}
 \partial_t \phi_\delta(t,x) 
 &= \Delta\phi_\delta(t,x) + \phi_\delta(t,x) - 
\Wick{\phi_\delta(t,x)^3}_{C_\delta} + \xi^\delta(t,x) \\
 &= \Delta\phi_\delta(t,x) + \phi_\delta(t,x) - \Bigpar{\phi_\delta(t,x)^3 - 
3C_\delta\phi_\delta(t,x)} + \xi^\delta(t,x)\;,
\label{eq:AC-3d-Wick} 
\end{align}
where $\xi^\delta$ denotes a mollified noise defined by 
\begin{equation}
 \xi^\delta(t,x) = (\varrho_\delta*\xi)(t,x)\;, 
 \qquad 
 \varrho_\delta(t,x) = \frac{1}{\delta^5} \varrho\biggpar{\frac{t}{\delta^2}, 
\frac{x}{\delta}}
\end{equation} 
for a compactly supported smooth function $\varrho:\R\times\Lambda \to \R$ of 
integral $1$. The renormalisation constant $C_\delta$ defining the Wick power 
in~\eqref{eq:AC-3d-Wick} is given as in two dimensions by 
\begin{equation}
 C_\delta := \int_{-\infty}^\infty P_\delta(t,x)^2\6x\6t\;,
\end{equation} 
where $P_\delta=P*\varrho_\delta$, except that $P$ is now the three-dimensional 
heat kernel 
\begin{equation}
\label{eq:heat_kernel_3d} 
 P(t,x) := \sum_{k\in\Z^3} P_{\,\R^3}(t,x-kL)\;, 
 \qquad 
 P_{\,\R^3}(t,x) := \frac{1}{(4\pi t)^{3/2}} \e^{-\norm{x}^2/(4t)} 
\indexfct{t>0}\;. 
\end{equation} 

\begin{exercise}
Use the fact that the Green function of the three-dimensional Laplacian is 
the periodicised version of 
\begin{equation}
 G_{\,\R^3}(x) = -\frac{1}{4\pi\norm{x}}
\end{equation} 
to show that $C_\delta$ diverges like $\delta^{-1}$ as $\delta\to0$. 
\Hint Recall \Cref{prop:asymp_Cdelta}.
\end{exercise}

Making the ansatz $\phi_\delta = \RSI_\delta + \psi_\delta$, where $\RSI_\delta 
= P*\xi^\delta$, and taking the limit $\delta\to0$, we obtain as in 
\Cref{sec:2d_existence} the fixed-point equation
\begin{equation}
\label{eq:FP_3d_psi} 
 \psi = P\psi_0 + P*\bigbrak{\RSI + \psi - \psi^3 
- 3\RSI\psi^2 - 3\RSV\psi - \RSW}\;.
\end{equation} 
The term with the worst regularity on the right-hand side should be  
\begin{equation}
 P*\RSW =: \RSIW\;.
\end{equation} 
Here $\RSW$ belongs to $\cC^\alpha_\fraks$ for any $\alpha<-\frac32$, so that 
the Schauder estimate shows that $\RSIW$ has H\"older regularity almost 
$\frac12$, and the same is expected to hold for $\psi$. This is a problem, 
however, for the definition of the product $\RSV\psi$. Indeed, since $\RSV$ has 
regularity almost $-1$, \Cref{thm:product_Holder} implies that this product is 
ill-defined!

One may think of circumventing this problem by making a second change of 
variables 
\begin{equation}
 \psi = -\RSIW + \eta\;,
\end{equation} 
in the hope that $\eta$ will be more regular than $\psi$. 
This turns the fixed-point equation into 
\begin{equation}
\label{eq:FP_eta} 
 \eta = P\eta_0 
 + P*\bigl[ \RSI - \RSIW + \eta - \eta^3 + 3\eta^2(\RSIW-\RSI)
 -3\eta (\RSV - 2\RSI\RSIW + \RSIW^2)
 + \RSIW^3 - 3\RSI\RSIW^2 + 3\RSV\RSIW % \\ 
 \bigr]\;.
\end{equation} 
The problem now is that the product $\RSV\RSIW$ is not well-defined, since the 
regularity of the two factors is slightly less than $-1$ and $\frac12$ 
respectively. 
This is a first indication that a second renormalisation constant will be 
needed. It further suggests that one has to rethink the way solutions are 
represented as linear combinations of iterated stochastic integrals.  

%%%%%%%%%%%%%%%%%%%%%%%%%%%%%%%%%%%%%%%%%%%%%%%%%%%%%%%%%%%%%%%%%%%%%%%%%%%%%%%%

\section{Perturbation theory for the Gibbs measure}
\label{sec:3d_Gibbs} 

One way to understand the need for a further renormalisation constant is to
consider, as in \Cref{sec:2d_metastability}, the Wick-renormalised potential
\begin{equation}
\label{eq:def_pot_3d} 
 V(\phi) := \int_\Lambda \biggbrak{\frac12\norm{\nabla\phi(x)}^2 + 
\frac14\Wick{\phi(x)^4} - \frac12\Wick{\phi(x)^2} + \frac14} \6x\;,
\end{equation} 
where we use a spectral Galerkin approximation of order $N$ for $\phi$, and 
Wick powers are with respect to $\eps C_N$. Here we choose 
\begin{equation}
\label{eq:CN_3d} 
 L^2C_N = \Tr\Bigpar{\bigpar{-\Delta_{\perp,N}-1}^{-1}}
 = \sum_{k\in\cK_N^*} \frac{1}{\lambda_k-1} 
 = \Order{N}\;, 
\end{equation} 
where, like before, $\lambda_k = (2\norm{k}\pi/L)^2$ and $\cK_N^* = 
\setsuch{k\in\Z^3}{0<\abs{k}\leqs N}$ with $\abs{k} = 
\abs{k_1}+\abs{k_2}+\abs{k_3}$.

To lighten the notation, we assume from now on that $L=1$. 
As we have seen, e.g., in the proof of \Cref{prop:Z-cap_AC-2d}, showing that 
the potential~\eqref{eq:def_pot_3d} is well-defined involves computing 
the expectation of $\e^{-R_\eps}$ under the Gaussian measure 
$\mu_\perp=\mu_\perp(\phi_0)$ having covariance 
$(-\Delta_{\perp,N}-1+3\phi_0^2)^{-1}$, %with $m^2(\phi_0)=3\phi_0^2$ 
where   
\begin{equation}
  R_\eps(\phi_0,\phi_\perp) := 
%\frac{3}{2}\phi_0^2 \int_{\Lambda} \Wick{\phi_\perp(x)^2} \6x +
 \sqrt{\eps}\phi_0\int_{\Lambda} \Wick{\phi_\perp(x)^3} \6x
+ \frac{\eps}{4}\int_{\Lambda} \Wick{\phi_\perp(x)^4} \6x\;.
\end{equation} 
We will represent this graphically as 
\begin{equation}
 R_\eps(\phi_0,\phi_\perp) = 
%\frac{3}{2}\phi_0^2 \FDVtwo{vedge}{vedge} +
\sqrt{\eps}\phi_0 \FDVthree{vedge}{vedge}{vedge}
+ \frac{\eps}{4} \FDVfour{vedge}{vedge}{vedge}{vedge}\;.
\end{equation} 
We can compute $\bigexpecin{\mu_\perp}{\e^{-R_\eps}}$ perturbatively in 
$\eps$ by using the \emph{cumulant expansion} 
\begin{align}
-\log \bigexpecin{\mu_\perp}{\e^{-R_\eps}}
={}& \bigexpecin{\mu_\perp}{R_\eps} \\
&{}+ \frac{1}{2!} \Bigpar{\bigexpecin{\mu_\perp}{R_\eps}^2 - 
\bigexpecin{\mu_\perp}{R_\eps^2}} \\
&{}+ \frac{1}{3!} \Bigpar{2\bigexpecin{\mu_\perp}{R_\eps}^3 - 
3\bigexpecin{\mu_\perp}{R_\eps}\bigexpecin{\mu_\perp}{R_\eps^2} + 
\bigexpecin{\mu_\perp}{R_\eps^3}} \\
&{}+ \dots 
\label{eq:cumulant} 
\end{align}
In order to compute these expectations, we recall 
(see~\eqref{eq:covariance_GFF}) that\footnote{We are going to be slightly sloppy 
here, by ignoring the fact that the covariance of the Gaussian measure 
$\mu_\perp$ depends on $\phi_0$, which affects $G_N$ and $C_N$. However, this 
results only in negligible extra terms, essentially thanks to 
\Cref{lem:Gaussian_shift}.}
\begin{equation}
 \bigexpecin{\mu_\perp}{\phi_\perp(x)\phi_\perp(y)} = G_N(x-y)\;,
\end{equation} 
where $G_N$ is the Green function of $\Delta_{\perp,N}-1$.
In a Fourier basis $\set{e_k}_{k\in\cK_N^*}$, it is given by 
\begin{equation}
 G_N(x) := \sum_{k\in\cK_N^*} \frac{1}{\lambda_k-1} e_k(x)
\end{equation} 
(note that we recover the fact that $G_N(0)=C_N$). 

Consider first the case $\phi_0=0$, when only the quartic term is present in 
$R_\eps$. We know that the fourth Wick power has zero expectation, while 
according to \Cref{lem:hermite_moments}, its variance is given by 
\begin{equation}
 \int_\Lambda \int_\Lambda \bigexpecin{\mu_\perp}{\Wick{\phi_\perp(x)^4} 
\Wick{\phi_\perp(y)^4}} \6x\6y 
= 4! \int_\Lambda G_N(x)^4\6x
\end{equation} 
(we have used translation invariance of the Green function and the fact that 
$L=1$). We denote this graphically (cf.~\eqref{eq:Wick_graphical}) by 
\begin{equation}
\label{eq:Wick_G4} 
 \biggexpecin{\mu_\perp}{\FDVfour{vedge}{vedge}{vedge}{vedge} 
\FDVfour{vedge}{vedge}{vedge}{vedge}} = 4! 
\FDQzero{vedge}{vedge}{vedge}{vedge}\;. 
\end{equation} 
In a similar way, we find 
\begin{equation}
\label{eq:Wick_G6} 
 \biggexpecin{\mu_\perp}{\FDVfour{vedge}{vedge}{vedge}{vedge} 
\FDVfour{vedge}{vedge}{vedge}{vedge} \FDVfour{vedge}{vedge}{vedge}{vedge}} =
\binom{4}{2}^3 2^3 
\FDTzero{vedge}{vedge}{vedge}{vedge}{vedge}{vedge}\;,
\end{equation} 
where the combinatorial factor counts the number of pairings of the legs, and 
\begin{equation}
\label{eq:Wick_C6} 
 \FDTzero{vedge}{vedge}{vedge}{vedge}{vedge}{vedge}
 := \int_{\Lambda}\int_{\Lambda} 
  G_N(x)^2 G_N(y)^2 G_N(x-y)^2 \6x\6y\;.
\end{equation} 
The graphs occurring in~\eqref{eq:Wick_G4} and~\eqref{eq:Wick_C6} are special 
cases of Feynman diagrams, called \emph{vacuum diagrams}. 

\begin{exercise}
Prove the bounds  
\begin{gather}
\label{eq:GNn_asymptotics} 
\int_{\Lambda} G_N(x)^n \6x \lesssim  
 \begin{cases}
  1 & \text{if $n=2$\;,} \\
  \log N & \text{if $n=3$\;,} \\
  N & \text{if $n=4$\;,} 
 \end{cases} \\
 \int_{\Lambda}\int_{\Lambda} 
  G_N(x)^2 G_N(y)^2 G_N(x-y)^2 \6x\6y
  \lesssim \log(N)\;.
\end{gather} 
\Hint One may use the Young-type inequality 
\begin{equation}
 \sum_{\substack{k_1,k_2\in\Z^d\setminus\set{0} \\ k_1+k_2=k}} 
\frac{1}{\norm{k_1}^n \norm{k_2}^m} \lesssim \frac{1}{\norm{k}^{n+m-d}}
\end{equation} 
valid whenever $0<n,m<d$ and $n+m>d$, see for 
instance~\cite[Lemma~3.10]{zhu2015three}.  
\end{exercise}

\begin{remark}
Equivalent estimates hold for the Green function $G_\delta = 
G*\varrho_\delta$ mollified on scale $\delta=1/N$, where $\varrho_\delta(x) = 
\delta^{-3}\varrho(\delta^{-1}x)$. In particular, we have $G_\delta(x) = 
\Order{(\norm{x}+\delta)^{-1}}$, see for 
instance~\cite[Lemma~10.17]{Hairer2014}.
\end{remark}

Inserting~\eqref{eq:Wick_G4} and~\eqref{eq:Wick_G6} in the cumulant 
expansion~\eqref{eq:cumulant}, we obtain the asymptotic expansion 
\begin{equation}
\label{eq:cumulant_phi_eq0} 
 -\log \bigexpecin{\mu_\perp(0)}{\e^{-R_\eps(0,\cdot)}} 
 = -\frac{4!}{2!4^2} \eps^2  \FDQzero{vedge}{vedge}{vedge}{vedge}
 + \frac{2^3}{3!4^3} \binom{4}{2}^3 \eps^3 
\FDTzero{vedge}{vedge}{vedge}{vedge}{vedge}{vedge}
 + \Order{\eps^4}\;,
\end{equation} 
where the two vacuum diagrams diverge, respectively, like $N$ and like 
$\log N$. 
This suggests that for $\phi_0=0$, the potential~\eqref{eq:def_pot_3d} can  
be renormalised by subtracting two terms of respective order $\eps^3 N$ and 
$\eps^4 \log(N)$, called \emph{energy renormalisation} terms. Of 
course, at this point we can tell nothing about the $N$-dependence 
of the higher-order terms in $\eps$. 

If $\phi_0\neq0$, the variance of the third Wick power in $R_\eps$ will 
add a new divergent term of order $\eps\phi_0^3\log N$ to the cumulant 
expansion. This is a symptom of the fact that the Gibbs measures obtained for 
different values of $\phi_0$ are not absolutely continuous with respect to one 
another. We cannot merely subtract this $\phi_0$-dependent term from the 
potential, as this would not result from a decomposition into mean and 
oscillating part. However, we can modify the potential~\eqref{eq:def_pot_3d} by 
adding a renormalisation term to the coefficient of the second Wick power, that 
is, 
\begin{equation}
\label{eq:def_pot_3d_mass} 
 V(\phi) = \int_\Lambda \biggbrak{\frac12\norm{\nabla\phi(x)}^2 + 
\frac14\Wick{\phi(x)^4} - \frac12\Bigpar{1-\eps^2C_N^{(2)}}\Wick{\phi(x)^2} + 
\frac14
%+ \eps^3 C_N^{(3)} + \eps^4 C_N^{(4)}
} \6x\;. 
\end{equation} 
This procedure, called \emph{mass renormalisation}, results in the expression 
\begin{equation}
 R_\eps(\phi_0,\phi_\perp) = 
\frac12\eps C_N^{(2)} \phi_0^2 + \frac12\eps^2 C_N^{(2)} \FDVtwo{vedge}{vedge}
+ \sqrt{\eps}\phi_0 \FDVthree{vedge}{vedge}{vedge}
+ \frac{\eps}{4} \FDVfour{vedge}{vedge}{vedge}{vedge}\;.
\end{equation} 
In particular, we now find for the first two moments
\begin{align}
\bigexpecin{\mu_\perp}{R_\eps} 
&= \frac12\eps C_N^{(2)} \phi_0^2\;, \\ 
\bigexpecin{\mu_\perp}{R_\eps^2} 
&= \biggpar{\frac12\eps C_N^{(2)}\phi_0^2}^2 
 + \eps\phi_0^2 3! \FDCthree{vedge}{vedge}{vedge}
 + \frac{\eps^2}{4^2} 4! \FDQzero{vedge}{vedge}{vedge}{vedge}
 + \Order{\eps^4}\;,
\end{align} 
where we have used the fact that terms with an odd number of legs have zero 
expectation (cf. \Cref{lem:hermite_moments}). Substituting in the cumulant 
expansion~\eqref{eq:cumulant}, we obtain 
\begin{equation}
 -\log \bigexpecin{\mu_\perp}{\e^{-R_\eps}}
 = \frac12\eps C_N^{(2)} \phi_0^2 
 + \frac{1}{2!} \biggpar{-\eps\phi_0^2 3! \FDCthree{vedge}{vedge}{vedge}
 - \frac{\eps^2}{4^2} 4! \FDQzero{vedge}{vedge}{vedge}{vedge}} 
 + \Order{\eps^3}\;.
\end{equation}
We thus see that choosing
\begin{equation}
 C_N^{(2)} = 3! \FDCthree{vedge}{vedge}{vedge}
\end{equation} 
allows to kill the divergent term of order $\eps$. Pushing the expansion one 
step further, one can show that the divergent part of the term of order 
$\eps^3$ is the same as in~\eqref{eq:cumulant_phi_eq0}. 

What this purely formal computation suggests, is that by choosing the 
renormalised potential of the form 
\begin{equation}
\label{eq:def_pot_3d_renorm} 
 V(\phi) = \int_\Lambda \biggbrak{\frac12\norm{\nabla\phi(x)}^2 + 
\frac14\Wick{\phi(x)^4}_{\eps C_N^{(1)}} 
-\frac12\Bigpar{1-\eps^2C_N^{(2)}}\Wick{\phi(x)^2}_{\eps C_N^{(1)}} + \frac14 
+ \eps^3 C_N^{(3)} - \eps^4 C_N^{(4)}} \6x\;,  
\end{equation}
the limit $N\to\infty$ may be well-defined, provided we choose the counterterms 
\begin{align}
C_N^{(1)} &:= \FDtadpolenolabel{vedge} = \Order{N}\;, \\
C_N^{(2)} &:= 3! \FDCthree{vedge}{vedge}{vedge} = \Order{\log N}\;,\\
C_N^{(3)} &:= \frac{4!}{2!4^2} \FDQzero{vedge}{vedge}{vedge}{vedge} 
= \Order{N}\;,\\
C_N^{(4)} &:= \frac{2^3}{3!4^3} \binom{4}{2}^3 
\FDTzero{vedge}{vedge}{vedge}{vedge}{vedge}{vedge} = \Order{\log N}\;.
\label{eq:CN_3d_potential} 
\end{align}
Of course, this perturbative computation up to order $\eps^3$ does not imply 
that no other counterterms of higher order in $\eps$ are needed. The fact that 
this is not the case is a highly nontrivial fact, which makes the static 
$\Phi^4_3$ model one of the important solved cases of (Euclidean) Quantum Field 
Theory. See \Cref{sec:3dbib} for some historical notes. A recent proof, which 
is relatively compact, of the well-posedness of the Gibbs measure is provided 
in~\cite{Barashkov_Gubinelli_18}. 

%%%%%%%%%%%%%%%%%%%%%%%%%%%%%%%%%%%%%%%%%%%%%%%%%%%%%%%%%%%%%%%%%%%%%%%%%%%%%%%%

\section{Regularity structures: introduction}
\label{sec:3d_regstruc_overview} 

%%%%%%%%%%%%%%%%%%%%%%%%%%%%%%%%%%%%%%%%%%%%%%%%%%%%%%%%%%%%%%%%%%%%%%%%%%%%%%%%

\subsection{H\"older spaces revisited}
\label{ssec:3d_Holder} 

One of the key ideas of the theory of regularity structures is to replace 
expansions as appearing in the fixed-point equation~\eqref{eq:FP_3d_psi} by  
more general expansions, which separate the roles of the coefficients, such as 
$\psi(t,x)$, and of the singular distributions, such as $\RSV$. These objects 
can be viewed as a kind of generalised Taylor expansion.

Consider for simplicity the case of a function $f_0\in\cC^{2+\alpha}(I)$, where 
$I\subset\R$ is an interval, and $\alpha\in(0,1)$. One way of defining 
$\cC^{2+\alpha}(I)$ is to require that $f_0$ is twice continuously 
differentiable, with a second derivative in $\cC^{\alpha}(I)$, 
cf. \Cref{exo:Holder_derivative}. An alternative definition, which is more in 
line with \Cref{def:Holder-positive}, is to require the existence of a triple 
$(f_0,f_1,f_2)$ of functions from $I$ to $\R$ satisfying for all $x,y\in I$ the 
bounds on Taylor expansions 
\begin{align}
 \bigabs{f_0(y)-f_0(x)-(y-x)f_1(x)-\tfrac12(y-x)^2f_2(x)} 
 &\leqs C \abs{x-y}^{2+\alpha}\;, \\
 \bigabs{f_1(y)-f_1(x)-(y-x)f_2(x)} 
 &\leqs C \abs{x-y}^{1+\alpha}\;, \\
 \bigabs{f_2(y)-f_2(x))} 
 &\leqs C \abs{x-y}^{\alpha}\;.
 \label{eq:poly_holder} 
\end{align}
Note that this definition, while it seems redundant, completely avoids the use 
of derivatives, though of course we necessarily have %$f_0(x) = f(x)$, 
$f_1(x)=f_0'(x)$ and $f_2(x)=f_0''(x)$. 

In the theory of regularity structures, one uses the abstract notation
\begin{equation}
\label{eq:f_poly} 
f(x) = f_0(x)\unit + f_1(x) \X + \tfrac12 f_2(x) \X^2\;,
\end{equation} 
where the symbols $\unit$, $\X$ and $\X^2$ are basis vectors of a vector space 
$T$ of dimension $3$, called a \emph{model space}, and $f:I\to T$ is considered 
as a map from $I$ to $T$. We will denote the collection of basis vectors by 
$\cF = \set{\unit,\X,\X^2}$. 

In order to be able to evaluate the abstract object~\eqref{eq:f_poly} at a 
particular point, we introduce the notion of \emph{model}. This is given by a 
collection of maps $\setsuch{\Pi_x\tau}{x\in I,\tau\in\cF}$ defined for all 
$z\in I$ by 
\begin{equation}
 (\Pi_x\unit)(z) := 1\;, \qquad 
 (\Pi_x\X)(z) := z-x\;, \qquad 
 (\Pi_x\X^2)(z) := (z-x)^2\;.
 \label{eq:Pi_poly} 
\end{equation} 
In this way, we have 
\begin{align}
 (\Pi_x f(x))(z) 
 &:= f_0(x)(\Pi_x\unit)(z) + f_1(x)(\Pi_x\X)(z) + \tfrac12 f_2(x) 
(\Pi_x\X^2)(z) \\
 &= f_0(x) + f_1(x)(z-x) + \tfrac12 f_2(x)(z-x)^2\;,
\end{align}
which is the second-order Taylor expansion of $f$ at $x$. In particular, we 
can recover $f_0$ thanks to the \emph{diagonal identity}  
\begin{equation}
\label{eq:poly_f0} 
 (\Pi_x f(x))(x) = f_0(x)\;.
\end{equation} 
However, $f$ contains more information than just the value of $f_0$ at any 
point $x$, since the higher-order terms of the Taylor expansion provide some 
description of how $f$ behaves in the vicinity of $x$. 

\begin{remark}
An aspect of the theory that may be confusing at first is that it is 
possible to define the quantity $(\Pi_x f(y))(z)$ which depends on three 
different arguments. This is not really of interest in the case of polynomial 
models considered here, but will play a role later on. 
\end{remark}

It will also be important to encode how the model $\Pi_x$ depends on the base 
point $x$. For this purpose, note that if we define a collection 
$\setsuch{\Gamma_{xy}}{x,y\in I}$ of linear maps from $T$ to $T$ by 
\begin{align}
 \Gamma_{xy}\unit &:= \unit\;, \\
 \Gamma_{xy}\X &:= \X + (x-y)\unit\;, \\
 \Gamma_{xy}\X^2 &:= \bigbrak{\X + (x-y)\unit}^2 
 := \X^2 + 2(x-y)\X + (x-y)^2\unit\;, 
 \label{eq:Gamma_poly} 
\end{align}
then we have the relations 
\begin{align}
 \Pi_y\tau &= \Pi_x\Gamma_{xy}\tau
 &
 &\forall\tau\in\cF\;, 
 \forall x,y\in I\;, \\
 \Gamma_{xy}\Gamma_{yz}\tau &= \Gamma_{xz}\tau
 &
 &\forall\tau\in\cF\;, 
 \forall x,y,z\in I\;. 
 \label{eq:PiGamma_poly} 
\end{align} 

\begin{exercise}
Check that the relations~\eqref{eq:PiGamma_poly} hold. What is the natural 
generalisation of \eqref{eq:Pi_poly} and~\eqref{eq:Gamma_poly} to general 
monomials $\X^k$, $k\in\N_0$? Check the analogue of~\eqref{eq:PiGamma_poly} for 
these general values of $k\in\N_0$ as well. 
\end{exercise}

The interest of this framework is that by combining~\eqref{eq:f_poly} 
and~\eqref{eq:Gamma_poly}, we obtain 
\begin{align}
f(y) - \Gamma_{yx}f(x) 
={}& \bigbrak{f_0(y) - f_0(x) - (y-x)f_1-x) - \tfrac12 (y-x)^2 f_2(x)} \unit \\
&{}+ \bigbrak{f_1(y) - f_1(x) - (y-x)f_2(x)} \X \\
&{}+ \tfrac12 \bigbrak{f_2(y) - f_2(x)} \X^2\;.
\end{align}
In view of~\eqref{eq:poly_holder}, we see that if ${\mathcal Q}_k$ denotes the 
projection on the subspace of $T$ spanned by $\X^k$, then we have 
\begin{equation}
\label{eq:Dalpha_poly} 
 f \in \cC^{2+\alpha} 
 \qquad \Leftrightarrow \qquad 
 \bigabs{{\mathcal Q}_k \bigbrak{f(y) - \Gamma_{yx}f(x)}}
 \lesssim \abs{x-y}^{2+\alpha-k} 
 \qquad \forall k\in\set{0,1,2}\;.
\end{equation} 
A central idea in the theory of regularity structures is that this 
characterisation of regularity can be extended to situations in which the 
monomials $\X^k$ are complemented by more singular objects, such as the 
stochastic convolution $\RSI$ and its powers. 

%%%%%%%%%%%%%%%%%%%%%%%%%%%%%%%%%%%%%%%%%%%%%%%%%%%%%%%%%%%%%%%%%%%%%%%%%%%%%%%%

\subsection{Overview of the approach}
\label{ssec:3d_overview} 

Our primary aim will be to construct solutions for the regularised Allen--Cahn 
SPDE on $\R_+\times\Lambda$ given by 
\begin{equation}
\label{eq:AC-3d-renorm} 
 \partial_t \phi_\delta(t,x) = \Delta \phi_\delta(t,x) + \phi_\delta(t,x) - 
 \phi_\delta(t,x)^3 + \bigbrak{3C_\delta^{(1)} - 9C_\delta^{(2)}} 
\phi_\delta(t,x) 
 + %\sqrt{2\eps} 
 \xi^\delta(t,x)\;,
\end{equation} 
where the renormalisation constants $\smash{C_\delta^{(1)}} = 
\Order{\delta^{-1}}$ and $\smash{C_\delta^{(2)}} = \Order{\log(\delta^{-1})}$ 
are analogues of the first two constants in~\eqref{eq:CN_3d_potential}, while 
$\xi^\delta = \varrho_\delta * \xi$ is a mollified version of space-time white 
noise. Indeed, the right-hand side corresponds to the derivative of the 
renormalised potential~\eqref{eq:def_pot_3d_renorm}, which does not see the 
constant terms. 

The general idea of the method developed in~\cite{Hairer2014} can be described 
by the following commutative diagram:

\begin{equation}
\label{eq:RS_diagram} 
 \begin{tikzcd}[column sep=large, row sep=large]
 (\phi_0,Z^\delta) \arrow[r, "\cS", mapsto] 
 & \Phi \arrow[d, "\cR", mapsto] \\
 (\phi_0,\xi^\delta) \arrow[u, "\Psi", mapsto] 
 \arrow[r, "\overline\cS"', mapsto] 
 & \phi_\delta 
\end{tikzcd}
\end{equation}

The maps appearing in this diagram are as follows:
\begin{itemize}
\item 	The \emph{classical solution map} $\overline\cS$ takes an initial 
condition $\phi_0$ and a realisation $\xi^\delta$ of mollified space-time white 
noise, and associates with them the solution $(\phi_\delta(t,x))_{t\geqs0}$ of 
the smooth PDE~\eqref{eq:AC-3d-renorm}. This solution exists globally in 
time for any positive $\delta$ and almost every realisation of the noise, 
because the PDE is smooth with a confining nonlinearity.

\item 	The lift $\Psi$ maps $\xi^\delta$ to a so-called \emph{canonical model} 
$Z^\delta = (\Pi^\delta,\Gamma^\delta)$, which extends the collections of maps 
introduced in~\eqref{eq:Pi_poly} and~\eqref{eq:Gamma_poly} to the stochastic 
convolution, its powers, and similar data. We will explain the construction of 
$Z^\delta$ in \Cref{ssec:3d_models}. 

\item 	The solution map $\cS$ associates with the initial data and the model 
an element $\Phi$ of an abstract space $\cD^\gamma$ of so-called \emph{modelled 
distributions} of regularity $\gamma$. This space can be thought of as an 
abstract analogue of the H\"older space $\cC^\alpha_\fraks$ on the level of 
coefficients of Taylor series. We will describe it in 
\Cref{ssec:3d_modeled_distributions}. The element $\Phi\in\cD^\gamma$ is 
obtained as the fixed point of a contracting map on $\cD^\gamma$. 

\item 	The \emph{reconstruction operator} $\cR$ takes a modelled distribution 
$\Phi\in\cD^\gamma$ and maps it to a distribution in a H\"older space 
$\cC^\alpha_\fraks$. We will describe it in \Cref{ssec:3d_reconstruction}. 
\end{itemize}

The point of the whole procedure is that we have the relation 
\begin{equation}
 \overline \cS = \cR \circ \cS \circ \Psi\;,
\end{equation} 
meaning that the fixed-point problems defining the classical solution 
$\phi_\delta$ and the modelled distribution $\Phi$ are equivalent for any 
$\delta>0$. In addition, both maps $\cS$ and $\cR$ are continuous in an 
appropriate topology. We explain this part of the theory in more detail in 
\Cref{ssec:3d_Schauder}.

Renormalisation comes into play when one modifies the lift $\Psi$. This will be 
an extension of the transformation~\eqref{eq:Wick_powers_limit} from powers of 
stochastic convolutions to their Wick powers, which removes divergencies. One 
can then show that under appropriate conditions, the renormalised models 
$\widehat Z^\delta =\Psi_\delta\xi^\delta$ converge, as $\delta\to0$, to some 
limiting model $\widehat Z$. The image of this limit under the map $\cR\circ\cS$ 
is then defined as the solution of the limiting stochastic PDE. We give more 
details on this procedure in \Cref{ssec:3d_renorm_conv}.

%%%%%%%%%%%%%%%%%%%%%%%%%%%%%%%%%%%%%%%%%%%%%%%%%%%%%%%%%%%%%%%%%%%%%%%%%%%%%%%%

\section{Regularity structures: algebraic aspects}
\label{sec:3d_regstruc_algebra} 

At the core of the theory lies the notion of a \emph{regularity structure}, 
which is the abstract Banach space $T$ spanned by the basis elements of 
generalised Taylor expansions. This space is equipped with two additional 
algebraic structures:

\begin{itemize}
\item 	the \emph{structure group} $\cG$, which allows to represent the link 
between expansions around different base points;
\item 	a \emph{renormalisation group}, which allows to encode the 
renormalisation procedure. 
\end{itemize}

We point out that while the model space is by definition infinite-dimensional, 
in practice only a finite-dimen\-sional subspace of $T$ will matter. This is in 
contrast with Feynman diagram expansions, which are in principle infinite. 
Both the structure group and renormalisation group will be idempotent (a 
sufficiently high power of any group element is the identity), which 
implies that the algebraic structures remain reasonably simple. 

%%%%%%%%%%%%%%%%%%%%%%%%%%%%%%%%%%%%%%%%%%%%%%%%%%%%%%%%%%%%%%%%%%%%%%%%%%%%%%%%

\subsection{The model space}
\label{ssec:3d_modelspace} 

The starting point of the construction is the following definition, 
cf.~\cite[Definition~2.1]{Hairer2014}. 

\begin{definition}[Regularity structure]
\label{def:regularity_structure}
A \emph{regularity structure} is a triple $(A,T,\cG)$ consisting of 
\begin{enumerate}
\item 	an \emph{index set} $A\subset\R$, containing $0$, which is bounded
from below and locally finite;
\item 	a \emph{model space} $T$, which is a graded vector space 
$T=\bigoplus_{\alpha\in A}T_\alpha$, where each $T_\alpha$ is a Banach space;
the space $T_0$ is isomorphic to $\R$ and its unit is denoted $\unit$;
\item 	a \emph{structure group} $\cG$ of linear operators
acting on $T$, such that 
\begin{equation}
 \label{eq:structure_group}
 \Gamma\tau - \tau \in \bigoplus_{\beta<\alpha} T_\beta =: T_\alpha^- 
\end{equation} 
holds for every $\Gamma\in \cG$, every $\alpha\in A$ and every $\tau\in
T_\alpha$;
furthermore, $\Gamma\unit = \unit$ for every $\Gamma\in \cG$. 
\end{enumerate}
\end{definition}

We have already encountered an example of regularity structure in 
\Cref{ssec:3d_Holder}:

\begin{example}[Polynomal regularity structure on $\R$]
This regularity structure is defined as follows:
\begin{itemize}
\item 	The index set $A = \N_0$ is the set of all degrees of monomials on 
$\R$.
\item 	For each $k\in\N_0$, $T_k$ is the one-dimensional real vector space 
spanned by a basis vector denoted $\X^k$ if $k>0$ and $\unit$ if $k=0$. This is 
indeed a Banach space for the norm $\abs{\cdot}$. 
\item 	The structure group is the group $\cG=\setsuch{\Gamma_h}{h\in\R}$ 
defined by 
\begin{equation}
\label{eq:Gamma_h} 
 \Gamma_h \X^k := (\X-h)^k 
 := \sum_{\ell=0}^k \binom{k}{\ell} (-h)^\ell\X^{k-\ell}\;. 
\end{equation} 
Note that the relation~\eqref{eq:structure_group} is indeed satisfied, and that 
$\cG$ is isomorphic to the group of translations on $\R$ via  
$\Gamma_{h_1}\circ\Gamma_{h_2} = \Gamma_{h_1+h_2}$. 
\end{itemize}
The model space $T$ is naturally equipped with the commutative product 
defined by $\X^k \X^\ell = \X^{k+\ell}$, with neutral element $\unit$. 
\end{example}

This construction can be easily extended to higher dimensions.

\begin{example}[Polynomal regularity structure on $\R^{d+1}$]
Fix a scaling $\fraks = (\fraks_0,\dots,\fraks_d) \in\N^{d+1}$. Then this 
regularity structure is defined in the following way: 
\begin{itemize}
\item 	The index set is again $A = \N_0$.
\item 	The model space is now the direct sum of $T_\ell$, $\ell\in\N_0$, where 
each $T_\ell$ is spanned by monomials $\X^k = \X_0^{k_0} \dots \X_d^{k_d}$ of 
scaled degree $\abs{k}_\fraks = \fraks_0k_0 + \dots + \fraks_dk_d = \ell$. 
\item 	The structure group is again defined by~\eqref{eq:Gamma_h}, where the 
binomial coefficients are defined via $k!=k_0!\dots k_d!$. 
\end{itemize}
Again, $T$ is naturally equipped with a commutative product, and $\cG$ is 
now isomorphic to the group of translations on $\R^{d+1}$. 
\end{example}

The \emph{degree} of a monomial $\X^k$ is by definition the quantity 
$\abss{\X^k} = \abss{k}$, which simply tells us to which $T_\ell$ 
the monomial belongs. 

In the case of the three-dimensional Allen--Cahn equation, we take $d=3$ and 
the parabolic scaling $\fraks=(2,1,1,1)$. The polynomial regularity structure 
now has to be enriched with additional elements representing the noise, the 
stochastic convolution, its powers and so on. We start by introducing a symbol 
$\Xi$ representing space-time white noise. In view of~\eqref{eq:reg_noise_d3}, 
we declare its degree to be given by $\abss{\Xi} := \alpha_0$, where $\alpha_0 
:= -\frac52 - \kappa$ for a fixed $\kappa>0$ that we can take as small as we 
like. 

The next step is to introduce an operator $\I$ creating new symbols $\I(\tau)$ 
from existing ones, which represent their convolution with the heat kernel. In 
order to avoid redundancies, we declare that $\I(\X^k)=0$ for elements of the 
polynomial structure. All these symbols can be multiplied, in a commutative 
way, yielding still more symbols. Their degree is defined by the two rules 
\begin{equation}
 \abss{\I(\tau)} := \abss{\tau} + 2\;, 
 \qquad 
 \abss{\tau_1 \tau_2} := \abss{\tau_1} + \abss{\tau_2}\;.
\end{equation} 
It will also be convenient to use graphical notations such as $\RSV:=\I(\Xi)^2$ 
and $\RSWW = \RSV \RSIW$. Unlike in~\eqref{eq:FP_eta}, this last object is 
just a new symbol instead of an undefined product of distributions. 

At this point, we realise that there may be a problem with the requirement that 
the index set $A$ be bounded below, since multiplying symbols of negative degree 
allows to generate terms whose degree is arbitrarily negative. The solution is 
to set $T := \vspan(\FAC)$, where $\FAC$ is the restriction of the set $\cF$ of 
all symbols to the set of those which can be generated by the map 
\begin{equation}
 \tau \mapsto \I(\Xi + \tau^3) + \sum_k \X^k\;,
\end{equation} 
which is an abstract representation of the fixed-point map~\eqref{eq:FP_2d}. 

\begin{table}
\begin{center}
\begin{tabular}{|l|c|l|}
\hline
\hlinespace
$\tau$ & Symbol & $\abss{\tau}$  \\
\hline
\hlinespaceplus
$\Xi$   & $\Xi$   & $-\frac52-\kappa$ \\
\hlinespace
$\I(\Xi)^3$ & \RSW & $-\frac32 -3\kappa$  \\
\hlinespace
$\I(\Xi)^2$ & \RSV & $-1-2\kappa$  \\
\hlinespace
$\I(\I(\Xi)^3)\I(\Xi)^2$ & \RSWW & $-\frac12 -5\kappa$ \\
\hlinespace
$\I(\Xi)$   & \RSI& $-\frac12 -\kappa$ \\
\hlinespace
$\I(\I(\Xi)^3)\I(\Xi)$ & \RSVW  & $\phantom{-}0-4\kappa$ \\
\hlinespace
$\I(\I(\Xi)^2)\I(\Xi)^2$ & \RSWV    & $\phantom{-}0-4\kappa$ \\
\hlinespace 
$\I(\Xi)^2 \X_i$ & \RSV$\X_i$  & $\phantom{-}0-2\kappa$  \\
\hlinespace
$\unit$ & $\unit$ & $\phantom{-}0$   \\
\hlinespace 
$\I(\I(\Xi)^3)$ & \RSIW & $\phantom{-}\frac12 - 3\kappa$ \\
\hlinespace 
$\I(\I(\Xi)^2)\I(\Xi)$ & \RSVV & $\phantom{-}\frac12 - 3\kappa$ \\
\hlinespace 
$\I(\I(\Xi))\I(\Xi)^2$ & \RSWI & $\phantom{-}\frac12 - 3\kappa$ \\
\hlinespace
$\I(\I(\Xi)^2)$ & \RSY & $\phantom{-}1-2\kappa$  \\
\hlinespace
$\I(\I(\Xi))\I(\Xi)$ & \RSVI & $\phantom{-}1-2\kappa$ \\
\hlinespace 
$\X_i$ & $\X_i$ & $\phantom{-}1$  \\
\hlinespace
$\I(\I(\Xi))$ & \RSII & $\phantom{-}\frac32-\kappa$  \\
\hline
\end{tabular}
\end{center}
\vspace{-2mm}
\caption[]{Elements of $\FAC$ of degree up to $\frac32$. Here $\X_i$ denotes 
any element of the form $\X^k$ where $k=e_i$ is a canonical basis vector.}
\label{tab:F_AllenCahn}
\end{table}

When iterating this map, starting with the empty set, one can check that indeed 
all symbols in $\FAC$ have degree at least $\alpha_0$. We have listed all 
symbols of degree up to $\frac32$ in \Cref{tab:F_AllenCahn}. The fact 
that the degree remains bounded below is a consequence of \emph{local 
subcriticality}, also called \emph{super-renormalisability} in Physics, 
cf.~\cite[Assumption 8.3 and Lemma~8.10]{Hairer2014}. To define this property, 
consider the formal Allen--Cahn equation 
\begin{equation}
 \partial_t \phi = \Delta \phi + \phi 
 - \phi^3 + \xi
\end{equation} 
on the $d$-dimensional torus, and perform a scaling 
$\tilde\phi(t,x) = \lambda^\alpha\phi(\lambda^\beta t,\lambda^\gamma x)$, 
ignoring boundary conditions. This yields the rescaled equation
\begin{equation}
 \partial_t \tilde\phi = \lambda^{\beta-2\gamma}\Delta \tilde\phi + 
 \lambda^\beta\tilde\phi - \lambda^{\beta-2\alpha}\tilde\phi^3 + 
 \lambda^{\alpha+\beta}\xi_{\lambda^\beta,\lambda^\gamma}\;.
\end{equation}
The scaling property of space-time white noise shows that $\tilde\xi = 
\lambda^{(\beta+\gamma d)/2} \xi_{\lambda^\beta,\lambda^\gamma}$ has the same 
law as $\xi$  (cf.~\Cref{prop:xi_scaling}). Choosing $\gamma=1$, $\beta=2$ and 
$\alpha=d/2-1$ thus yields 
\begin{equation}
 \partial_t \tilde\phi = \Delta \tilde\phi + 
 \lambda^2\tilde\phi - \lambda^{4-d}\tilde\phi^3 + 
 \tilde\xi\;.
\end{equation}
The equation is called \emph{locally subcritical} if the coefficient of the 
nonlinear term vanishes in the limit $\lambda\to0$, which corresponds to 
zooming in on small scales. This is the case if and only if $d<4$, which is why 
the Allen--Cahn equation can be renormalised in dimensions two and three, but 
not in higher dimension.

\begin{exercise}
Determine the degree of the symbols listed in \Cref{tab:F_AllenCahn} in the 
case of the two-dimensional Allen--Cahn equation, when $\Xi$ has degree 
$\alpha_0 = -\frac32-\kappa$. What happens in the four-dimensional case?
\end{exercise}

%%%%%%%%%%%%%%%%%%%%%%%%%%%%%%%%%%%%%%%%%%%%%%%%%%%%%%%%%%%%%%%%%%%%%%%%%%%%%%%%

\subsection{The structure group}
\label{ssec:3d_structure_group} 

We have already defined the structure group for the polynomial part $\overline 
T = \vspan\setsuch{\X^k}{k\in\N_0^4}$ of the model space as being given by 
all maps $\Gamma_h$ acting as 
\begin{equation}
 \Gamma_h \X^k = (\X-h)^k\;, \qquad h\in\R^4\;.
\end{equation} 
Extending the structure group to the non-polynomial elements is a somewhat 
tricky point, since it has to reflect the way Taylor expansions change under 
translations of the base point. Since space-time white noise is 
translation-invariant, it seems reasonable to set 
\begin{equation}
\label{eq:XI_invariant} 
 \Gamma \Xi = \Xi 
 \qquad \forall \Gamma\in\cG\;.
\end{equation} 
It turns out that the same holds for the element $\RSI$. This is related to the 
fact that the stochastic convolution has negative regularity, so that there is 
no need to subtract any polynomial terms when testing it against a scaled test 
function. We have encountered this situation in~\eqref{eq:proof_Schauder_bound} 
in the proof of Schauder's \Cref{thm:Schauder}. Therefore, we set 
\begin{equation}
\label{eq:IXI_invariant} 
 \Gamma \RSI = \RSI\;, \qquad  
 \Gamma \RSV = \RSV\;, \qquad  
 \Gamma \RSW = \RSW
 \qquad \forall \Gamma\in\cG\;.
\end{equation} 
The situation is different, however, for elements such as $\RSIW$. Indeed, 
$\RSIW$ represents a function with positive regularity, from which one has to 
subtract a term before testing it, as in~\eqref{eq:proof_Schauder_bound2}, a 
procedure known as \emph{recentering}. Therefore, the structure group cannot 
act trivially on such elements. 

To extend the structure group to all elements of $T$, it is helpful to take a 
look at the structure group of $\Tbar$ in a more algebraic way. The dual 
$\Tbarstar$ can be identified with the space 
$\setsuch{\diff{k}}{k\in\N_0^4}$ of differential operators with constant 
coefficients via 
\begin{equation}
\label{eq:def_dual_T} 
 \pscal{\diff{k}}{\X^\ell} 
 := \dpar{\X^\ell}{\X^k} \biggr|_{\X=0}
 = k! \delta_{k\ell}\;.
\end{equation} 
We can define an action $(\diff{\ell},\tau)\mapsto\Gamma_{\diff{\ell}}\tau$ of 
$\Tbarstar$ onto $\Tbar$ by 
\begin{equation}
\label{eq:def_action_Dk} 
 \pscal{\diff{k}}{\Gamma_{\diff{\ell}}\X^m}
 := \pscal{\diff{k}\diff{\ell}}{\X^m}
 = m! \delta_{k+\ell,m}
 \qquad 
 \forall \diff{k}\in\Tbarstar, \forall \X^m\in\Tbar\;,
\end{equation} 
which by~\eqref{eq:def_dual_T} is equivalent to 
\begin{equation}
 \Gamma_{\diff{\ell}}\X^m = \frac{m!}{(m-\ell)!} \X^{m-\ell}\;,
\end{equation} 
with the convention that the right-hand side vanishes unless $\ell_i\leqs m_i$ 
for each $i\in\set{0,1,2,3}$, which we write for short $\ell\leqs m$. 
Note that we have the semigroup property $\Gamma_{\diff{\ell}\diff{m}} = 
\Gamma_{\diff{\ell}}\Gamma_{\diff{m}}$, since for every 
$\smash{\diff{k}\in\Tbarstar}$ and $\X^n\in\Tbar$,
\begin{equation}
 \pscal{\diff{k}}{\Gamma_{\diff{\ell}\diff{m}} \X^n}
 = \pscal{\diff{k}\diff{\ell}\diff{m}}{\X^n}
 = \pscal{\diff{k}\diff{\ell}}{\Gamma_{\diff{m}}\X^n}
 = \pscal{\diff{k}}{\Gamma_{\diff{\ell}}\Gamma_{\diff{m}} \X^n}\;.
\end{equation} 
However, not all $\Gamma_{\diff{k}}$ are invertible, so that the set 
$\setsuch{\Gamma_g}{g\in\Tbarstar}$ does not form a group. We thus define 
$\cG$ as the set of $\smash{g\in\Tbarstar}$ which are \emph{group-like}, 
meaning that they satisfy 
\begin{equation}
\label{eq:group_like} 
 \pscal{g}{\X^\ell\X^m} = \pscal{g}{\X^\ell} \pscal{g}{\X^m}
 \qquad 
 \forall \X^\ell, \X^m \in \Tbar\;.
\end{equation} 
One can show (cf.~\cite[Section~4.3]{Hairer2014}) that $\cG$ is exactly the Lie 
group generated by the Lie algebra of first-order differential operators. 
Indeed, if for $h\in\N_0^4$ we set 
\begin{equation}
\label{eq:def_exph} 
 \pscal{h}{\DD} := \sum_{i=0}^3 h_i \diff{e_i} \in \Tbarstar\;,
 \qquad 
 g = \e^{-\pscal{h}{\DD}} := \sum_{k\in\N_0^4} 
\frac{\bigpar{-\pscal{h}{\DD}}^k}{k!}\;, 
\end{equation} 
then we have 
\begin{equation}
 \Gamma_g(\X^\ell) 
 = \sum_{k\in\N_0^4} \frac{(-h)^k}{k!} \Gamma_{\diff{k}}(\X^\ell)
 = \sum_{k\leqs \ell} \frac{(-h)^k}{k!} 
\frac{\ell!}{(\ell-k)!} \X^{\ell-k}
 = \bigpar{\X-h}^\ell\;,
\end{equation} 
showing that the group action $\Gamma$ of $\cG$ onto $\Tbar$ is 
indeed equivalent to the action of the structure group of the polynomial 
regularity structure. This is related to the fact that first-order derivatives 
generate the group of translations. 

\begin{exercise}
Prove that the element $g$ defined by~\eqref{eq:def_exph} is group-like in the 
sense of~\eqref{eq:group_like} in the case where the indices $k$ belong to 
$\N_0$ instead of $\N_0^4$. 
\end{exercise}

A useful way of encoding the Leibniz rule is to define a bilinear operator 
$\Delta^+: \Tbar \to \Tbar \otimes \Tbar$ by 
\begin{equation}
\label{eq:coproduct_X} 
 \Delta^+(\unit) := \unit\otimes\unit\;, 
 \qquad 
 \Delta^+(\X_i) := \X_i\otimes\unit + \unit\otimes\X_i\;,
\end{equation} 
which is then extended to all of $\Tbar$ by requiring that 
\begin{equation}
\label{eq:coproduct_mult} 
 \Delta^+(\X^k\X^\ell) := \Delta^+(\X^k)\Delta^+(\X^\ell) 
 \qquad 
 \forall \X^k, \X^\ell \in \Tbar\;.
\end{equation} 
The operator $\Delta^+$ is called a \emph{coproduct}. 
One can indeed check that $\Delta^+$ is \emph{coassociative}, meaning that 
\begin{equation}
 (\Id\otimes\Delta^+)\Delta^+\X^k 
 = (\Delta^+\otimes\Id)\Delta^+\X^k
 \qquad \forall \X^k\in\Tbar\;,
\end{equation} 
which provides $\Tbar$ with a \emph{Hopf algebra} structure\footnote{In 
addition to a product and coproduct and associated neutral elements, a Hopf 
algebra is characterised by a linear map $\cA:\Tbar\to\Tbar$ called the 
\emph{antipode}. It should satisfy the identity $M(\cA\otimes\Id)\Delta^+ = 
M(\Id\otimes\cA)\Delta^+$, where $M$ is the multiplication map defined by 
$M(\tau_1\otimes\tau_2) = \tau_1\tau_2$. In this setting, it is given by 
$\smash{\cA(\X^k) = (-1)^k\X^k}$.}. Furthermore, we have the duality property 
\begin{equation}
\label{eq:duality_Delta} 
 \pscal{\diff{k}\diff{\ell}}{\X^m}
 = \pscal{\diff{k}\otimes\diff{\ell}}{\Delta^+\X^m}\;, 
\end{equation} 
where by definition 
\begin{equation}
 \pscal{\diff{k}\otimes\diff{\ell}}{Y^{(1)}\otimes Y^{(2)}}
 := \pscal{\diff{k}}{Y^{(1)}} \pscal{\diff{\ell}}{Y^{(2)}}\;.
\end{equation} 
Note that in general, $\Delta^+\X^m$ is a sum of several terms 
$Y_i^{(1)}\otimes Y_i^{(2)}$, but it is convenient to use \emph{Sweedler's 
notation} suppressing the sum. The point is that by 
combining~\eqref{eq:duality_Delta} with~\eqref{eq:def_action_Dk}, one obtains 
the relation 
\begin{equation}
 \Gamma_{\diff{\ell}} \X^m = (\Id\otimes\diff{\ell}) \Delta^+ \X^m\;,
\end{equation} 
where by definition, $(\Id\otimes\diff{\ell})Y^{(1)}\otimes Y^{(2)} 
:= Y^{(1)} \pscal{\diff{\ell}}{Y^{(2)}}$.

The advantage of this rather formal approach, which may seem like overkill in 
the polynomial case, is that it allows to define the structure group in the 
general case. Define a set of formal expressions
\begin{equation}
\label{eq:def_FACplus} 
 \FAC^+ = \Biggsetsuch{\X^k \prod_j \J_{k_i}\tau_j}{k \in\N_0^4,\; 
\tau_j\in\FAC, \;\abss{\tau_j}+2-\abss{k_j} > 0}\;.
\end{equation} 
Here the $\J_k\tau$ are new symbols of degree $\abss{\tau}+2-\abss{k}$, 
which is strictly positive by assumption. Their meaning will become clearer 
later on. By convention, $\smash{\FAC^+}$ also contains the unit $\unit$ and 
all monomials $\X^k$, corresponding to an empty product. 

\begin{table}
\begin{center}
\begin{tabular}{|l|l|l|}
\hline
\hlinespace
$\tau$ & $\Delta^+(\tau)$ & $\Gamma_g\tau$  \\
\hline
\hlinespaceplus
\RSWW & $\RSWW\otimes\unit + \RSV\otimes\J_0\RSW$ & 
$\RSWW + \RSV \, \pscal{g}{\J_0\RSW}$ \\
\hlinespace
\RSVW  & $\RSVW\otimes\unit + \RSI\otimes\J_0\RSW$ &  
$\RSVW + \RSI \, \pscal{g}{\J_0\RSW}$ \\
\hlinespace
\RSWV  & $\RSWV\otimes\unit + \RSV\otimes\J_0\RSV$ &  
$\RSWV + \RSV \, \pscal{g}{\J_0\RSV}$ \\
\hlinespace 
\RSV $\X_i$ & $\RSV \X_i\otimes\unit + \RSV\otimes\X_i$ &  
$\RSV \X_i + \RSV \, \pscal{g}{\X_i}$ \\
\hlinespace 
\RSIW & $\RSIW\otimes\unit + \unit\otimes\J_0\RSW$ &  
$\RSIW + \unit \, \pscal{g}{\J_0\RSW}$ \\
\hlinespace 
\RSVV & $\RSVV\otimes\unit + \unit\otimes\J_0\RSV$ &  
$\RSVV + \RSI \, \pscal{g}{\J_0\RSV}$ \\
\hlinespace 
\RSWI & $\RSWI\otimes\unit + \RSV\otimes\J_0\RSI + \RSV\X_i\otimes\J_i\RSI$ &  
$\RSWI + \RSV \, \pscal{g}{\J_0\RSI} + \RSV\X_i \, \pscal{g}{\J_i\RSI}$ \\
\hlinespace 
 & $\phantom{\RSWI\otimes\unit} {}+{}\RSV\otimes\X_i\J_i\RSI$ &  
$\phantom{\RSWI} {}+{} \RSV \, \pscal{g}{\X_i\J_i\RSI}$ \\
\hlinespace
\RSY & $\RSY\otimes\unit + \unit\otimes\J_0\RSV$ &  
$\RSY + \unit \, \pscal{g}{\J_0\RSV}$ \\
\hlinespace
\RSVI & $\RSVI\otimes\unit + \RSI\otimes\J_0\RSI + \RSI\X_i\otimes\J_i\RSI + 
\RSI\otimes\X_i\J_i\RSI$ &  
$\RSVI + \RSI \, \pscal{g}{\J_0\RSI} + \RSI\X_i \, \pscal{g}{\J_i\RSI} + 
\RSI \, \pscal{g}{\X_i\J_i\RSI}$ \\
\hlinespace
$\X_i$ & $\X_i\otimes\unit + \unit\otimes\X_i$ & 
$\X_i + \unit\,\pscal{g}{\X_i}$ \\
\hlinespace
\RSII & $\RSII\otimes\unit + \unit\otimes\J_0\RSI + \X_i\otimes\J_i\RSI + 
\unit\otimes\X_i\J_i\RSI$ &  
$\RSII + \unit \, \pscal{g}{\J_0\RSI} + \X_i \, \pscal{g}{\J_i\RSI} + 
\unit \, \pscal{g}{\X_i\J_i\RSI}$ \\
\hline
\end{tabular}
\end{center}
\vspace{-2mm}
\caption[]{Elements of $\FAC$ of degree up to $\frac32$ on which the structure 
group has a non-trivial effect and their coproducts. Repeated indices $i$ imply 
summation over all values of $i$, and we have written $\J_i$ instead of 
$\J_{e_i}$.}
\label{tab:Gamma_AllenCahn}
\end{table}

The coproduct $\Delta^+$ is extended from $\Tbar$ to $T=\vspan(\FAC)$ by 
requiring that, in addition to~\eqref{eq:coproduct_X} and multiplicativity as 
in~\eqref{eq:coproduct_mult}, it satisfies 
\begin{align}
\label{eq:Delta_XI} 
\Delta^+(\Xi) &:= \Xi\otimes\unit\;, \\
\Delta^+(\I\tau) &:= (\I\otimes \Id)\Delta^+(\tau) 
+ \sum_{\ell, m} \frac{\X^\ell}{\ell!} \otimes \frac{\X^m}{m!} 
\J_{\ell+m}\tau\;.
\label{eq:Delta_Itau} 
\end{align}
The sum in the last equation is necessarily finite, because of the limitation 
on the degree in~\eqref{eq:def_FACplus}. Denoting by $T_+^*$ the set of 
linear maps $\tau\mapsto\pscal{h}{\tau}$ from $T_+=\vspan(\FAC^+)$ to $\R$, the 
structure group $\cG$ is given by the set of group-like elements in $T^*$, 
acting  on $T$ as 
\begin{equation}
\label{eq:def_Gamma_g} 
 (g,\tau) \mapsto \Gamma_g\tau := (\Id\otimes g)\Delta^+\tau\;.
\end{equation} 
Note that since $\pscal{g}{\unit}=1$ for any group-like $g$, 
\eqref{eq:Delta_XI} yields $\Gamma_g(\Xi) = \Xi$, so that this definition is 
compatible with the invariance of $\Xi$ required in~\eqref{eq:XI_invariant}. 
Furthermore, the invariance of $\RSI$, $\RSV$ and $\RSW$ required 
in~\eqref{eq:IXI_invariant} is satisfied, because 
\begin{equation}
 \Delta^+(\RSI) = \RSI\otimes\unit\;, \qquad 
 \Delta^+(\RSV) = \RSV\otimes\unit\;, \qquad 
 \Delta^+(\RSW) = \RSW\otimes\unit\;. 
\end{equation} 
Indeed, since $\RSI$ has negative degree, the sum in~\eqref{eq:Delta_Itau} is 
empty, and the other two identities follow from the product 
rule~\eqref{eq:coproduct_mult}.
A first non-trivial case is 
\begin{equation}
 \Delta^+(\RSIW) = (\I\otimes\Id)(\Delta^+(\RSW)) + \unit\otimes\J_0\RSW 
 = \RSIW \otimes \unit + \unit \otimes \J_0\RSW\;.
\end{equation}
Indeed, since $\abss{\RSW}+2 = \frac12 - 3\kappa$, the term $\ell=m=0$ is 
permitted in~\eqref{eq:Delta_Itau}. Substituting in~\eqref{eq:def_Gamma_g} 
yields 
\begin{equation}
\label{eq:Gamma_RSIW} 
 \Gamma_g \RSIW = \RSIW +  \unit\,\pscal{g}{\J_0\RSW}\;.
\end{equation} 
The additional term $ \unit\,\pscal{g}{\J_0\RSW}$ will be the one responsible 
for recentering. More examples of nontrivial structure group actions are listed 
in  \Cref{tab:Gamma_AllenCahn}. Note that the indempotency 
requirement~\eqref{eq:structure_group} is indeed always satisfied. 

\begin{exercise}
Check the expressions for $\Delta^+\tau$ and $\Gamma_g\tau$ listed in 
\Cref{tab:Gamma_AllenCahn}. What is the matrix representing $\Gamma_g$ in a 
basis given by the elements in \Cref{tab:F_AllenCahn}?
\end{exercise}

\begin{remark}
If $\FAC^+$ were equal to $\FAC$, the coproduct $\Delta^+$ would again endow 
$T$ with a Hopf-algebra structure. Since these sets are different, $T$ is now a 
comodule over $T_+ = \vspan(\FAC^+)$. However, $T_+$ can also be endowed with a 
Hopf algebra structure --- see~\cite[Section~8.1]{Hairer2014}. Strictly 
speaking, since $\FAC$ and $\FAC^+$ are not stable under multiplication, the 
comodule should rather be defined on $\vspan(\cF)$ and its analogue 
$\vspan(\cF^+)$, but this has no influence on 
the theory. 
\end{remark}

%%%%%%%%%%%%%%%%%%%%%%%%%%%%%%%%%%%%%%%%%%%%%%%%%%%%%%%%%%%%%%%%%%%%%%%%%%%%%%%%

\subsection{The renormalisation group}
\label{ssec:3d_renorm_group} 

The renormalisation group also acts on the model space $T$, but has a different 
purpose than the structure group. In the two-dimensional case, we have seen 
that Wick renormalisation, as appearing in~\eqref{eq:Wick_powers_limit}, allows 
to convert divergent stochastic integrals into convergent ones by subtracting 
appropriate terms. Here we will perform the same subtractions, but as we have 
seen in \Cref{sec:3d_Gibbs} and in~\eqref{eq:AC-3d-renorm}, it is necessary to 
introduce a second renormalisation constant. 

The second renormalisation constant $C_\delta^{(2)}$ is in fact related to the 
symbol $\RSWV$. This is because similarly to \eqref{eq:stoch_conv_integral}, 
this symbol represents the stochastic integral 
\begin{equation}
\label{eq:stoch_int_RSWV}
\begin{split}
 \Bigpar{P*(P_\delta*\xi)^2 (P_\delta*\xi)^2} (z)
 := \idotsint P(z-z_0) P_\delta(z_0-z_1) \xi(\6z_1)  P_\delta(z_0-z_2) 
\xi(\6z_2) \\
 \times P_\delta(z-z_3) \xi(\6z_3) P_\delta(z-z_4) \xi(\6z_4) \6z_0\;.
\end{split}
\end{equation} 
By Isserlis' theorem, we formally have 
\begin{align}
 \bigexpec{\xi(\6z_1)\xi(\6z_2)\xi(\6z_3)\xi(\6z_4)}
 ={}& \delta(z_1-z_2)\delta(z_3-z_4)\6z_1\6z_3 \\ 
 &{}+ \delta(z_1-z_3)\delta(z_2-z_4)\6z_1\6z_2 \\
 &{}+ \delta(z_1-z_4)\delta(z_2-z_3)\6z_1\6z_2\;.
\end{align} 
The expectation of~\eqref{eq:stoch_int_RSWV} is thus given by the sum of three 
integrals. The first one actually vanishes for symmetry reasons, while the last 
two integrals are equal, and given by 
\begin{equation}
\label{eq:Pz1z2z3} 
 \iiint P(z-z_0) 
P_\delta(z_0-z_1)P_\delta(z-z_1)P_\delta(z_0-z_2)P_\delta(z-z_2) 
\6z_0\6z_1\6z_2\;.
\end{equation} 
We can then use the identities~\eqref{eq:Pdelta_convolution} from the proof of 
\Cref{prop:asymp_Cdelta} to obtain 
\begin{align}
\int P_\delta(z_0-z_1)P_\delta(z-z_1)\6z_1 
&= \int_{-\infty}^\infty\int_\Lambda 
P_\delta(t_0-t_1,x_0-x_1)P_\delta(t-t_1,x-x_1)\6x_1\6t_1 \\
&= \int_{-\infty}^\infty \tilde P_\delta(t_0+t-2t_1,x_0-x)\6t_1 
= -\frac12 \tilde G_\delta(x_0-x)\;.
\end{align}
Substituting in~\eqref{eq:Pz1z2z3} and changing variables from $x-x_0$ to 
$-x_0$, we thus obtain 
\begin{equation}
\label{eq:def_Cdelta2} 
 C_\delta^{(2)} = \Bigexpec{\RSWV_\delta} 
 = 2 \biggpar{-\frac12}^2 \int_{-\infty}^\infty\int_\Lambda P(t-t_0,-x_0) 
\tilde G_\delta(-x_0)^2\6x_0\6t_0 
 = \frac12 \int_\Lambda \tilde G_\delta(x_0)^3 \6x_0 
 = \frac12 \FDCthree{vedge}{vedge}{vedge}\;.
\end{equation} 
Note that graphically, this amounts to pairing leaves of $\RSWV$, which 
represent the noise. 

The renormalisation group will primarily concern the elements of $T$ having 
negative degree, namely those in the set
\begin{equation}
 \cF_- = \bigset{\RSI, \RSV, \RSW, \RSWV, \RSWW, \RSVW}\;.
\end{equation} 
Formally, an element $M^\delta$ of the renormalisation group is defined as 
\begin{equation}
 M^\delta := \exp\Bigset{-C_\delta^{(1)}L_1 - C_\delta^{(2)}L_2}\;, 
\end{equation} 
where $L_1$ and $L_2$ represent the substitutions 
\begin{equation}
 L_1: \RSV \mapsto \unit\;, 
 \qquad 
 L_2: \RSWV \mapsto \unit\;.
\end{equation} 
The understanding is that each substitution is applied as often as possible, so 
that $L_1(\RSW) = 3\RSI$, because there are three ways to extract $\RSV$ from 
$\RSW$, each one leaving exactly $\RSI$, which is exactly what happens 
in Wick renormalisation, cf.~\eqref{eq:Wick_powers_limit}.
Applying $M^\delta$ to the elements of $\cF_-$, we get 
% $M^\delta(\RSI) = \RSI$ and 
\begin{align}
M^\delta(\RSI) &= \RSI \\
M^\delta(\RSV) &= \RSV - C_\delta^{(1)}\unit\;, \\
M^\delta(\RSW) &= \RSW - 3C_\delta^{(1)}\RSI\;, \\
M^\delta(\RSWV) &= \RSWV -C_\delta^{(1)}\RSY -C_\delta^{(2)}\unit\;, \\
M^\delta(\RSWW) &= \RSWW - 3C_\delta^{(1)}\RSWI - C_\delta^{(1)}\RSIW + 
3\Bigpar{C_\delta^{(1)}}^2\RSII - 3C_\delta^{(2)}\RSI\;, \\
M^\delta(\RSVW) &= \RSVW - 3C_\delta^{(1)}\RSVI\;.
% M(\RSV X_i) &= \RSV X_i - C_1 X_i\;.
\label{eq:M-AC} 
\end{align}
Observe that now, the idempotency results from the fact that all additional 
terms have a \emph{higher} degree than the original term. 

\enlargethispage{\baselineskip}

\begin{remark}
The renormalisation group can also be introduced in a more abstract way, via a 
coproduct $\Delta^-$ and an associated Hopf algebra. See in 
particular~\cite{BrunedHairerZambotti,ChandraHairer16}. In some sense, this 
indicates that recentering and renormalisation are \lq\lq dual\rq\rq\ 
procedures. 
\end{remark}

%%%%%%%%%%%%%%%%%%%%%%%%%%%%%%%%%%%%%%%%%%%%%%%%%%%%%%%%%%%%%%%%%%%%%%%%%%%%%%%%

\section{Regularity structures: analytic aspects}
\label{sec:3d_regstruc_analytic} 

Now that the abstract algebraic framework is in place, we have to complete it 
with analytic objects, describing the different maps in 
the commutative diagram ~\eqref{eq:RS_diagram}. 

%%%%%%%%%%%%%%%%%%%%%%%%%%%%%%%%%%%%%%%%%%%%%%%%%%%%%%%%%%%%%%%%%%%%%%%%%%%%%%%%

\subsection{Models}
\label{ssec:3d_models} 

We describe in this section the map denoted $\Psi$ in the 
diagram~\eqref{eq:RS_diagram}, which associates to a realisation of (mollified) 
space-time white noise a collection $Z^\delta=(\Pi^\delta,\Gamma^\delta)$ 
of objects called a \emph{model}. The following definition is 
\cite[Definition~2.17]{Hairer2014} particularised to the parabolic scaling 
$\fraks=(2,1,1,1)$, where we denote space-time points by 
$z=(t,x)\in\R\times\R^3\simeq\R^4$.  

\begin{definition}[{Model}]
\label{def:model} 
A \defwd{model} for a regularity structure $(A,T,\cG)$ is
a pair $Z=(\Pi,\Gamma)$, defined by a collection
$\set{\Pi_z:T\to\cS'(\R^4)}_{z\in\R^4}$ of continuous linear maps and a
map $\Gamma:\R^4\times\R^4\to \cG$ with the following properties. 
\begin{enumerate}
\item 	$\Gamma_{zz} = \Id$ is the identity of
$\cG$ and
$\Gamma_{zz'}\Gamma_{z'z''}=\Gamma_{zz''}$ for
all $z,z',z''\in\R^4$. 

\item 	$\Pi_{z'}=\Pi_z \Gamma_{zz'}$ for all $z,z'\in\R^4$. 

\item 	Let $r = \intpartplus{-\inf A}$.
For any $\gamma\in\R$ and any compact set
$\fraK\subset{\R^4}$, one has 
\begin{equation}
 \label{eq:def_normPi}
 \norm{\Pi}_{\gamma;\fraK} := 
 \sup_{z\in\fraK}~ \sup_{\alpha<\gamma} ~\sup_{\tau\in T_\alpha}~
\sup_{\ph\in B_r}~\sup_{0<\lambda\leqs1}
 \frac{\bigabs{\pscal{\Pi_z\tau}{\cS^\lambda_{z}\ph}}}{\lambda^{
\alpha} \norm{\tau}_\alpha} 
 < \infty\;.
\end{equation}

\item 	For any $\gamma\in\R$ and any compact set
$\fraK\subset{\R^4}$, one has 
\begin{equation}
 \label{eq:def_normGamma}
 \norm{\Gamma}_{\gamma;\fraK} := 
 \sup_{z,z'\in\fraK} ~\sup_{\alpha<\gamma}~
\sup_{\beta<\alpha}~\sup_{\tau\in T_\alpha}
 \frac{\norm{\Gamma_{zz'}\tau}_\beta}
 {\norm{\tau}_\alpha\,\norm{z-z'}^{\alpha-\beta}_{\fraks}} < \infty\;.
\end{equation} 
\end{enumerate}
\end{definition}

\goodbreak

We recall that the set $B_r$ occurring in~\eqref{eq:def_normPi} is the set of 
smooth test functions supported in the unit $\norm{\cdot}_\fraks$-ball and of 
unit $\cC^r$-norm, already encountered in \Cref{def:Holder-negative}. The norm 
$\norm{\tau}_\alpha$ in~\eqref{eq:def_normPi} and~\eqref{eq:def_normGamma} is 
any norm on the finite-dimensional vector space $T_\alpha$. 

The requirement~\eqref{eq:def_normPi} essentially asks that $\Pi_z \tau$ should 
belong to $\cC^\alpha_\fraks$ whenever $\tau\in T_\alpha$. There is a slight 
difference, however, since the bound is only required to hold when zooming in 
at the particular point $z$ where the model is evaluated. 

An important feature of the theory is that the definition allows for different 
models for a given regularity structure. This will become important when 
considering renormalisation. A particular role is played by the \emph{canonical 
model} $Z^\delta=(\Pi^\delta,\Gamma^\delta)$ for mollified noise $\xi^\delta = 
\varrho^\delta*\xi$, which is defined by 
\begin{align}
(\Pi^\delta_z\Xi)(\bar z) &:= \xi^\delta(\bar z)\;, \\
(\Pi^\delta_z \X^k)(\bar z) &:= (\bar z-z)^k
&&\forall k\in\N_0^4\;, \\
(\Pi^\delta_z \tau_1\tau_2)(\bar z) &:= (\Pi^\delta_z \tau_1)(\bar 
z)(\Pi^\delta_z \tau_2)(\bar z)
&&\forall\tau_1,\tau_2\in T\;. 
\label{eq:model_product} 
\end{align}
Note that for polynomial symbols $\X^k$, this definition is exactly as 
in~\eqref{eq:Pi_poly}. 

The tricky part is again to extend this to elements of the form $\I(\tau)$. 
This is done in~\cite[Section~5]{Hairer2014} via a decomposition 
\begin{equation}
\label{eq:decomp_P} 
 P(t,x) = K(t,x) + R(t,x)
\end{equation} 
of the heat kernel $P$ into a smooth part $R$, and a singular part $K$ with 
special algebraic properties. These properties are as follows:
\begin{itemize}
\item 	$K$ is supported in the set $\set{\abs{x}^2 + \abs{t} \leqs 1}$ where 
$\abs{x}=\sum_{j=1}^3\abs{x_j}$;
\item 	$K(t,x)=0$ for $t\leqs0$ and $K(t,-x)=K(t,x)$ for all $(t,x)$;
\item we have	
\begin{equation}
 \label{eq:K_heat}
 K(t,x) = \frac{1}{\abs{4\pi t}^{3/2}} \e^{-\abs{x}^2/(4t)}
 \qquad
 \text{for $\abs{x}^2+\abs{t} \leqs \frac12$}
\end{equation}
and $K(t,x)$ is smooth for $\abs{x}^2+\abs{t} > \frac12$;
\item finally, the integral of $K$ against any polynomial of parabolic 
degree less or equal some fixed $\zeta\geqs2$ vanishes; this last condition is 
for compatibility with the non-redundancy condition $\I(\X^k)=0$. 
\end{itemize}
See~\cite[Lemma~5.5]{Hairer2014} for a proof that such a decomposition is 
indeed possible. This lemma also shows that $K$ can be decomposed, as 
in~\eqref{eq:P_sum}, as a sum of kernels $K_n$, where each $K_n$ is supported 
in a ball of radius $2^{-n}$, and has $k$th derivatives bounded by 
$C2^{(1+\abss{k})n}$, as in~\eqref{eq:Pn_norm}. 
Then the construction of the canonical model is completed by requiring that 
\begin{equation}
 \label{eq:Pi_xI} 
(\Pi^\delta_z \I\tau)(\bar z) := 
\pscal{\Pi^\delta_z \tau}{K(\bar z-\cdot)}
+ \sum_\ell \frac{(\bar z-z)^\ell}{\ell!} \pscal{f^\delta_z}{\J_\ell\tau}
\qquad\forall\tau\in T\;,
\end{equation} 
where the $f^\delta_z\in T_+^*$ are linear forms defined in~\eqref{eq:f_x} 
below. This definition makes sense, even if $\Pi^\delta_z \tau$ is replaced by a 
genuine distribution $\Pi_z \tau$, by interpreting~\eqref{eq:Pi_xI} as a sum 
over all $K_n$, each of which is tested against $\Pi_z \tau$. 

As suggested by the construction of the structure group in 
\Cref{ssec:3d_structure_group}, the linear forms $f^\delta_z$ will be used to 
construct the elements $\Gamma^\delta_{zz'}$ of the model. They are defined by  
\begin{align}
\pscal{f^\delta_z}{\unit} &:= 1\;, \\
\pscal{f^\delta_z}{\X_i} &:= -z_i\;, \\
\pscal{f^\delta_z}{\tau_1\tau_2} &:= \pscal{f^\delta_z}{\tau_1} 
\pscal{f^\delta_z}{\bar\tau_2} 
&&\forall\tau_1,\tau_2\in T\;, \\
\pscal{f^\delta_z}{\J_\ell\tau} &:= 
- \pscal{\Pi^\delta_z\tau}{\DD^\ell K(z-\cdot)}
&&\forall\tau\in T\;.
\label{eq:f_x} 
\end{align}
The maps $\Gamma^\delta_{zz'}$ of the canonical model are then defined, as 
in~\eqref{eq:def_Gamma_g}, by
\begin{equation}
 \label{eq:Gamma_F}
 \Gamma^\delta_{zz'} := (F^\delta_z)^{-1} F^\delta_{z'}
 \qquad 
 \text{where 
 $F^\delta_z := \Gamma_{f^\delta_z} = (\Id\otimes f^\delta_z)\Delta^+$\;.}
\end{equation} 
Note that the first property of \Cref{def:model} is automatically satisfied. 
To check the second property, let $\gamma_{zz'}$ denote the element of the 
structure group such that $\Gamma_{zz'} = \Gamma_{\gamma_{zz'}}$. 
Writing $\Delta^+\tau = \smash{\tau^{(1)}\otimes\tau^{(2)}}$ in Sweedler's 
notation, we have by~\eqref{eq:def_Gamma_g}
\begin{equation}
 \label{eq:Gamma_F01}
 \Pi^\delta_z \Gamma^\delta_{zz'}\tau = \Pi^\delta_z 
(\Id\otimes\gamma^\delta_{zz'})\Delta^+\tau
 = \Pi^\delta_z\tau^{(1)} \pscal{\gamma^\delta_{zz'}}{\tau^{(2)}}\;.
\end{equation} 
The second property of \Cref{def:model} thus amounts to the relation
\begin{equation}
 \label{eq:Gamma_F02}
 \Pi^\delta_{z'}\tau = \Pi^\delta_z\tau^{(1)} \pscal{\gamma^\delta_{z\bar 
z}}{\tau^{(2)}}\;,
\end{equation} 
which provides some intuition for the meaning of $\Delta^+$. The fact that
$(\Pi^\delta,\Gamma^\delta)$ is indeed a model is proved 
in~\cite[Proposition~8.27]{Hairer2014}. 

\begin{exercise}\hfill
\label{exo:model} 
\begin{itemize}
\item 	Check that in the case $\tau=\X_i$, the definition~\eqref{eq:f_x} is 
compatible with the expressions~\eqref{eq:Gamma_poly} of the $\Gamma_{xy}$ 
introduced in the one-dimensional polynomial case. 

\item 	Compute $(\Pi^\delta_z \RSIW)(\bar z)$. 
Compare with~\eqref{eq:proof_Schauder_bound2} and interpret the result. 

\item 	Compute $F^\delta_z \RSIW$. Representing this in the basis 
$(\unit,\RSIW)$, compute its inverse, and determine a function $\chi^\delta$ 
such that 
\begin{equation}
 \Gamma^\delta_{zz'} \RSIW
 = \RSIW + \bigbrak{\chi^\delta(z) - \chi^\delta(z')}\unit\;.
\end{equation} 
What is the associated expression for $\gamma^\delta_{zz'}$? Compare 
with~\eqref{eq:Gamma_RSIW}. 
\qedhere
\end{itemize}
\end{exercise}

\begin{remark}
The only rule that we are going to bend for general models is the product 
rule~\eqref{eq:model_product}. For instance, in the case of Wick 
renormalisation, the model for $\RSV$ is different from the product of two 
models for $\RSI$, owing to the subtraction a counterterm. 
\end{remark}

%%%%%%%%%%%%%%%%%%%%%%%%%%%%%%%%%%%%%%%%%%%%%%%%%%%%%%%%%%%%%%%%%%%%%%%%%%%%%%%%

\subsection{Modelled distributions}
\label{ssec:3d_modeled_distributions} 

In this section, we describe the space in which the fixed-point equation 
defining the solution map $\cS$ lives, which appears in the upper part of the 
commutative diagram~\eqref{eq:RS_diagram}. The following definition 
is a particularisation of~\cite[Definition~3.1]{Hairer2014}. 

\begin{definition}[Modelled distributions]
\label{def:D_gamma}
Let $\gamma\in\R$. 
The space of modelled distributions $\cD^\gamma$ consists of all function 
$f:\R^4\to T_\gamma^-$ such that for every compact set 
$\fraK\subset\R^4$ one has 
\begin{equation}
 \label{eq:D_gamma}
 \normDgamma{f}_{\gamma;\fraK} := 
 \sup_{z\in\fraK}~ \sup_{\beta < \gamma} \norm{f(z)}_\beta 
 + \sup_{\substack{z,z'\in\fraK\\ 
 \norm{z-z'}_{\fraks}\leqs1}}
 \sup_{\beta<\gamma}\frac{\norm{f(z)-\Gamma_{zz'}f(z')}_\beta}
 {\norm{z-z'}^{\gamma-\beta}_{\fraks}} < \infty\;.
\end{equation}
\end{definition}

Here $\norm{f(z)}_\beta$ denotes the norm of the projection of $f(z)$ on 
$T_\beta$. A modelled distribution can thus be written as a Taylor-series-like 
expression 
\begin{equation}
\label{eq:f_expansion} 
 f(z) = \sum_{\alpha < \gamma} \sum_{\tau\in T_\alpha} f_\tau(z) \tau\;,
\end{equation} 
where the projection of $f(z)-\Gamma_{zz'}f(z')$ on $T_\beta$ satisfies a 
H\"older condition of exponent $\gamma - \beta$. This is indeed a 
generalisation of~\eqref{eq:f_poly}, while the condition~\eqref{eq:D_gamma} is 
of the same form as the condition~\eqref{eq:Dalpha_poly} that ensured that $f$ 
belong to a given H\"older space. One important difference between expansions of 
the form~\eqref{eq:f_expansion} and usual Taylor expansions is that the \lq\lq 
basis vectors\rq\rq\ $\tau$ in~\eqref{eq:f_expansion} can be associated with 
very irregular distributions, while Taylor expansions use monomials which are 
analytic functions. One should thus distinguish between the \emph{regularity} 
$\alpha_0=\inf A$ of the expansion~\eqref{eq:f_expansion}, and the H\"older 
exponent $\gamma$, which in applications will always be larger than $\alpha_0$ 
(and, in fact, positive). 

It may happen that the smallest value of $\alpha$ occurring in the 
expansion~\eqref{eq:f_expansion} is some degree $\alpha_1 > \alpha_0$. In that 
case, we say that $f$ has regularity $\alpha_1$ (or that it is defined on a 
\emph{sector} $V\subset T$ of regularity $\alpha_1$). This is important, in 
particular, for the following result, which corresponds 
to~\cite[Theorem~4.7]{Hairer2014}. 

\begin{theorem}[Multiplication of modelled distributions]
\label{thm:mult_Dgamma} 
Let $f_1\in\cD^{\gamma_1}$ and $f_2\in\cD^{\gamma_2}$ be modelled distributions 
of respective regularity $\alpha_1\leqs0$ and $\alpha_2\leqs0$, and let 
$\gamma=(\gamma_1+\alpha_1)\wedge(\gamma_2+\alpha_2)$. Then one can define a 
modelled distribution $f_1 f_2 \in \cD^{\gamma}$ of regularity 
$\alpha_1+\alpha_2$, which satisfies 
\begin{equation}
 \normDgamma{f_1 f_2}_{\gamma;\fraK}
 \lesssim \normDgamma{f_1}_{\gamma_1;\fraK}
 \normDgamma{f_2}_{\gamma_2;\fraK}
 \Bigpar{1+\norm{\Gamma_{\gamma_1+\gamma_2;\fraK}}}^2\;.
\end{equation} 
for every compact set $\fraK\subset\R^4$. 
The product $f_1f_2$ is obtained by multiplying formally the expansions of 
$f_1$ and $f_2$ and truncating them to order $\gamma$. 
\end{theorem}

Note that in \Cref{def:D_gamma}, $\cD^\gamma$ depends on the $\Gamma$ of the 
chosen model, so that it should be written $\cD^\gamma(\Gamma)$ in case one 
works with several models, as will be the case when dealing with 
renormalisation. In order to compare elements $f\in\cD^\gamma(\Gamma)$ and 
$\bar f\in\cD^\gamma(\Gammabar)$, we introduce the quantities 
\begin{align}
 \label{eq:D_gamma_normf}
 \norm{f-\bar f}_{\gamma;\fraK} 
 &:= \sup_{z\in\fraK}~ \sup_{\beta < \gamma}
 ~\norm{f(z) - \bar f(z)}_\beta\;, \\
 \label{eq:D_gamma_seminorm}
 \seminormff{f}{\bar f}_{\gamma;\fraK}
 &:= \norm{f-\bar f}_{\gamma;\fraK} 
 + \sup_{\substack{z,z'\in\fraK \\ \norm{z-z'}_{\fraks}\leqs1}}
 \sup_{\beta < \gamma}
 \frac{\norm{f(z) - \bar f(z) - \Gamma_{zz'}f(z') +
\Gammabar_{zz'}\bar f(z')}_\beta}{\norm{z-z'}_{\fraks}^{\gamma-\beta}}\;.
\end{align}
The reason for the semicolon in the last expression is that it is not a 
function of $f-\bar f$, and hence is not symmetric under exchange of $f$ and 
$\bar f$. The quantity $\seminormff{\cdot}{\cdot}_{\gamma;\fraK}$ is thus 
not strictly speaking a norm, though it has similar properties for all 
practical purposes.  

%%%%%%%%%%%%%%%%%%%%%%%%%%%%%%%%%%%%%%%%%%%%%%%%%%%%%%%%%%%%%%%%%%%%%%%%%%%%%%%%

\subsection{The reconstruction theorem}
\label{ssec:3d_reconstruction} 

A key result of the theory of regularity structures is the reconstruction 
theorem, which states the existence of a map $\cR$ from a space $\cD^\gamma$ of 
modelled distributions to a \lq\lq classical\rq\rq\ H\"older space. This map 
appears on the right-hand side of the commutative 
diagram~\eqref{eq:RS_diagram}. 
The following result is part of~\cite[Theorem~3.10]{Hairer2014}. 

\begin{theorem}[Reconstruction theorem]
\label{thm:reconstruction}
Let $(A,T,\cG)$ be a regularity structure, let $\alpha_0=\inf A$, and  
fix $r>\abs{\alpha_0}$ as well as a model $(\Pi,\Gamma)$. Then for every 
$\gamma\in\R$, there exists a continuous map 
$\cR:\cD^\gamma\to\cC^{\alpha_0}_\fraks$ such that for every compact set 
$\fraK\subset\R^4$,  
\begin{equation}
\label{eq:reconstruction} 
 \bigabs{\pscal{\cR f-\Pi_zf(z)}{\cS^\lambda_z\ph}} 
 \lesssim \lambda^\gamma \norm{\Pi}_{\gamma;\bar{\fraK}}
 \normDgamma{f}_{\gamma;\bar{\fraK}}
\end{equation} 
holds uniformly over all test functions $\ph\in B_r$, all $\lambda\in(0,1]$, all 
$f\in\cD^\gamma$  and all $z\in\fraK$. Here $\bar\fraK$ is the $1$-fattening of 
$\fraK$, that is, the set of points at $\norm{\cdot}_\fraks$-distance at most 
$1$ from $\fraK$. 

\noindent
Furthermore, if $\gamma>0$, then $\cR f$ is unique. 
\end{theorem}

If $f(z)$ is written as in~\eqref{eq:f_expansion}, then the term $\Pi_zf(z)$ 
in~\eqref{eq:reconstruction} has the expression 
\begin{equation}
\label{eq:f_expansion_Pi} 
 \Pi_zf(z) = \sum_{\alpha < \gamma} \sum_{\tau\in T_\alpha} f_\tau(z) 
\Pi_z\tau\;,
\end{equation} 
which is in general a distribution. In the particular case where $f$ has 
non-negative regularity, one has in fact 
\begin{equation}
\label{eq:reconstruction_function} 
 \cR f(z) = \bigpar{\Pi_zf(z)}(z)
 = \pscal{\unit}{f(z)} := f_\unit(z)\;,
\end{equation} 
cf.~\cite[Proposition~3.28]{Hairer2014}, which is a generalisation 
of~\eqref{eq:poly_f0}. More generally, if every $\Pi_z\tau$ happens to belong 
to a H\"older space $\cC^\alpha_\fraks$ with $\alpha>0$, then one has 
\begin{equation}
 \cR f(z) = \bigpar{\Pi_zf(z)}(z) 
 = \sum_{\alpha < \gamma} \sum_{\tau\in T_\alpha} f_\tau(z) 
\bigpar{\Pi_z\tau}(z)\;, 
\end{equation} 
see~\cite[Remark~3.15]{Hairer2014}. This holds in particular for the 
canonical model for mollified noise $\Pi^\delta$ constructed in the previous 
section. In general, $\cR f$ and $\Pi_zf(z)$ may differ by a 
remainder term of degree $\gamma$, which is small if $\gamma$ is large. 

The proof of \Cref{thm:reconstruction} given in~\cite{Hairer2014} is based on 
wavelet analysis. There exists by now an alternative proof, presented 
in~\cite{Otto_Weber_2016}, which is based on smoothing properties of the heat 
kernel. 

%%%%%%%%%%%%%%%%%%%%%%%%%%%%%%%%%%%%%%%%%%%%%%%%%%%%%%%%%%%%%%%%%%%%%%%%%%%%%%%%

\subsection{Multilevel Schauder estimates}
\label{ssec:3d_Schauder} 

The last missing piece needed to lift the classical solution map $\overline \cS$ 
to the space of modelled distributions, and defining the map $\cS$ making the 
diagram~\eqref{eq:RS_diagram} commute, is an operator lifting the operation of 
convolution with the heat kernel $P$. Owing to the 
decomposition~\eqref{eq:decomp_P} of $P$ into a singular part $K$ and a smooth 
part $R$, this problem can be decomposed into two separate problems, which are 
to make the following diagrams commute:

\begin{equation}
\label{eq:K_diagram} 
 \begin{tikzcd}[column sep=large, row sep=large]
 \cD^\gamma \arrow[r, "\cK_\gamma"] \arrow[d, "\cR"]
 & \cD^{\gamma+2} \arrow[d, "\cR"] \\
 \cC^{\alpha_0}_\fraks  
 \arrow[r, "K*"'] 
 & \cC^{\alpha_0+2}_\fraks 
\end{tikzcd}
\qquad\qquad
 \begin{tikzcd}[column sep=large, row sep=large]
 \cD^\gamma \arrow[r, "R_\gamma\circ\cR"] \arrow[d, "\cR"]
 & \cD^{\gamma+2} \arrow[d, "\cR"] \\
 \cC^{\alpha_0}_\fraks  
 \arrow[r, "R*"'] \arrow[ru, "R_\gamma"]
 & \cC^{\alpha_0+2}_\fraks 
\end{tikzcd}
\end{equation}
The smooth part $R$ is actually comparatively easy to deal with. Indeed, one 
can define an operator $R_\gamma$, acting on distributions 
$\zeta\in\cC^{\alpha_0}_\fraks$ via the Taylor-series-like expression
\begin{equation}
 \label{eq:R_gamma}
 (R_\gamma \zeta)(z) 
 := 
 \sum_{\abss{k}<\gamma} \frac{\X^k}{k!} 
 \pscal{\zeta}{\DD^k R(z-\cdot)}\;.
\end{equation} 
The fact that the right-hand side belongs to $\cD^{\gamma+2}$ is shown 
in~\cite[Lemma~7.3]{Hairer2014}. Since $R_\gamma\zeta$ has positive regularity, 
it follows from~\eqref{eq:reconstruction_function} that $(\cR R_\gamma 
\zeta)(z)$ is just the convolution of $\zeta$ with $R(z-\cdot)$, so that for 
$\zeta=\cR f$ we have indeed 
\begin{equation}
 \label{eq:R_gamma2}
 \cR R_\gamma \cR f = R * \cR f
\end{equation} 
as required. 

The construction of the operator $\cK_\gamma$ lifting the convolution with the 
singular part $K$ of the heat kernel is substantially more involved. It turns 
out that $\cK_\gamma$ has to be defined as follows. For any $f\in\cD^\gamma$, 
it is composed of three parts, that is  
\begin{equation}
 \label{eq:K_gamma}
 (\cK_\gamma f)(z) := \I f(z) + \cJ(z)f(z) + (\cN_\gamma f)(z)\;. 
\end{equation} 
The first term $\I f(z)$ is simply given by applying the abstract integration 
operator $\I$ to each symbol $\tau$ in the expansion~\eqref{eq:f_expansion} of 
$f(z)$. The second term, which is again associated with recentering, is 
obtained by setting, for each $\tau\in T_\alpha$,
\begin{equation}
\label{eq:def_Jz}  
\cJ(z)\tau := \sum_{\abss{k} < \alpha+2}
\frac{\X^k}{k!} \pscal{\Pi_z\tau}{\DD^k K(z-\cdot)}\;.
\end{equation} 
Note the similarity with~\eqref{eq:f_x}, which is no accident! 
The last part of $\cK_\gamma$ is a nonlocal operator given by 
\begin{equation}
\label{eq:def_Ngamma}  
(\cN_\gamma f)(z) := \sum_{\abss{k} < \gamma+2}
\frac{\X^k}{k!} \pscal{\cR f - \Pi_z f(z)}{\DD^k K(z-\cdot)}\;.
\end{equation}
As before, both operators~\eqref{eq:def_Jz} and~\eqref{eq:def_Ngamma} 
actually are to be interpreted by replacing $K$ by a sum of $K_n$.  
Note that they both have values in $\overline T$, the polynomial
part of the regularity structure. 

With all these objects in place, we have the following far-reaching 
generalisation of the Schauder estimate of \Cref{thm:Schauder}, which is an 
instance of~\cite[Theorem~5.12]{Hairer2014}.

\begin{theorem}[Multilevel Schauder estimate]
\label{thm:multilevel_Schauder} 
If $\gamma+2\notin\N$, then the operator $\cK_\gamma$ defined 
in~\eqref{eq:K_gamma} maps $\cD^\gamma$ into $\cD^{\gamma+2}$, and satisfies 
\begin{equation}
 \label{eq:multilevel_Schauder}
 \cR\cK_\gamma f = K*\cR f
\end{equation} 
for any $f\in\cD^\gamma$. Furthermore, if $\overline Z = (\overline 
\Pi,\overline \Gamma)$ is a second model, then one has 
\begin{equation}
 \seminormff{\cK_\gamma f}{\cK_\gamma \bar f}_{\gamma+2;\fraK}
 \lesssim \seminormff{f}{\bar f}_{\gamma;\bar\fraK}
 + \norm{\Pi-\Pibar}_{\gamma;\bar\fraK}
 + \norm{\Gamma-\Gammabar}_{\gamma;\bar\fraK}
\end{equation} 
for any $\bar f\in\cD^\gamma(\Gammabar)$, where $\bar\fraK$ is the 
$1$-fattening of $\fraK$. 
\end{theorem}

The proof of this result is similar in spirit to the proof we have given of 
\Cref{thm:Schauder}, but is significantly more involved because of the required 
bookkeeping of all the polynomial terms. In fact, bounding components of 
$\cK_\gamma f$ in $T_\alpha$ of non-integer degree $\alpha$ is relatively 
straightforward, and all the difficulty lies in integer values of $\alpha$.

%%%%%%%%%%%%%%%%%%%%%%%%%%%%%%%%%%%%%%%%%%%%%%%%%%%%%%%%%%%%%%%%%%%%%%%%%%%%%%%%

\subsection{Convergence of renormalised models}
\label{ssec:3d_renorm_conv} 

The tools introduced so far suffice to set up a fixed-point equation in a space 
$\cD^\gamma$, associated with the canonical model 
$Z^\delta=(\Pi^\delta,\Gamma^\delta)$ for mollified noise $\xi^\delta$, built in 
\Cref{ssec:3d_models}. In order to be able to take the limit $\delta\searrow0$, 
one has to incorporate the renormalisation procedure. 

We already introduced a group of renormalisation transformations $M^\delta$ in 
\Cref{ssec:3d_renorm_group}, see in particular~\eqref{eq:M-AC}. The $M^\delta$ 
are linear maps from the model space $T$ into itself. It is now necessary to 
translate this to a transformation from the model $Z^\delta$ to a renormalised 
model $\widehat Z^\delta=(\widehat\Pi^\delta,\widehat\Gamma^\delta)$. 

A useful remark here is that the second property in 
\Cref{def:model} of models and~\eqref{eq:Gamma_F} imply that for all 
$z,z'\in\R^4$, 
\begin{equation}
\Pi^\delta_{z}\bigpar{F^\delta_{z}}^{-1}  
= \Pi^\delta_{z'}\bigpar{F^\delta_{z'}}^{-1} 
=: \bPi^\delta\;,
\end{equation} 
where $\bPi^\delta$ is independent of $z$. Therefore, the model can also 
be specified by the pair $(\bPi^\delta,f^\delta_z)$, and 
$(\Pi^\delta,\Gamma^\delta)$ can be recovered via 
\begin{equation}
 \Pi^\delta_z = \bPi^\delta \, \Gamma_{f^\delta_z}\;, \qquad 
 \Gamma^\delta_{zz'} = \Gamma_{f^\delta_z}^{-1} 
\Gamma_{f^\delta_{z'}}\;.
\end{equation} 
The renormalised model associated with the renormalisation transformation 
$M^\delta$ will thus be defined in such a way that  
\begin{equation}
 \widehat\bPi^\delta \tau = \bPi^\delta M^\delta\tau 
 \qquad 
 \forall \tau\in T\;.
\end{equation} 
While it is possible to compute the renormalised model \lq\lq by hand\rq\rq, 
there exists a more efficient way of doing this algebraically, which is 
described in~\cite[Section~8.3]{Hairer2014}. Here we just illustrate the result 
by giving the expressions of $\widehat\Pi^\delta_z\tau$ for the negative-degree 
elements of $\cF_-$:
\begin{align}
\label{eq:PiM-AC-I} 
\widehat\Pi^\delta_z(\RSI) &= \Pi^\delta_z(\RSI)\;, \\
\label{eq:PiM-AC-V} 
\widehat\Pi^\delta_z(\RSV) &= \Pi^\delta_z(\RSI)^2 - C_\delta^{(1)}\;, \\
\label{eq:PiM-AC-W} 
\widehat\Pi^\delta_z(\RSW) &= \Pi^\delta_z(\RSI)^3 - 
3C_\delta^{(1)}\Pi^\delta_z(\RSI)\;, \\
\label{eq:PiM-AC-WV} 
\widehat\Pi^\delta_z(\RSWV) &= \Pi^\delta_z(\RSY)\widehat\Pi^\delta_z(\RSV) - 
C_\delta^{(2)}\;, \\
\label{eq:PiM-AC-WW} 
\widehat\Pi^\delta_z(\RSWW) &= \Pi^\delta_z(\RSIW)\widehat\Pi^\delta_z(\RSV) - 
3 C_\delta^{(2)} \Pi^\delta_z(\RSI) 
+ 3C_\delta^{(1)} \Pi^\delta_z(\RSV X_i) \pscal{f_z}{\J_i\RSI)}\;, \\
%\label{eq:PiM-AC-VW} 
\widehat\Pi^\delta_z(\RSVW) &= \Pi^\delta_z(\RSIW)\Pi^\delta_z(\RSI) - 3 
C_\delta^{(1)} \Pi^\delta_z(\RSVI) 
+ 3C_\delta^{(1)} \Pi^\delta_z(\RSI X_i) \pscal{f_z}{\J_i\RSI}\;.
\label{eq:model_renormalised} 
\end{align}
Note that the first three relations are compatible with Wick renormalisation as 
introduced in~\eqref{eq:Wick_powers_limit}, while the last three relations are 
specific to the three-dimensional case. 

The question that arises now is whether this sequence of models converges to a 
meaningful limit as $\delta\searrow0$. Here the key result is the following 
one, see~\cite[Theorem~10.7]{Hairer2014}. 

\begin{theorem}[Convergence criterion for renormalised models]
\label{thm:convergence_renorm} 
Assume that there exists $\kappa>0$ such that, for every test function $\ph\in 
B_r$, every $z\in\R^4$ and every $\tau\in T$ of negative degree, there exists a 
random variable $\pscal{\widehat \Pi_z\tau}{\ph}$ such that
\begin{align}
\Bigexpec{\abs{\pscal{\widehat \Pi_z\tau}{\cS^\lambda_z\ph}}^2}
&\lesssim \lambda^{2\abss{\tau}+\kappa}\;, \\
\Bigexpec{\abs{\pscal{\widehat \Pi_z\tau - \widehat
\Pi_z^{\delta}\tau}{\cS^\lambda_z\ph}}^2} 
&\lesssim \delta^{2\theta}\lambda^{2\abss{\tau}+\kappa}
 \label{eq:rmodel_conv1} 
\end{align}
holds for some $\theta>0$. Then there exists a unique model $\widehat 
Z=(\widehat\Pi,\widehat\Gamma)$ such that 
\begin{equation}
\Bigexpec{\normDgamma{\widehat Z}^p_{\gamma;\fraK}} \lesssim 1\;, 
\qquad 
\Bigexpec{\seminormff{\widehat Z}{\widehat Z^\delta}^p_{\gamma;\fraK}} \lesssim
\delta^{\theta p}
 \label{eq:rmodel_conv2} 
\end{equation} 
holds for every compact $\fraK\subset\R^4$ and every $p\geqs1$. Here the 
quantities $\normDgamma{Z}_{\gamma;\fraK} = 
\norm{\Pi}_{\gamma;\fraK} + 
\norm{\Gamma}_{\gamma;\fraK}$ and $\smash{\seminormff{Z}{\Zbar}_{\gamma;\fraK} 
= \norm{\Pi-\Pibar}_{\gamma;\fraK} + \norm{\Gamma-\Gammabar}_{\gamma;\fraK}}$
measure the magnitude of models and their difference.
\end{theorem}

To be more precise, each random variable $\pscal{\widehat \Pi_z\tau}{\ph}$ 
should belong to a specific Wiener chaos, namely the $n$th inhomogeneous Wiener 
chaos if $\tau$ contains $n$ symbols $\Xi$. This means that it can be 
written as an integral involving $n$ space-time white noise variables, just as 
in~\eqref{eq:stoch_int_RSWV} which is an instance with $n=4$. 
The conditions~\eqref{eq:rmodel_conv1} can then be reformulated in terms of 
bounds on the Wiener chaos expansion coefficients, 
(cf.~\cite[Proposition~10.11]{Hairer2014}), which can be represented by Feynman 
diagrams similar to those we encountered in \Cref{sec:3d_Gibbs}.

%%%%%%%%%%%%%%%%%%%%%%%%%%%%%%%%%%%%%%%%%%%%%%%%%%%%%%%%%%%%%%%%%%%%%%%%%%%%%%%%

\section{Existence and uniqueness of solutions}
\label{sec:3d_existence} 

We return now to our aim of constructing limits of solutions of regularised 
Allen--Cahn SPDEs 
\begin{equation}
\label{eq:AC-3d-renorm-rep} 
 \partial_t \phi_\delta(t,x) = \Delta \phi_\delta(t,x) + \phi_\delta(t,x) - 
 \phi_\delta(t,x)^3 + \bigbrak{3C_\delta^{(1)} - 9C_\delta^{(2)}} 
\phi_\delta(t,x) 
 + %\sqrt{2\eps} 
 \xi^\delta(t,x)\;,
\end{equation} 
with renormalisation constants $\smash{C_\delta^{(1)}} = 
\Order{\delta^{-1}}$ and $\smash{C_\delta^{(2)}} = \Order{\log(\delta^{-1})}$ 
defined as in~\eqref{eq:def_Cdelta} and~\eqref{eq:def_Cdelta2}. 
One of the main results of~\cite{Hairer2014} is the following local existence 
theorem (which was formulated for the $\Phi^4$ model, but the proof extends 
trivially to the Allen--Cahn equation). 

\begin{theorem}[Local existence of solutions for the three-dimensional 
Allen--Cahn equation] Assume $\phi_0\in\cC^\eta_\fraks$ for some 
$\eta>-\frac23$. Then the Allen--Cahn equation~\eqref{eq:AC-3d-renorm-rep} with 
initial condition $\phi_0$ admits a sequence $\phi_\delta$ of local solutions 
converging in probability, as $\delta\searrow0$, to a limit $\phi$. This limit 
is independent of the choice of mollifier $\varrho$. 
\end{theorem}
\begin{proof}[\Sketch]
We will only outline the main steps of the proof, ignoring some technical 
subtleties, which are discussed in detail in~\cite[Sections~9.4 and 
10.5]{Hairer2014}. 

In a first step, we ignore the counterterms $C_\delta^{(i)}$ 
in~\eqref{eq:AC-3d-renorm-rep}. Given $\gamma>\bar\gamma>0$, consider the 
fixed-point equation  
\begin{equation}
\label{eq:FP_Dgamma} 
 \Phi = P\phi_0 + 
 (\cK_{\bar\gamma} + \cR R_\gamma\cR) \bigbrak{\Xi + \Phi - \Phi^3}\;,
\end{equation} 
where we make a slight abuse of notation by writing $P\phi_0$ also for 
the lift of $P\phi_0$ to $\cD^\gamma$, defined as in~\eqref{eq:R_gamma}. First 
we want to show that solving~\eqref{eq:FP_Dgamma} is indeed equivalent 
to solving the regularised equation~\eqref{eq:AC-3d-renorm-rep} without the 
counterterms. Applying the reconstruction operator $\cR$ for the canonical 
model $(\Pi^\delta,\Gamma^\delta)$ to both sides of~\eqref{eq:FP_Dgamma} and 
using~\eqref{eq:multilevel_Schauder} and~\eqref{eq:R_gamma2}, we obtain 
\begin{equation}
 \phi := \cR\Phi = P\phi_0 + P*\cR\bigbrak{\Xi+\Phi-\Phi^3}\;.
\end{equation} 
Since we work with the canonical model, we can identify $P*\cR\Xi$ with 
the stochastic convolution $P*\xi^\delta$. It thus remains to show that 
$\cR(\Phi^3)$ is equal to $(\cR\Phi)^3=\phi^3$. To this end, we 
rewrite~\eqref{eq:FP_Dgamma}, by separating non-polynomial from polynomial 
terms, as 
\begin{equation}
\label{eq:FP_Dgamma2} 
 \Phi = \I \bigbrak{\Xi + \Phi - \Phi^3} + \sum_k \Phi_{\X^k}\X^k\;.
\end{equation} 
Iterating this map, it is not hard to see that any fixed point 
of~\eqref{eq:FP_Dgamma2} should have the expression 
\begin{equation}
\label{eq:Phi_expansion} 
 \Phi(z) = \RSI + \Phi_\unit(z)\unit - \RSIW - 3 \Phi_\unit(z)\RSY 
 + \Phi_{\X_i}(z)\X_i + \dots 
\end{equation} 
and thus 
\begin{equation}
\label{eq:Phi_R} 
 (\cR\Phi)(z) = \Pi^\delta_z \RSI(z) + \Phi_\unit(z)\;,
\end{equation} 
since all other terms in~\eqref{eq:Phi_expansion} have strictly positive 
degree, cf.~\eqref{eq:reconstruction_function}. Furthermore, we have
\begin{equation}
 \Phi(z)^3 = \RSW + 3\Phi_\unit(z)\RSV + 3\Phi_\unit(z)^2\RSI 
 - 3 \RSWW - 9\Phi_\unit(z)\RSWV + 3\Phi_{\X_i}(z)\RSV\X_i + 
\Phi_\unit(z)^3\unit + \dots\;, 
\end{equation} 
where the dots stand for terms of strictly positive degree. Applying the 
reconstruction operator and comparing to the cube of~\eqref{eq:Phi_R}, we see 
that $\cR(\Phi^3)-(\cR\Phi)^3$ is a linear combination of 
\begin{equation}
 \bigpar{\Pi^\delta_z \RSWW}(z)\;, \qquad 
 \bigpar{\Pi^\delta_z \RSWV}(z)\;, \qquad\text{and}\qquad  
 \bigpar{\Pi^\delta_z \RSV\X_i}(z)\;.
\end{equation} 
However, all these terms are equal to zero. For instance, $\bigpar{\Pi^\delta_z 
\RSWW}(z) = \bigpar{\Pi^\delta_z \RSIW}(z)\bigpar{\Pi^\delta_z \RSV}(z)$, where 
the first term vanishes as shown in \Cref{exo:model}. We conclude that indeed 
$\phi=\cR\Phi$ solves the regularised equation without counterterms. 

The second step is to show that~\eqref{eq:FP_Dgamma} admits indeed a unique 
fixed point. First we determine admissible values of $\gamma$ and $\bar\gamma$.
Assume that $\Phi\in\cD^\gamma$ has regularity $\alpha\leqs0$. Then 
\Cref{thm:mult_Dgamma} shows that $\Phi^3\in\cD^{\gamma+2\alpha}$, and has 
regularity $3\alpha$. The multilevel Schauder estimate of 
\Cref{thm:multilevel_Schauder} shows that the right-hand side 
of~\eqref{eq:FP_Dgamma} is in $\cD^{\gamma+2\alpha+2}$ and has regularity 
$(-\frac12-\kappa)\wedge(3\alpha+2)$, provided $\bar\gamma\geqs\gamma+2\alpha$. 
We can thus choose $\alpha=-\frac12-\kappa$, which is the regularity of the 
stochastic convolution, $\gamma>1+2\kappa$ and $\bar\gamma=\gamma-1-2\kappa$, 
to ensure that both sides of~\eqref{eq:FP_Dgamma} belong to the same spaces. 

Here we encounter a new difficulty, which is that the term $P\phi_0$ is not 
necessarily small in $\cD^\gamma$ near $t=0$. The solution is to modify 
\Cref{def:D_gamma} of the spaces of modelled distributions $\cD^\gamma$ in order 
to allow coefficients that diverge near $t=0$ in a way controlled by a power 
$\eta$. This yields new spaces $\cD^{\gamma,\eta}$ with similar results on 
multiplication, reconstruction and Schauder estimates as for $\cD^\gamma$, 
discussed in~\cite[Section~6]{Hairer2014}. By slightly changing the parameter 
$\eta$, one can also obtain bounds which are small for small times, as required 
to show that the fixed-point map is a contraction for small enough times. 

The third step is to incorporate the counterterms $C_\delta^{(i)}$ and 
take the limit $\delta\searrow0$. This implies that we replace the commutative 
diagram~\eqref{eq:RS_diagram} by
\begin{equation}
\label{eq:RS_diagram_renorm} 
 \begin{tikzcd}[column sep=large, row sep=large]
 (\phi_0,\widehat Z^\delta) \arrow[r, "\cS_{\,\,M^\delta}", mapsto] 
 & \Phi_{M^\delta} \arrow[d, "\widehat\cR", mapsto] \\
 (\phi_0,\xi^\delta) \arrow[u, "\Psi\circ M^\delta", mapsto] 
 \arrow[r, "\overline\cS_{\!\!\!M^\delta}"', mapsto] 
 & \phi_\delta 
\end{tikzcd}
\end{equation}
where $\overline\cS_{\!\!\!M^\delta}$ is the solution map of the renormalised 
equation~\eqref{eq:AC-3d-renorm-rep}, $\widehat Z^\delta = Z^\delta M^\delta$ is 
the renormalised model described in~\eqref{eq:model_renormalised}, and 
$\smash{\widehat\cR}$ is the corresponding reconstruction operator. The map 
$\cS_{\,\,M^\delta}$ is defined by solving the fixed-point 
equation~\eqref{eq:FP_Dgamma}, with $\cR$ replaced by $\smash{\widehat\cR}$. 
Then a similar computation as above, taking into account the additional terms in 
the model~\eqref{eq:model_renormalised}, shows that the 
diagram~\eqref{eq:RS_diagram_renorm} indeed commutes. 

Finally, we have to check that the renormalised model satisfies the 
convergence conditions in \Cref{thm:convergence_renorm}. This is done 
in~\cite[Section~10.5]{Hairer2014} via a lengthy computation involving Feynman 
diagrams. However, recent results 
in~\cite{ChandraHairer16,BrunedHairerZambotti,Bruned_Chandra_Chevyrev_Hairer_18}
based on so-called BPHZ renormalisation allow to make the procedure much more 
systematic. 
\end{proof}

\begin{remark}
One of the technicalities we have glossed over is that the term in brackets on 
the right-hand side of~\eqref{eq:FP_Dgamma} has to be multiplied by an 
indicator function $\indexfct{t>0}$. The resulting objects have then to be 
shown to belong to the right spaces of distributions. 
\end{remark}

%%%%%%%%%%%%%%%%%%%%%%%%%%%%%%%%%%%%%%%%%%%%%%%%%%%%%%%%%%%%%%%%%%%%%%%%%%%%%%%%

\section{Large deviations}
\label{sec:3d_LDP} 

We finally consider the weak-noise variant of~\eqref{eq:AC-3d-renorm-rep}, 
given by 
\begin{equation}
\label{eq:AC-3d-regeps} 
 \partial_t \phi(t,x) = \Delta \phi(t,x) + \phi(t,x) - 
 \phi(t,x)^3 + \bigbrak{(2\eps) 3 C_\delta^{(1)} - (2\eps)^2 9C_\delta^{(2)}} 
\phi(t,x) 
 + \sqrt{2\eps} \xi^\delta(t,x)\;.
\end{equation} 
The work~\cite{HairerWeber} also contains the following three-dimensional 
analogue of \Cref{thm:LDP_2d}.

\begin{theorem}[Large-deviation principle for the renormalised equation]
\label{thm:LDP_3d} 
Let $\phi_{\delta,\eps}$ denote the solution of~\eqref{eq:AC-3d-regeps}, and 
let $\eps\mapsto\delta(\eps)\geqs0$ be a function such that
\begin{equation}
 \label{eq:lim_delta3}
 \lim_{\eps\to0} \delta(\eps) = 0\;.
\end{equation} 
Then the family $(\phi_{\delta(\eps),\eps})_{\eps>0}$ with fixed initial 
condition $\phi_0$ satisfies an LDP on $[0,T]$ with good rate function 
\begin{equation}
\label{eq:lpd_AC-3d} 
\cI_{[0,T]}(\gamma) := 
 \begin{cases}
 \displaystyle
 \frac12 \int_0^T \int_\Lambda \biggbrak{\dpar{}{t} \gamma(t,x) - 
\Delta\gamma(t,x) - \gamma(t,x) + \gamma(t,x)^3}^2 \6x\6t
 & \text{if the integral is finite\;, } \\
 + \infty 
 & \text{otherwise\;.}
 \end{cases}
\end{equation}
\end{theorem}

Note that just as in the two-dimensional case, the counterterms are absent from 
the large-deviation rate function. 

As in \Cref{sec:2d_ldp}, for $\delta>0$, we may also consider the variant 
of~\eqref{eq:AC-3d-regeps} without counterterms 
\begin{equation}
 \partial_t \phi(t,x) 
 = \Delta\phi(t,x) + C\phi(t,x) - \phi(t,x)^3 
 + \sqrt{2\eps}\xi^\delta(t,x)\;,
\label{eq:AC-3d-noC} 
\end{equation}
which does not admit a limit as $\delta\searrow0$. 
Here, the large-deviation result reads as follows. 

\begin{theorem}[Large-deviation principle for equation without renormalisation]
Let $\tilde\phi_{\delta,\eps}$ denote the solution of~\eqref{eq:AC-3d-noC}, and 
let $\eps\mapsto\delta(\eps)\geqs0$ be a function 
satisfying~\eqref{eq:lim_delta3} as well as 
\begin{equation}
 \lim_{\eps\to0} \eps \delta(\eps)^{-1} = 
 %\frac{\eps}{\delta(\eps)} = 
\lambda\in[0,\infty)\;.
\end{equation} 
Then the family $(\tilde\phi_{\delta(\eps),\eps})_{\eps>0}$ with fixed initial 
condition $\phi_0$ satisfies an LDP on $[0,T]$ with good rate function 
\begin{equation}
\cI_{[0,T]}(\gamma) := 
 \begin{cases}
 \displaystyle
 \frac12 \int_0^T \int_\Lambda \biggbrak{\dpar{}{t} \gamma(t,x) - 
\Delta\gamma(t,x) + C^{(\lambda)} \gamma(t,x) + \gamma(t,x)^3}^2 \6x\6t
 & \text{if the integral is finite\;, } \\
 + \infty 
 & \text{otherwise\;,}
 \end{cases}
\end{equation}
where $C^{(\lambda)} = C - 3\lambda^2 C_\delta^{(1)}$.
\end{theorem}

%%%%%%%%%%%%%%%%%%%%%%%%%%%%%%%%%%%%%%%%%%%%%%%%%%%%%%%%%%%%%%%%%%%%%%%%%%%%%%%%

\section{Bibliographical notes}
\label{sec:3dbib} 

Perturbation theory for the Gibbs measure of the $\Phi^4$ model in three 
dimensions is an important problem in Euclidean Quantum Field theory that has 
been studied for a long time with various methods. The earliest works by Glimm 
and Jaffe and by Feldman approached the problem via a detailed combinatorial 
analysis of Feynman 
diagrams~\cite{Glimm_Jaffe_68,Glimm_Jaffe_73,Feldman74,Glimm_Jaffe_81}. The 
works~\cite{BCGNOPS78,BCGNOPS_80} introduced the idea of using a renormalisation 
group approach, consisting in a decomposition of the covariance of the Gaussian 
field into scales, which then allows to integrate sucessively over one scale 
after the other. This method was further perfected 
in~\cite{Brydges_Dimock_Hurd_CMP_95}, using polymers to control error terms, an 
approach based on ideas from Statistical Physics~\cite{Gruber_Kunz_71}. 

In another direction, the approach provided 
in~\cite{Brydges_Frohlich_Sokal_CPM83,Brydges_Frohlich_Sokal_CMP83_RW} allows 
to bound correlation functions without having to compute the partition 
function explicitly, by using it as a generating function. This involves the 
derivation of skeleton inequalities, which were obtained up to third order 
in~\cite{Brydges_Frohlich_Sokal_CPM83}, and later extended to all orders 
in~\cite{Bovier_Felder_CMP84}. A relatively compact 
derivation of bounds on the partition function 
based on the Bou\'e--Dupuis formula was recently 
obtained in~\cite{Barashkov_Gubinelli_18}. 

The solution theory for the three-dimensional $\Phi^4$ or Allen--Cahn equation 
we presented is based on the theory of regularity structures developed 
in~\cite{Hairer2014}. Alternative approaches include paracontrolled 
calculus~\cite{Gubinelli_Imkeller_Perkowski_15,Catellier_Chouk_18} and a 
Wilsonian renormalisation group~\cite{Kupiainen_16}. The large-deviation 
results are from~\cite{HairerWeber}. 

Regularity structures have been successfully applied to several other 
equations, including the KPZ equation~\cite{Hairer_KPZ,Hairer_Shen_17} and its 
generalisations to polynomial nonlinearities~\cite{Hairer_Quastel_18}, 
the continuum parabolic Anderson model~\cite{Hairer_Labbe_15}, 
the Navier--Stokes equation~\cite{ZhuZhu_2015}, 
the motion of a random string on a curved surface~\cite{Hairer_random_string},
the FitzHugh--Nagumo SPDE~\cite{Berglund_Kuehn_15}, 
the dynamical Sine--Gordon model~\cite{Hairer_Shen_16},
the heat equation driven by space-time fractional noise~\cite{Deya_2016},
reaction-diffu\-sion equations with a fractional Laplacian~\cite{BK2017}, and 
the multiplicative stochastic heat equation~\cite{Hairer_Labbe_18}. 
Recent progress on the theory includes a systematic approach to 
renormalisation~\cite{ChandraHairer16,BrunedHairerZambotti,
Bruned_Chandra_Chevyrev_Hairer_18} based on 
BPHZ-renormalisation, see 
also~\cite{Hairer_BPHZ,Bruned_Gabriel_Hairer_Zambotti_2019}. 

%%%%%%%%%%%%%%%%%%%%%%%%%%%%%%%%%%%%%%%%%%%%%%%%%%%%%%%%%%%%%%%%%%%%%%%%%%%%%%%%

\appendix 

\chapter{Potential theory for reversible Markov chains}
\label{ch:app_markov} 

%%%%%%%%%%%%%%%%%%%%%%%%%%%%%%%%%%%%%%%%%%%%%%%%%%%%%%%%%%%%%%%%%%%%%%%%%%%%%%%%

\section{First-hitting time and capacity}

Consider a reversible Markov chain on a finite set $\XX$, with transition 
probabilities $p(x,y)$ and invariant probability measure $\pi$. Reversibility 
means that 
\begin{equation}
 \pi(x) p(x,y) = \pi(y) p(y,x) 
 \qquad 
 \forall x, y\in\XX\;,
\end{equation} 
which is equivalent to the transition matrix $P=(p(x,y)_{x,y\in\XX})$ being 
self-adjoint with respect to the $L^2(\XX,\pi)$ inner product 
\begin{equation}
 \pscal{f}{g}_\pi := \sum_{x\in\XX} \pi(x) f(x) g(x)\;.
\end{equation} 
The \emph{generator} of the chain is given by 
\begin{equation}
 (Lf)(x) := \sum_{y\in\XX} p(x,y) \brak{f(y) - f(x)}\;,
\end{equation}
and is also self-adjoint in $L^2(\XX,\pi)$. 

Fix a nonempty set $A\subset \XX$ and let 
\begin{equation}
 w_A(x) := \expecin{x}{\tau_A}\;, 
 \qquad 
 \tau_A = \inf\setsuch{n\geqs0}{X_n\in A}
\end{equation} 
be the expected first-hitting time of $A$. It satisfies the Poisson boundary 
value problem 
\begin{equation}
\begin{cases}
 (Lw_A)(x) = -1 &\qquad x\in A^c\;, \\
 w_A(x) = 0 &\qquad x\in A\;.
\end{cases}
\end{equation} 
The solution of this problem can be written as 
\begin{equation}
 w_A(x) = - \sum_{y\in A^c} G_{A^c}(x,y)\;,
\end{equation} 
where the Green function (or fundamental matrix) $G_{A^c} = L_{A^c}^{-1}$ is 
the inverse of the restriction of the generator $L$ to $A^c$. 

Fix now two disjoint nonempty sets $A, B\subset \XX$ and let 
\begin{equation}
 h_{AB}(x) := \bigprobin{x}{\tau_A < \tau_B} 
\end{equation} 
be the \emph{committor function} (or \emph{equilibrium potential}), where 
$\tau_A$ is the first-hitting time of $A$. It satisfies the Dirichlet boundary 
value 
problem 
\begin{equation}
\begin{cases}
 (Lh_{AB})(x) = 0 &\qquad x\in (A\cup B)^c\;, \\
 h_{AB}(x) = 1 &\qquad x\in A\;, \\
 h_{AB}(x) = 0 &\qquad x\in B\;.
\end{cases}
\end{equation} 
The \emph{equilibrium measure} is defined by 
\begin{equation}
 e_{AB}(x) := -(Lh_{AB})(x) 
 \qquad 
 \forall x\in A\;.
\end{equation} 
This measure also has a probabilistic interpretation, namely 
\begin{equation}
 e_{AB}(x) = \bigprobin{x}{\tau^+_A < \tau^+_B}\;,
\end{equation} 
where $\tau^+_A = \inf\setsuch{n\geqs1}{X_n\in A}$ denotes the first-return 
time to $A$.

\begin{proposition}
The committor has the representation
\begin{equation}
 h_{AB}(x) = -\sum_{y\in A} G_{B^c}(x,y) e_{AB}(y) 
 \qquad 
 \forall x\in(A\cup B)^c\;.
\end{equation} 
\end{proposition}
\begin{proof}
The restriction of $Lh_{AB}$ to $B^c$ satisfies 
\begin{equation}
 L_{B^c} 
 \begin{pmatrix}
 h_{AB} \\ 1 
 \end{pmatrix}
 = 
 \begin{pmatrix}
 0 \\ -e_{AB}
 \end{pmatrix}
\end{equation} 
where the upper part of the vectors represents elements of $(A\cup B)^c$, and 
the lower part elements of $A$. Multiplying on the left by $G_{B^c} = 
L_{B^c}^{-1}$ yields the result. 
\end{proof}

The \emph{capacity} is defined by 
\begin{equation}
\label{eq:def_cap} 
 \capacity(A,B) := \sum_{x\in A} \pi(x) e_{AB}(x)\;. 
\end{equation} 
As a consequence, $\nu_{AB}(x) = \pi(x)e_{AB}(x)/\capacity(A,B)$ defines a 
probability measure on $A$. 

\begin{lemma}
\label{lem:cap} 
The capacity admits the two equivalent representations 
\begin{align}
\capacity(A,B) &= \pscal{h_{AB}}{-Lh_{AB}}_\pi\;, \\
\capacity(A,B) &= \frac12 \sum_{x,y\in\XX} \pi(x) p(x,y) \bigbrak{h_{AB}(x) - 
h_{AB}(y)}^2\;.
\end{align}
\end{lemma}
\begin{proof}
For the first identity, we write 
\begin{equation}
 \pscal{h_{AB}}{-Lh_{AB}}_\pi 
 = \sum_{x\in A\cup B\cup(A\cup B)^c} \pi(x) h_{AB}(x) (-Lh_{AB})(x)\;.
\end{equation} 
Since $h_{AB}(x)=1$ in $A$, while $h_{AB}(x)=0$ in $B$ and $Lh_{AB}(x)=0$ in 
$(A\cup B)^c$, we recover~\eqref{eq:def_cap}. 
The second identity is a consequence of summation by parts (or discrete Green's 
identity). Indeed, 
\begin{align}
\pscal{h_{AB}}{-Lh_{AB}}_\pi 
&= \sum_{x,y\in\XX} \pi(x) p(x,y) h_{AB}(x) \bigbrak{h_{AB}(x)-h_{AB}(y)} \\
&= \sum_{x,y\in\XX} \pi(x) p(x,y) \bigbrak{\frac12 h_{AB}(x)^2 - 
h_{AB}(x)h_{AB}(y) + \frac12 h_{AB}(x)^2} \\
&= \sum_{x,y\in\XX} \pi(x) p(x,y) \bigbrak{\frac12 h_{AB}(x)^2 - 
h_{AB}(x)h_{AB}(y) + \frac12 h_{AB}(y)^2}\;,
\end{align}
where we have used reversibility and a permutation of summation indices to 
replace $x$ by $y$ in the last line. 
\end{proof}

The main result making the potential-theoretic approach work is the following 
link between expected hitting times and capacity.

\begin{theorem}
For any disjoint sets $A,B\subset\XX$, one has 
\begin{equation}
\expecin{\nu_{AB}}{\tau_B} := 
\displaystyle \sum_{x\in A} \nu_{AB}(x)w_B(x)
= \frac{1}{\capacity(A,B)} \sum_{x\in B^c} \pi(x) h_{AB}(x)\;.
\end{equation} 
\end{theorem}
\begin{proof}
We have the identity
\begin{align}
\sum_{x\in A}\pi(x) e_{AB}(x)w_B(x) 
&= -\sum_{x\in A}\sum_{y\in B^c} \pi(x)G_{B^c}(x,y)e_{AB}(x) \\
&= -\sum_{y\in B^c}\sum_{x\in A} \pi(y)G_{B^c}(y,x)e_{AB}(x) 
= \sum_{y\in B^c} \pi(y) h_{AB}(y)\;.
\end{align}
Dividing on both sides by the capacity yields the result. 
\end{proof}

A more probabilistic proof of this result can be found in~\cite{Slowik_12}. 

%%%%%%%%%%%%%%%%%%%%%%%%%%%%%%%%%%%%%%%%%%%%%%%%%%%%%%%%%%%%%%%%%%%%%%%%%%%%%%

\section{Dirichlet principle}

Define the \emph{Dirichlet form} (or energy) by 
\begin{equation}
 \cE(f) := \pscal{f}{-Lf}_\pi
 = \frac12 \sum_{x,y\in\XX} \pi(x) p(x,y) \bigbrak{f(x) - f(y)}^2\;.
\end{equation} 
We associate with it the bilinear form 
\begin{equation}
\label{eq:MC-Dirichlet} 
 \cE(f,g) := \frac12 \sum_{x,y\in\XX} \pi(x) p(x,y) 
 \bigbrak{f(x) - f(y)}\bigbrak{g(x) - g(y)}\;.
\end{equation} 
Note that the polarisation identity and self-adjointness yield 
\begin{equation}
 \cE(f,g) = \frac14 \bigpar{\cE(f+g) - \cE(f-g)}
 = \frac12 \bigpar{\pscal{f}{-Lg}_\pi + \pscal{g}{-Lf}_\pi}
 =  \pscal{f}{-Lg}_\pi\;.
\end{equation} 
By the Cauchy--Schwarz inequality (and reversibility), we have 
\begin{equation}
\label{eq:Cauchy-Schwarz} 
 \cE(f,g)^2 \leqs \cE(f) \cE(g)\;.
\end{equation} 
Furthermore, Lemma~\ref{lem:cap} shows that 
\begin{equation}
 \capacity(A,B) = \cE(h_{AB})\;.
\end{equation} 

\begin{proposition}[Dirichlet principle]
\label{prop:Dirichlet}
We have 
\begin{equation}
 \capacity(A,B) = \min_{h\in\cH_{AB}} \cE(h)\;,
\end{equation} 
where $\cH_{AB} = \setsuch{h:\XX\to[0,1]}{h\vert_A=1, h\vert_B=0}$.
The minimum is reached in $h=h_{AB}$. 
\end{proposition}
\begin{proof}
Pick any $h\in\cH_{AB}$ and observe that 
\begin{align}
 \cE(h,h_{AB}) 
 = \pscal{h}{-Lh_{AB}}_\pi
 &= \sum_{x\in\XX} \pi(x) h(x) (-Lh_{AB})(x) \\
 &= \sum_{x\in A} \pi(x) (-Lh_{AB})(x) \\
 &= \sum_{x\in A} \pi(x) e_{AB}(x) \\
 &= \capacity(A,B)\;,
\end{align} 
where we have used the boundary conditions for $h$ and the fact that $h_{AB}$ 
vanishes in $(A\cup B)^c$ to obtain the second line. 
The Cauchy--Schwarz inequality~\eqref{eq:Cauchy-Schwarz} yields 
\begin{equation}
 \capacity(A,B)^2 
 = \cE(h,h_{AB})^2
 \leqs \cE(h)\cE(h_{AB}) 
 = \cE(h)\capacity(A,B)
\end{equation} 
and hence $\capacity(A,B) \leqs \cE(h)$. We already know that equality is 
reached for $h=h_{AB}$. 
\end{proof}

%%%%%%%%%%%%%%%%%%%%%%%%%%%%%%%%%%%%%%%%%%%%%%%%%%%%%%%%%%%%%%%%%%%%%%%%%%%%%%

\section{Thomson principle}

The \emph{Thomson principle} provides a complementary variational principle for 
the capacity. It comes in two versions.

\begin{proposition}[Thomson principle, super-harmonic version]
\label{prop:thomson1}
We have 
\begin{equation}
\label{eq:thomson1} 
 \capacity(A,B) 
 = \max_{h\in\mathfrak{H}_{AB}} 
 \frac{\bigpar{\sum_{x\in A} \pi(x) (-Lh)(x)}^2}{\cE(h)}\;,
\end{equation} 
where $\mathfrak{H}_{AB} = \setsuch{h:\XX\to[0,1]}{(Lh)(x)\leqs 0 \; \forall 
x\in B^c}$. The maximum is reached for $h=h_{AB}$. 
\end{proposition}

\begin{proof}
Pick $h\in\mathfrak{H}_{AB}$. We have 
\begin{equation}
 \cE(h,h_{AB}) 
 = \pscal{-Lh}{h_{AB}}_\pi 
 = \sum_{x\in\XX} \pi(x) (-Lh)(x) h_{AB}(x)\;.
\end{equation} 
Since $\pi$, $-Lh$ and $h_{AB}$ are non-negative, we obtain 
\begin{equation}
 \cE(h,h_{AB}) \geqs 
 \sum_{x\in A} \pi(x) (-Lh)(x) \geqs 0\;.
\end{equation} 
It follows from the Cauchy--Schwarz inequality that 
\begin{equation}
 \Bigpar{\sum_{x\in A} \pi(x) (-Lh)(x)}^2 
 \leqs \cE(h,h_{AB})^2 \leqs \cE(h) \capacity(A,B)\;.
\end{equation} 
This proves the inequality, while~\eqref{eq:def_cap} shows that equality is 
again reached for $h=h_{AB}$. 
\end{proof}

The second version involves the notion of (discrete) unit flow.

\begin{definition}[Unit $AB$-flow]
\label{def:ABflow} 
A \defwd{unit $AB$-flow} is a map $\ph:\XX\times\XX\to\R$ satisfying 
\begin{itemize}
\item 	antisymmetry: $\ph(x,y) = -\ph(y,x)$ for all $x,y\in\XX$;
\item 	compatibility with $p$: if $p(x,y) = 0$, then $\ph(x,y)=0$; 
\item 	Kirchhoff's law (divergence-freeness): 
\begin{equation}
\label{eq:Kirchhoff} 
 (\divergence \ph)(x) := \sum_{y\in\XX} \ph(x,y) = 0 
 \qquad \forall x\in (A\cup B)^c\;;
\end{equation} 
\item 	unit intensity: 
\begin{equation}
\label{eq:unit_intensity} 
 \sum_{x\in A} (\divergence \ph)(x) = 1 = - \sum_{x\in B}(\divergence \ph)(x)\;.
\end{equation} 
\end{itemize}
We denote by $\mathfrak{U}^1_{AB}$ the set of unit $AB$-flows. 
\end{definition}

One should think of $\ph$ as being defined on the edges of the graph associated 
with $P$ (i.e., the $(x,y)$ on which $p(x,y)>0$). 

A special role is played by the \defwd{harmonic unit flow} 
\begin{equation}
 \ph_{AB}(x,y) := \frac{\pi(x)p(x,y)}{\capacity(A,B)} 
 \bigbrak{h_{AB}(x) - h_{AB}(y)}\;.
\end{equation} 

\begin{lemma}
\label{lem:harmonic_flow} 
$\ph_{AB}$ is a unit $AB$-flow.  
\end{lemma}
\begin{proof}
$\ph_{AB}$ is clearly antisymmetric and compatible with $p$. It is 
divergence-free, since for all $x\in(A\cup B)^c$, 
\begin{equation}
 (\divergence \ph_{AB})(x)
 = \frac{\pi(x)}{\capacity(A,B)} 
 \sum_{y\in\XX} p(x,y)  \bigbrak{h_{AB}(x) - h_{AB}(y)}
 = \frac{\pi(x)}{\capacity(A,B)} (-Lh_{AB})(x) = 0\;.
\end{equation} 
Finally, it has unit intensity, because 
\begin{align}
 \sum_{x\in A} (\divergence \ph_{AB})(x)
 &= \frac1{\capacity(A,B)} 
 \sum_{x\in A} \pi(x) (-Lh_{AB})(x) \\
 &= \frac1{\capacity(A,B)} 
 \sqrt{\capacity(A,B) \cE(h_{AB})} 
 = 1
\end{align} 
by Proposition~\ref{prop:thomson1}. The analogous statement for $B$ holds by 
a discrete version of the divergence theorem. 
\end{proof}

We now define a bilinear form $\cD$ on the set of unit $AB$-flows by 
\begin{equation}
 \cD(\ph,\psi) := \frac12 \sum_{x,y\in\XX} \frac{1}{\pi(x)p(x,y)} 
\ph(x,y)\psi(x,y)\;,
\end{equation} 
and introduce the shortcut $\cD(\ph,\ph) =: \cD(\ph)$. 
Again, the Cauchy--Schwarz inequality applies, yielding 
\begin{equation}
 \cD(\ph,\psi)^2 \leqs \cD(\ph)\cD(\psi)\;.
\end{equation} 

\begin{proposition}[Thomson principle, unit flow version]
\label{prop:thomson2}
We have 
\begin{equation}
 \capacity(A,B) = \max_{\ph\in\mathfrak{U}^1_{AB}} \frac{1}{\cD(\ph)}\;.
\end{equation} 
The maximum is reached for $\ph=\ph_{AB}$. 
\end{proposition}
\begin{proof}
For a function $h:\XX\to\R$, define the flow 
\begin{equation}
 \Psi_h(x,y) := \pi(x)p(x,y)\bigbrak{h(x)-h(y)}\;.
\end{equation} 
Then $\ph_{AB} = \Psi_{h_{AB}}/\capacity(A,B)$. Thus $\Psi_{h_{AB}}$ is a 
(non-unit) $AB$-flow, and we find  
\begin{align}
 \cD(\Psi_{h_{AB}}) 
 &= \frac12 \sum_{x,y\in\XX} \frac{1}{\pi(x)p(x,y)} \Psi_{h_{AB}}(x,y)^2 \\
 &= \frac12 \sum_{x,y\in\XX} \pi(x)p(x,y) \bigbrak{h_{AB}(x) - h_{AB}(y)}^2 \\
 &= \capacity(A,B)\;.
\end{align}
By bilinearity, we have  
\begin{equation}
 \cD(\ph_{AB}) = \frac{1}{\capacity(A,B)}\;.
\end{equation} 
Furthermore, given any unit $AB$-flow $\ph$, we obtain
\begin{align}
\cD(\Psi_{h_{AB}},\ph)
&= \frac12 \sum_{x,y\in\XX} \bigbrak{h_{AB}(x) - h_{AB}(y)} \ph(x,y) \\
&= \sum_{x,y\in\XX} h_{AB}(x) \ph(x,y) \\
&= \sum_{x\in A} \sum_{y\in\XX} \ph(x,y) + 
\sum_{x\in (A\cup B)^c} h_{AB}(x) \sum_{y\in\XX} \ph(x,y) \\
&= \sum_{x\in A} (\divergence \ph)(x) + 
\sum_{x\in (A\cup B)^c} h_{AB}(x) \underbrace{(\divergence \ph)(x)}_{=0} \\
&= 1\;.
\end{align}
The Cauchy--Schwarz inequality thus yields 
\begin{equation}
 1 =  \cD(\Psi_{h_{AB}},\ph)^2 \leqs \cD(\Psi_{h_{AB}})\cD(\ph) 
 = \capacity(A,B) \cD(\ph)\;,
\end{equation} 
showing that $\capacity(A,B) \geqs 1/\cD(\ph)$ as required. 
\end{proof}

%%%%%%%%%%%%%%%%%%%%%%%%%%%%%%%%%%%%%%%%%%%%%%%%%%%%%%%%%%%%%%%%%%%%%%%%%%%%%%%%

\section{Bibliographical notes}
\label{sec:appendixA_bib} 

A comprehensive account of the potential-theoretic approach to reversible 
Markov chains is given in~\cite[Sections~7.1 and 
7.3]{Bovier_denHollander_book}. An overview can also be found 
in~\cite{Slowik_12}. 

Let us mention that there exists an improved version of the Thomson principle, 
called the \emph{Berman--Konsowa principle}~\cite{Berman_Konsowa_90}. The fact 
that the Berman--Konsowa principle is sharper than the Thomson principle is a 
consequence of Jensen's inequality, as shown in~\cite{Slowik_thesis}. An 
extension of the Berman--Konsowa principle to continuous-space Markov chains 
has been obtained in~\cite{denHollander_Jansen_16}. 

%%%%%%%%%%%%%%%%%%%%%%%%%%%%%%%%%%%%%%%%%%%%%%%%%%%%%%%%%%%%%%%%%%%%%%%%%%%%%%%%

%\listoftheorems[ignoreall,show=theorem]

% \bibliographystyle{plain}
\cleardoublepage
\phantomsection
\bibliographystyle{alpha}
\addcontentsline{toc}{chapter}{Bibliography}
\bibliography{Sarajevo}

\vfill

\bigskip\bigskip\noindent
{\small
Nils Berglund \\
Institut Denis Poisson (IDP) \\ 
Universit\'e d'Orl\'eans, Universit\'e de Tours, CNRS -- UMR 7013 \\
B\^atiment de Math\'ematiques, B.P. 6759\\
45067~Orl\'eans Cedex 2, France \\
{\it E-mail address: }
{\tt nils.berglund@univ-orleans.fr} \\
{\tt https://www.idpoisson.fr/berglund} 

\end{document}